\documentclass[12pt,letterpaper]{article}

\usepackage{amsmath,amsthm,amssymb,euscript}
\usepackage{graphicx,psfrag,epsfig, color, url, verbatim, eufrak}

\usepackage{hyperref}
\theoremstyle{plain}
\newtheorem{theorem}{Theorem}%[section]
\newtheorem{lemma}{Lemma}
\newtheorem{corollary}{Corollary}
\newtheorem{proposition}{Proposition}
\newtheorem{keytheorem}{Key Theorem}
\theoremstyle{definition}
\newtheorem{definition}[theorem]{Definition}
\newtheorem{remark}[theorem]{Remark}

\def\bdef{\begin{definition}}
\def\endef{\end{definition}}
\def\bthm{\begin{theorem}}
\def\ethm{\end{theorem}}
\def\blm{\begin{lemma}}
\def\elm{\end{lemma}}
\def\brm{\begin{remark}}
\def\erm{\end{remark}}
\def\bprop{\begin{proposition}}
\def\eprop{\end{proposition}}
\def\bcor{\begin{corollary}}
\def\ecor{\end{corollary}}
\def\beq{\begin{eqnarray}}
\def\eeq{\end{eqnarray}}
\def\beal{\begin{aligned}}
\def\enal{\end{aligned}}
\def\be{\begin{eqnarray}}
\def\ee{\end{eqnarray}}
\def\beal{\begin{aligned}}
\def\enal{\end{aligned}}

\def\dps{\displaystyle}
\def\om{\omega}
\def\al{\alpha}
\def\bt{\beta}
\def\sg{\sigma}
\def\eps{\varepsilon}

\def\phi{\varphi}

\def\R{\mathbb R}
\def\RR{\mathbb R}
\def\C{\mathbb C}
\def\DD{\mathbb D}

\def\T{\mathbb T}
\def\Q{\mathbb Q}
\def\Z{\mathbb Z}

\def\bI{\mathbf I}
\def\ZZ{\mathbb Z}
\def\V{\mathbb V}
\def\NN{\mathbb N}
\def\SS{\mathbb S}
\def\PP{\mathbb P}
\def\L{\mathcal L}
\def\LL{\mathcal L}

\def\CCC{\mathcal C}

\def\cS{\mathcal S}
\def\bS{\mathbf S}
\def\DDD{\mathcal D}
\def\cD{\mathcal D}
\def\NNN{\mathcal N}
\def\ZZZ{\mathcal Z}
\def\UU{\mathcal U}
\def\MM{\mathcal M}
\def\cF{\mathcal F}
\def\FF{\mathcal F}
\def\VV{\mathcal V}
\def\JJ{\mathcal J}
\def\II{\mathcal I}
\def\RRR{\mathcal R}
\def\AAA{\mathcal A}
\def\TT{\mathbb T}
\def\TTT{\mathcal T}

\def\OO{\mathcal O}
\def\KK{\mathcal K}
\def \SSS{\mathcal S}

\def\F{\mathcal F}
\def\D{\mathcal D}
\def\B{\mathcal B}

\def\cG{\mathcal G}
\def\HH{\mathcal H}
\def\~{\tilde}
\def\gm{\gamma}
\def\ga{\gamma}
\def\de{\delta}
\def\ii{^{-1}}
\def\rr{\rho}

\def\la{\lambda}
\def\Lb{\Lambda}
\def\Om{\Omega}
\def\Gm{\Gamma}
\def\Ga{\Gamma}
\def\th{\theta}
\def\dt{\delta}
\def\lb{\lambda}
\def\wt{\widetilde}
\def\ol{\overline}
\def\wh{\widehat}
\def\kk{\kappa}
\def\eps{\varepsilon}
\def\pa{\partial}
\def\Vor{\mathrm{Vor}}
\def\Leb{\mathrm{Leb}}
\def\mol{\mathrm{mol}}
\def\sr{\mathrm{sr}}

\def\dr{\mathrm{dr}}
\def\shad{\mathrm{shad}}
\def\dist{\mathrm{dist}}
\def\Pos{\mathrm{P\ddot{o}s}}
\def\KAM{\mathrm{KAM}}
\def\Id{\mathrm{Id}}
\def\loc{\mathrm{loc}}
\def\glob{\mathrm{glob}}
\def\Cr{\mathrm{Cr}}
\def\Cor{\mathrm{Cor}}
\def\Tr{\mathrm{Sr}}
\def\sc{\mathrm{sc}}
\def\sf{\mathrm{sf}}
\def\trunc{\mathrm{trunc}}

\def\beq {\begin{equation}}
\def\eeq {\end{equation}}
\def\bdef{\begin{definition}}
\def\endef{\end{definition}}
\def\blm{\begin{lemma}}
\def\elm{\end{lemma}}
\def\beal{\begin{aligned}}
\def\enal{\end{aligned}}

\newcommand{\Phg}{\Phi_{\mathrm{glob}}}

\def\vf{\varphi}
\def\Gm{\Gamma}
\def\th{\theta}
\def\bt{\beta}
\def\sg{\sigma}
\def\dt{\delta}
\def\lb{\lambda}
\def\Lb{\Lambda}

\begin{document}
\title{Orbits of nearly integrable systems accumulating to KAM tori}
\author{M. Guardia, V. Kaloshin}

\maketitle

\begin{abstract}
Consider a sufficiently smooth nearly integrable Hamiltonian system of two and a half degrees of freedom in action-angle coordinates
\[
 H_\eps(\varphi,I,t)=H_0(I)+\eps H_1(\varphi,I,t), \qquad \varphi\in\TT^2,\ I\in U\subset \RR^2,\ t\in \TT=\RR/\ZZ.
\]
Kolmogorov-Arnold-Moser Theorem asserts that a set of nearly 
full measure in phase space consists of  three dimensional invariant tori carrying quasiperiodic dynamics.

In this paper we prove that for a class of nearly integrable Hamiltonian systems there is an orbit which contains these KAM tori in its $\om$-limit set. This implies that the closure of the orbit has almost full measure in the phase space. As byproduct, we show that KAM tori are Lyapunov unstable. The proof relies in the recent developments in the study of Arnold diffusion in nearly integrable systems \cite{BernardKZ11, KaloshinZ12}. It is a combination of geometric and variational techniques.
\end{abstract}

\tableofcontents

\section{Introduction}

\subsection{The main result} %time-periodic setting}

Let $U$ be a convex bounded open subset of $\R^n$.
Consider a $\CCC^2$ smooth strictly convex 
Hamiltonian $H_0(I), \ I\in U$, i.e. for some $D>1$ we have 
\begin{equation}\label{def:Convexity:OriginalHam}
D^{-1}\|v\| \le \langle \partial_I^2 H_0(I) v,v\rangle \le D\|v\|\ 
\text{ for any }I \in U \text{ and }v\in \R^n. 
\end{equation}

Fix $r\ge 2$ and consider the space of $\CCC^r$--perturbations: 
$\CCC^r(\T^n \times U \times  \T) \ni H_1(\varphi,I, t).$
Denote the unit sphere with respect to the standard $\CCC^r$ norm,
given by maximum of all partial derivative of order up to $r$, by 
$\SSS^r=\left\{H_1 \in \CCC^r:\ \|H_1\|_{\CCC^r}=1\right\}.$
%For a non-integer $r$ be use the standard H\"older norm (see Section \ref{sec:holder}). 

In this paper we study dynamics of nearly integrable systems 
\be \label{def:Ham:Original:0}
H_\eps(\varphi,I,t)=H_0(I)+\eps H_1(\varphi,I,t).
\ee
Assume that for some $r\ge 2$,  
\begin{equation}\label{def:PropertiesH0}
\|H_0\|_{\CCC^{3r+9}}\le 1,\qquad \|H_1\|_{\CCC^r}\le 1.
\end{equation}
%$\mathcal V$
Consider $\eta>0$, $\tau>0$. A vector $\om\in\RR^n$ is called $(\eta,\tau)$-{\it Diophantine} if 
$|\om \cdot k+k_0| \ge \eta\,|(k,k_0)|^{-n-\tau}\quad
\text{ for each }\quad (k,k_0)\in (\Z^n\setminus \{0\})\times \Z.$
Denote by 
\begin{equation}\label{def:Diophantine}
\DDD_{\eta,\tau}= 
\left\{\omega: |\om \cdot k+k_0| \ge \eta\,|(k,k_0)|^{-n-\tau}\right\}
\end{equation}
this set of frequencies. Denote 
$U'=\nabla H_0(U)$ the set of values of the gradient of $H_0$.
It is a bounded open set in $\R^n$. Let $U'_\eta \subset U'$ 
be the set of points whose $\eta$-neighborhoods belong to $U'$.
Denote by $\Leb$ the Lebesgue measure on $\T^n \times \R^n \times  \T$
and $\DDD^U_{\eta,\tau}:=\DDD_{\eta,\tau}\cap U'_\eta.$

\bthm \label{time-periodic-KAM}
(time-periodic KAM) Let $\eta,\tau>0,\ r > 2n+2\tau+2$. Assume \eqref{def:Convexity:OriginalHam} and  \eqref{def:PropertiesH0}.
Then there exist $\eps_0=\eps_0(H_0,\eta,\tau)>0$ and 
$c_0=c_0(H_0,r,n)>0$ such that for any $\eps$ with $0<\eps<\eps_0$ 
and any $\om \in \DDD^U_{\eta,\tau}$ 
the Hamiltonian $H_0+\eps H_1$ has a $(n+1)$-dimensional (KAM) 
invariant torus $\TTT_\om$ and dynamics restricted to $\TTT_\om$ 
is smoothly conjugate to the constant flow $(\dot \varphi,\dot t)=(\om,1)$ 
on $\TT^{n+1}$.  Moreover, 
\[
 \Leb(\cup_{\om \in \DDD_{\eta,\tau}\cap U'_\eta}\ \TTT_\om)
 > (1 - c_0 \eta)\,\Leb(\T^n \times U'\times  \T).
\]
\ethm 

This theorem was essentially proven by P\"oschel \cite{Poschel82}.
We add the actual derivation in Appendix \ref{sec:KAM-time-periodic}.
One can enlarge the set of KAM tori by lowering $\eta$ and improve 
the lower bound to $1-c_0 \sqrt \eps$, but we do not rely on 
this improvement. In this paper we study the Arnold diffusion phenomenon 
for these systems in the case $n=2$ and consider only the tori with 
frequencies in $\DDD^U_{\eta,\tau}$.  
Denote the set of all such KAM tori by 
%\begin{equation}\label{def:KAMUnion}
$ \KAM^U_{\eta,\tau}:=\cup_{
 \om \in \DDD^U_{\eta,\tau}}\ \TTT_\om.$
%\end{equation}

\vskip 0.1in 

Recall a well known question about ergodic behavior 
of Hamiltonians systems:

\vskip 0.1in 

{\bf The Quasi-ergodic Hypothesis} (Ehrenfest \cite{Ehr}, 
Birkhoff \cite{Birkh})  
{\it A typical Ha\-mil\-tonian has a dense orbit in a typical energy surface.}
%\medskip 
\vskip 0.1in 

\bthm\label{thm:MainResultNonAutonomous} (The First Main Result) 
Let $\eta,\tau>0.$ Then, for any $r\ge 2018$  
and any $\CCC^{3r+9}$ smooth strictly convex Hamiltonian $H_0$ 
there exists a $\CCC^r$ dense set of perturbations 
$D \subset \SSS^r$ such that for any $H_1 \in D$ there is 
$\eps=\eps(\eta,\tau,r,H_0,H_1)\in(0,\eps_0)$ (where $\eps_0$ 
is the constant from Theorem \ref{time-periodic-KAM}) with 
the property that the Hamiltonian $H_0+\eps H_1$ has an orbit 
$(\varphi_\eps(t),I_\eps(t),t)$ accumulating to all KAM tori 
from $\KAM^U_{\eta,\tau}$, i.e.
\[
 \KAM^U_{\eta,\tau} \subset \  
 \overline{\cup_{t\in \R}\ (\varphi_\eps(t),I_\eps(t),t)}.
\]
\ethm 

An autonomous version of this result  can be obtained using 
the standard energy reduction (see e.g. \cite[Sect 45]{Arnold89}). 

This result can be considered {\it a weak form of the quasi-ergodic hypothesis}.

\vskip 0.1in 

 The result presented in this paper does not obtain full dense orbits but orbits dense in a set of large measure, and deals 
with nearly integrable systems only. Note that the quasiergodic hypothesis is not always true. Herman showed a counterexample in $\TT^{2n}\times[-1,1]^2$. He obtained open sets of Hamiltonian systems with a KAM persistent tori of codimension 1 in the energy surface (see \cite{Yoccoz91}).

Theorem \ref{thm:MainResultNonAutonomous} is a considerable improvement of the result from 
\cite{KaloshinZZ09}, where it is constructed a Hamiltonian of 
the form $\frac 12 \sum_{j=1}^3 I^2_j+\eps H_1(\varphi,I)$ 
having an orbit accumulating to a positive measure set of KAM tori in a fixed energy level. 
An example of this form with an orbit accumulating to a fractal 
set of  tori of maximal Hausdorff dimension is constructed in \cite{KaloshinS12}. 

\paragraph{Lyapunov stability of KAM tori} For the union of 
KAM tori $\KAM^U_{\eta,\tau}$ one of the basic 
questions, studied in this paper, is the question of 
{\it Lyapunov stability of these tori}. In particular, we show 
that  

\medskip 

{\it For a class of nearly integrable Hamiltonian systems all KAM tori 
in fixed Diophantine class $\KAM^U_{\eta,\tau}$ are Lyapunov unstable.} 

\medskip 

Lyapunov stability of a KAM torus is closely related to a question 
of Lyapunov stability of a totally elliptic fixed point. 
An example of a $4$-dimensional map with a Lyapunov unstable 
totally elliptic fixed point was constructed by 
P. Le Calvez and R. Douady \cite{LeCalvezDou83}.  For generic 
resonant elliptic points of $4$-dimensional maps instability was 
established in \cite{KaloshinMV04}. A class of examples of 
nearly integrable Hamiltonians having an unstable KAM torus 
was recently obtained by J. Zhang and C.-Q. Cheng \cite{ChengZ13}. 

\medskip 

Douady \cite{Douady88} proved that 
the stability or instability property of a totally elliptic point is 
a flat phenomenon for $\CCC^\infty$ mappings. Namely, if 
a $\CCC^\infty$ symplectic mapping $f_0$ satisfies certain 
nondegeneracy hypotheses, then there are two mapping $f$ and $g$ 
such that

$\bullet$ $f_0-f$ and $f_0-g$ are flat mappings at 
the origin and

$\bullet$ the origin is Lyapunov unstable for $f$ and 
Lyapunov stable for $g$.

This shows that Lyapunov stability is {\bf not} an open property. 
Thus, the only chance to have robustness in the Main Theorem
is to support perturbations of $H_1$ away from KAM tori.
Moreover, the closer we approach to KAM tori the smaller the 
size of those perturbations has to be. This naturally leads us to 
a  Whitney topology relative to the union of KAM tori.

\subsection{Whitney KAM topology and an improvement 
of the main result}

Denote $\CCC^s(\T^3 \times B^2,\KAM^U_{\eta,\tau})$ --- 
the space of $\CCC^s$ functions with the natural $\CCC^s$-topology 
such that they tend to $0$ as $(\phi,I)$ approaches 
$\KAM^U_{\eta,\tau}$ inside the complement.

Let $s$ be a positive integer. Let $M$ be one of
$U\setminus \DDD^U_{\eta,\tau}\subset \R^2,\ 
\T^3\times (U \setminus  \DDD^U_{\eta,\tau})$, or  
$\T^3 \times U \setminus \KAM^U_{\eta,\tau}$. 
If $f$ is a $\CCC^s$ real valued function on $M$, the $\CCC^s$-norm of $f$
\[
 \|f\|_{\CCC^s}=\sup_{x\in M,\ |\al|\le s} \|\partial^\al f(x)\|,
\]
where the supremum is over the absolute values of
all partial derivatives $\partial^\al$ of order $\le s$.
%Definition of $\CCC^s$-norm for non-integer $s$ is in section \ref{sec:holder}.  

Introduce a strong $\CCC^s$-topology. We endow it with
the strong $\CCC^s$-topology on the space of functions on
a non-compact manifold or the $\CCC^s$ Whitney topology.
A base for this topology consists of sets of the following
type. Let $\Xi =\{ \phi_i,U_i \}_{i \in \Lambda}$ be a locally
finite set of charts on $M$, where $M$ is as above.  Let
$K=\{K_i\}_{i \in \Lambda}$ be a family of compact subsets of
$M,\ K_i \subset U_i$. Let also $\eps=\{\eps_i\}_{i \in \Lambda}$
be a family of positive numbers. A strong basic neighborhood
$\mathcal N^s(f,\Xi,K,\eps)$ is given by
\[
\forall i\in \Lambda \qquad  \| (f\phi_i)(x)-(g\phi_i)(x)\|_s \le \eps_i \qquad 
\forall x\in K_i.
\]
The strong topology has all possible sets of this form. 
Let $\eps_0$ be small positive. Endow 
$\CCC^s(\T^3\times U,\KAM^U_{\eta,\tau})$ 
with the strong topology.

\bthm \label{thm:2nd-main} (The Second Main result) 
Let $\eta,\tau>0.$ Then, for any $r\ge 2018$  
and any $\CCC^{3r+9}$ smooth strictly convex Hamiltonian $H_0$ 
there exists a $\CCC^r$ dense set of perturbations $D \subset \SSS^r$ 
such that for any $H_1 \in D$ there is $\eps=\eps(\eta,\tau,H_0,H_1)>0$ 
with the property that there is a set $W$ open in $\CCC^r$-Whitney KAM 
topology such that for any $\Delta H_1\in W$ the Hamiltonian 
$H_0+\eps (H_1+\Delta H_1)$ has an orbit $(I_\eps(t),\varphi_\eps(t),t)$ 
such that 
\[
 \KAM^U_{\eta,\tau} \subset \  
 \overline{\cup_{t\in \R}\ (I_\eps(t),\varphi_\eps(t),t)}.
\]
\ethm 
Theorem \ref{thm:2nd-main} certainly implies Theorem \ref{thm:MainResultNonAutonomous}.

\subsection{Open problems:}

{\bf Is there a strong form of diffusion for prevalent perturbations?}
The class of perturbations for which the main result holds is 
not generic. It is expected that the conclusion of 
Theorem \ref{thm:2nd-main} holds for a prevalent set of 
perturbations. Loosely speaking,  choose any $\CCC^r$ perturbation $\eps H_1$ and add a potential perturbation 
$\Delta H_1$ whose Fourier coefficients are chosen 
randomly inside a certain Hilbert cube with respect to 
the Lebesgue product measure. Then, one expects 
that for most coefficients there is an orbit accumulating 
to all tori from $\KAM^U_{\eta,\tau}$\footnote{For discussion 
of prevalence see e.g. \cite{HK10} and references therein}. 

\vskip 0.1in 

{\bf Is there diffusion for prevalent perturbations?}
The question of prevalent diffusion is open even for ``simple Arnold diffusion'', i.e. diffusion along finitely many resonances,
as all results \cite{Mather03,Mather08, Cheng13,KaloshinZ12,
KaloshinZ14B, KaloshinZ14} study only 
``a cusp residual set of perturbations'' (see e.g. Thm. 1
\cite{KaloshinZ12}).

\vskip 0.1in 

{\bf Is there diffusion for analytic perturbations?}
All papers aforementioned in the previous item 
deal with smooth Hamiltonian systems. Proving 
diffusion for analytic Hamiltonian systems is 
a deep open problem. 

\vskip 0.1in 

{\bf Lyapunov stability of totally elliptic points} \, 
Is a nonresonant totally elliptic fixed point of a prevalent 
$4$-dimensional symplectic map is Lyapunov unstable? 
In \cite{KaloshinMV04} this is shown for a generic resonant 
totally elliptic point in dimension 4\footnote{
Even examples of analytic non-resonant totally elliptic 
Lyapunov unstable fixed points  are unknown}. 

\vskip 0.1in 

{\bf Essential coexistence of zero and nonzero Lyapunov exponents} 
In \cite{CHP13} (resp. \cite{HPT13})  volume preserving flows (resp. diffeomorphisms)
of a $5$-dimensional compact manifold are constructed such that 
there are positive and zero Lyapunov exponents both on a set of positive 
measure. One of interesting open problems is to construct such examples
for nearly integrable systems.

\subsection{Arnold diffusion}
%The present paper is build upon the recent literature in Arnold diffusion. 
Recall that the quasi-ergodic hypothesis asks for the existence 
of a dense orbit in a typical energy surface of a typical 
Hamiltonian system. Arnold diffusion asks a much weaker question: 
whether  a typical nearly integrable Hamiltonian has orbits whose actions make a drift with size independent of the perturbative parameter. 

The study of Arnold diffusion was initiated by Arnold in his seminal paper \cite{Arnold64}, where he obtained a concrete Hamiltonian system with an orbit undergoing a small drift in action. In the last decades there has been a huge progress in the area. The works of Arnold diffusion can be classified in two different groups, the ones which deal with \emph{a priori unstable systems} and  the ones dealing with \emph{a priori stable systems} (as defined in \cite{ChierchiaG94}). The first ones are those whose first order presents some hyperbolicity. They have been studied both by geometric methods \cite{DelshamsLS06a, GideaL06, DelshamsLS08,  DelshamsH09, GelfreichT08, Treschev04, Treschev12, DelshamsLS13} and variational methods \cite{Be, ChengY04, ChengY09}. \emph{A priori stable systems} are those who are close to a completely integrable Hamiltonian system. That is systems close to a Hamiltonian system whose phase space is foliated by quasiperiodic invariant tori.  The existence 
of Arnold diffusion in \emph{a priori stable systems} was only known in concrete examples \cite{Bessi96, Bessi97, BessiCV01, Douady88,FontichM00,FontichM01, FontichM03, KaloshinL08, KaloshinL08a, Kaloshin-Levi-Saprykina, Zhang11} until 
the  recent  works \cite{Mather03, BernardKZ11, KaloshinZ12, KaloshinZ14B, KaloshinZ14, Cheng13}. 

The present paper relies on the techniques that have developed in the last decades in the field of Arnold diffusion. We follow the approach initiated in \cite{BernardKZ11, KaloshinZ12} which mixes geometric  and variational techniques. In the geometric part we rely on the theory of  normally hyperbolic invariant manifolds \cite{HirschPS77, Fenichel71} (see also \cite{Moeckel96, DelshamsLS00, Bernard10}). In the variational part we rely on the theory developed by J. Mather \cite{Mather91, Mather91a, Mather04,  MaSh}, A. Fathi \cite{FathiBook}, P. Bernard \cite{Be} and 
Ch.-Q.Cheng--J.Yan \cite{ChengY04, ChengY09}.

{\it Acknowledgement:} We would like to thank  
Abed Bounemoura, Georgi Popov, Kostya Khanin, 
Hakan Eliasson, John Mather for useful conversations. 
Regular discussions with Ke Zhang are warmly 
acknowledged and highly appreciated. The authors thank 
the Institute for Advanced Study for its hospitality, where 
a part of the work was done.  The first  author is partially
 supported by the  Spanish MINECO-FEDER Grant
 MTM2012-31714 and the Catalan Grant 2014SGR504.
The second author acknowledges support 
of a NSF grant DMS-1157830.

%\newpage 
%\section{Constants involved}
%\begin{itemize}
%\item $r$: Regularity. It's the free parameter.
%\item $s$: Regularity after P\"oschel. It's $s=r-6-2\tau$.
%\item $m=r/10$. It was the step of Birkhoff normal form. Now it's used to separate the single and double resonance regime.
%\item $d$: Hyperbolicity. According to the deformation procedure should be $d=14$.
%\item $q$: Size of the normal form remainder. From Aubry sets localization we need $q=18 d$.
%\item $\theta$: order of strong double resonances.
%\end{itemize}

%\newpage

\section{Outline of the proof of Theorem \ref{thm:MainResultNonAutonomous}}\label{sec:Outline}
The proof of Theorem \ref{thm:MainResultNonAutonomous} 
has several parts, which we describe in this section. In 
the proof we establish some robust mechanism of diffusion
and, in particular, derive Theorem \ref{thm:2nd-main}. The proofs 
of each part are given in Sections 
\ref{sec:ConstructionNetResonances} --\ref{sec:LocAubrySets}. We start by listing these parts. Recall that 
$U'_\eta \subset \nabla H_0(U)$\ is the set of points 
whose $\eta$-neighborhoods belong to $\nabla H_0(U)$ and  
%In other words, this is the set of frequencies in the domain 
%of definition of (\ref{def:Ham:Original:0}). 
$\DDD^U_{\eta,\tau}:=\DDD_{\eta,\tau}\cap U_\eta'$. 
Fix $R_0\gg 1$ (to be determined) and a sequence of radii
$\{R_n\}_{n\in \Z_+}$, where $R_{n+1}=R^{1+2\tau}_n$
for each $n\in \Z_+$. Consider another ``reciprocal''
sequence $\{\rho_n\}_{n\in \Z_+},\ \rho_n=R_n^{-3+5\tau}$
for each $n\in \Z_+$. For each $k\in (\Z^2\setminus 0)\times \Z$
denote $\Gm_{k}=\{\om\in U':k\cdot (\om,1)=0\}$
the corresponding resonant segment. 
\begin{enumerate}
\item \textit{A tree of  Dirichlet resonant segments}: 
\begin{itemize}
\item We construct a sequence of grids of Diophantine 
frequencies in $\DDD^n_{\eta,\tau}\subset \DDD^U_{\eta,\tau}$ 
such that $3\rho_n$-neighborhood of  $\DDD^n_{\eta,\tau}$ 
contains $\DDD^U_{\eta,\tau}$ and $\rho_n$-neighborhoods of 
points of $\DDD^n_{\eta,\tau}$ are pairwise disjoint. 

\item To each grid $\DDD^n_{\eta,\tau}$ we associate a collection 
of pairwise disjoint Voronoi cells (see Figure \ref{fig:voronoi-cells}): open neighborhoods 
$\{\Vor_n(\om_n)\}_{\om_n\in \DDD^n_{\eta,\tau}}$ such that 
$B_{\rho_n}(\om_n)\subset \Vor_n(\om_n)$  and 
$\Vor_n(\om_n)\cap \DDD_{\eta,\tau} \subset 
B_{3\rho_n}(\om_n)$.

We construct a collection of generations of resonant segments 
\[
\bS=\cup_{n\in \Z_+} \bS_n,\qquad \bS_n=
\cup_{k_n\in \FF_n} \{\SSS^{\om_n}_{k_n}\},
\]
where $\FF_n\subset (\Z^2\setminus 0)\times \Z$ is a collection of 
resonances, $\SSS^{\om_n}_{k_n}\subset \Gm_{k_n}$. 
We choose this collection such that most of the intersections between 
the different resonant segments are empty and if nonempty
we have quantitative estimates.

\item By construction each segment of generation $n$ 
is contained in a Voronoi cell $\Vor_{n-1}(\om_{n-1})$ for some $\om_{n-1}\in \DDD^{n-1}_{\eta,\tau}$, i.e. 
$\SSS^{\om_n}_{k_n}\subset \Vor_{n-1}(\om_{n-1}).$
 We say that segments 
$\SSS^{\om_n}_{k_n}$ from $\bS_n$ belong to the {\it generation} $n$. 
\end{itemize}
We call a resonance segment from $\bS$  {\it a   Dirichlet resonant segment}.

By construction  Dirichlet segments accumulates to  all
the Diophantine frequencies in $\DDD^U_{\eta,\tau}$ 
and satisfy several properties (see Key Theorem \ref{keythm:ResonancesWithNoTriple} for details). 

  \begin{figure}[thb]
     \psfrag{omj}[c]{\small $\om_j$}
      \psfrag{omj1}[c]{\small $\om_j^1$}
      \psfrag{omj2}[c]{\small $\om_j^2$}
      \psfrag{omj3}[c]{\small $\om_j^3$}
      \psfrag{omj4}[c]{\small $\om_j^4$}
      \psfrag{omj5}[c]{\small $\om_j^5$}
      \psfrag{omj6}[c]{\small $\om_j^6$}
      \psfrag{omj7}[c]{\small $\om_j^7$}
      \psfrag{omj8}[c]{\small $\om_j^8$}
      \psfrag{omj9}[c]{\small $\om_j^9$}
      \psfrag{omj10}[c]{\small $\om_j^{10}$}
      \psfrag{omj11}[c]{\small $\om_j^{11}$}
      \psfrag{omj12}[c]{\small $\om_j^{12}$}
      \psfrag{omj13}[c]{\small $\om_j^{13}$}
      \psfrag{omj14}[c]{\small $\om_j^{14}$}
      \psfrag{omj15}[c]{\small $\om_j^{15}$}
        \center{ \includegraphics[scale=0.3]{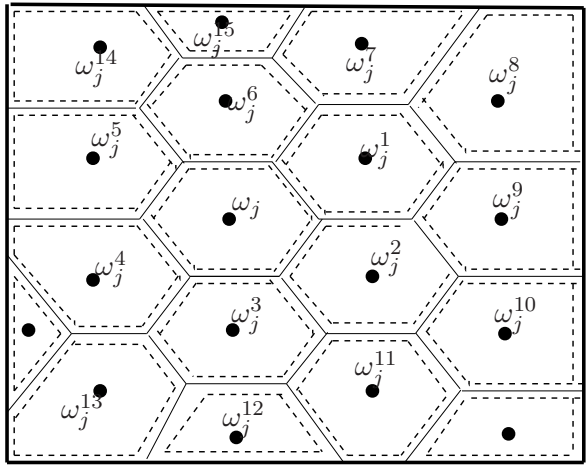}}
        \caption{  Voronoi cells }
        \label{fig:voronoi-cells}
\end{figure}

These  Dirichlet segments have a tree structure except 
that segments of the same generation can have 
{\it multiple intersections} (junctions)! See Section \ref{sec:ConstructionNetResonances}.
Notice that this part is {\it number theoretic} and 
 independent of the Hamiltonian $H_0$!

%\medskip 

\item \textit{Global (P\"oschel) normal form and molification:}
Our main result does not have loss of derivatives. To compensate the loss of derivatives caused by the normal forms we mollify our Hamiltonian $H_\eps$  to $H'_\eps$ so that it is 
$\CCC^\infty$ away from KAM tori $\KAM^U_{\eta,\tau}$ .
% away from $\KAM^U_{\eta,\tau}$\footnote{Certainly high norms of $H'_\eps$ blow up to infinity as we approach $\KAM^U_{\eta,\tau}$} (see Section \ref{sec:Molification}). 

In order to construct diffusion orbits, we consider the Hamiltonian 
$H'_\eps$ in the P\"oschel normal form 
\[
N'(\varphi,I,t)=H'_\eps \circ \Phi_\Pos(\varphi,I,t)=H'_0(I)+\eps R(\varphi,I,t)
\]
where $R$ vanishes on $\Phi_\Pos\ii(\KAM^U_{\eta,\tau})$ 
(see Theorem \ref{thm:Poschel}). 

It is convenient to identify resonances in the action space. 
To each Dirichlet segment
$\SSS^{\om_n}_{k_n}\subset \bS_n$ with $k_n\in \FF_n$
we associated a Diophantine frequency 
$\om_n\in \DDD^n_{\eta,\tau}$. For each 
$\om_n\in \DDD^n_{\eta,\tau}$ denote 
$I_n:=(\partial_I H_0')\ii(\om_n)$ and    
\[
\II^{\om_n}_{k_n}:=\{I\in U_\eta: \ \partial_I H_0'(I)\in \SSS^{\om_n}_{k_n}\}.
\]
Similarly, we can define Voronoi cells in action space:
\[
\Vor_n(I_n)=\{I\in U_\eta:\ \partial_I H_0'(I)\in \Vor_n(\om_n)\}.
\]
By construction we have $\SSS^{\om_n}_{k_n}\subset \Vor_n(\om_n)$. 
In the proof we always perturb away from KAM tori $\KAM^U_{\eta,\tau}$
and this does not affect $H_0'$. It is convenient to describe 
perurbations as well as diffusing conditions in P\"oschel's coordinates. 
Notice that in P\"oschel coordinates all KAM tori are flat, i.e. 
\[
\Phi_{\Pos}(\KAM^U_{\eta,\tau})=\{(\varphi,I,t):
\partial_I H_0'(I)\in \DDD^U_{\eta,\tau}\}.
\]

\item \textit{Deformation of the Hamiltonian $N^*$}: We modify 
the Hamiltonian $N'$ to $N^*$ by a small $\CCC^r$ 
perturbation supported away from KAM tori $\KAM^U_{\eta,\tau}$
so that the resulting Hamiltonian 
$N^*$ is non-degenerate (in a way that we specify later). 
This perturbation is done in Section \ref{sec:ConstructionOfDeformation}.
Notice that this modification procedure is also done by 
induction: 
\[
N^*(\varphi,I,t)=N'(\varphi,I,t) + \sum_{n\in\NN\cup\{0\}}  
\Delta N^n(\varphi,I,t), \  \text{ where }
\]
\begin{itemize}
\item the $0$-th generation perturbation $\Delta N^0$ is supported in a $\OO(\sqrt{\eps})$-neighbor\-hood of all horizontal and vertical resonant lines $\SSS_k$ with $|k|\leq R_0$. 
\item the $n$-th generation perturbation $\Delta N^n$
is supported in $n$-th order Voronoi cells, i.e. $\Delta N^n(\varphi,I,t)=0$ for any $I \not\in \Vor_n(I_n)$ 
for any $I_n$  with 
$\partial_I H_0'(I_n) \in \DDD^n_{\eta,\tau}.$
\end{itemize}
The zero step is done as in \cite{KaloshinZ12} and it allows us to construct a net of normally hyperbolic invariant cylinders (NHIC from now on\footnote{Note that in some cases
we obtain normally hyperbolic invariant manifolds diffeomorphic
to a cylinder minus a ball. In this case we call them normally 
hyperbolic invariant manifolds and abbreviate NHIMs}) along 
the resonances in $\bI_0$. 

In the $n$-th step we modify our Hamiltonian in Voronoi cells 
of order $n$ close to Dirichlet resonant segments of order $n$ 
so that it satisfies some non-degeneracy conditions, while 
non-degeneracy conditions of orders $k<n$
from previous steps still hold true. The result of the Deformation
is given in Key Theorem \ref{keythm:Transition:Deformation}.
In particular, it gives existence of NHICs for the truncated 
Hamiltonian (\ref{def:SingleResonance:TruncatedHamiltonian}).

\item \textit{Resonant Normal Forms}: Fix a generation $n$
and a Dirichlet segment $\II^{\om_n}_{k_n}$ of this generation.  
In Section \ref{sec:NormalForms}, we derive Resonant Normal Forms 
for $N'$ for points $(\varphi,I,t)$ such that $I$ is close 
to $\II^{\om_n}_{k_n}$. We obtain normal forms of two different types: 
\begin{itemize}
\item A single resonant one 
\be \label{eq:SR-NF}
N' \circ \Phi^{k_n}(\psi,J,t) = \HH_0^{k_n}(J)+\ZZZ^{k_n}(\psi^s,J)+\RRR^{k_n}(\psi^s,\psi^f,J,t)
\ee 
where $\ZZZ^{k_n}$ only depends on the slow angle 
$\psi^s=k_n\cdot (\psi,t),\ (\psi^s,\psi^f,t)$ form a basis and 
$\RRR^{k_n}$ is small relatively to non-degeneracy of $\ZZZ^{k_n}$. 
This normal form is similar to the normal form obtained in \cite{BernardKZ11} 
(see Key Theorem \ref{keythm:Transition:NormalForm}). 

\item A double resonant one 
\be\label{eq:DR-NF} 
N' \circ \Phi^{k_n, k'}  (\psi,J,t) = \HH_0^{k_n, k'} (J)
+\ZZZ^{k_n, k'} (\psi,J)
%+\RRR^{k_n, k'} (\psi,J,t)
\ee
where $\ZZZ^{k_n, k'} $ only depends on the two slow angles 
$\psi^s_1=k_n\cdot (\psi,t),\ \psi^s_2=k'\cdot (\psi,t)$. 
This normal form is valid near the double resonance 
$$
\nabla H_0'(I^*)\cdot k_n=\nabla H_0'(I^*)\cdot k'=0.
$$
Here we use the normal form procedure 
from \cite{Bounemoura10} (see Key Theorem \ref{keythm:CoreDR:NormalForm}).

We consider this normal form in the \emph{cores of 
the double resonances} (see Fig. \ref{fig:ResonancesBalls}). Thanks to the deformation procedure 
we can completely remove the time dependence in the normal form in there.  In \cite{KaloshinZ12} 
the double resonant normal form Hamiltonian is a small pertubation of 
a two degree of freedom mechanical system. Here  the  ``non-mechanical'' remainder is not negligible and 
the leading order has the form (\ref{eq:DR-NF}). To study 
such systems we apply a generalized Mapertuis principle  
(see Appendix \ref{app:NonMechanicalAtDR}) and use 
the technique developed in Appendix A of \cite{KaloshinZ12}.
\end{itemize}

\item \textit{A tree of NHICs}:  
In Section \ref{sec:NHIC}, using the non-degenera\-cy from
the Deformation step and the Resonant Normal Forms of 
the previous step, we build a tree of NHICs along all Dirichlet resonant segments. 

In the single resonant regime, we obtain NHICs  along the resonant segments  $\II_{k_n}^{\om_n}$. We  denoted these cylinders by 
$\CCC_{k_n}^{\om_n,i}$ (see  Key Theorem \ref{keythm:Transition:IsolatingBlock}). In the double resonance 
regime we obtain several NHICs by analogy with 
\cite{KaloshinZ12} (see 
Key Theorems \ref{keythm:DR:HighEnergy} and \ref{keythm:DR:LowEnergy:MapComposition} and 
Corollary \ref{coro:DR:NHIMs} for details).

\item \textit{Localization of Aubry sets}: We prove the existence of 
certain Aubry sets localized inside the NHICs from the previous 
step. We also establish a graph property, which essentially says that 
these Aubry sets have the same structure as Aubry-Mather 
sets of twist maps (see Key Theorems \ref{keythm: Transition:AubrySet} and \ref{keythm:DR:AubrySets} for details 
and Section \ref{sec:LocAubrySets} for proofs).

\item \textit{Shadowing}: It turns out that to construct orbits 
accumulating to all KAM tori KAM$^U_{\eta,\tau}$ it suffices 
to find diffusing orbits shadowing a chain of Aubry sets 
$\AAA(c)$'s with $c$'s accumulating to every point of 
$\Omega^{-1}(\DDD^U_{\eta,\tau})$, where $\Omega=H_0'$ 
is the frequency map in P\"oschel coordinates, see  in 
(\ref{eq:frequency-map}). It is sufficient to choose $c$'s from 
the tree of Dirichlet resonant segments 
$\Omega^{-1}({\bf S})$\footnote{near 
$\theta$-strong double resonance resonances, defined in 
\ref{definition:StrongWeakDR}, we need to be 
more careful with choice of $c$'s. See Key Theorem 
\ref{keythm:DR:AubrySets} and Section \ref{sec:outline:Aubrysets:doubleresonance} for details}. 

To shadow a chain of Aubry sets we study so-called 
$c$-equivalence proposed by Bernard. It turns out that 
$c$-equivalence is an equivalence relation. Moreover, all 
sets within an equivalence class can be shadowed by an orbit 
(see Proposition \ref{prop:c-equiv}). To prove $c$-equivalence 
we study three regimes: a priori unstable (diffusion along 
one NHIC), bifurcations (jump from one NHIC to another 
in the same homology), a turn --- a jump from a NHIC with one 
homology to a NHIC with  a different one. 
The first two regimes are essentially done in \cite{BernardKZ11}. %using a lemma from \cite{ChengY09}. 
The last one is similar to \cite{KaloshinZ12} (see 
Key Theorem 
\ref{keythm:equivalence} and Section \ref{sec:shadowing}
for details). 
\end{enumerate}

Now we describe each step and, in particular, state 
all ten Key Theorems. 
%Here is a contens of these Theorems:
%\begin{itemize}
%\item In Key Theorem \ref{keythm:ResonancesWithNoTriple}
%prove existence of Dirichlet resonances which we use for diffusion. 
%\item In Key Theorem \ref{keythm:Transition:NormalForm}
%we establish the normal for along single resonances. 
%\item Key Theorem \ref{keythm:Transition:Deformation}
%we find a family of NHICs for the truncated system.  
%\item Key Theorem \ref{keythm:Transition:IsolatingBlock}
%\item Key Theorem \ref{keythm:CoreDR:NormalForm}
%\item Key Theorem \ref{keythm:DR:HighEnergy}
%\item Key Theorem \ref{keythm:DR:LowEnergy:MapComposition}
%\item Key Theorem \ref{keythm: Transition:AubrySet}
%\item Key Theorem \ref{keythm:DR:AubrySets}
%\item Key Theorem \ref{keythm:equivalence}
%\end{itemize}

\subsection{A tree of Dirichlet resonances}\label{sec:OutlineResonances}
The first step is to construct a tree of generations of special 
resonant segments 
\[\bS=\{\bS_{n}\}_{n\in \Z},\qquad 
\bS_n=\{\SSS^{\om_n}_{k_n}\}_{k_n\in \FF_n}
\] 
in the frequency space $U'\subset \R^2$ such that 
the following properties hold 
\begin{itemize}
\item Each segment $\SSS^{\om_n}_{k_n}$ belongs to 
the intersection of the resonant segment and 
the corresponding Voronoi cell 
$\SSS^{\om_n}_{k_n}\subset 
\Gm_{k_n}\cap \Vor_{n-1}(\omega_{n-1})\cap B_{3\rr_{n-1}(\om_{n-1})}$, 
%=\{I\in U: k_n\cdot (\partial_I H_0(I),1)=0\},
i.e. there are $k_n\in (\Z^2\setminus 0)\times \Z$
and $\om_{n-1}\in \DDD^{n-1}_{\eta,\tau}$ such that 
this inclusion holds. 

\item Any segment $\SSS^{\om_n}_{k_n}$ of generation $n$ 
intersects only  segments of the  generations $n-2$ and $n-1$ and the segments of generation $n+1$ and $n+2$ associated to frequencies $\DDD^{n+1}_{\eta,\tau}\cap  \Vor_{n-1}(\omega_{n-1})$ and  $\DDD^{n+2}_{\eta,\tau}\cap  \Vor_{n-1}(\omega_{n-1})$ respectively. However, it does not intersect any segment of any prevous  generation generation $n+k$ or $n-k$, $k\geq 3$. 

\item We have quantitative information about properties 
of intersections of segments of generation $n$ inside 
of the same Voronoi cell. 

\item The union $\cup_n \cup_{k_n\in \FF_n}\SSS^{\om_n}_{k_n}$ 
is connected. 
 
\item The closure 
$\overline{\cup_n \cup_{k_n\in \FF_n}\SSS^{\om_n}_{k_n}}$  contains all 
Diophantine frequencies in $\DDD^U_{\eta,\tau}$ 
(see \eqref{def:Diophantine}).
\end{itemize}

We shall call these segments Dirichlet resonant segments. 
We define them to satisfy quantitative 
estimates on speed of approximation. Recall that an elementary 
pigeon hole principle show that for any bounded set 
$U\subset \R^2$ and any
any $\om=(\om_1,\om_2)\in U$ there is a sequence 
$k_n(\om)\in (\Z^2\setminus 0)\times \Z$ such that 
\[
|k_n(\om)\cdot (\om,1)|\le |k_n(\om)|^{-2} \  \ 
\text{ and }\ \ |k_n(\om)|\to \infty \text{ as }n\to \infty.
\] 
We choose our segments so that  
\begin{itemize}
\item For each $\SSS^{\om_n}_{k_n}\subset \Gm_{k_n}$ 
there is a Diophantine number $\om_n \in \DDD^n_{\eta,\tau}$ 
such that 
\[
|k_n\cdot (\om,1)|\le |k_n|^{-2+3\tau},
\] 

\item If $\SSS^{\om_n}_{k_n}\cap \SSS^{\om'_n}_{k'_n}\ne \emptyset$, 
then we have an upper bound on ratio $|k_n|/|k'_{n}|$ as well
as a uniform lower bound for the angle of intersection between $\Ga_{k_n}$ and $\Ga_{k'_n}$.  
\end{itemize}

The tree of Dirichlet resonances is obtained in two steps. 
\begin{itemize}

\item  We fix a Diophantine $\om^*\in\DDD^U_{\eta,\tau}$ 
and construct a connected ``zigzag'' approaching $\om^*$
(see Section \ref{sec:SelectionResonances}). 

\item We define Dirichlet resonant segments and show how 
to deal with all frequencies in $\DDD^U_{\eta,\tau}$
simultaneously (see Section \ref{sec:ResonancesWithNoTriple}).  
\end{itemize}

It turns out that we cannot just construct the ``zigzag'' for each frequency 
and then consider the union over all frequencies in $\DDD^U_{\eta,\tau}$,  
because we need estimates on the complexity of
the intersection among resonant segments\footnote{Too many 
segments, too many intersections to control.}. Now we define 
the sequence of discrete sets $\{\DDD^n_{\eta,\tau}\}_{n\in \Z_+}$.

\medskip 
 
Consider that $R_0\gg 1$ (to be determined) and we denote
\begin{equation}\label{def:rho}
R_{n+1}=R_n^{1+2\tau}, \ \rho_n = R_n^{-(3-5\tau)}, \qquad 
\forall n\in \Z_+
\end{equation}
and $B_r(\om)$ be an $r$-ball centered at $\om$. 
Note that the sequence $\rr_n$ also satisfies $\rr_{n+1}=\rr_n^{1+2\tau}$. Define a sequence of discrete sets 
$\cD^n=\cD^n_{\eta,\tau}\subset \cD^U_{\eta,\tau}$ with $n\ge 1$. 

By the Vitali covering lemma from any finite cover of $\cD^U_{\eta,\tau}$ 
by $\rho_n$-balls, there is a subcover $\mathcal B_n$ 
having the following two properties:
\begin{enumerate}
\item Replacing $\rho_n$-balls of subcover $\mathcal B_n$ by concentric 
$3\rho_n$-balls we cover  $\cD^U_{\eta,\tau}$.
\item The  $\rho_n$-balls of $\mathcal B_n$ are pairwise disjoint. 
\end{enumerate}
Moreover, by the Besicovitch covering theorem from the cover by balls 
of radius $3\rho_n$, one can choose a finite subcover such that each 
point is covered by at most $\kappa$ balls for some $\kappa$
which depends only on dimension. In \cite{Sullivan94}
it is shown that in the 2-dimensional case $\kappa = 19$.
Denote the set of the centers of the balls of this cover by $\cD^n_{\eta,\tau}$.  For each point $\om_n \in \cD^n_{\eta,\tau}$ consider all neighbors, i.e. 
points $\le 6\rho_n$ away from $\om_n$.  Define the set of points 
that are closer to $\om_n$ than to any other neighbor and 
at most $3\rho_n$ away from $\om_n$. Denote this set by 
$\Vor_n(\om_n)$ and, following the standard terminology, call it
{\it Voronoi cells}.  The properties of the coverings by balls and the Voronoi cells, are summarized in the following 

\begin{lemma}\label{lemma:VoronoiCells}
Consider the set of frequencies $U_\eta'\subset U'$. 
Then, there exists a sequence of discrete sets 
$\cD^n_{\eta,\tau}\subset\cD^U_{\eta,\tau}\subset U_\eta'$, $n\geq 1$, satisfying 
\[
\cD^U_{\eta,\tau}=\ol{\bigcup_{n\geq 1}\cD^n_{\eta,\tau}}
\]
such that for 
the Voronoi cell $\mathrm{Vor}_n(\om^\ast)\subset U'$ associated 
to each frequency $\om^*$ of the set $\cD^n_{\eta,\tau}$ we have 
\[
 \om^\ast\in B_{\rr_n}(\om^\ast)\subset\Vor_n(\om^\ast)\,\,\text{ and }\,\, \Vor_n(\om^\ast)\cap \cD_{\eta,\tau}\subset B_{3\rr_n}(\om^\ast).
\]
Moreover, each $\om\in \cD^U_{\eta,\tau}$   belongs to 
at least one and at most 19 balls  $B_{3\rr_n}(\om^\ast)$,
%\[ \cD^U_{\eta,\tau}\subset 
%\bigcup_{\om^\ast\cD^n_{\eta,\tau}}B_{3\rr_n}(\om^\ast),
%\]
and $B_{\rr_n}(\om^\ast)\cap B_{\rr_n}(\om')=\emptyset$ 
for any $\om^*\neq \om'\in \cD^n_{\eta,\tau}$.
\end{lemma}

We use this sequence of discrete sets $\DDD^n_{\eta,\tau}$ 
to construct a tree of resonances. The key theorem of this 
first step  is the following.

\begin{figure}[thb]
        \center{ \includegraphics[scale=1.0]{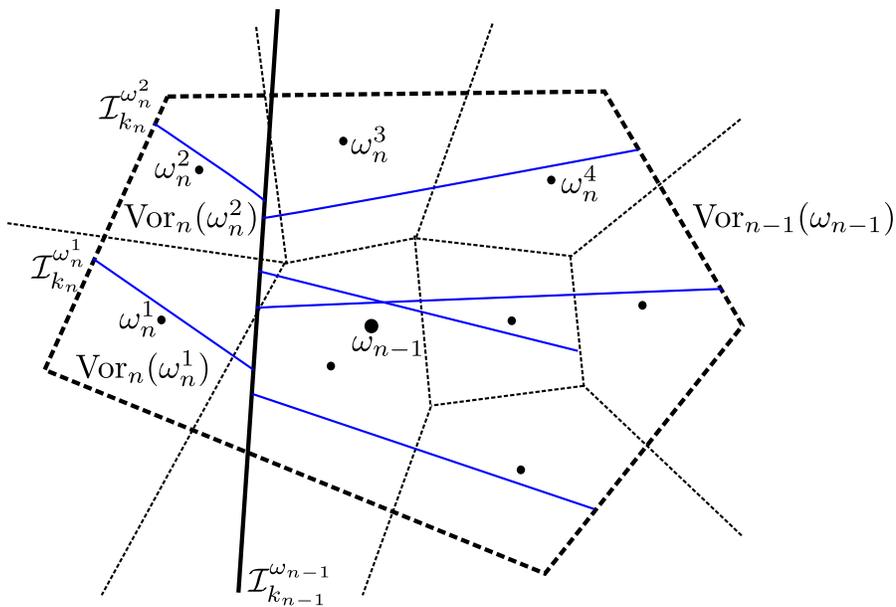}}
        \caption{A branch of a resonant tree. The solid black (resp. blue)
lines correspond to resonant segments associated to $\om_{n-1}$
(resp. of $\om^k_n$'s). The dashed lines correspond to 
boundaries of the Voronoi 
cells $\Vor_{n-1}(\om_{n-1})$ and $\Vor_{n}(\om^k_{n})$'s. Recall that resonant segments of the same generation may intersect}
        \label{fig:ResonanceTree}
\end{figure}

\begin{keytheorem}\label{keythm:ResonancesWithNoTriple}
Consider the sequence of sets $\{\cD^n_{\eta,\tau}\}_{n\leq 1}$ and 
the associated Voronoi cells $\Vor_n(\om_n^*)$, 
$\om^*_n\in\cD^n_{\eta,\tau}$. Then, there are constants 
$c_1,c_2>0$, $R^*\gg1$ and $\ 0<\tau^*\ll 1$,  such that, for any 
$R_0\in [R_*,+\infty)$ and $\tau\in (0,\tau^*)$ and taking 
$R_{n+1}=R_n^{1+2\tau}$, $\rr_n=R_n^{-(3-5\tau)}$, for 
each $n\in \Z_+$ and 
$\om_n^*\in \cD^n_{\eta,\tau}$, there exists 
$k_n^*\in(\ZZ^2 \setminus 0)\times \Z$ 
which satisfies the following properties:
\begin{enumerate}
\item The vectors belong to the following annulus 
$ \frac{R_{n}}{4}\leq \left|k_{n}^{*}\right|\leq R_{n}.$ 
\item The vectors satisfy  $ \eta R_{n}^{-(2+\tau)}\leq \left|k_{n}^{*}\cdot (\omega^*,1)\right|\leq  R_{n}^{-(2-3\tau)}$.
\item The associated resonant segment $\SSS_{k_n^*}^{\om_n^*}$ 
satisfies $\SSS_{k_n^*}^{\om_n^*}\cap B_{\rr_n}(\om_n^*)\neq \emptyset$.
\item Let $\om_{n-1}^*\in \cD^{n-1}_{\eta,\tau}$ be a frequency in 
the previous generation so that $\om_n^*\in \Vor_{n-1}(\om_{n-1}^*)$ 
and $k^*_{n-1}$ the associated resonant vector, i.e. 
$\SSS_{k^*_{n-1}}^{\om^*_{n-1}}\cap \Vor_{n-1}(\om_{n-1}^*)\neq \emptyset$. 
Then, the angle $k_n^*\not \parallel k_{n-1}^{*}$  satisfies
$\measuredangle (k_n^*,k_{n-1}^{*})\geq \dfrac{\pi}{10}.$
\item The resonances $\Ga_{k_n^*}$ and $\Ga_{k_{n-1}^{*}}$ 
intersect in $\Vor_{n-1}(\om_{n-1}^*)$ at only one point.
\item Let $\SSS^{\om_n^*}_{k^*_n}\subset \Ga_{k^*_n}$ be 
the minimal segment, which contains the above intersection and 
also all the intersections of $\Ga_{k_n^*}$ with the resonant lines 
associated to all frequencies in 
$\cD^{n-1}_{\eta,\tau}\cap \Vor_n(\om_{n}^*)$. Then,
$\SSS^{\om_n^*}_{k^*_n}\subset B_{\rr_{n-1}}(\om_{n-1}^*).$

\item 
Any other  resonant line  $\Ga_k$ intersecting 
$B_{3 \rr_n}(\om_{n}^*)$, 
satisfies $\frac{|k_n|}{|k|}\leq  R_n^{c_1\tau}. $
%for some constants $c_1>0$ independent of $R_0$ and $n$.
\item Let $\FF_n\subset (\Z^2\setminus 0)\times \Z$ be 
the set of integer vector $k_n$'s selected above.  
Let $\om'$ be an intersection point of $\SSS_{k_n}^{\om_n}$
and $\SSS_{k'_n}^{\om'_n}$ for some $\om_n,\om_n'\in \cD_{\eta,\tau}^n$ and $k_n,k_n'\in \FF_n$. Then, 
there are at most $\rr_{n}^{c_2\tau}$ resonant lines 
$\SSS_{k_n''}^{\om_n''}$ from the $n$ generation
passing through $\om'$.
\end{enumerate}
\end{keytheorem}
This Theorem is proved in Section \ref{sec:ConstructionNetResonances}. 

As a consequence of this Theorem we obtain a set of Dirichlet 
resonant vectors and a tree of resonances in the frequency 
space. Nevertheless the tree of resonances formed by 
the Dirichlet resonant segments, obtained  in 
Key Theorem \ref{keythm:ResonancesWithNoTriple}, is not 
connected since the segments associated to the first 
generation belong to different Voronoi cells and therefore 
they do not intersect. To connect them, 
we consider a $0$-generation.

\begin{definition}\label{def:ResonanceNet}
Fix $R_0\gg 1$ and $\rr_0=R_0^{-(3-5\tau)}$. Define sets $\FF_n\subset(\Z^2\setminus 0)\times \Z$ as follows. 
\[
\FF_0 = \{k\in(\ZZ^2\setminus\{0\})\times\ZZ: k\cdot (1,0,0)=0\,\text{ or }k\cdot(0,1,0)=0,\,\,\text{gcd}(k)=1\footnote{where gcd$(k)$ is 
the greatest common divisor of absolute values of components of $k$. 
},\,\,\,|k|\leq R_0\},
\]
which corresponds to vertical and horizontal lines in frequency space. The associated net of resonances in frequency space $\bS_0\subset U'_\eta\subset\RR^2$ is given by
\[
 \bS_0=\bigcup_{k\in\FF_0}\left\{\om\in U'_\eta\subset\RR^2:(\om,1)\cdot k=0\right\}.
\]
For $n\geq 1$ we define the set $\FF_n$  as the set of Dirichlet resonant vectors $k\in(\ZZ^2\setminus\{0\})\times\ZZ$ obtained in Key Theorem \ref{keythm:ResonancesWithNoTriple} and 
\begin{equation}\label{def:UnionResonanceInFreq}
 \bS_n=\bigcup_{k_n\in\FF_n}\SSS_{k_n}^{\om_n}, \qquad 
 \bS=\bigcup_{n\geq 0}\bS_n.
\end{equation}
where $\SSS_{k_n}^{\om_n}$ are the resonant segments obtained in  
Key Theorem \ref{keythm:ResonancesWithNoTriple}.
%Finally, we define
%\[
%\]
\end{definition}

By the definition of $\bS_0$ and Key Theorem \ref{keythm:ResonancesWithNoTriple} we have constructed a net of resonances $\bS$ which satisfies the needed conditions. They are summarized in the next corollary of Key Theorem \ref{keythm:ResonancesWithNoTriple}.
% , whose proof is a direct consequence of Theorem \ref{thm:ResonancesWithNoTriple} and Definition \ref{def:ResonanceNet}. 
\begin{corollary}\label{coro:ResonancesFinal}
The resonance net $\bS\subset U_\eta'\subset\RR^2$ is connected, 
%satisfies the following properties
%\begin{enumerate}
%\item It is connected
%\item 
$\DDD^U_{\eta,\tau}\subset\ol{\bS}$. Moreover, 
each $\bS_n=\{\SSS_{k^*_{n}}^{\om^*_{n}}:k\in\FF_n\}$, $n\geq 1$, satisfies all the properties of Key 
Theorem \ref{keythm:ResonancesWithNoTriple}.
%\end{enumerate}
\end{corollary}

\subsection{Deformation of the Hamiltonian and a tree of NHICs}\label{sec:outline:deformationandcylinders}
Step 2 of the proof is to deform the Hamiltonian. 
The deformation has two goals. 
\begin{itemize}
\item Recover the loss of regularity due to the normal form procedures.
\item Obtain a properly non-degenerate Hamiltonian. 
\end{itemize}
Using these non-degeneracies we prove the existence of a tree of NHICs in Step 3. 
In this section we explain both Steps 2 and 3.

Fix small  $\eta, \tau>0$. Denote $a_\tau:=a+\OO(\tau)$, 
where $|\OO(\tau)|<C\tau$ for some $C>0$ indep\-endent 
of $\eps$, $\rr_n$ and $n$. Since $\tau$ is small, 
the notation is used  if a multiple in front of $\tau$ is irrelevant.

Consider the nearly integrable Hamiltonian \eqref{def:Ham:Original:0}. 
Assume it satisfies conditions  \eqref{def:Convexity:OriginalHam} and
\eqref{def:PropertiesH0}. We construct  a  small 
$\CCC^r$-perturbation $\Delta H$, so that for small enough $\eps$ the Hamiltonian
\[ H_0  + \eps H_1 + \eps \Delta H
\]
satisfies the required properties. This perturbation consists of four parts:
\begin{equation} \label{eqn:perturbation-form}
 \Delta H = \Delta H^{\mol}+\Delta H^{\sr}
 +\Delta H^{\dr} +\Delta H^{\shad}.
\end{equation}
To perform these deformations consider the net of resonances given by Key Theorem \ref{keythm:ResonancesWithNoTriple}.

\begin{enumerate}
 \item \emph{Mollification $\Delta H^{\mol}$}: Throughout the proof of Theorem \ref{thm:MainResultNonAutonomous} we need to perform normal forms, which decrease regularity. To stay in the $\CCC^r$ regularity class, 
we mollify the Hamiltonian $H_\eps$ so that it is $\CCC^\infty$.
%$\CCC^\infty$.
% away from KAM tori. 
See Section \ref{sec:ConstructionOfDeformation}.
 
 \item \emph{Perturbation  $\Delta H^{\sr}$}: It is supported near single resonances. In Section \ref{sec:NormalForms} we construct NHICs along single resonances. The normal forms
show that, to construct such NHICs, it is sufficient to have hyperbolicity of the first order approximation.
Using the deformation we indeed attain required hyperbolicity. See Section \ref{sec:deformationSR&DRstep0}.
%--\ref{sec:deformationSR&DRstepn}.

 \item \emph{Perturbation $\Delta H^{\dr}$}: It is supported in the neighborhoods of  
 double resonances. Analogous to the previous one but for the double resonances. In double resonances, we need also some hyperbolicity properties to construct certain NHICs.  This third deformation is done 
simultanenously with $\Delta H^{\sr}$ in Section \ref{sec:deformationSR&DRstep0}.
%--\ref{sec:deformationSR&DRstepn}.

 \item \emph{Perturbation $\Delta H^{\shad}$}: It is supported near chosen   resonant lines. This deformation allows us to shadow a chosen collection of (Aubry)  invariant sets. See Section \ref{sec:shadowing}.
\end{enumerate}

\vskip 0.1in 

\subsubsection{The P\"oschel normal form and molification of the Hamiltonian}\label{sec:MolificationAndPoschel}
We start by  mollifying $H_1$ and  performing the P\"oschel normal form.
% away from KAM tori. 

%Now we mollify the Hamiltonian  $H_\eps\circ \Phi_{\Pos}$ away from $E_{\eta,\tau}$ or, equivalently, away from $\KAM^U_{\eta,\tau}$.
%One possibility is just to mollify $H_1$. However,
%in the future we would like to state an open dense result 
%with respect to KAM Whitney topology. This requires to make
%perturbations away from the KAM tori. We denote the mollification of $H_\eps$, by  $H_\eps^{\mol}$ and
%the set of points in the phase space $\rr_n$ far away from $\KAM^U_{\eta,\tau}$
%by  $\Upsilon_{\rr_n}(\KAM^U_{\eta,\tau})$.

\begin{lemma} \label{lm:molification}
%Let  $\{\rr_n\}_{n\geq 1}$ be the sequence defined in Key Theorem \ref{keythm:ResonancesWithNoTriple}. 
Fix any $\gm\ll 1$ independent of $\eps$. There exists a $\CCC^r$ small perturbation 
$\Delta H^{\mol}$ 
%vanishing on $\KAM_{\eta,\tau}$ 
%and  $c=c(H_0,r)$ 
such that 
\[
 \|\Delta H^{{\mol}}_1\|_{\CCC^r}\le \gm \ \ 
\text{ and }\ \ 
\| H_1+\Delta H^{{\mol}}_1 \|_{\CCC^{r+\kappa}}<\infty
\ \ \text{ for any }\kappa>0.
\]
\end{lemma}
%This lemma is a direct consequence of Lemma \ref{lm:SalamonZehnder} stated in Section \ref{sec:ConstructionOfDeformation}. 
All estimates in the proof will depend on the parameter $\ga$. Since it is a fixed parameter we do not keep track of this dependence. After this lemma, the regularity of the Hamiltonian $H_\eps'=H_0+\eps H_1+\eps\Delta H_1^\mol$ is the same as the original regularity of $H_0$, that is $\CCC^{3r+9}$.
%From now on, we deal with the Hamiltonian
%\begin{equation}\label{def:HamAfterMolification}
%\begin{split}
% N'(I,\varphi,t)&=H_\eps\circ \Phi_{\Pos}(I,\varphi,t)+\Delta H^{{\mol}}_1\circ \Phi_{\Pos}(I,\varphi,t)\\
%&=H'_0(I)+R(I,\varphi,t)+\Delta H^{{\mol}}_1\circ \Phi_{\Pos}(I,\varphi,t).
%\end{split}
%\end{equation}

Now we apply the P\"oschel normal form. It was obtained in \cite{Poschel82} for autonomous Hamiltonian Systems. In Section \ref{sec:KAM-time-periodic}, it is explained how to adapt it to the non-autonomous setting.

\begin{theorem}\label{thm:Poschel} 
Let $\eta,\tau>0$ and $\{\rr_n\}_{n\geq 1}$ be the sequence defined in Key Theorem \ref{keythm:ResonancesWithNoTriple}. There exist 
$\eps_0>0$ such that there exists a $\CCC^{3r -2\tau}$ 
change of coordinates $\Phi_{\Pos}$ so that the Hamiltonian 
$N'=(H_\eps+\Delta H^\mol)\circ \Phi_{\Pos}$ is of the form
\begin{equation}\label{def:HamAfterMolification}
 N'(\varphi,I,t)=H'_0(I)+R(\varphi,I,t)
\end{equation}
where $H'_0$ satisfies  
\begin{itemize}
\item $D\ii\Id\leq \pa_I^2H_0'(I)\leq D\,\Id$
%\[(2D)^{-1}\|v\| \le \langle \partial^2 H_0'(I) v,v\rangle \le 2D\|v\|
%\]
with the constant $D$ introduced in \eqref{def:Convexity:OriginalHam}, and 
\item $R$ vanishes on the set 
\[
 E_{\eta,\tau}=\left\{(\varphi,I,t):\pa_I H_0'(I)\in \DDD^U_{\eta,\tau}\right\}, \ 
\textup{i.e. }R|_{E_{\eta,\tau}}\equiv 0.
\]
\end{itemize}
Moreover, if $U_\la(E_{\eta,\tau})$ is the set of  points $\la$-close to $E_{\eta,\tau}$, for each $n\geq 1$, we have the following estimates
\[
 \|R\|_{\CCC^{j}(U_{\rr_n}(E_{\eta,\tau}))}\leq C\,\eps\,\rr_n^{3r-2\tau-j}
\]
where $C>0$ is a constant independent of $\eps$ and $\rr_n$ and $j=0,\ldots, \lfloor 3r-2\tau\rfloor$.

\end{theorem}
Note that  $ E_{\eta,\tau}=\Phi_{\Pos}\ii(\KAM^U_{\eta,\tau})$ is the union 
of KAM tori $\KAM^U_{\eta,\tau}$ with fixed Diophantine frequency. In 
the P\"oschel coordinates all these tori are flat, $I=\text{constant}$. 
Moreover, $R$ being zero at  $E_{\eta,\tau}$ implies that 
it is flat at these points. Indeed, each such a point is a Lebesgue 
density point.

In Step 1 we constructed a tree of resonances segments in the frequency space. 
Now we need to pull it back to the action space. We do it using the frequency map 
associated to the integrable Hamiltonian $H_0'$ obtained in Theorem \ref{thm:Poschel}. 
We define the frequency map 
\be \label{eq:frequency-map}
\Omega(I)=\pa_IH_0'(I).
\ee
By Theorem \ref{thm:Poschel} , 
$H_0'$ is strictly convex and 
%therefore 
$\Omega$ is a global diffeomorphism on $U$. As given in Definition \ref{def:ResonanceNet}, 
we have two type of resonances: 
\begin{itemize}
\item horizontal and vertical resonant segments of low order
from $\FF_0$, 

\item the resonances from $\FF_n$, $n\geq1$, approaching the Diophantine frequencies. 
\end{itemize}
In both cases we pull back these segments to the frequency space by $\Om$.
For each $k\in\FF_0$ we define 
\[
\II_k=\left\{I\in U_\eta:\Omega(I)\in \SSS_k\right\}.
\]
For each $\om_n\in \DDD_{\eta,\tau}^n$, we define  the resonant segment 
$\II_{k_n}^{\om_n}=\left\{I\in U_\eta:\Omega(I)\in\SSS_{k_n}^{\om_n}\right\}.$
Analogously, we can define
\be \label{eq:action-resonances}
 \bI_n=\left\{I\in U_\eta:\Omega(I)\in \bS_n\right\}\quad ,
\quad \bI=\bigcup_{n\geq 0}\bI_n.
\ee
where $\bS_n$ is the union of Dirichlet resonant segments defined in \eqref{def:UnionResonanceInFreq}.
 
%\subsubsection{Vertical and horizontal resonances far from 
%$\DDD_{\eta,\tau}$: the results of \cite{KaloshinZ12}} The first step 
%to build the cylinder nets is to built it for al horizontal and vertical 
%resonances $\Ga_k$ with $|k|\leq R_0$. Then, we will depart from 
%such resonances to approach Diophantine frequencies using the trees 
%introduced in Step 1. To build the cylinder along this first resonances it is 
%enough to apply \cite{KaloshinZ12}. 

%\begin{theorem}\label{thm:KaloshinZhangCylinders}
 
%\end{theorem}

\subsubsection{Different regimes along resonances}
To prove the existence of NHICs along the resonant segments $\II_{k_n}^{\om_n}$ in the action space, we need to divide the tree of resonant segments in two different regimes. We call these regimes \emph{single resonance zones} and \emph{core of double resonances}. The difference between them is whether the points are very close to intersections of resonant lines  of comparable order or are not close to such intersections.

Fix a generation $n$, a segment $\II_{k_n}^{\om_n}$
and $\theta>0$.
% (to be determined later). 
Define the $\theta$-strong resonances in $\II_{k_n}^{\om_n}$. 
%To this end, we first need to define the order of a double resonance. 

\begin{definition}\label{definition:RatioDR}
Let $k',k''\in (\Z^2\times 0)\times \Z$. Define
\[
\langle k',k''\rangle=\{k\in (\Z^2\times 0)\times \Z:\ k=pk'+qk'',\ p,q\in\Z\}
\]
the lattice generated by integer combinations of $k'$ and $k''$ and call
\[
\NNN(k',k'') = \min \{ |K| : \ K\ne 0,\  K \in \langle k', k'' \rangle \}
\]
the order of  the double resonance.
We also define $ \langle k\rangle=\langle k,0\rangle$.
%We say that the double resonance $\langle L,M\rangle$  has $\kk$-bounded ratio if
%\[
%\frac{|L|}{|\NNN(L,M)|}<\kk\,\,\,\text{ and }\,\,\, \frac{|M|}{|\NNN(L,M)|}<\kk.
%\]
\end{definition}

\begin{definition}\label{definition:StrongWeakDR}
A double resonance $\II_{k_n}^{\om_n}\cap\Ga_{k'}$ is called 
%M I think it worth to call it $\th$-strong
$\th$-strong if there exists 
$k''\in\NNN(\langle k_n,k'\rangle\setminus \langle k_n\rangle)$ such that 
\[
|k'|\leq |k|^{\theta}.
\]
Otherwise, the double resonance is called $\theta$-weak.
%On the contrary, a double resonance $\Ga_k\cap\Ga_{k'}$ 
%is called $\th$-weak if for any $k''\in\NNN(k,k')$ 
%\[
%|k'|> |k|^{1+\theta} 
%\]
\end{definition}

The parameter $\theta$ will be fixed later on. It depends on other parameters 
which appear throughout the proof. The list of all these parameters is given in 
Appendix \ref{sec:Notations}.

Consider the sequence $\{\rr_n\}_{n\geq 1}$ defined in Key Theorem \ref{keythm:ResonancesWithNoTriple}. By Lemma \ref{lemma:AbundanceDoubleResonances}, we can cover the segment $\II_{k_n}^{\om_n}$ by $\mu_n$ balls $B_{\mu_n}(I_i)$ where 
\[
 C\ii\rr_n^{(2\theta-1)\frac{1+2\tau}{3-5\tau}}\leq \mu_n\leq C\rr_n^{(2\theta-1)\frac{1+2\tau}{3-5\tau}}
\]
for some constant $C>0$, and $I_0\in\II_{k_n}^{\om_n}$ is a $\theta$-strong double resonance.

%that is, it satisfies that there exists $k'\in(\ZZ^2\setminus 0)\times\ZZ$ with $|k'|\leq R_%n^\theta$ and  $(\Omega(I_0),1)\cdot k'=(\pa_IH'_0(I_0),1)\cdot k'=0$.

\begin{figure}[thb]
        \center{ \includegraphics[scale=0.9]{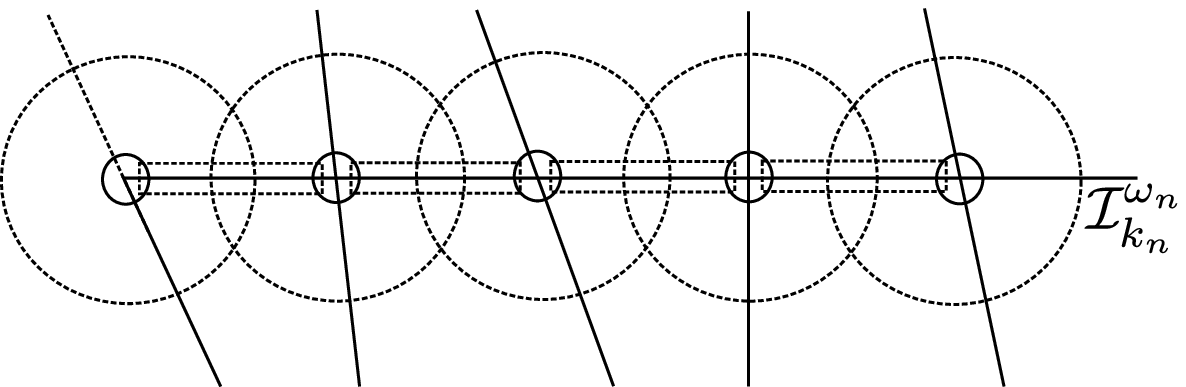}}
        \caption{Different regimes along a resonances:
\newline The larger balls  have radii $\mu_n$ and are centered at $\theta$-strong double resonances. 
\newline The smaller balls represent the core of these double resonances $\Cor_n(k,k')$.
\newline The rectangles are  the single resonance zones $\Tr_n(k_n,k',k'')$ along single resonances.
}
        \label{fig:ResonancesBalls}
\end{figure}

Now we define the single resonance zones and core of the double resonances. The core of the double resonance is defined as  
\begin{equation}\label{def:CoreDR-m}
\Cor_n(k,k')=\TT^2\times B_{C\rr_n^{m}}(I_0)\times \TT, 
%\end{equation}
\ \textup{ where }\ 
%\begin{equation}\label{def:m}
m=\theta+1.
\end{equation}
Lemma \ref{lemma:NoNeighboringTripleThetaStrongIntersections} implies that two consecutive $\theta$-strong double resonances are $\sim\rr^{\frac{2}{3}_\tau\theta+1_\tau}$ far away from each other. Thus the choice of the parameter $m$ implies two consecutive cores of double resonance do not overlap. This implies that there exists a single resonance zone in between.

Let $I_0', I_0''\in \II_{k_n}^{\om_n}$ be two consecutive 
$\theta$-strong double resonances corresponding to 
$(k_n,k')$ and $(k_n,k'')$. Then, the single resonance zone 
in between is given by 
\begin{equation}\label{def:TransitionZone}
 \Tr_n(k_n,k',k'')=\TT^2\times\left\{ I\in  \AAA_n(k_n,k')\cup  \AAA_n(k_n,k''): \dist (I,\II_{k_n}^{\om_n})\leq  c\rr_n^{m+\frac{1}{3}\theta+1} \right\}\times\TT,
\end{equation}
where $\AAA_n(k_n,k')$ is the annulus around the core of the double resonance, 
\begin{equation}\label{def:Transition:Annulus}
 \AAA_n(k_n,k')=\left\{I\in U: c\rr_n^{m}\leq| I-I_0|\leq \mu_n\right\}.
\end{equation}
The parameter  $c$ is chosen so that $c>C$ to have overlap with \eqref{def:CoreDR-m}.

In each of these zones we will obtain different NHICs. 
\begin{itemize}
\item $\Cor_n(k,k')$: the dynamics  is essentially the same as in the double resonances in \cite{KaloshinZ12}. Nevertheless, it turns out that in this regime, the systems are not close to a  a mechanical system as in \cite{KaloshinZ12}. Thus, we have to deal with  more general 2 degree of freedom Hamiltonian systems. This requires modifying the proof in \cite{KaloshinZ12}. See Section \ref{sec:NF:DR} and Appendix \ref{app:NonMechanicalAtDR}.

\item $ \Tr_n(k_n,k',k'')$: the dynamics is essentially the same as in the single resonance zones studied in \cite{BernardKZ11, KaloshinZ12}. 
%On a technical level, we do not use  single resonant normal forms as in \cite{BernardKZ11}. Instead, the whole resonant segment is covered by double resonance normal forms.}

%Note that here we do not need the purely single resonance regime which is used in \cite{KaloshinZ12}. The reason is that the double resonances are so abundant that the whole resonance can be covered by core of double resonances and transition zones. 
\end{itemize}
In the next two sections we deal with the two different regimes. In both cases we apply a normal form procedure to the Hamiltonian \eqref{def:HamAfterMolification}. The normal form leads to new Hamiltonian  which now can be treated as a perturbation of a new first order and is ``easier'' to analyze.

\subsection{Analysis of the  single resonance zones}
We consider a  single resonance zone \eqref{def:TransitionZone}, that is  a tubular neighborhood  along the resonant segment $\II_{k_n}^{\om_n}$ such that it overlaps the core of two consecutive double resonances (see Figure \ref{fig:ResonancesBalls}).  Then, we  prove that in such neighborhood there is a NHIC. It suffices to
\begin{itemize}
\item perform a change of coordinates and show that our system is a small perturbation of a system which  depends on actions and on the slow angle $k_n\cdot (\varphi,t)$ only.
\item perturb the Hamiltonian so that this first order has sufficient nondegeneracy to imply the existence of a NHIC (see Key Theorem \ref{keythm:Transition:Deformation}).
\item use an isolating block argument, as in \cite{BernardKZ11}, to prove that the full system has a NHIC (see  Key Theorem \ref{keythm:Transition:IsolatingBlock}).

\end{itemize}

\subsubsection{Normal forms along single resonances}\label{sec:outline:transition:normalform}
%First, we perform a preliminar change of coordinates which separates the fast and slow angles. Thus, we define $\varphi^s=k_n\cdot\varphi$ and $\varphi^f=\varphi_2$ (if the change is singular, take $\varphi^f=\varphi_1$) and conjugated actions $I^s$ and $I^f$ so that the change of coordinates is symplectic. To simplify notation, we still call $N'_\eps$ to the Hamiltonian \eqref{def:HamAfterMolification} after the change of coordinates. Note that, this change of coordinates changes the $\CCC^1$ (and higher regularity) norms of  $N'_\eps$ since $k_n$ is large with respect to $\rr_n$. 

The next Theorem shows the existence of the normal form. A more precise statement is done in Section \ref{sec:NF:Transition}. Recall that $m$ has been defined in \eqref{def:CoreDR-m} (see also Appendix \ref{sec:Notations} for the values of the considered constants) and $\theta$ is used in the definition of 
$\theta$-strong double resonances, which turn is used
in the definition of the single resonant zones (\ref{def:TransitionZone}). 
\begin{keytheorem}\label{keythm:Transition:NormalForm}
Fix  $q=252$,  $\theta=3q+1$, $m=\theta+1$ and  $r\geq m+5q$. Let $N'=(H_\eps+\Delta H^\mol)\circ \Phi_\Pos$ be the Hamiltonian defined in \eqref{def:HamAfterMolification}, $k_n\in\FF_n$ and $\Tr_n=\Tr_n(k_n,k',k'')$ be a single 
resonance zone defined in \eqref{def:TransitionZone}. 
Then, there exists constants $C,D\geq 1$ independent 
of $\rr_n$ and a symplectic change of coordinates $\Phi_n:=\Phi^{k_n}$ such that 
\begin{equation}\label{def:Ham:Transition:NF}
\begin{split}
\HH^{k_n}(\psi,J,t)&= N'\circ \Phi_n(\psi,J,t)\\
&=\HH_0^{k_n}(J)+\ZZZ^{k_n}(k_n\cdot(\psi,t),J)+\RRR^{k_n}(\psi,J,t)
\end{split}
\end{equation}
with the following properties
\begin{itemize}
%\item The angle $\psi^s$  is close to $\varphi\cdot k_n$: $\|\psi-k_n\cdot\Phi^{(\varphi,t)}_n(\psi,J,t)\|_{\CCC^0(\Phi_n(\Tr_n))}\leq C\rr^m$.
\item $\HH_0^{k_n}$ is strictly convex: $D\ii\Id\leq \pa_I^2\HH_0^{k_n}\leq D\,\Id$.
\item $\ZZZ^{k_n}$ and $\RRR^{k_n}$ satisfy
\[
%\begin{split}
\|\ZZZ^{k_n}\|_{\CCC^{2}(\Phi_n\ii(\Tr_n))}\leq C\rr_n^{\frac{4}{3}r-2m}, 
\qquad 
\|\RRR^{k_n}\|_{\CCC^{2}(\Phi_n\ii(\Tr_n))}\leq C\rr_n^{q(r+1)-2m}.
%\end{split} 
\]
%for any $0\leq \kk\leq r-b=6+2\tau+3m$.
\item $\Phi_n$ satisfies $\|\Phi_n-\mathrm{Id}\|_{\CCC^1}\leq C\rr^{2r-2m}$.
\end{itemize}
\end{keytheorem}
%This theorem is proven  in Section \ref{sec:NF:Transition}. 

%The change of coordinates $\Phi_n$ is just the composition of 
%the rescaling $J=\rr^m I$ and a canonical transformation, see 
%Section \ref{sec:NF:Transition} the proof and for more details.

\subsubsection{NHICs for the truncated system}
Thanks to the normal form given in Key Theorem \ref{keythm:Transition:NormalForm}, one can look at the Hamiltonian  \eqref{def:Ham:Transition:NF} as a perturbation of the truncated Hamiltonian $\HH_0^{k_n}+\ZZZ^{k_n}$. Now we deform the Hamiltonian \eqref{def:HamAfterMolification} by a small $\CCC^r$ perturbation so that the truncated Hamiltonian  has a NHIC. Then, in Section \ref{sec:outline:transition:cylinder} we show that it persists for Hamiltonian \eqref{def:Ham:Transition:NF}.

To analyze the truncated system, we consider slow-fast variables. The slow-fast angles are defined as 
\begin{equation}\label{def:ChangeToFastSlow:StrongWeak:0}
 \left(\begin{matrix}\phi^s\\\phi^f\\t\end{matrix}\right)=
 \wt A\left(\begin{matrix}\phi_1\\\phi_2\\t\end{matrix}\right)\,\,
 \text{ with }\wt A=\left(\begin{matrix}k^*\\ e_2\\ e_3\end{matrix}\right) 
\end{equation}
where $e_i$ are the standard coordinate vectors (if this change is singular, 
replace some of the coordinate vectors). To have a symplectic 
change of coordinates, we perform the change of coordinates  
\begin{equation}\label{def:ChangeToFastSlow:StrongWeak2:0}
\left(\begin{matrix}J^s\\ J^f\\ E\end{matrix}\right)=\wt A^{-T}
\left(\begin{matrix}J_1\\J_2\\E\end{matrix}\right)
\end{equation}
to the conjugate actions. $E$ is the variable conjugate to time, which 
is not modified  when the change \eqref{def:ChangeToFastSlow:StrongWeak2} is performed. 
We call $L^\sf$ to this {\it symplectic} linear change of variables 
and denote $J=(J^s,J^f)$. As shown in Section \ref{sec:NormalForm:DR:StrongWeak} it satisfies $\|L^\sf\|,\|(L^\sf)\ii\|\leq C\rr^{-\frac{1}{3}-\tau}$. After this change of coordinates, 
we have
\begin{equation}\label{def:Ham:Single:slowfast}
\wh\HH^{k_n}=N'\circ \Phi_n\circ L^\sf=\HH^{k_n}\circ L^\sf=\wh\HH^{k_n}_0+\wh\ZZZ^{k_n}+\wh\RRR^{k_n}
\end{equation}
where $\HH^{k_n}$ is the Hamiltonian \eqref{def:Ham:Transition:NF}. Now, the truncated system is just 
\begin{equation}\label{def:SingleResonance:TruncatedHamiltonian}
\HH^\trunc(\psi,J)= \wh\HH_0^{k_n}(J)+\wh\ZZZ^{k_n}(\psi^s,J).
\end{equation}
and the resonance $\II_{k_n}^{\om_n}$ can be parameterized as a graph $J^s=J^s_*(J^f)$ with 
\be \label{eq:resonance-parametrization}
\II_{k_n}^{\om_n}=\{(J^s(J^f),J^f):\ J^f\in 
[b_{k_n}^-, b_{k_n}^+]\}.
\ee

\begin{keytheorem}\label{keythm:Transition:Deformation}
Let $d=14$. With the notation of Key Theorem \ref{keythm:Transition:NormalForm}, there exists  $C,D>1$ and a $\CCC^r$ function  $\Delta H^\sr$  supported in the $\Tr_n$ and satisfying 
$\|\Delta H^\sr\|_{\CCC^r(\Tr_n)}\leq C\rr_n^{1+4\tau}$
such that the conclusion of Key Theorem \ref{keythm:Transition:NormalForm} applied to 
$$
N''=(H_\eps+\Delta H^\mol+\Delta H^\sr)\circ\Phi_\Pos
$$ 
holds along with the following properties.  Consider  
$b_{k_n}^-<b_{k_n}^+$, so that 
$J^f\in [b_{k_n}^-, b_{k_n}^+]$ parameterizes the single resonance $\II_{k_n}^{\om_n}\,\cap\, \Tr_n$. Then, there exists a sequence of action values $\{J_i^f\}_{i=1}^N\subset [b_{k_n}^-, b_{k_n}^+]$, such that
  for each  $J^f_*\in[b_{k_n}^-, b_{k_n}^+]$, the Hamiltonian  $\HH^\trunc(\psi^s,J^s,J_*^f)$, defined by 
(\ref{def:Ham:Single:slowfast}) with $N'$ replaced by $N''$, 
satisfies the following conditions: 
\begin{itemize}
\item has one degree of freedom, is Tonelli,
\item for each $J^f_*\in (J_i^f, J_{i+1}^f)$, there is a unique critical point $(\psi_*^s, J_*^s)$ (as defined in 
\eqref{eq:min-crit-point})  which is of saddle type. The same is true for $[b_{k_n}^-, J_1^f)$ and $(J_N^f,b_{k_n}^+]$.
\item  for $J^f=J_i^f$, there are exacly two critical points $(\psi_{*}^s, J_{*}^s)$ and $(\bar\psi_{*}^s, \bar J_{*}^s)$ of saddle type.
\item In both cases, eigenvalues of the hyperbolic periodic orbits  $\la$, $-\la$, where $$ \la\geq C\rr_n^{\frac{dr+1}{2}}.$$
\end{itemize}
\end{keytheorem}

The family of critical points discussed here  
are minimal in a proper sense (see (\ref{eq:min-crit-point})). 
Notice also that properties such as hyperbolicity of 
periodic orbits are independent of coordinate system
and, therefore, apply to the Hamiltonian $N''$.

We fix $d=14$ throughout the proof. Nevertheless, we keep the notation $d$ to simplify the exposition. Recall that in  Appendix \ref{sec:Notations} we list all the involved constants and the relations between them.

Key Theorem \ref{keythm:Transition:Deformation} implies that 
the truncated Hamiltonian $\HH^\trunc$, given by  \eqref{def:SingleResonance:TruncatedHamiltonian}, has a sequence of NHICs
\begin{equation}\label{def:Cylinder:truncated}
\left\{(\psi,J,t):(\psi^s, J^s)=(\psi_*^s(J^f), J_*^s(J^f)), J^f\in (J_i^f,J_{i+1}^f)\right\}. 
\end{equation}
Thanks to hyperbolicity these NHICs can be actually slightly extended to intervals $[J_i^f-\de,J_{i+1}^f+\de]$ for small ($\rr_n$-dependent) $\de>0$. For the first and last one can take the intervals $[b_{k_n}^-, J_1^f+\de]$ and $[J_N^f-\de,b_{k_n}^+]$  respectively.

\subsubsection{NHICs along single resonances}
\label{sec:outline:transition:cylinder}
Now we show that the NHIC \eqref{def:Cylinder:truncated} persists for the Hamiltonian \eqref{def:Ham:Transition:NF}. The normal hyperbolicity of the NHIC \eqref{def:Cylinder:truncated} is very weak since the eigenvalues of the saddle of $\HH^\trunc(\psi^s,J^s,J_*^f)$   are of size $\rr_n^{(dr+1)/2}$. The weak hyperbolicity is the reason why we need to reduce the size of the perturbation $\RRR$ to $\|\RRR\|_{\CCC^2}\leq\rr_n^{q(r+1)}$ to ensure that these NHICs persist for the full system. We state the persistence of the NHICs in the intervals $[J_i^f-\de,J_{i+1}^f+\de]$ to simplify the statement. The same Theorem holds for the first and last interval.

\begin{keytheorem}\label{keythm:Transition:IsolatingBlock}
With notations introduced in Key Theorem \ref{keythm:Transition:Deformation} fix any  
$i=1,\ldots N$, then there exists a $\CCC^1$ map
\[
(\Psi_i^s, \JJ_i^s)(\psi^{f},J^{f},t):\T\times 
[J_i^f-\small{\frac \de 2},J_{i+1}^f+\small{\frac \de 2}]
\times\TT  \to \T \times \R
\]
such that the NHIC
\[
\CCC_{k_n}^{\om_n,i}=\{ (\psi^s, J^s)=(\Psi_i^s, \JJ_i^s)
(\psi^f, J^f,t); \quad
(\psi^f,J^f,t)\in\T\times [J_i^f-\small{\frac \de 2},J_{i+1}^f+\small{\frac \de 2}]\times\TT \}
\]
is weakly invariant with respect to the vector field associated 
to the Hamiltonian \eqref{def:Ham:Single:slowfast}, in the sense 
that the  vector field is tangent to $\CCC_{k_n}^{\om_n,i}$. Moreover, $\CCC_{k_n}^{\om_n,i}$ is contained in the set
\beq
\beal
V_i:=\big\{&(\psi, J); J^{f}\in  [J_i^f-\small{\frac \de 2},J_{i+1}^f+\small{\frac \de 2}], \\
&|\psi^s-\psi^s_*(J^{f})|\leq
C\rr_n^{ m-1+2dr},
\quad |J^s-J^s_*(J^f)|\le
C\rr_n^{ 2m-1+2dr}
\big\},
\enal
\eeq 
for some constant $C>0$ independent of $\rr_n$
and contains all the full orbits  of \eqref{def:Ham:Transition:NF}
contained in $V_i$. We also have the estimates
\[
%\begin{split}
|\Psi_i^s(\psi^f,J^f)-\psi^s_*( J^f)|\leq
C\rr_n^{m-1+5dr}\quad ,\quad 
| J^s_i(\psi^f,J^f,t)- J^s_*(J^f)|\leq
C\rr_n^{2m-1+5dr},
%\end{split}
\]
and  
\[ \left|\frac{\partial\Psi^s}{\partial  J^f}\right| \leq C\rr_n^{\frac{3}{2}dr-\frac{3m+1}{2}}
,
\quad \left|\frac{\partial \Psi^s}{\partial\psi^f}
\right| \leq C
\rr_n^{\frac{3}{2}dr-\frac{m+1}{2}}.
\]
\end{keytheorem}
This theorem is proved in Section \ref{sec:NHIC}. Notice that

\begin{remark}
The eigendirections of the minimal saddle $(\psi_*^s,J_*^s)$ are twisted in the phase space. Thus, the localizations on the angle and action components are not independent.
\end{remark}

To drift along these cylinders later in Section \ref{sec:shadowing} we need to have a good localization of these cylinders in P\"oschel coordinates. This can be deduced from Key Theorems \ref{keythm:Transition:NormalForm} and  \ref{keythm:Transition:IsolatingBlock}. To simplify the statement we consider again the slow-fast variables. Call $L^{\sf}$ the change given by \eqref{def:ChangeToFastSlow:StrongWeak:0}-- \eqref{def:ChangeToFastSlow:StrongWeak2:0}. The next 
corollary gives estimates for the cylinder for the Hamiltonian 
$$
(H_\eps+\Delta H^\mol+\Delta H^\sr)\circ\Phi_\Pos\circ L^\sf.
$$ 
%Note that since it is stated in slow-fast variables, 
Note that the resonant segment $\II_{k_n}^{\om_n}$ can be 
still parameterized by $I^s=I^s_*(I^f)$, $I^f\in [a_{k_n}^-,a_{k_n}^+]$, where this interval is just the rescaling by $J=\rr^m I$ of the former interval $[b_{k_n}^-,b_{k_n}^+]$. We call $\{I_i\}_{i=1}^N$ the bifurcation values.

\begin{corollary}\label{coro:cylinderPoschelCoordinates}
Take any  $i=1,\ldots N$, where $N$ has been introduced in Key Theorem \ref{keythm:Transition:Deformation}.  There exist  $\CCC^1$ maps
\[
(\wt\Psi_i^s, \wt\JJ_i^s)(\varphi^{f},I^{f},t):\T\times \left[I_i^f-
\small{\frac \de 2},I_{i+1}^f+\small{\frac \de 2}\right]
\times\TT  \to \T \times \R
\]
such that the NHICs
\[
\wt\CCC_{k_n}^{\om_n,i}=\left\{ (\varphi^s, I^s)=
(\Psi_i^s, \JJ_i^s)
(\varphi^f, I^f,t); \quad
(\varphi^f,I^f,t)\in\T\times \left[I_i^f-\small{\frac \de 2},
I_{i+1}^f+
\small{\frac \de 2}\right]\times\TT\right \}
\]
are weakly invariant with respect to the vector field associated 
to the Hamiltonian 
$$
(H_\eps+\Delta H^\mol+\Delta H^\sr)\circ\Phi_\Pos\circ L^\sf,
$$ 
in the sense that the  vector field is tangent to $\wt\CCC_{k_n}^{\om_n,i}$. 
The NHICs $\wt\CCC_{k_n}^{\om_n,i}$ are contained in the set
\beq
\beal
V_i:=& \left\{(\varphi, I); I^{f}\in  \left[I_i^f-\small{\frac \de 2},I_{i+1}^f+\small{\frac \de 2}\right],\right. \\ & \left.
|\varphi^s-\psi^s_*(J^{f})|\leq
C\rr_n^{2r-2 m},
\quad |I^s-I^s_*(J^f)|\le
C\rr_n^{2r-2m}
\right\},
\enal
\eeq 
for some constant $C>0$ independent of $\rr$
and they contains all the full orbits  of \eqref{def:Ham:Transition:NF}
contained in $V_i$. \end{corollary}

Note that in this corollary the slow angle and action localization worsens with respect to Key Theorem \ref{keythm:Transition:IsolatingBlock}. The reason is that we perform the change of coordinates $\Phi_n$ given by Key Theorem \ref{keythm:Transition:NormalForm}. Note that nevertheless we still have slow angle localization since $2r-2m\geq 10q$ (see Apendix \ref{sec:Notations}).

%\textbf{Here, shall we add a corollary of existence of the cylinder in the original coordinates? But then, in the original coordinates we don't have a good location of the isolating block}

\subsection{Analysis of the core of double resonances}
\subsubsection{Normal forms in the core of double resonances}\label{sec:outline:DR:normalform}
To study the core of the double resonances we need to proceed as in the single resonance zone analysis and first perform a normal form.  We follow \cite{Bounemoura10} (see Section \ref{sec:NF:DR}). This Theorem involves the constants  $\theta$, $m$ and $r$ whose values can be found in Appendix \ref{sec:Notations}.

\begin{keytheorem}\label{keythm:CoreDR:NormalForm}
Consider the constants $\theta$, $m$ and $r$ as in Key Theorem \ref{keythm:Transition:NormalForm}. Let 
\[
N''=(H_\eps+\Delta H^\mol+\Delta H^\sr)\circ \Phi_{\Pos}
\]
be the Hamiltonian from Key Theorem \ref{keythm:Transition:Deformation}, $k_n\in\FF_n$, $\rr_n$ defined in Key Theorem \ref{keythm:ResonancesWithNoTriple}. Let $(k_n,k')$ form a $\theta$-strong double resonance, i.e. $|k'|\leq |k|^\theta$ and $\Cor_n=\Cor_n(k_n,k')$ be the core of the double resonance \eqref{def:CoreDR-m}. 

Then, there exists constants  $C,D>1$ independent of $\rr_n$, a $\CCC^r$ function $\Delta H^\dr$  supported in a $\rr^m$ neighborhood of $\Cor_n$ satisfying  
\[\|\Delta H^\dr\|_{\CCC^r(\Cor_n)}\leq C\rr_n \ \ \text{ and }\ \ 
\|\Delta H^\dr\|_{\CCC^2(\Cor_n)}\leq C\rr_n^{2q(r+1)}\]
 and a  symplectic change of coordinates $\Phi_n=\Phi^{k_n, k'}$ satisfying
\[
 \|\Phi_n-\Id\|_{\CCC^1}\leq C\rr^{2r-3m-\frac{2}{3}_\tau (1+\theta)},
\]
such that
\begin{equation}\label{def:Ham:DR:NF}
\begin{split}
\HH^{k_n,k'}(\psi,J,t)&= (H_\eps+\Delta H^\mol+\Delta H^\sr+\Delta H^\dr)\circ \Phi_{\Pos}\circ \Phi_n(\psi,J,t)\\
&=\HH_0^{k_n,k'}(J)+\ZZZ^{k_n,k'}(\psi,J)=: N^{\dr} \circ \Phi_n(\psi,J,t).
\end{split}
\end{equation}
where $\HH_0^{k_n,k'}$ and $\ZZZ^{k_n,k'}$ satisfy
\be \label{eq:DR-Hamiltonians}
\beal
 D\ii\rr_n^{\frac{2}{3}_\tau (1+\theta)}\Id\leq \pa_J^2 \HH_0^{k_n,k'}&\leq D\rr_n^{-\frac{2}{3}_\tau (1+\theta)}\Id,\\
\\
 \left\|\ZZZ^{k_n,k'}\right\|_{\CCC^{r}(\Cor_n)}&\leq C\rr_n^{2r+\frac{1}{3}_\tau (r-2)-3m-\frac{2}{3}_\tau (1+\theta)}
\enal
\ee
where the first set of inequalities is in the sense of quadratic forms.

Moreover, \eqref{def:Ham:DR:NF} satisfies hypotheses [H0]-[H3] and  [A1]-[A2], defined in the next section.
\end{keytheorem}
Our choice of parameters (see Appendix \ref{sec:Notations}) implies that the whole Hamiltonian $\HH^{k_n,k'}$ is convex with respect to the actions.

\begin{comment}
\begin{keytheorem}\label{keythm:CoreDR:NormalForm:OLD}
OLD VERSION
Let $N''=(H_\eps+\Delta H^\mol+\Delta^\sr)\circ \Phi_{\Pos}$ be the Hamiltonian from Key Theorem \ref{keythm:Transition:Deformation}, $k_n\in\FF_n$, $\rr_n$ defined in Key Theorem \ref{keythm:ResonancesWithNoTriple}. Let $(k_n,k')$ form a $\theta$-strong double resonance, i.e. $|k'|\leq |k|^\theta$ and $\Cor_n=\Cor_n(k_n,k')$ be the core of the double resonance \eqref{def:CoreDR-m}. 

Then, there exists  a $\CCC^r$ function $\Delta H^\dr$  supported in $\Cor_n$ satisfying  $\|\Delta H^\dr\|_{\CCC^r(\Cor_n)}\leq C\rr_n^{18(r+1)}$, with $C,D>1$ independent of $\rr_n$ and a conformal symplectic change of coordinates $\Phi_n=\Phi^{k_n, k'}$ such that
\begin{equation}\label{def:Ham:DR:NF}
\begin{split}
\HH^{k_n,k'}(\psi,J,t)&= (H_\eps+\Delta H^\mol+\Delta H^\sr+\Delta H^\dr)\circ \Phi_{\Pos}\circ \Phi_n(\psi,J,t)\\
&=\HH_0^{k_n,k'}(J)+\ZZZ^{k_n,k'}(\psi,J)
\end{split}
\end{equation}
where $\HH_0^{k_n,k'}$ satisfies
\[
 D\ii\rr_n^{m+1}\Id\leq \pa_J^2 \HH_0^{k_n,k'}\leq D\rr_n^{m-1}\Id,
\]
 in the sense of quadratic forms, and $\ZZZ^{k_n,k'}$  satisfies
\[
 \left\|\ZZZ^{k_n,k'}\right\|_{\CCC^{b}(\Cor_n)}\leq C\rr_n^{m+(1+2\tau)b/3}.
\]
where $b=r-6-2\tau-m$.

Moreover, \eqref{def:Ham:DR:NF} satisfies the generic hypotheses [H0]-[H3] and  [A1]-[A2] defined in the next section.
\end{keytheorem}
\end{comment}
\begin{remark}
 The perturbation  $\Delta H^\dr$ is nonzero in the overlapping zone between the core of the double resonances $\Cor_n(k_n,k')$ and 
the single resonance zones $\Tr_n(k_n,k',k'')$. Nevertheless, using 
smallness of the $\CCC^2$ norm, one can see that Key Theorems \ref{keythm:Transition:NormalForm} and \ref{keythm:Transition:IsolatingBlock} remain true. 
\end{remark}
The new Hamiltonian \eqref{def:Ham:DR:NF} is 
%$t$-inedependent 
autonomous and, therefore, has  two degrees of freedom.
In particular, its energy is conserved. We study dynamics within 
energy levels
\[
\SSS_E=\left\{(\psi,J):\HH^{k_n,k'}(\psi,J)=E\right\}.
\]

\subsubsection{Cylinders in the core of double resonances}\label{sec:outline:cylinders}

Fix $D>1$ and a small $\rho>0$. Consider the class 
of Hamiltonians
\begin{equation}\label{def:Ham2dofNonMech}
 \HH(\psi,J)=\HH_0(J)+\ZZZ(\psi,J)\,\qquad \psi\in\TT^2,J\in\RR^2
\end{equation}
satisfying (\ref{eq:DR-Hamiltonians}). Notice that all these Hamiltonians 
are Tonelli. 

A property of a Tonelli Hamiltonian is Ma$\tilde{\mbox n}\acute{\mbox e}$'s $\CCC^r$ generic if there is a $\CCC^r$ generic set of potential $\UU\in\CCC^r(\TT^2)$ such that $\HH_\UU=\HH+\UU$ satisfies this property. 

Let $\al_\HH(0)$ be the critical energy level defined in Appendix \ref{app:NonMechanicalAtDR}. Shifting $\HH$ by a constant assume $\al_\HH(0)=0$. As in \cite{KaloshinZ12} we study two regimes:
\begin{itemize}
\item {\it high energy:} $E\in (E_0,E_*)$ for small $E_0>0$ and  large $E_*$,
\item {\it low energy:} $E\in (0, 2E_0)$.
\end{itemize}

\begin{comment}The perturbation $\Delta H^\dr$ makes the Hamiltonian $\HH^{k_n,k'}$ nondegenerate. This implies that we can prove the existence of several hyperbolic objects in the different levels of energies. Recall that $\HH_0^{k_n,k'}$ is convex and, making a translation if necessary, we have that $\nabla\HH_0^{k_n,k'}(0,0)=0$, $\HH_0^{k_n,k'}(\psi,J)\geq 0$ and $\HH_0^{k_n,k'}(0,0)= 0$.

The point $(\psi,J)=(0,0)$ is a critical point. In Appendix \ref{app:NonMechanicalAtDR}, we show that due to nondegeneracy given by $\Delta H^\dr$ this point 
$(0,0)$ is hyperbolic with distinct eigenvalues $-\la_1<-\la_2<0<\la_2<\la_1$.  As done in \cite{KaloshinZ12}, 
we split the study into two parts. First we deal with case of high energy, 
i.e. we study the dynamics in with $E\in (E_0,E_*)$ with $E_0>0$ small 
and $E_*$ large. Then, we analyze the low energy regime, i.e. $E\in [0,e_0]$. 
We take $e_0>2E_0$ so that we have some overlap between the two regions.
\end{comment}

\paragraph{Ma$\tilde{\mbox n}\acute{\mbox e}$'s generic properties of the high energy regime}
In the high energy regime we obtain similar results to the ones obtained in \cite{KaloshinZ12}. 
Nevertheless note that now we do not deal with a mechanical system as in that paper since 
now the potential also depends on the actions. The proof in \cite{KaloshinZ12} relies on identifying 
certain  periodic orbits of the Hamiltonian with orbits of the geodesic flow on $\TT^2$ given by 
the Mapertuis metric associated to the mechanical system restricted to a fixed (regular) energy level. 
This is not possible in the present case. However, following \cite{ContrerasIPP98}, one can obtain 
a Finsler metric $\de_E$ associated to the Hamiltonian $\HH$ in the regular energy levels $\SSS_E$. 
We call $\ell_E$ the length associated to the metric $\de_E$. This metric allows us to obtain similar 
results to those in \cite{KaloshinZ12} (see Appendix \ref{app:NonMechanicalAtDR}).

We fix a homology class $h\in H^1(\TT^2,\ZZ)$ and we look for 
the shortest  geodesics in this homology class with respect to 
the length  $\ell_E$.
\begin{keytheorem}\label{keythm:DR:HighEnergy}
For any Ma$\tilde{\mbox n}\acute{\mbox e}$'s generic Hamiltonian $\HH$ of the form \eqref{def:Ham2dofNonMech}, satisfying (\ref{eq:DR-Hamiltonians}), we have  
\begin{itemize}
\item[{[H1]}] On the critical energy level $\SSS_{\al_\HH(0)}$ there is a unique minimal 
critical point of saddle type with distinct eigenvalues $-\la_1<-\la_2<0<\la_2<\la_1$.
\end{itemize}
Fix $E_0=E^*_0(\la_1,\la_2,\|\HH\|_{\CCC^2})>0$ small, 
$E_*>0$ large, and a homology class $h\in H^1(\TT^2,\ZZ)$. 
Then, there exist a monotone sequence of  energies 
$\{E_j\}_{j=1}^N\subset (E_0,E_*)$ such that 
\begin{itemize}
\item[{[H2]}] For each $E\in\cup_{j=0}^N (E_j,E_{j+1})\cup (E_N,E_*)$, there exists a unique 
shortest geodesic $\ga_h^E$ with respect to the length $\ell_E$. Moreover, it is non-degenerate 
in the sense of Morse, i. e. the corresponding periodic orbit is hyperbolic. 
\item[{[H3]}] For $E=E_j$, there exists two shortest geodesics $\ga_h^E$ and $\ol\ga_h^E$ with 
respect to the length $\ell_E$. Moreover, they are non-degenerate in the sense of Morse, i. e. 
the corresponding periodic orbit is hyperbolic. Moreover they satisfy
\[
 \left.\frac{d(\ell_E(\ga_h^E))}{dE}\right|_{E=E_j}\neq\left.\frac{d(\ell_E(\ol\ga_h^E))}{dE}\right|_{E=E_j}.
\]
\end{itemize}
\end{keytheorem}
In Appendix \ref{app:NonMechanicalAtDR}, this Theorem is reduced to a similar theorem in \cite{KaloshinZ12}. 
Thanks to the hyperbolicity of the periodic orbits,  $\ga_E^h$ has a unique smooth local continuation to 
a energy interval $[E_j-\de,E_{j+1}+\de]$ for small $\de>0$. For the energies outside of $[E_j,E_{j+1}]$ 
the orbit is still hyperbolic but it is not the shortest geodesic anymore. We consider the union 
\be \label{eq:DR-cylinder}
 \MM_h^{E_j,E_{j+1}}=\bigcup_{E\in [E_j-\de,E_{j+1}+\de]} \ga_h^E,\,\qquad i=1,\ldots, N.
\ee
For the boundary intervals we take $[E_0, E_1+\de]$ and $[E_N-\de,E_*]$. The hyperbolicity of $\ga_h^E$ 
implies that $ \MM_h^{E_j,E_{j+1}}$ is a NHIC of $\HH$.

\paragraph{Ma$\tilde{\mbox n}\acute{\mbox e}$'s generic properties of the low energy regime}
Now we deal with the low energy regime, that is, energy close to  critical energy $E=0$. Let $(\psi^*,J^*)$ be the minimal saddle which  belongs to $\SSS_0$. Without loss of generality, assume that $(\psi^*,J^*)=(0,0)$. The Finsler metric $\de_0$ is singular at $\psi=0$. Let $\ga_h^0$ be the shortest geodesic in the homology $h$ with respect to the Finsler metric induced on $\TT^2$ from $\SSS_0$. 
% In Appendix \ref{app:NonMechanicalAtDR} is proved the following. 
The result for mechanical systems was proven by Mather in \cite{Mather03}.

\begin{lemma}\label{lemma:SingularGeodesics}
The geodesic $\ga_h^0$ satisfies one of the following statements:
\begin{itemize}
 \item $0\not\in\ga_h^0$ and $\ga_h^0$ is not self-intersecting. We call such homology  class $h$ simple non-critical and 
the corresponding geodesic a simple loop.
\item $0\in\ga_h^0$ and $\ga_h^0$ is self-intersecting. We call such homology class $h$ 
non-simple and the corresponding geodesic a non-simple.
\item $0\in\ga_h^0$ and $\ga_h^0$ is a regular geodesic. We call such homology class $h$ simple critical.
\end{itemize}
 \end{lemma}

As explained in \cite{KaloshinZ12} in the mechanical systems setting, there are at least 
two homology classes which are simple non-critical, that is, with associated geodesics 
which contain the saddle and are non-intersecting. 
In the phase space, such geodesics are homoclinic orbits to the saddle 
at $(0,0)$. The next lemma describes what happens generically in 
the self-intersecting case. It was proved by Mather in the mechanical systems setting 
in \cite{Mather08} (see also Appendix \ref{app:NonMechanicalAtDR}).

%Recall that a property of a Tonelli Hamiltonian is {\it Ma$\tilde{\mbox n}\acute{\mbox e}$'s  $\CCC^r$ generic } if there is a $\CCC^r$ generic set of potentials $U$ on $\T^2$ such that $\HH_U = \HH + U$ satisfies this property.

\begin{lemma}\label{lemma:SelfIntersectingGeodesics}
Let $\HH$ satisfy (\ref{eq:DR-Hamiltonians}) and condition [H0] of 
Key Theorem \ref{keythm:DR:HighEnergy}. 
Let $\de_0$ be the Jacobi-Finsler metric on the critical energy level.
Consider $h\in H_1(\TT^2,\ZZ)$ such that the associated geodesic $\ga_h^0$, 
of  $\de_0$, is  non-simple. Then for a Ma$\tilde{\mbox n}\acute{\mbox e}$ generic 
Hamiltonian $\HH$, there exist homology classes $h_1,h_2\in H_1(\TT^2,\ZZ)$, 
integers $n_1,n_2$ with  $h=n_1h_1+n_2h_2$ such that the associated geodesics 
$\ga_{h_1}^0$ and $\ga_{h_2}^0$ are simple critical.
\end{lemma}

Recall that the condition [H0]  is open. 

For $E>0$, the geodesic has no self-intersections. Therefore, this lemma implies that 
there is a unique way of writing $\ga_{h}^0$ as a concatenation of $\ga_{h_1}^0$ and 
$\ga_{h_2}^0$. Denote $n=n_1+n_2$. 

\begin{lemma}
There exists a sequence $\sigma=(\sigma_1,\ldots , \sigma_n)\in \{1,2\}^n$, unique up 
to cyclical permutation, such that 
\[
 \ga_{h}^0=\ga_{h_{\sigma_1}}^0\ast\ga_{h_{\sigma_2}}^0
\ldots\ast\ga_{h_{\sigma_n}}^0.
\]
\end{lemma}

These lemmas describe what happens at the critical energy level. 
From this information we describe what happens for small energy. 
Assume  the Hamiltonian \eqref{def:Ham2dofNonMech} 
satisfies [H0]. In a neighborhood of the minimal saddle $(0,0)$ there 
is a local system of coordinates $(s,u)=(s_1,s_2,u_1,u_2)$ such that 
the $u_i$ axes correspond to the 
eigendirections of $\la_i$ and  the $s_i$ axes correspond to the eigendirections of $-\la_i$  for $i = 1, 2$.

Now we need to impose some assumptions on the geodesics 
$\ga_h^0$. They are different depending on the properties stated in 
Lemma \ref{lemma:SingularGeodesics} satisfying the geodesic.
 We denote by $\ga^+=\ga_{h,+}^0$ be a homoclinic to $(0,0)$ of homology $h$ and 
 $\ga^-=\ga_{h,-}^0$ the symmetric one with respect to the involution $I\mapsto -I$, $t\mapsto -t$.
Assume that
\begin{itemize}
\item [[A1]] The homoclinics $\ga^\pm$ are not tangent to the $u_2$ or $s_2$ axis at $(0,0)$. 
\end{itemize}
Since the eigenvalues are all different, this implies that they are tangent to the $u_1$ and $s_1$ axis. Assume without loss of generality, that $\ga^+$ approaches $(0,0)$  along $s_1>0$ in forward time and along $u_1>0$ in backward time. Using reversibility, $\ga^-$ approaches $(0,0)$  along $s_1<0$ in forward time and along $u_1<0$ in backward time.

For the nonsimple case we assume the following. We consider two homoclinics $\ga_1$ and $\ga_2$ which are in the same direction instead of being in the opposite direction.
\begin{itemize}
\item [[A$1'$]] The homoclinics $\ga_1$ and $\ga_2$ are not tangent to the $s_2$ and $u_2$ axis at $(0,0)$. Both approach $(0,0)$  along $s_1>0$ in forward time and along $u_1>0$ in backward time.
\end{itemize}
In the  case the homology $h$ is simple and $0\not\in\ga_h^0$ we assume the following.
\begin{itemize}
\item [[A$1''$]]  The closed geodesic $\ga_h^0$ is hyperbolic
\end{itemize}

%\blm\label{keythm:DR:LowEnergy:Properties}
%The Hamiltonian \eqref{def:Ham:DR:NF} satisfies properties [A1]-[A2]. 
%\elm
%This lemma is proved in Appendix \ref{app:NonMechanicalAtDR}.

Fix any $d>0$ and $\de\in (0,d)$. Let $B_d$ be the $d$-neighborhood of $(0,0)$ and let
\[
 \Sigma_\pm^s=\{s_1=\pm\de\}\cap B_d,\,\,\,\Sigma_\pm^u=\{u_1=\pm\de\}\cap B_d
\]
be four local sections transversal to the flow of \eqref{def:Ham:DR:NF}. We define the local maps
\[
\begin{aligned} 
\Phi_\loc^{++} &:U^{++}\subset\Sigma_+^s\rightarrow \Sigma_+^u, &\qquad\Phi_\loc^{-+} &:U^{-+}\subset\Sigma_-^s\rightarrow \Sigma_+^u\\
\Phi_\loc^{+-} &:U^{+-}\subset\Sigma_+^s\rightarrow \Sigma_-^u, &\qquad\Phi_\loc^{--} &:U^{--}\subset\Sigma_-^s\rightarrow \Sigma_-^u
\end{aligned}
\]
These maps are defined by the first time the flow of \eqref{def:Ham:DR:NF} hits 
the sections $\Sigma_\pm^u$. These maps are not defined in the whole sets $U^{\pm\pm}$, 
since some orbits might escape. For such points we consider that the maps are undefined. 

For the simple case, by assumption [A1], we know that for $\de$ small enough, 
$\ga^+$ intersects $\Sigma_+^u$ and $\Sigma_+^s$ and $\ga^-$ intersects 
$\Sigma_-^u$ and $\Sigma_-^s$. Let $p^+$ and $q^+$ (respectively 
$p^-$ and $q^-$) the intersection of  $\ga^+$ (resp. $\ga^-$) with $\Sigma_+^u$ 
and $\Sigma_+^u$ (resp. $\Sigma_-^u$ and 
$\Sigma_-^u$). Then, for small neighborhoods $V^\pm\ni q^\pm$ 
there are well defined Poincar\'e maps
\[
 \Phi_\glob^\pm : V^\pm \longrightarrow \Sigma_\pm^s.
\]
In the nonsimple case, [A$2'$] implies that $\ga^i$, $i=1,2$, intersect 
$\Sigma_+^u$ at $q_i$ and  $\Sigma_+^s$ at $p_i$. The associated global maps are denoted
\[
 \Phi_\glob^i:V^i\longrightarrow \Sigma_+^s\,\,\,i=1,2.
\]
Analyzing the composition of these maps at energy surfaces of small energy, we show 
that there exist periodic orbits shadowing the homoclinics and concatenation of homoclinics.

We now assume that the global maps are ``in general position''.
We will only phrase our assumptions [A3] for the homoclinic
$\gamma^+$ and $\gamma^-$. The assumptions for $\gamma^1$ and
$\gamma^2$ are identical, only requiring different notations and will
be called [A3$'$]. Let $W^s$ and $W^u$ denote the local
stable and unstable manifolds of $(0,0)$. Note that $W^u\cap\Sigma^u_\pm$
is one-dimensional and contains $q^\pm$. Let $T^{uu}(q^\pm)$ be the
tangent direction to this one dimensional curve at $q^\pm$. Similarly,
we define $T^{ss}(p^\pm)$ to be the tangent direction to $W^s\cap
\Sigma^s_\pm$ at $p^\pm$.

\begin{itemize}
\item [[A3]]  Image of strong stable and unstable directions under
$D\Phg^\pm(q^\pm)$ is transverse to strong stable and unstable
directions at $p^\pm$ on the energy surface $S_0=\{\HH=0\}$.
For the restriction to $S_0$ we have
\[
D\Phg^+(q^+)|_{TS_0} T^{uu}(q^+) \pitchfork T^{ss}(p^+), \quad
D\Phg^-(q^-)|_{TS_0} T^{uu}(q^-) \pitchfork T^{ss}(p^-).
\]
\end{itemize}
This condition is equivalent to $W^s(0)\pitchfork W^u(0)$.
\begin{itemize} 
\item [[A$3'$]] Suppose condition [A3] hold for both $\gm_1$ and $\gm_2$.
\end{itemize}

In the case that the homology $h$ is simple and $0 \notin \gamma_h^0$, 
we assume
\begin{itemize}
\item [[A$3''$]] The closed geodesic  $\gamma_h^0$ is hyperbolic.
\end{itemize}

\begin{figure}[thb]
        \center{ \includegraphics[scale=1]{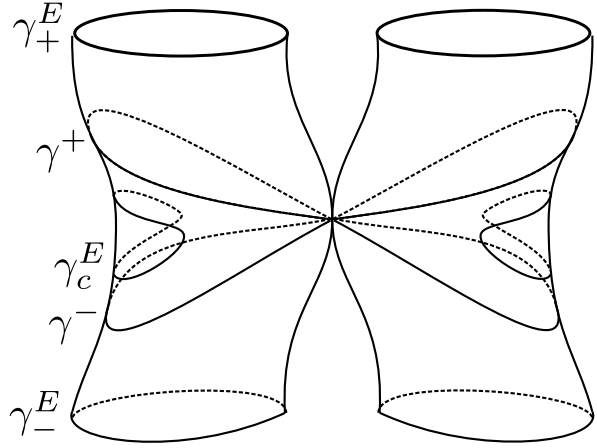}}
        \caption{The cylinders by Corollary \ref{coro:DR:NHIMs} in the low energy regime in the core of double resonances. They are foliated by the periodic orbits obtained in Key Theorem \ref{keythm:DR:LowEnergy:MapComposition}. The two cylinders are tangent to the plane span by the weak stable and unstable directions of the saddle $(0,0)$. Such cylinders are called kissing cylinders in \cite{KaloshinZ12}.}
        \label{fig:KissingCylinders}
\end{figure}

\begin{keytheorem}\label{keythm:DR:LowEnergy:MapComposition}
Consider a Hamiltonian $\HH$ of the form (\ref{def:Ham2dofNonMech})
satisfying conditions (\ref{eq:DR-Hamiltonians}). 
\begin{enumerate}
\item Fix one homology class $h\in H_1(\T^2,\Z)$ (resp. 
two classes $h_1,h_2$). Then 
for Ma$\tilde{\mbox n}\acute{\mbox e}$  $\CCC^r$ generically 
$\HH$ satisfies [H0] and  assumptions [A1-A3]. 

For Hamiltonians satisfying these assumptions we have:    
 
\item Consider the geodesics $\ga^\pm$ associated to 
the Hamiltonian \eqref{def:Ham2dofNonMech}. Then, 
there exists $e_0>0$ such that 
\begin{itemize}
\item For each $E\in (0,e_0]$, there exists a unique periodic 
orbit $\ga_+^E$  corresponding to a fixed point of the map 
$\Phi_\glob^+\circ\Phi_\loc^{++}$ restricted to the energy 
surface $\SSS_E$. 
\item For each $E\in (0,e_0]$, there exists a unique periodic 
orbit $\ga_-^E$  corresponding to a fixed point of the map 
$\Phi_\glob^-\circ\Phi_\loc^{--}$ 
restricted to the energy surface $\SSS_E$.
\item For each $E\in [-e_0,0)$, there exists  a unique periodic 
orbit $\ga_c^E$  corresponding to a fixed point of the map 
$\Phi_\glob^-\circ\Phi_\loc^{+-}\circ \Phi_\glob^+\circ\Phi_\loc^{-+}$ 
restricted to the energy surface $\SSS_E$.
\end{itemize}
\item Consider the geodesics $\ga^1$ and $\ga^2$ associated to 
the Hamiltonian \eqref{def:Ham2dofNonMech}. Then, there exists 
$e_0>0$ such that for each $E\in (0,e_0]$, the following hold. 
For any $\sigma=(\sigma_1,\ldots,\sigma_n)\in \{1,2\}^n$, 
there is a unique periodic orbit, denoted by $\ga_\sigma^E$, corresponding to a unique fixed point of the map
\[
\Pi_{i=1}^n \left(\Phi_\glob^{\sigma_i}\circ \Phi_\loc^{++}\right)
\]
restricted to the energy surface $\SSS_E$.  (Product stands for composition of maps.)
\end{enumerate}
\end{keytheorem}
The second part of this Theorem is proved in \cite{KaloshinZ12}, Theorem 6. 
Namely, if conditions [A1-A3] holds then conclusions 2 and 3 hold. 
A proof that condition 1 is Ma$\tilde{\mbox n}\acute{\mbox e}$ generic is Theorem 4, 
\cite{KaloshinZ12}. 

As a corollary of Key Theorem \ref{keythm:DR:LowEnergy:MapComposition}, 
we have NHICs.
\begin{corollary}\label{coro:DR:NHIMs}
Consider a Hamiltonians $\HH$ of the form (\ref{def:Ham2dofNonMech})
satisfying (\ref{eq:DR-Hamiltonians}). 
\begin{enumerate}\item
Fix a simple noncritical homology class $h\in H_1(\TT^2,\ZZ)$. Consider the homoclinic orbits $\ga_\pm$ and family of periodic orbits $\ga_+^E$, $\ga_{-}^E$, $\ga_c^E$ obtained in Key Theorem \ref{keythm:DR:LowEnergy:MapComposition}. Then, the set 
\[
 \MM_{h}^{e_0}=\bigcup_{0<E\leq e_0}\ga_+^E\cup \ga^+\bigcup_{-e_0\leq E<0}\ga_c^E\cup\ga^-\bigcup_{0<E\leq e_0}\ga_{-}^E
\]
is a $\CCC^1$ normally hyperbolic invariant manifold (NHIM) 
with boundaries $\ga_\pm^{e_0}$ and $\ga_c^{e_0}$  for 
the flow of Hamiltonian \eqref{eq:DR-Hamiltonians}.

\item Fix a nonsimple homology class $h\in H_1(\TT^2,\ZZ)$.  Consider the  periodic orbits $\ga_\sigma^E$ obtained in Key Theorem \ref{keythm:DR:LowEnergy:MapComposition}, where $\sigma$ is the sequence determined by Lemma \ref{lemma:SelfIntersectingGeodesics}. Then, for any $0<e<e_0$, the set 
\[
 \MM_{h}^{e,e_0}=\bigcup_{e<E\leq e_0}\ga_\sigma^E
\]
is a $\CCC^1$ NHIM with boundary for the flow of 
Hamiltonian \eqref{eq:DR-Hamiltonians}.
\end{enumerate}
\end{corollary}

All cylinders associated to simple homology classes are 
tangent along a two dimensional plane at the original formed 
by the weak hyperbolic directions $(s_1,u_1)$. For this reason, 
in \cite{KaloshinZ12}, the cylinders $\MM_{h}^{e_0}$ are called \emph{kissing cylinders}.

Taking the constants $E_0$ and $e_0$ introduced in Key Theorems \ref{keythm:DR:HighEnergy} and \ref{keythm:DR:LowEnergy:MapComposition} respectively, such that $E_0<e_0$ implies that there is overlapping between the low and the high energy regime. Take a simple non-critical homology class $h$. Then, $\MM_h^{E_0,E_1}$,  the last of the cylinders given in Key Theorem \ref{keythm:DR:HighEnergy} can be continued by $\MM_h^{e_0}$  thanks to Corollary \ref{coro:DR:NHIMs} and the uniqueness given by Key Theorem \ref{keythm:DR:LowEnergy:MapComposition}. That is, 
\[
 \MM_h^{E_0,E_1}\cup \MM_h^{e_0}\cup \MM_{-h}^{E_0,E_1}
\ \ \text{ 
is a NHIM.}
\]

Once we have defined these cylinders we have to put them into the framework of this paper. Fix $\om_{n-1}\in \DDD_{\eta,\tau}^{n-1}$ and a Dirichlet resonant segment $\II^{\om_n}_{k_n}\subset\Vor_{n-1}(\om_{n-1})$. Select a $\theta$-strong double resonance $(k_n,k')$ in $\II^{\om_n}_{k_n}$ given in action space by a point $I_0\in\II^{\om_n}_{k_n}\cap\II^{\om_n}_{k'}$. For all $k^*\in\ZZ^3$ such that $I_0\in \II^{\om_n}_{k_n}\cap\II^{\om_n}_{k*}$ we have an associated homology class $h_n=h_n(k^*, k_n,k')$. Then, we can define the sequence of NHICs
\begin{equation}\label{def:DR:UnionNHICs}
\MM^{k_n^*, k_n}=\MM_{h_n}^{E_n^0,E_n^*}\cup \MM_{h_n}^{e_n, e_n^0}.
\end{equation}
where the second object can be omitted if $h_n$ is simple. All the limit energies depend on the properties of the averaged system.
 
\subsection{Localization of the Aubry sets}\label{sec:outline:Aubrysets}
In this section we localize the Aubry sets of certain cohomologies and show that they belong to the NHICs obtained in Section \ref{sec:outline:deformationandcylinders}. First in Section \ref{sec:outline:Aubrysets:transition} we deal with the single resonance zones $\Tr(k_n,k',k'')$ and then, in Section \ref{sec:outline:Aubrysets:doubleresonance}, we deal with the core of double resonances $\Cor(k_n,k')$ (see \eqref{def:TransitionZone} and \eqref{def:CoreDR-m} for the definitions).

\subsubsection{Aubry sets in the single resonance zones}\label{sec:outline:Aubrysets:transition}
We need to improve the estimates of \cite{BernardKZ11} since the hyperbolicity of the NHIC $\wt\CCC_{k_n}^{\om_n,i}$ given by Corollary \ref{coro:cylinderPoschelCoordinates} is very weak. First we state it for the Hamiltonian obtained in Key Theorem \ref{keythm:Transition:NormalForm} by a normal form procedure. 
Then we deduce analogous results for the Hamiltonian 
\begin{equation}\label{def:HamAfterSRDR}
 N^{\dr}:=\left(H_\eps+\Delta H^\mol+\Delta H^\sr+\Delta H^\dr\right)\circ\Phi_\Pos
\end{equation}
The Aubry $\wt\AAA_{N^{\dr}}(c)$ and 
the Ma$\tilde{\mbox n}\acute{\mbox e}$ sets 
$\wt\NNN_{N^{\dr}}(c)$ are defined in Appendix \ref{app:WeakKAM}.

Fix a resonant segment $\II_{k_n}^{\om_n}$, given 
as a parametrized curve in 
(\ref{eq:resonance-parametrization}). Consider 
the Hamiltonian $\HH=\HH^{k_n}$ obtained in Key Theorem \ref{keythm:Transition:NormalForm} and  the cylinders $\CCC_{k_n}^{\om_n,i}$ obtained in Key Theorem \ref{keythm:Transition:IsolatingBlock}.

\begin{theorem}\label{thm: Transition:AubrySet}
With these notations for  
each cohomology class $c=(c^s,c^f)=(J_*^s(c^f),c^f)$ 
with $c^f\in [\wt a_-,\wt a_+]$, the Aubry $\wt\AAA_\HH(c)$ 
and the Ma$\tilde{\mbox n}\acute{\mbox e}$  $\wt\NNN_\HH(c)$ sets satisfy the following.
\begin{itemize}
 \item There exists a constant  $C>0$ independent of $\rr_n$ and $c$ and $b>0$ which might depend on $\rr$ such that
\begin{enumerate}
 \item If $c=(J_*^s(c^f),c^f)$ with  $c^f\in [J_{i}+b,J_{i+1}-b]$, $\wt\AAA_\HH(c)\subset\wt\NNN_\HH(c)\subset\CCC_{k_n}^{\om_n,i}$ and 
 \[
\wt\AAA_\HH(c) \subset \wt\NNN_\HH(c)\subset B_{ C\rr^{\frac{9}{4}dr}}(\psi^s_*(c^f))\times \TT^2\times B_{C\rr^{\frac{2}{3}r+m}}(c).\]
\item If $c=(J_*^s(c^f),c^f)$ with  $c^f\in [J_{i+1}-b,J_{i+1}+b]$,  $\wt\AAA_\HH(c)\subset\CCC_{k_n}^{\om_n,i}\cap\CCC_{k_n}^{\om_n,i+1}$ and 
 \[
\wt\AAA_\HH(c) \subset\left(B_{ C\rr^{\frac{9}{4}dr}}(\psi^s_*(c^f))\cup B_{ C\rr^{\frac{9}{4}dr}}(\ol\psi^s_*(c^f))\right)\times \TT^2\times B_{C\rr^{\frac{2}{3}r+m}}(c).\]
\end{enumerate}
\item Let $\pi_{(\psi^f,t)}$ be the projection onto the fast and time components $(\psi^f,t)$. Then, we have that $\left.\pi_{(\psi^f,t)}\right|_{\wt\AAA_\HH(c)}$ is one-to-one and the inverse is Lipschitz. 
\end{itemize}
\end{theorem}
The proof of this Theorem is  in Section \ref{sec:LocAubrySets}. It follows similar lines as  the results obtained in 
\cite{BernardKZ11} but additional estimates are required 
to deal with very small hyperbolicity.  Note that close to bifurcation we do not have localization of  Ma$\tilde{\mbox n}\acute{\mbox e}$ sets since it can contain heteroclinic orbits between orbits belonging to both cylinders.

From this Theorem we can deduce an analogous result for the Hamiltonian $N^{\dr}$ in \eqref{def:HamAfterSRDR}. It is enough to use the following Theorem from \cite{Be2}.
\begin{theorem}\label{thm:Bernard}
 Let $L:TM\longrightarrow \RR$ be a Tonelli Lagrangian and $G:T^*M\longrightarrow \RR$ the associated Hamiltonian. If $\Psi:T^*M\longrightarrow T^*M$ is an exact symplectomorphism, then
\[
\wt \AAA_{G\circ \Psi}(c)=\Psi^{-1}\wt\AAA_{G}(c)\ \ \ \text{ and }\ \ \ 
\wt\NNN_{H\circ \Psi}(c)=\Psi^{-1}\wt\NNN_{G}(c)
\]
 \end{theorem}

If we apply this Theorem to the results of Theorem \ref{thm: Transition:AubrySet}, we obtain the following 

%\begin{corollary}\label{keythm: Transition:AubrySet:Original}

\begin{keytheorem}\label{keythm: Transition:AubrySet}
Fix a resonant segment $\II_{k_n}^{\om_n}$ and consider the Hamiltonian  $N^\#=N^{\dr}\circ L^\sf$, where $N^{\dr}$ has been defined in \eqref{def:HamAfterSRDR} and $L^\sf$ is the linear symplectic change 
of coordinates given in \eqref{def:ChangeToFastSlow:StrongWeak:0}, \eqref{def:ChangeToFastSlow:StrongWeak2:0}. Consider also 
the cylinders $\wt\CCC_{k_n}^{\om_n,i}$ obtained in Corollary \ref{coro:cylinderPoschelCoordinates}.
For each cohomology class $c=(c^s,c^f)=(I_*^s(c^f),c^f)$ with $c^f\in [\wt a_-,\wt a_+]$,  the Aubry $\wt\AAA_{N^\#}(c)$ 
and the Ma$\tilde{\mbox n}\acute{\mbox e}$ set $\wt\NNN_{N^\#}(c)$ satisfy the following.
\begin{itemize}
\item There exists a constant  $C>0$ independent of $\rr_n$ and $c$ and a constant $b$ which might depend on $\rr$  such that
\begin{enumerate}
 \item If $c=(J_*^s(c^f),c^f)$ with  $c^f\in [I_{i}+b,I_{i+1}-b]$, $\wt\AAA_{N^\#}(c)\subset\wt\NNN_{N^\#}(c)\subset\wt\CCC_{k_n}^{\om_n,i}$ and 
 \[
\wt\AAA_{N^\#}(c) \subset \wt\NNN_\HH(c)\subset B_{C\rr^{2r-2m}}(\psi^s_*(c^f))\times \TT^2\times B_{ C\rr^{\frac{2}{3}r+m}+C\rr^{2r-2m}}(c).\]
\item If $c=(J_*^s(c^f),c^f)$ with  $c^f\in [I_{i+1}-b,I_{i+1}+b]$,  $\wt\AAA_{N^\#}(c)\subset\wt\CCC_{k_n}^{\om_n,i}\cap
\wt\CCC_{k_n}^{\om_n,i+1}$ and 
 \[
\wt\AAA_{N^\#}(c)\subset 
\left(B_{C\rr^{2r-2m}}(\psi^s_*(c^f))
\cup B_{C\rr^{r-\frac{5}{2}m}}(\ol\psi^s_*(c^f)\right)\times 
\TT^2\times B_{ C\rr^{\frac{2}{3}r+m}+C\rr^{2r-2m}}(c).
\]
\end{enumerate}
\item Let $\pi_{(\varphi^f,t)}$ be the projection onto the fast and time 
components $(\varphi^f,t)$. Then, we have that 
$\left.\pi_{(\varphi^f,t)}\right|_{\wt \AAA_{N^\#}(c)}$ is one-to-one and 
the inverse is Lipschitz. 
\end{itemize}
\end{keytheorem}
As happened in Corollary \ref{coro:cylinderPoschelCoordinates}, the localization worsens with respect to Key Theorem
\ref{keythm: Transition:AubrySet} since we perform the change of coordinates $\Phi_n$ given by Key Theorem \ref{keythm:Transition:NormalForm}.
% Note that nevertheless we still have slow angle localization since with the values of the parameters $r$, $m$, $q$ and $\theta$ (see Appendix \ref{app:WeakKAM}), we have $r-5m/2\geq 5q-3m/2=q/2-1>0$.

Note that the change of coordinates given by Key Theorem \ref{keythm:Transition:NormalForm} is  a canonical transformation which is close to the identity. Thus, the graph property is still true 
when we go back to P\"oschel coordinates.

\subsubsection{Aubry sets in the core of the double resonances}\label{sec:outline:Aubrysets:doubleresonance}
In the double resonance regime, after a normal form procedure, we deal with the Hamiltonian 
$$
\HH(\psi,J)=\HH_0(J)+\ZZZ(\psi,J),\quad \psi\in\T^2,\ J\in \R^2,
$$
satisfying (\ref{eq:DR-Hamiltonians}). In Corollary \ref{coro:DR:NHIMs} 
we have constructed a variety of Normally Hyperbolic Invariant Manifolds. 
Now, we need to see that certain Aubry sets  belong to those NHIMs. 
We proceed as in \cite{KaloshinZ12}. Nevertheless, recall that now, 
at first order, we do not have a mechanical system.

The cylinders obtained in Key Theorem \ref{keythm:DR:HighEnergy} 
and  Corollary \ref{coro:DR:NHIMs} are related to the integer homology 
classes whereas the Aubry sets are related to the cohomology classes. 
Therefore, the first step is to relate them. Each minimal geodesic 
$\ga_h^E$ corresponds to a minimal measure of the system associated 
to Hamiltonian \eqref{def:Ham:DR:NF}. 
Moreover, it has associated a cohomology class. Assume that  $\ga_h^E$ is parameterized 
so that it satisfies the Euler-Lagrange equation and call $T=T(\ga_h^E)$ to its period. 
Then, the probability measure supported on  $\ga_h^E$ is a minimal measure and 
its rotation number is $T\ii$. Then, the associated cohomology class is the convex subset
\[
 \LL\FF_\beta (h/T(\ga_h^E)),
\]
of $H^1(\TT^2,\RR)$, where $\LL\FF_\beta$ is 
the Legendre-Fenchel transform of the $\beta$-function defined 
in Appendix \ref{app:WeakKAM}. Recall that 
in Key Theorems \ref{keythm:DR:HighEnergy} and \ref{keythm:DR:LowEnergy:MapComposition} 
we have seen that for $E\in (0,E_0)$ or $E\in (E_j,E_{j+1})$ $j=0,\ldots,N-1$, there exists a unique 
minimal geodesic $\ga_h^E$ for energy $E$. In this case, we define 
\[
 \la_h^E=\frac{1}{T(\ga_h^E)}.
\]
For the bifurcation values $E=E_j$ there are two minimal geodesics $\ga_h^E$ and $\ol\ga_h^E$. 
We write  $\la_h^E=1/T(\ga_h^E)$, where the choice of geodesic is arbitrary. In \cite{KaloshinZ12} it is shown that the set 
$\LL\FF_\beta (\la_h^E h)$ is independent of the choice of geodesic 
in the setting of mechanical systems. This fact is still true for 
Hamiltonians of the form (\ref{def:Ham2dofNonMech}) satisfying 
(\ref{eq:DR-Hamiltonians}). The set $\LL\FF_\beta (\la_h^E h)$ 
is a well defined set function of $E$.

We call the union 
\[
 \bigcup_{E>0}\LL\FF_\beta (\la_h^E h)
\]
the channel associated to the homology $h$ and we choose a curve of cohomologies along this channel. 
For cohomologies in such curve, the Aubry sets will belong to the NHICs obtained in 
Key Theorem \ref{keythm:DR:HighEnergy} and Corollary \ref{coro:DR:NHIMs}. 
Note that if the Hamiltonian formulation the reparameterized geodesic $\ga_h^E$ becomes a two dimensional torus. 
Indeed, since we are in the non-autonomous setting and, 
therefore, we need to add the time component. We denote 
this torus by $\TT^2_{\ga_h^E}\subset\TT^2\times\RR^2\times \TT$. 
Analogously the critical point $\psi=0$ now becomes a periodic orbit, which we denote by $\TT_0$.

\begin{keytheorem}\label{keythm:DR:AubrySets}
(see Prop. 4.1 \cite{KaloshinZ12})
Consider the Hamiltonian $\HH$ of the form (\ref{def:Ham2dofNonMech})
satisfying (\ref{eq:DR-Hamiltonians}).  
There exists a continuous function 
$\ol c_h: [0,\ol E]\longrightarrow H^1(\TT^2,\RR)$ satisfying $\ol c_h(E)\in \LL\FF_\beta (\la_h^E h)$ 
with the following properties
\begin{enumerate}
 \item For $E\in (0,E_0)$ or $E\in(E_j,E_{j+1})$, $j=0,\ldots,N-1$, 
$ \wt\AAA (\ol c_h(E))=\TT^2_{\ga_h^E}$.
\item For $E=E_j$, $j=0,\ldots,N-1$, 
$ \wt\AAA (\ol c_h(E_j))=\TT^2_{\ga_h^E}\cup \TT^2_{\ol\ga_h^E}$.
\item If $h$ is simple critical, then, 
$ \wt\AAA (\ol c_h(0))=\TT^2_{\ga_h^0}$ and for $\la \in [0,1)$,
%$ \AAA (\la\ol c_h(0))=\{\psi=0\}.$
$ \wt\AAA (\la\ol c_h(0))=\TT_0.$
\item If $h$ is simple non-critical, then, 
$ \wt\AAA (\ol c_h(0))=\ga_h^0\cup
\TT_0$
%\{\psi=0\}$
and for $\la \in [0,1)$, $\wt \AAA (\la\ol c_h(0))=\TT_0.$
\end{enumerate}
\end{keytheorem}

\subsection{Diffusion mechanism, shadowing, equivalent forcing classes}\label{sec:shadowing}

We construct diffusing orbits using variational methods.
The method is the same as in \cite{BernardKZ11} and \cite{KaloshinZ12}
and inspired by Mather \cite{Mather93, Mather96}. The general 
idea is to shadow a chain of Aubry sets $\wt\AAA(c)$'s parametrized 
by cohomologies $c\in H^1(\T^2,\R)$. In \cite{Be} definitions of forcing 
and $c$-equivalence of two cohomology class are proposed. They 
have the following important property. For two cohomologies 
$c,c' \in H^1(\T^2,\R)$ if $c$ forces $c'$ and $c'$ forces $c$, 
denoted $c\vdash c'$, and $c'\vdash c$, then $c$ and $c'$ are 
$c$-equivalent. It  turns out that $c$-equivalence is an equivalence 
relation (see Proposition \ref{prop:c-equiv} in Appendix \ref{app:WeakKAM}). Moreover, if $c \dashv \vdash c'$, 
then there are two heteroclinic orbits going from $\wt\AAA(c)$ to $\wt\AAA(c')$ and 
from $\wt\AAA(c')$ to $\wt \AAA(c)$ (see Proposition \ref{prop:c-equiv}). Forcing relation and its dynamical consequences are explained in Appendix \ref{app:WeakKAM}.

For ``Diophantine'' cohomologies we have the following 
\blm \label{lem:aubry-kam}
(see e.g. \cite{Sorrentino11}) For the Hamiltonian $N$ in the P\"oschel 
normal form (\ref{def:HamAfterMolification}) for every $c$ such 
that $\partial H_0'(c)=\om \in \DDD^U_{\eta,\tau}$ 
the Aubry set $\wt\AAA_N(c)$, 
the Ma$\tilde{\mbox n}\acute{\mbox e}$ set
$\wt\NNN_N(c)$, and the KAM torus $\TTT_\om$ 
of the corresponding rotation vector $\om$
coincide.
\elm 
%\begin{proof} Consider the Legendre transformation 
%of the Hamiltonian $N$. We have  
%\[
%L_c(\varphi,v,t)=L(\varphi,v,t)-c\cdot v=
%\max_I  v(I-c) - N(\varphi,I,t)=
%\]
%\[
%=\max_I ( v\cdot (I-c) - H_0'(I)-\eps R(\varphi,I,t))=
%\]
%\[
%=\max_{I'} \left( v\cdot I' - H_0'(I'+c)-\eps R(\varphi,I'+c,t)
%\right)=
%\]
%\[
%=\max_{I'} \left((v-\om) \cdot I'-H_0'(c)- \dfrac \kappa 2
%\langle \partial^2 H_0'(I'+c) I',I'\rangle 
%-\eps R(\varphi,I'+c,t)\right)
%\]
%for some $0<\kappa<1$ and $\kappa\to 1$ as 
%$I'\to 0$. By Theorem \ref{thm:Poschel} we have that 
%\[
%\|R(\varphi,I'+c,t)\|_{C^0}\le C \eps \|I'\|^{3r-3\tau-j}.
%\]
%Therefore, the global maximum of $L_c$ takes place at 
%$v\equiv \om$ and corresponds to the KAM invariant 
%torus $\TTT_\om$. Since dynamics on $\TTT_\om$ is 
%minimal, there is exactly one invariant set there. Thus, 
%the Aubry set $\wt\AAA(c)$ concides with $\TTT_\om$. 
%Due to the Mather Graph Theorem the Aubry $\wt\AAA(c)$ 
%and the Ma$\tilde{\mbox n}\acute{\mbox e}$ $\wt\NNN(c)$ 
%sets coincide too.  
%\end{proof}

Consider the tree of Dirichlet resonances 
${\bf S}=\cup_{n} {\bf S}_n$, defined in  (\ref{def:UnionResonanceInFreq}). 
Using $c$-equivalence 
it suffices to show that for any $h$ and $h'$ in $\bS$ close 
enough and $c\in \L\F_N(h)$ and $c'\in \L\F_N(h')$ we have
that $c$ and $c'$ are $c$-equivalent. We show that for each 
$h$ away from $\theta$-strong resonances and $c\in \L\F(h)$ 
we have that $\wt\AAA(c)$ belongs to a NHIC.  It turns out 
that away from junctions at $\theta$-strong resonances 
there are two essentially different cases: \begin{itemize}
\item diffuse along 
one homology class and 
\item jump from one homology class
to another.
\end{itemize}
%Prove $c$-equivalence 
%\begin{itemize}
%\item in single and double resonances along one cylinder,
%
%\item in single and double resonances for a pair 
%of cylinders of the same homology class,
%\end{itemize}
%
%To prove that we can switch from one resonance 
%to another, close enough to a $\theta$-strong resonance, 
%we prove $c$-equivalence 
%\begin{itemize}
%\item in the core of double resonances between (kissing) 
%cylinders of different homology.
%\end{itemize}

\vskip 0.15in

Consider a Hamiltonian $N^{\dr}=(H_\eps+\Delta H^\mol+\Delta H^\sr+\Delta H^\dr)\circ \Phi_\Pos$ introduced in 
Key Theorem \ref{keythm:CoreDR:NormalForm}.
Consider the collection of Dirichlet resonances ${\bf S}$ 
and their preimage in the action space (\ref{eq:action-resonances})
$$
{\bf I}=\Omega^{-1}({\bf S})\subset U.
$$
For each $n\in\NN$, $k_n\in\FF_n$, let $k'$ be 
such that $(k_n,k')$ form a $\theta$-strong resonance,  
i.e. $|k'|\leq |k_n|^\theta$ (see 
definition \ref{definition:StrongWeakDR}). 
Consider the core of this resonance Cor$_n(k_n,k')$, 
given by (\ref{def:CoreDR-m}). Denote the projection onto 
the action space of the $C\rho^m_n$-neighborhood 
of this double resonance by Cor$^I_n(k_n,k')$. 

Define 
$$
{\bf I}^*={\bf I} \setminus \cup_{n\in \NN} 
\cup_{k_n\in\FF_n} \cup_{k'} \text{ Cor}^I_n(k_n,k'),
$$
where the last union is taken over $k'$ such 
that $(k_n,k')$ form a $\theta$-strong resonance. 

Since $N^{\dr}$ satisfies  
%Key Theorem \ref{keythm:CoreDR:NormalForm}, by 
Key Theorem \ref{keythm:Transition:IsolatingBlock}, 
we have the existence of NHICs $\CCC_{k_n}^{\om_n}$
away from $\theta$-strong resonances 
and by Corollary \ref{coro:DR:NHIMs} 
existence of NHI Manifolds $\MM^{k',k_n}$ 
as those resonances
(see also \eqref{def:DR:UnionNHICs}).
%Let $\dt_n, \ n\ge 1$ be a sequence 
%$\dt_n \to 0$ as $n\to \infty$.

\begin{keytheorem}\label{keythm:equivalence}
With the above notations suppose 
$N^{\dr}$ satisfies conditions of 
Key Theorem \ref{keythm:CoreDR:NormalForm}, then 
there is a $\CCC^\infty$ small 
perturbation $\Delta H^\shad$ supported away from 
pull backs of all NHICs $\CCC_{k_n}^{\om_n}$ and 
of all NHIMs $\MM^{k',k_n}$, such that for  
\[N^*=(H_\eps+\Delta H^\mol+\Delta H^\sr
+\Delta H^\dr+\Delta H^\shad)\circ \Phi_\Pos
\]
 all cohomology classes $c\in {\bf I}^*$ 
are contained in a single forcing class and the closure 
$\LL\FF_{N^*}(\DDD^U_{\eta,\tau})\subset\overline {\bf I}^*$.

Moreover, the same conclusion holds for $N^{\bullet}$ which are 
close to $N^*$ in Whitney KAM topology. In particular, for each 
$n\in\NN$, in each Dirichlet resonance $\II_{k_n}^{\om_n}$ 
away from $\theta$-strong double resonances,
the single resonant normal form 
$$
\HH^{k_n}_0(J)+\ZZZ^{k_n}(\psi_n,J)+\RRR^{k_n}(\psi,J,t)
$$
from Key Theorem \ref{keythm:Transition:NormalForm},
and at each $\theta$-strong double resonance $(k_n,k')$, the normal form
$$
\HH^{k_n,k'}_0(J)+\ZZZ^{k_n,k'}(\psi,J)
$$
from Key Theorem \ref{keythm:CoreDR:NormalForm} obtained from  $N^{\bullet}$ and of $N^*$ are $\CCC^r$-close.
\end{keytheorem}

This theorem implies that there are orbits which drift along the single resonance segment and jump from one segment to the other. In Section \ref{sec:shadowing} we explain how to adapt the ideas from \cite{BernardKZ11, KaloshinZ12} to our setting. We need to pay attention to both single and double resonance regimes.  Notice that the set ${\bf I}^*$ is not connected: it only becomes connected if  we add 
the cores of double resonances. The reason is that the choice of the  cohomology path at $\theta$-strong 
double resonances which belongs to a single forcing class and  connects different single resonant segments  is somewhat involved (see Key Theorem \ref{keythm:DR:AubrySets}). Thus, to simplify the exposition we do not describe these paths in this theorem and  we defer its description to Section \ref{sec:shadowing}. Thus, even if Key Theorem \ref{keythm:equivalence} may makes an illusion that proving $c$-equivalence for $c$ close but on ``different'' sides of the core of a $\theta$-strong double resonance is straightfoward, it is certainly not the case: it is fairly involved and  occupies 
chapters 6,11,12 in \cite{KaloshinZ12} (see Section \ref{sec:shadowing}).

\vskip 0.1in

%\subsubsection

\subsubsection{Key Theorem \ref{keythm:equivalence} 
implies Theorem \ref{thm:MainResultNonAutonomous}} \ \ 
For every $c$ such that 
$\om=\partial H_0'(c)\in \DDD^U_{\eta,\tau}$ 
by Lemma  \ref{lem:aubry-kam} the corresponding 
Aubry set $\wt\AAA_{N^*}(c)$ coincides with 
the Ma$\tilde{\mbox n}\acute{\mbox e}$ $\wt\NNN_{N^*}(c)$.
Therefore, for any neighborhood of $V$ of 
$\wt\AAA_{N^*}(c)$ in the phase space and any $c_n\to c$ 
we have that $\wt\AAA_{N^*}(c_n)$ belongs to a neighborhood 
$V$ for $n$ sufficiently large (see e.g. Thm 1, \cite{Ber10b}). 
Since dynamics on $\wt\NNN_{N^*}(c)=\TTT_\om$ is minimal, 
the limit $\lim_{c_n\to c} \wt\AAA_{N^*}(c_n)=\wt\AAA_{N^*}(c)$.
Therefore, finding orbits shadowing invariant sets 
$\wt\AAA_{N^*}(c_n)$ for any sequence $c_n\to c$ implies 
that such orbits accumulate to 
the KAM torus $\TTT_\om$ and contain it in its closure. 
 
By Key Theorem \ref{keythm:equivalence} we know that all $c$ from 
the pullbacks ${\bf I}^*$ of the Dirichlet resonant segments ${\bf S}$ 
without the cores of the $\theta$-strong double resonances are 
$c$-equivalent. In particular, for any Diophantine 
$\om \in \DDD^U_{\eta,\tau}$ we have a ``zigzag'' 
$\SSS_{k_n}^{\om},\ n\in \NN$ accumulating to $\om$. 
Pull backs $\II_{k_n}^{\om},\ n\in \NN$ contain $c$'s
which accumulate to $c=\Omega^{-1}(\om)$.

Now it suffices to prove $c$-equivalence for all $c$'s in ${\bf I}^*$, 
which by Proposition 5.1 implies existence of shawoding orbits. 
It turns out that the proof of the property of $c$-equivalence 
depends on local properties of the underlying Hamiltonian 
and is robust with respect to $\CCC^2$-small perturbations. 
Therefore, the ``moreover'' part holds. 

Now we start the detailed description of the proof and,
in particular, prove all Key Theorems.  Recall that Key 
Theorems ~\ref{keythm:DR:HighEnergy}, ~\ref{keythm:DR:LowEnergy:MapComposition}, and 
~\ref{keythm:DR:AubrySets} are proven in \cite{KaloshinZ12}. 
Therefore, we need to prove only seven Key Theorems.

\section{First step of the proof: Selection of resonances}\label{sec:ConstructionNetResonances}
In this section we prove Key Theorem \ref{keythm:ResonancesWithNoTriple}. That is, we construct a net of resonances 
$\{\SSS_n\}_{n\in \Z}$ in the frequency spaces 
$U_\eta'$ such that the following properties hold 
\begin{itemize}
\item each segment $\SSS_n$ belongs to a resonant segment 
\[
\SSS_n\subset\Gm_{k_n}=\{\om\in U_\eta': k_n\cdot  (\om,1)=0\}
\] 
for some $k_n\in (\Z^2\setminus 0)\times \Z$. 

\item The union $\cup_n \SSS_n$ is connected 
 
\item The closure $\overline{\cup_n \SSS_n}$  contains all 
Diophantine frequencies in $\DDD_{\eta,\tau}^U$ 
(see \eqref{def:Diophantine}).
\end{itemize}
Moreover, we need these resonances satisfy quantitative 
estimates on speed of approximation. 
%Recall that an elementary pigeon hole principle show that there is $C>0$ such that for  any $\om=(\om_1,\om_2)\in \R^2$ there is a sequence  $k_n(\om)\in (\Z^2\setminus 0)\times \Z$ such that 
%$$
%|k_n(\om)\cdot (\om,1)|\le C|k_n(\om)|^{-2} \  \text{ and }\ |k_n(\om)|\to \infty \text{ as }n\to \infty.
%$$ 
We choose our segments so that  
\begin{itemize}
\item for each $\SSS_n\subset\Gm_{k_n}$ there is 
a Diophantine number $\om \in \DDD_{\eta,\tau}$ 
such that 
\[
|k_n\cdot (\om,1)|\le |k_n|^{-2+3\tau}.
\] 
\item if $\SSS_n\cap \SSS_{n'}\ne \emptyset$, then 
we have a bound on ration $|k_n|/|k_{n'}|$ as well
as angle of intersection.  
\end{itemize}
%The net of resonances is obtained in two steps. 
%\begin{itemize}
%\item in Section \ref{sec:SelectionResonances} we fix a diophatine 
%$\om^*\in\DDD_{\eta,\tau}$ and we construct a connected ``zigzag''
%approaching $\om^*$. 
%\item in Section \ref{sec:ResonancesWithNoTriple} we show 
%how to deal with all frequencies in $\DDD_{\eta,\tau}$.  
%\end{itemize}

%It turns out that we cannot the union over all frequencies in 
%$\DDD_{\eta,\tau}$.  We shall need additional estimates on 
%the intersection between resonant segments. 
%The main theorem of this section is Theorem \ref{thm:ResonancesWithNoTriple}, where 
%the resonance net and its properties is given.

\subsection{Approximation of Diophantine vectors}\label{sec:SelectionResonances}
In this section we explain how to construct a sequence of resonant segments approaching one single Diophantine frequency. Later, in Section \ref{sec:ResonancesWithNoTriple} we explain how to deal with all $(\eta,\tau)$-Diophantine frequencies at the same time and we prove Key Theorem \ref{keythm:ResonancesWithNoTriple}.

Take $\om =(\om_1,\om_2)\in \DDD_{\eta,\tau} \subset\R^2$. For each $k\in (\Z^2\times 0)\times \Z$ denote 
$\Gm_k=\{\om: (\om,1)\cdot k=0\}$ the resonant line.  We state the results in the frequency space. Therefore all results are indepenent of the Hamiltonian.
%We distinguish the resonant line $\Gm_k\in \R^2$ from 
%the corresponding resonant segment 
%$\SSS(\Gm_k)=\{I: \ k\cdot (\partial H_0(I),1)=0\}$. 
%In the case $H_0(I)=(I_1^2+I_2^2)/2$ they coincide, but 
%in general, are different. 

\begin{comment}

For any $a,b\in \Gm_k$ denote $\Gm_k(a,b)$ the segment 
inside $\Gm_k$ with endpoint $a$ and $b$ resp. Similarly, 
for any $a,b\in \SSS(\Gm_k)$ denote $\SSS_k(a,b)$ the curve 
inside $\SSS(\Gm_k)$ with endpoints $a$ and $b$. 

\begin{definition}\label{def:RatioDR}
Let $k',k''\in (\Z^2\times 0)\times \Z$. Define
\[
\langle k',k''\rangle:=\{k\in (\Z^2\times 0)\times \Z:\ k=pk'+qk'',\ p,q\in\Z\}
\]
 the lattice generated by integer combinations of $k'$ and $k''$ and call
\[
\NNN(k',k'') = \min \{ |K| : \ K\ne 0,\  K \in \langle k', k'' \rangle \}
\]
the order of  the double resonance.

%We say that the double resonance $\langle L,M\rangle$  has $\kk$-bounded ratio if
%\[
%\frac{|L|}{|\NNN(L,M)|}<\kk\,\,\,\text{ and }\,\,\, \frac{|M|}{|\NNN(L,M)|}<\kk.
%\]
\end{definition}

\end{comment}

%Let $H_0$ be a convex Hamiltonian. Fix an energy surface
%$S=\{H_0=1\}$. 
Fix a $(\eta,\tau)$-Diophantine frequency $\om\in\DDD_{\eta,\tau}$. We construct 
a zigzag of curves $\{\Gm_n^\om\}_{n\ge 1}$ approximating 
$\om$ and satisfying certain quantitative estimates.

\begin{theorem}\label{KeyThm:SelectionResonances}
Fix $\om^\ast\in \DDD_{\eta,\tau}$ and 
a sequence of radii $\{R_n\}_{n\in\ZZ}$ defined as
\[
 R_{n+1}=R_n^{1+2\tau}, \,\, R_0\gg 1, \kk\geq 1.
\]
Then, there exist $\xi>0$ and $c>0$ and a sequence 
of $(\ZZ^2\setminus 0)\times \Z$-vectors $\{k_n\}_{n\in\NN}$, 
such that for each positive integer $n$ we have 
\begin{enumerate}
\item The $n$-th vector $k_n$ belong to the following annulus  $\frac{R_n}{4}\leq \left|k_n\right|\leq R_n$.
\item The vectors satisfy
\[
 \eta R_n^{-(2+\tau)}\leq \left|k_n\cdot \omega\right|\leq 16 \eta^{-2} R_n^{-(2-2\tau)}.
\]
\item Consecutive vectors satisfy $k_n\not \parallel k_{n+1}$ 
and the angles between them satisfy
\[
\measuredangle (k_n,k_{n+1})\geq \frac \pi 6.
\]
\item Let \begin{equation}\label{def:DiskAroundDiophantine}
 \DD_n(\om^\ast)=\left\{|\om^\ast-\om|\leq \xi\eta^{-7/2} R_n^{-(3-4\tau)}\right\}.
\end{equation}
Then 
$\om_{n+1}:=\Ga_{k_n}\cap \Ga_{k_{n+1}}\subset \DD_n(\om^\ast)$. 
Denote $\Gm^{\om^\ast}_n \subset \Ga_{k_n}$ 
the segment between $\om_n$ and $\om_{n+1}$. Then 
$\Gm^{\om^\ast}(\om_n,\om_{n+1})\subset\DD_n(\om^\ast)$.

\item Any other resonance $\Ga_k$ intersecting $\DD_n(\om^\ast)$, 
satisfies 
\[
\frac{|k_n|}{|k|}\leq c_2 \eta^{-2} R_n^{2\tau}. 
\]
\end{enumerate}
\end{theorem}

This Theorem is proved in Section \ref{sec:proofSelectionResonances}.  From item 2 there, we 
obtain upper and lower 
bounds for the distance between $\omega^\ast$ and any point 
in $\Gm^{\om^\ast}_n$.

%\bdef Call resonance segments $\Gm_k^\om$ satisfying 
%items 1, 2, 4 weak Dirichlet. 
%\endef 

\begin{corollary}\label{coro:DistanceToDiophantine}
 For any point $\om\in \Gm^{\om^\ast}_n$ we have 
\[
\eta R_n^{-3-\tau}\leq\left|\om-\om^\ast\right|\leq \xi \eta^{-7/2} R_n^{-(3-4\tau)}
\]

\end{corollary}
\begin{proof}
 The second estimate is given in the Theorem. For the first one, it is enough to use 
 the Diophantine property \ref{def:Diophantine} to obtain
\[
 \left|\om-\om^\ast\right|\geq \frac{|k_n\cdot\om^\ast|}{|k_n|}
 \geq \frac{\eta}{|k_n|^{3+\tau}}\geq \eta R_n^{-3-\tau}.
\]
\end{proof}

%Then, in Appendix \ref{sec:KLDMAlgorithm} we show how the results we obtain can be put in the framework of the  Multidimensional 
%Continued Fraction Algorithm of \cite{KLDM}.  First,   we explain the Multidimensional Continued Fraction Algorithm of \cite{KLDM}. Then,  we show that we can obtain solutions of  such algorithm which satisfy the properties stated in Theorem \ref{KeyThm:SelectionResonances}.

\subsection{Proof of Theorem \ref{KeyThm:SelectionResonances}: approaching one Diophantine frequency}\label{sec:proofSelectionResonances}
The key point of the proof of Theorem \ref{KeyThm:SelectionResonances} 
is a non-homogeneous Dirichlet Theorem, which can be found 
in \cite[Chapter 5, Theorem VI]{Cassels57}. We state a simplified version, 
which is sufficient for our purposes. In this section, the norm $\|\cdot\|$ 
measures the maximal distance to the integers, where the maximum is also 
taken over each component, i.e. for $v=(v_1,\dots,v_n)\in \R^n$ 
we denote 
\[
\|v\|:=\min_{(m_1,\dots,m_n)\in \ZZ^n} \max_{1\le j\le n} |v_j-m_j|.
\]

\begin{theorem}\label{thm:Cassels}
Let $L(x)$, $x=(x_1,\ldots, x_n)$ be a linear form and fix $C,X>0$. 
Suppose that there does not exist any $x\in \ZZ^n\setminus 0$ such that 
\[
 \|L(x)\|\leq A\,\,\,\text{ and }\,\,\,|x_i|\leq X.
\]
Then, for any $\al\in\RR$, the equations
\[
 \|L(x)-\al\|\leq A_1\,\,\,\text{ and }\,\,\,|x_i|\leq X_1.
\]
have an integer solution, where
\[
 A_1=\frac{1}{2}(h+1)A,\ \ \ X_1=\frac{1}{2}(h+1)X
\,\,\textup{ and }\,\,h=X^{-n}A^{-1}.
\]
\end{theorem}
This Theorem allows us to prove Theorem \ref{KeyThm:SelectionResonances}
\begin{proof}[Proof of Theorem \ref{KeyThm:SelectionResonances}]
Let $\om^\ast \in \DDD^U_{\eta,\tau}$, i.e.  
$|(\om^\ast,1) \cdot k |\ge \eta|k|^{-2-\tau}$ for each $k\in \Z^2\times (\ZZ\setminus 0)$. 
Consider the sequence of $\{R_n\}_{n\in\NN}$. We prove 
Theorem \ref{KeyThm:SelectionResonances} by induction. 
Assume that we have $k_n$ and we need to obtain $k_{n+1}$
satisfying items 1--5. 
%We also need to have a uniform bound on the angle between 
%$k_n$ and $k_{n+1}$.

Consider a point $\widetilde k$ at distance $R_{n+1}/2$ in 
the direction perpendicular to $k_n$ along the plane 
$T_{\om^\ast}=\{x:\ x\cdot {\om^\ast}=0\}$. Take a ball of radius 
$R_{n+1}/4$ around it.  We want to find a good approximating 
vector in this ball. This vector must be of the form 
$k_{n+1}=\widetilde k+x$ with $|x|\leq R_{n+1}/4$.
This implies that for any $k_{n+1}$ from this ball $k_{n}$ 
and $k_{n+1}$ make an angle at least $\pi/6$ and item 3 holds. 

Define the linear form $L(x)=\omega^\ast \cdot x$ and  
$\al=\omega^\ast\cdot \widetilde k$. Then, 
\[
 \omega^\ast\cdot k_{n+1}=L(x)+\al.
\]
Now we apply Theorem \ref{thm:Cassels} taking 
\[
X=\left(\frac{\eta R_{n+1}}{4}\right)^{1/(1+\tau)}.
\]
Then, using the Diophantine condition we know that taking
\[
 A=\frac{\eta}{X^{2+\tau}},
\]
the equation $\|L(x)\|\leq A$ has no solution for  $|x|\leq X$. Thus, we get, 
\[
h=\eta\ii X^\tau.
\]
Then, using the formula for $X$, we get a vector $x$ satisfying $|x|\leq X_1$ with 
\[
 X_1=\frac{1}{2}(h+1)X\leq \eta\ii X^{1+\tau}=\frac{R_{n+1}}{4}
\]
which is solution of 
\[
 |k_{n+1}\cdot\omega^\ast|=\|L(x)-\al\|\leq hA\leq \frac{1}{X^2}\leq \frac{16}{\eta^2 R_{n+1}^{2-2\tau}}.
\]
Moreover, $|k_{n+1}|=|\wt k+x|\geq R_{n+1}/4$.
%Finally, it is enough to use the just obtained bound and the Diophantine estimate 
%\eqref{def:Diophantine} to obtain a lower bound for $|k_{n+1}|$. We have
%\[
% \frac{\eta}{|k_{n+1}|^{2+\tau}}\leq|k_{n+1}\cdot\omega|\leq \frac{16}{\eta^2 R_{n+1}^{2-2\tau}}
%\]
%and, therefore,
%\[
% |k_{n+1}|\geq \left(\frac{\eta^3}{16}\right)^{1/(2+\tau)}R_{n+1}^{\frac{2-\tau}{2+\tau}}
% \geq \eta^{3/2} R_{n+1}^{\frac{2-2\tau}{2+\tau}}\geq \eta^{3/2} R_{n+1}^{1-5\tau/3}.
%\]
%taking $\tau$ small enough. 
Thus, $k_{n+1}$ satisfies items  1 and 2.

In order to obtain the first $k_1$, it is enough to apply the classical 
Dirichlet Theorem. This $k_1$ satisfies better estimates.

Now, to prove that the resonant segment $\Ga^{\om^\ast}_{k_{n+1}}$ 
intersects $\Ga^{\om^\ast}_{k_n}$. We start by pointing out that, 
using the already obtained estimates from item 4
\[
 \mathrm{dist}\left(\Ga^{\om^\ast}_{k_n},\om\right)=
\frac{|k\cdot\om^\ast|}{|k|}\leq 
 \eta^{-7/2}R_{n+1}^{-(3-4\tau)}.
\]
Then, taking into account that we have a uniform bound of the angle 
between $k_n$ and $k_{n+1}$, one can take $\xi$ large enough 
independent of $R_{n+1}$ so that the two resonant segments intersect.

It only remains to see that any other resonance crossing 
$\Ga^{\om^\ast}_{k_n}$ is of almost the same order. 
Assume that there is a resonance $\Ga_k$ intersects the disk
\[
 \mathbb{D}\left(\om^\ast\right)=\left\{|\om-\om^\ast|\leq  \xi\eta^{-7/2}R_{n+1}^{-(3-4\tau)}\right\},
\]
This implies that the distance to $\omega^\ast$ has to satisfy
\[
 \mathrm{dist}(\Ga_k,\om^\ast)=\frac{|k\cdot\om^\ast|}{|k|}\leq \xi\eta^{-7/2}R_{n+1}^{-(3-4\tau)}.
\]
which, using the Diophantine condition \eqref{def:Diophantine} implies
\[
 |k|\geq \xi\ii\eta^{7/2}R_{n+1}^{3-4\tau}|k\cdot\om^\ast|\geq \xi\ii\eta^{9/2}R_{n+1}^{3-4\tau} |k|^{-(2+\tau)}
\]
and then
\[
 |k|^{3+\tau}\geq \xi\ii\eta^{9/2}R_{n+1}^{3-4\tau}.
\]
Then, taking $\tau$ small enough, one can get the last estimate of 
Theorem \ref{KeyThm:SelectionResonances}.
%Using the definition of $\nu_2(\tau)$ in Theorem \ref{KeyThm:SelectionResonances} and choosing a suitable $c$ one can get 
%item 5 of Theorem \ref{KeyThm:SelectionResonances}.
\end{proof}

\subsection{Net of resonances: proof of Key Theorem \ref{keythm:ResonancesWithNoTriple}}\label{sec:ResonancesWithNoTriple}
In Section \ref{sec:SelectionResonances} we have obtained a sequence 
of  Dirichlet resonances approaching a single Diophantine 
frequency $\om^\ast\in \DDD_{\eta,\tau}^U$. Now we generalize this result to obtain 
sequences of resonances approaching all frequencies in 
$\DDD_{\eta,\tau}^U$ and we prove Theorem \ref{keythm:ResonancesWithNoTriple}. Since the number of  Dirichlet resonant 
segments grows we need to control the complexity of possible 
intersections. We modify the strategy of the proof of Theorem \ref{KeyThm:SelectionResonances}.

In Section \ref{sec:OutlineResonances}, we have defined the  sequence of sets of Diophantines frequencies $\cD^n_{\eta,\tau}$, which are $3\rho_n$-dense subsets of $\cD_{\eta,\tau}^U$, such that no two points in $\cD^n_{\eta,\tau}$ are closer than $\rho_n$. In Lemma \ref{lemma:VoronoiCells} we have given some properties of these sets. Consider  $\om_n \in \cD^n_{\eta,\tau}$ and its Voronoi cell $\Vor_n(\om_n)$ (see Section \ref{sec:OutlineResonances}). Assume that at the stage $n$ we have obtained a resonance $\Gm_{k_{n}}$ which intersects $B_{\rho_n}(\om_n)\subset\Vor_n(\om_n)$ and we call $\SSS(\Gm_{k_{n}}^{\om_n})$ a segment of it. Now, at the following step, we consider the frequencies of $\cD^{n+1}_{\eta,\tau}$ which are contained in $\Vor_n(\om_n)$,
\[
\{\om^{j}_{n+1}\}_{j\in \FF^*}=\cD^{n+1}_{\eta,\tau} \cap \Vor_n(\om_n),  
\]
which have associated Voronoi cells $\Vor_n(\om_{n+1}^{j})$. From Lemma \ref{lemma:VoronoiCells} and using the definition of the sequence $\{\rr_n\}_{n\in\NN}$ one can easily deduce the following lemma.
\begin{lemma}\label{lemma:VoronoiDivision}
Take $\rr_0>0$ small enough. Then, for all $n\in\NN$ and any $\om_n\in \cD^{n}_{\eta,\tau}$. The set of points in $\om^*\in\cD^{n+1}_{\eta,\tau}$ such that $B_{3\rr_{n+1}}(\om^*)\cap \Vor_n(\om_n)\neq\emptyset$ contains at most $K\rr_n^{-4\tau}$ points, for some constant $K>0$ independent of $\rr_n$.
\end{lemma}
Now at each step $n$ we need to place resonant segments that intersect the corresponding Voronoi cells. We place them using the results in Section \ref{sec:SelectionResonances} in such a way that the placed resonant segments do not intersect resonant segments which do not belong to the neighboring generations. This, jointly with Lemma \ref{lemma:VoronoiDivision}, will give a good bound of the possible multiple intersections of Dirichlet resonant lines. 
%them successively avoiding that they produce triple intersections with the already placed resonant segments of the same step or the previous ones. 

Now we are ready to prove Key Theorem \ref{keythm:ResonancesWithNoTriple}. Recall that we are looking for  Dirichlet resonant lines $\Gm_{k^{*,j}_{n+1}}$
such that
\begin{itemize}
 \item $\Gm_{k^{*,j}_{n+1}}$ intersects $B_{\rho_{n+1}}(\om^{*,j}_n)\subset\Vor_n(\om_{n+1}^{*,j})$. 
 
 \item $\Gm_{k^{*,j}_{n+1}}$ intersects the previous resonant segment $\Gm_{k^{*}_{n}}$ inside 
 $B_{3\rho_n}(\om^*_n)$.

 \item Does not intersect resonant segments which do not belong to the neighboring generations.
%No three lines out of $\{\Gm_{k^{*,j}_{n+1}}\}_{j\in J^*}$ 
% and $\Gm_{k^*_n}$ nor any already placed resonant segments intersect or form a triangle with vertices $\rr_{n+1}^3$ close.
\end{itemize}

\begin{proof}[Proof of Key Theorem \ref{keythm:ResonancesWithNoTriple}]
%To prove Theorem \ref{thm:ResonancesWithNoTriple} we need to to modify the prove of Theorem \ref{KeyThm:SelectionResonances}. 
As in the proof of Theorem \ref{KeyThm:SelectionResonances},  the norm $\|\cdot\|$ measures the distance to the integers. 
We prove this Theorem inductively. Assume that we have placed the Dirichlet  resonant segments for the frequencies of the  first $n$ generations and fix a frequency $\om_n\in \cD_{\eta,\tau}^{n}$. Now we place the Dirichlet resonant segments
%Order the points in $\cD_{\eta,\tau}^n$ and consider the associated Voronoi cells. Asssume that we have already placed the resonant segments for the frequencies in $\cD_{\eta,\tau}^{n+1}$ which belong to the Voronoi cells of the first $i-1$ frequencies in $\cD_{\eta,\tau}^n$. Call $\om_n^*$ to the $i$-th frequency and $k_n$ to the Dirichlet integer vector associated to it. Now we need to obtain the Dirichlet resonant segments 
associated with the frequencies $\cD_{\eta,\tau}^{n+1}\cap\Vor_n(\om_n)$. 
%Order such frequencies and call them $\om_{n+1}^{*,1}, \ldots, \om_{n+1}^{*,N}$. Assume that we have already obtained the  Dirichlet resonant lines associated to the first $j-1$ frequencies.
Assume, by inductive hypothesis, that all the already placed Dirichlet resonant lines satisfy the properties stated in Theorem \ref{keythm:ResonancesWithNoTriple}.

Now, proceeding as in the proof of Theorem \ref{KeyThm:SelectionResonances}, for each $\om_{n+1}^{*,j}\in\cD_{\eta,\tau}^{n+1}\cap\Vor_n(\om_n)$ we can obtain a Dirchlet resonant integer satisfying 
\[
 \frac{R_{n+1}}{4}\leq |k_{n+1}^{*,j}|\leq R_{n+1}
\]
and 
\[
 \eta R_n^{-(2+\tau)}\leq |k_{n+1}^{*,j}\cdot \om_{n+1}^{*,j}|\leq 16\eta^{-2}R_n^{-(2-2\tau)}.
\]
Thus, taking $R_0$ large enough (independently of $n$), we obtain the first two properties of the vectors $k_{n+1}^{*,j}$. Statements 4 and 5, 6 and 7 are also proved as in Theorem \ref{KeyThm:SelectionResonances}. 

Now we prove the last statement.  We prove that  that the  resonant line $\Ga_{k_n}$ can only intersect segments which  belong to the previous, next to previous and future generations. We first show that $\Ga_{k_n}$ cannot intersect any $\rr_m$ ball around any frequency of $\cD_{\eta,\tau}^{m}$  with $m\geq n+2$. Take $\om_m^*\in \cD_{\eta,\tau}^{m}$. Then, the distance 
\[
 \begin{split}
\mathrm{dist}(\Ga_{k_n},\om_m^*)&=\frac{|k_n\cdot \om_m^*|}{|k_n|}\geq \frac{\eta}{|k_n|^{3+\tau}}\\
&\geq \frac{1}{R_n^{3+\tau}}\geq \rr_n^{1+3\tau}> \rr_{n+2}
 \end{split}
\]
Therefore $\mathrm{dist}(\Ga_{k_n},\om_m^*)>\rr_m$ for any $m\geq n+2$. Since a Dirichlet resonant segment $\SSS_{k_m}^{\om_m}$  associated to a frequency $\om_m$ of the generation $m\geq n+2$ satisfies $\SSS_{k_m}^{\om_m}\subset B_{\rr_m}(\om_{m-1})$, where $\om_{m-1}$ is such that $\om_m\in\Vor_{m-1}(\om_{m-1})$, it cannot intersect $\SSS_{k_n}^{\om_n}$.

Recall that, by statement 6, we know that $\SSS_{k_n}^{\om_n}\subset B_{\rr_n}(\om_{n-1}^*)$. Reasoning analogously, the resonant lines associated to frequencies in $\cD_{\eta,\tau}^{m}$ with $m\leq n-3$ cannot intersect  $B_{\rr_n}(\om_{n-1}^*)$ and therefore, they cannot intersect $\SSS_{k_n}^{\om_n}$ either. This completes the proof of Theorem \ref{keythm:ResonancesWithNoTriple}.
\end{proof}

Analyzing the resonances belong to a generation $n\geq 1$, we can see that in suitable neighborhoods of the intersections of resonances, there cannot exist other intersections.  This fact will be crucial then, in the analysis of the resonances. It is stated in the next two lemmas.

\begin{lemma}\label{lemma:NoNeighboringTripleEssentialIntersections}
Consider $\SSS_{k'}$ and $\SSS_{k''}$, two of the resonant segments obtained in Theorem \ref{keythm:ResonancesWithNoTriple} which intersect  at a point  $\om'=\SSS_{k'}\cap\SSS_{k''}$. Assume $k$ belongs to the $n$-th generation and $k'$ belongs to the $n+1$ generation. Then, any other resonant segment $\SSS_{k}$ of those   obtained in Theorem \ref{keythm:ResonancesWithNoTriple} intersecting the ball $B_{\rr_n^2}(\om')$ must also contain $\om'$.
\end{lemma}

\begin{lemma}\label{lemma:NoNeighboringTripleThetaStrongIntersections}
Fix $\theta>1$, $n\in\NN$ and consider $\SSS_{k'}$ one of the resonant segments of the $n$ generation obtained in Theorem \ref{keythm:ResonancesWithNoTriple}. Assume that a resonant line $\Ga_k''$ with $|k''|\leq R_n^\theta$ intersects $\SSS_{k'}$ and call $\om'$ to the intersection point. Then, 
\begin{itemize}
\item Any other resonant line $\Ga_k$ with $|k|\leq R_n^\theta$  intersecting the ball $B_{\rr_n^{(1+2\theta)/(3-5\tau)}}(\om')$ must also contain $\om'$.
\item Any other resonant segment $\SSS_{k}$ of those   obtained in Theorem \ref{keythm:ResonancesWithNoTriple} intersecting the ball  $B_{\rr_n^{(1+2\theta)/(3-5\tau)}}(\om')$ must also contain  $\om'$.
\end{itemize}
\end{lemma}

We devote the rest of the sections to prove
% Theorem \ref{thm:ResonancesWithNoTriple} and 
Lemmas \ref{lemma:NoNeighboringTripleEssentialIntersections} and \ref{lemma:NoNeighboringTripleThetaStrongIntersections}.

\begin{proof}[Proof of Lemma \ref{lemma:NoNeighboringTripleEssentialIntersections}]
Let $\om' \in\SSS_{k'}\cap \SSS_{k''}$ 
be  a $\theta$-strong double resonance. Let 
$|\om'|\ge \nu$. 

By definition $\om'\cdot k' = \om' \cdot k''=0$.
Therefore, $\om \parallel k' \times k''$, where 
$\times$ is the wedge product. Let $k^*$ be a resonance such that $\Gm_{k^*}$ passes 
through the $\rho_n^2$-neighborhood of $\om'$ but does not contain $\om'$. Then, $k^*$ satisfies $|\om \cdot k^*|\le \rho_n^2 |k^*|$. On the other hand, since $k^*$ is not a linear combination of $k'$ and $k''$ (otherwise  $\Gm_{k^*}$ would contain $\om'$, we have that $(k'\times k'')\cdot k^*\neq0$. Since they are integer vectors, this implies that $|(k'\times k'')\cdot k^*|\geq 1$. Then, 
\[
 |\om' \cdot k^*| \ge 
 \dfrac{\nu |(k'\times k'')\cdot k^*|}{|k' \times k''|}
\ge\dfrac{\nu}{|k' \times k''|} 
\ge \nu R^{-(2+2\tau)}_n. 
\]
Thus, recalling that $|\om \cdot k^*|\le \rho_n^2 |k^*|$ and the definition of $\rr_n$, we obtain
\[
 \nu R_n^{-(2+2\tau)} \le \rr_n^{2} |k^*| \le R_n^{-(6-10\tau)} |k^*|,
\]
which implies that $|k^*| \geq \nu R_n^{4-12\tau}$ and therefore, it cannot belong to the generations $n$ or $n+1$. Assume that it belongs to a generation $m\geq n+2$ and we reach a contradiction. 

If $k^*$ belongs to the $m$ generation, we have that $R_m\geq \nu R_n^{4-12\tau}$. Then, using the relation between $R_m$ and $\rr_m$, we can ensure that there is $\om_m\in D_{\eta,\tau}^m$ which satisfies $\mathrm{dist}(\Ga_{k^*},\om_m)\leq\rr_m\leq \rr_n^{4-13\tau}$. This implies that $\om' \in\SSS_{k'}\cap \SSS_{k''}$ satisfies
$\mathrm{dist}(\om',\om_m)\leq 2\rr_n^2$. Therefore, $\mathrm{dist}(\SSS_{k'},\om)\leq 2\rr_n^2$ and $\mathrm{dist}(\SSS_{k''},\om)\leq 2\rr_n^2$. Nevertheless, this violates the Diophantine condition satisfied by $\om_m$ and therefore we have reached a contradiction.
\end{proof}

\begin{proof}[Proof of Lemma \ref{lemma:NoNeighboringTripleThetaStrongIntersections}]
It follows the same lines as the proof of Lemma \ref{lemma:NoNeighboringTripleEssentialIntersections}.  We take $\nu$ such that 
$|\om'|\ge \nu$ and   $\Gm_{k^*}$ which passes 
through the $\rho_n^{1+\theta}$-neighborhood of $\om'$ but does not contain $\om'$. Then, proceeding as in the proof of Lemma \ref{lemma:NoNeighboringTripleEssentialIntersections}, one can see that $k^*$ must satisfy 
\[
 \mathrm{dist}(\om', \Ga_{k^*})\geq \frac{\nu}{R_n^{1+\theta}|k^*|}\geq\frac{\nu}{R_n^{1+2\theta}}\geq \nu\rr_n^{(1+2\theta)/(3-5\tau)} ,
\]
This proves the first statement of the lemma. The second one can be proved as in Lemma \ref{lemma:NoNeighboringTripleEssentialIntersections}.
\end{proof}

\section{The deformation procedure}\label{sec:ConstructionOfDeformation}

We devote this section to explain the different deformations that we need to perform to the Hamiltonian $H_1$ in \eqref{def:Ham:Original:0}. These deformations have been explained in Section \ref{sec:Outline}.

The first deformation is to mollify $H_1$, as stated in Lemma  \ref{lm:molification}. 

%\begin{comment}
%The proof of this lemma is a direct consequence of the following result, which is proven in \cite{SalamonZ89}.
%
%\begin{lemma}\label{lm:SalamonZehnder}
%There is a family of convolution operators
%\[
%S_r f (x) = r^{-n} \int_{\R^n} 
%K_r(r^{-1}(x - y))f (y) dy,\quad  0 < r \le 1,
% \]
%from $\CCC^0 (\R^n)$ into the space of entire functions 
%on $\R^n$ with the following property. For every 
%$d\ge \ell\ge 0$, there exists a constant $c = c(\ell,n,d) > 0$ 
%such that, for every $f \in \CCC^\ell (R^n )$ we have 
%\[
% \|S_r f - f \|_{\CCC^s} \le 
% c\|f\|_{\CCC^\ell} r^{\ell-s}, \quad s\le \ell, 
%\]
%\[
% \|S_r f \|_{\CCC^s} \le c\|f\|_{\CCC^\ell} r^{\ell-s}, 
% \quad  \ell\le s\le d. 
%\]
%Finally, if $f$ is periodic in some variables, then so are 
%the approximating functions $S_r f$ in the same variables. 
%\end{lemma}
%\textcolor{blue}{Either refer to density of $\CCC^{2r}$ or mention closeness in $\CCC^r$ norm}.
%\end{comment}
We devote the rest of this section to explain the other deformations.

\subsection{Perturbation along resonant segments}\label{sec:deformationSR&DRstep0}
As explained in Section \ref{sec:Outline}, the perturbation along resonances has two different terms. One supported in the neighborhoods of single resonances and the other one supported in the neighborhood of double resonances. In fact, this perturbation is constructed as  two infinite sums 
\be \label{eqn:perturbation-single&double}  
 \Delta H^{\sr}+\Delta H^{\dr}=
 \sum_{n=0}^\infty \Delta H^{\sr}_n+\Delta H^{\dr}_n.
\ee
%All perturbations are localized in the phase space and their localizations have different nature. 
The support
of $\Delta H^{\sr}_0$ is localized along the single resonances in $\II_{k}$, $k\in\FF_0$ (see Definition \ref{def:ResonanceNet}). The support of $\Delta H^{\dr}_0$  is localized in the neighborhoods of the double resonances which appear along the resonances in $\bI_0$. The supports 
of $\Delta H^{\sr}_n$, $n\geq 1$ are localized in action space in a neighborhood of the
Dirichlet resonances $\II_{k_n}^{\om_n}$ of the  $n$-th generation selected in Key Theorem \ref{keythm:ResonancesWithNoTriple}.
The supports of $\Delta H^{\dr}_n$ are localized in action space 
in the neighborhoods of  $\th$-strong double resonances.
%intersections of resonant segments (double resonances)
%of two types: 
%\begin{itemize}
%\item Essential double resonances --- intersection of Dirichlet 
%resonances $\SSS_{k_n}^{\om_n}\cap \SSS_{k_m}^{\om_m}$ given by Key Theorem \ref{keythm:ResonancesWithNoTriple}.
%above.
%\item $\th$-strong double resonances --- intersection of Dirichlet 
%resonance $\SSS_{k_n}^{\om_n}$  and a resonance $\SSS_{k'}$, 
%which is not Dirichlet of order $n$ and such that $|k'|<|k|^\th$ (see Section \ref{sec:Outline}).
%\end{itemize}
By construction supports of $\Delta H^{\sr}_n$ (resp. $\Delta H^{\dr}_n$ ) 
and $\Delta H^{\sr}_m$ (resp. $\Delta H^{\dr}_m$ ) are 
disjoint if $|n-m|>1$. 

Recall that we are dealing with
the $\CCC^\infty$ Hamiltonian 
\[
H_\eps' = H_0+\eps H_1 + \eps \Delta H^{\mol},
\]
given by Lemma \ref{lm:molification}. 
%While the $\CCC^r$-norm of $H_1 + \Delta H^{\mol}$  is bounded, its $\CCC^{r+\kappa}$-norms with $\kappa>0$ are bounded by inverse powers of the mollification parameter $\rr_n^{-\kappa}$ (see Lemma \ref{lm:molification}). 

\subsubsection{The perturbation along the resonance of $\bI_0$}
 The analysis of the step 0 resonances is done in \cite{BernardKZ11} and \cite{KaloshinZ12}. The step 0  perturbation 
%\be \label{eqn:1st-perturbation-single&double}  
$\Delta H^{(0)}=\Delta H^{\sr}_0+\Delta H^{\dr}_0$
%\ee
is designed to construct diffusion along the zero generation 
of resonant segments $\{\II_{k}\}_{k\in \FF_0}$. Let $\bI_0=\{\II_{k}\}_{k\in \FF_0}$ be all single resonant
segments of the zero generation. Then 
\be \label{eqn:1st-perturbation-single-local}  
  \Delta H^{\sr}_0 = \sum_{k\in \FF_0} 
 \Delta H^{k}_0 \circ \Phi_{k}^{-1}, 
\ee
where $\Phi_{k}$ is a canonical change of variables which is obtained in \cite{BernardKZ11} and  is
defined in 
%\[
$\T^2\times B_{\eps^{1/4}}(\bI_0)\times  \T$. 
%\]
As explained in \cite{BernardKZ11}, this change is chosen such that 
\[
 H_\eps'\circ \Phi_{k}= H_0 + \eps \ZZZ_{k}
 + \dt \eps \RRR_k
\]
for some small $\dt>0$ and 
%\[
$\|\ZZZ_k\|_{\CCC^2}\le C$,
%\qquad  
$\|\RRR_k\|_{\CCC^2}\le C$,
%\]
for some $C$ independent of $\rr$ and $\eps$.

Let $I^{ij}=\II_{k_j}\cap \II_{k_i},\ k_i,k_j\in \FF_0$ be 
all double resonances given by non-empty intersections 
of resonant segments of the first generation with other first generation resonant segments, or other resonant segments creating  double strong resonances, as defined in \cite{BernardKZ11, KaloshinZ12}. Let 
$E=E(H_0,H_1)$ be a large number. Then 
\be \label{eqn:1st-perturbation-double-local}   
   \Delta H^{\dr}_0 = \sum_{k_i,k_j\in \FF_0} 
 \Delta H^{ij}_0 \circ \Phi_{ij}^{-1}, 
\ee
where $\Phi_{ij}$ is obtained in \cite{KaloshinZ12}. It is the composition of a canonical change of variables and a rescaling and is
defined in 
%\[
$\T^2\times B_{E\sqrt \eps}(\om^{ij})\times  \T$.
%\]
As explained in \cite{KaloshinZ12}, is constructed so that
\[
 H_\eps'\circ \Phi_{ij}= K_{ij} + \eps \ZZZ_{ij}
 + \eps^{3/2} \RRR_{ij},
\]
where $K_{ij}$ is a positive definite quadratic form,
$\ZZZ_{ij}$ is a potential depending on a $2$-dimensional 
slow angle $\phi^s_{ij}$, and $\|\RRR_{ij}\|_{\CCC^2}\le C$.

After this step 0 deformation, we have the Hamiltonian 
\be \label{eqn:Hamiltonian-after-1st-stage}
 H_\eps^{(0)}:=H_\eps'+\Delta H^{\dr}_0 +\Delta H^{\sr}_0.
\ee

%\begin{comment}
%\subsection{The second perturbation. Step 1. Birkhoff Normal Form}
%
%
%Recall that for each integer $n\ge 1$ we have 
%a discrete subset $\DDD^n_{\eta,\tau}$ of 
%$(\eta,\tau)$-Diophantine numbers such that 
%it is $\rho_n$-dense in $\DDD_{\eta,\tau}$ and 
%any two distinct points are at least $\rho_n/3$ apart.
%
%By construction of resonant segments of the first 
%generation the union $\cup_{j\in J_0} \SSS(k_j)$ is 
%a connected set and it enters $\rho_1$-neighborhood 
%of each $\om_j\in \DDD^1_{\eta,\tau}$. 
%
%In the {\it frequency space} $\R^2$ we construct Voronoi cells 
%$\Vor_1(\om_j)$ around each $\om_j\in \DDD^1_{\eta,\tau}$. 
%Loosely speaking, this is a $\rho_1$-neighborhood of $\om_i$.  
%We construct a Birkhoff Normal Form around a KAM torus $\TTT_{\om^*}$.
%
%\bthm \label{thm:Birkhoff-Normal-Form}
%Fix $0<m<r$. For each $\rho<\rho_1$ and each 
%$\om^* \in \DDD_{\eta,\tau}$
%there is a canonical change of coordinates 
%$ \Phi^m_{\om^*}:U_{\rho}(\TTT_{\om^*})\to B^2 \times \T^2 \times \T$ 
%and $c=c(H_0^{\om^*})$  such that 
%\be \label{eqn:BNF}
% \bar H^{(1)}=H_\eps^{(1)} \circ (\Phi^m_{\om^*})^{-1}= 
% H^{\om^*}_0  + \rho^m H^{\om^*}_1.
%\ee 
%Moreover, we have 
%\[
% \|H^{\om^*}_0\|_{\CCC^{r+z-3m-1}}<+\infty,
%\] 
%\[ 
% \| H^{\om^*}_1 \|_{\CCC^{r-3m-1}}\le c,\ \  
% \|H^{\om^*}_1 \|_{\CCC^{r-3m-1+\kappa}}\le c\, \rho^{-\kappa(1+4\tau)},\ 
% \ \ \text{ for any }\kappa>0.
% \]
% \ethm 
% 
%To simplify notations we omit dependence on $\rho$ and $\om^*$.
%
%\vskip 0.1in 
%\end{comment}

\subsubsection{The perturbation along resonances of $\bI_n$ for $n\geq 1$}\label{sec:deformationSR&DRstepN}

The rest of the  perturbations 
\be \label{eqn:2nd-perturbation-single&double}  
\Delta H^{(n)}=\Delta H^{\sr}_ n+\Delta H^{\dr}_n
\ee
are designed to construct diffusion along the 
generations of resonant segments $\{\II_{k_n}^{\om_n}\}_{k_n\in \FF_n}$, $n\geq 1$. 
Each such segment is a Dirichlet resonance 
of the $n$ generation. We construct these deformations inductively. Assume that one has already modified the Hamiltonian along the Dirichlet resonant segments belonging to the generations $1, 2,\ldots, n-1$.  We denote by $ H_\eps^{(n-1)}$ the Hamiltonian after adding the generation $1, 2,\ldots, n-1$ deformations. Now we show how to construct the generation $n$ deformation.

Since the Voronoi cells constructed in Lemma \ref{lemma:VoronoiCells} are pairwise
disjoint, we restrict our discussion to one Voronoi cell 
$\Vor_n(\om_n^*)$ with $\om_n^*\in \DDD^n_{\eta,\tau}$.  Denote by  $\FF_n(\om^*_{n})$ the set of resonant vectors of the $n$ generation such that the associated resonant segments  intersect the Voronoi cell $\Vor_n(\om_{n}^*)$. That is,
\[
\FF_n(\om_{n}^*)=\left\{k_n\in \FF_n: \II_{k_n}^{\om_n}\cap  \Vor_n(\om_{n}^*)\neq \emptyset \right\}.
\]
%These are all Dirichlet resonant segment of the first generation located inside the Voronoi cell $\Vor_1(\om^*)$. 
%Let $k_n\in\FF_1(\om^*)$. Denote by $\FF_1'(\om^*,k_n)$ the set of 
%$k'\in \Z^3,$ $k'$ is not Dirichlet of the first generation,  $k'\notin J_1(\om^*)$, $|k|<|k'|^\th$, and $\SSS_{k_n}^{\om_n}\cap \Gm_{k'}$ is nonempty and intersects $\Vor_1(\om^*)$. Call such intersections 
%{\it $\th$-strong double resonance}. 
%Call intersections of two Dirichlet segments of order one 
%\[
%\SSS(k_j)\cap \SSS(k_i) \subset \Vor_1(\om^*),
%\ i,j\in J_1(\om^*). 
%\] 
%{\it essential double resonance} of the first order. 
Let $\th>1$ be specified later. Then we define 
\be \label{eqn:2nd-perturbation-single-local}  
  \Delta H^{\sr,\om^*}_n = \sum_{k_n\in \FF_n(\om^*)} 
 \Delta H^{k_n}_n \circ \Phi_{k_n}^{-1}, 
\ee
where $\Phi_{k_n}$ is a change of variables defined 
for  the single resonance zone $\Tr_n(k_n,k',k'')$. The functions $\Delta H^{k_n}_n$ are obtained in Section \ref{sec:DeformationSinglePotential} and the change $\Phi_{k_n}$ is obtained through a normal form procedure in Section \ref{sec:NF:Transition}.
%Notice that in order to identify location of resonances 
%$\{\SSS(k_j)\}_j,\ j\in J_1(\om^*)$ we need to consider 
%the Hamiltonian $\bar H^{(1)}$ in the Birkhoff Normal Form
%$\bar H^{(1)} \circ \Phi^m_{\om^*}$.
Now we define
\[
  \ol H_\eps^{(n)}:= H_\eps^{(n-1)} +  
 \sum_{\om^* \in \DDD^n_{\eta,\tau}} \Delta H^{\sr,\om^*}_n.
\]
Let $W_n^{\om^*}=\{\om^{j}\}_{j\in\JJ_n}$ be the set of all 
%essential and 
$\th$-strong double resonances inside 
the Voronoi cell $\Vor_n(\om^*)$ except the ones involving Dirichlet resonant segments belonging to the previous generations. 
Define  
\begin{equation} \label{eqn:2nd-perturbation-double-local}  
 \Delta H^{\dr,\om^*}_n = \sum_{j\in\JJ_n} 
 \Delta H^{j}_n \circ\Phi_{j}^{-1}
\end{equation}
where the function $\Delta H^{j}_n$ is obtained in Appendix \ref{app:NonMechanicalAtDR}. The change of variables  $\Phi_{j}$  is defined for 
%\[
$(\phi,I,t)\in  \T^2 \times B_{C\rho_n^m}(\om^{j})\times  \T)$ for some $C\geq 0$ independent of $\rr_n$. Recall that $m=\theta+1$ (see \eqref{def:CoreDR-m} and Appendix \ref{sec:Notations} for the value of the constants $\theta$ and $m$). 
%\]
The change of coordinates $\Phi_{j}$ is given by  Bounemoura Normal Form and is obtained in Section \ref{sec:NF:DR} (see Theorem \ref{thm:BounemouraNormalForm:linear}, which is based on \cite{Bounemoura10}). As explained in that section, the change $\Phi_j$ is constructed such that
\[
 \bar H^{(n)}\circ \Phi_{j}= 
 K_{j} + \ZZZ_{j} +  \RRR_{j}.
\]
where $K_{j}$ is Hamiltonian which only depends on the actions,
$\ZZZ_{j}$ is function which only depends on the actions and the $2$-dimensional 
slow angle $\phi^s_{j}$, and $\|\RRR_{j}\|_{\CCC^2}\le C\rho_n^{q(r+1)}$ (see Section \ref{sec:NF:DR}). The deformation is constructed so that   $\RRR_j$ is eliminated. Lemma \ref{lemma:NoNeighboringTripleThetaStrongIntersections} implies that the sets $\T^2 \times B_{C\rho_n^m}(\om^{j})\times  \T$ are disjoint and thus we can deform the core of all the double resonances without overlapping the deformations (the deformations along single resonances and in double resonances certainly overlap).

Thus, now we consider the Hamiltonian 
\be \label{eqn:Hamiltonian-after-2nd-stage}
H^{(n)}_\eps :=\bar H^{(n)}_\eps +  
 \sum_{\om^* \in \DDD^n_{\eta,\tau}} \Delta H^{\sr,\om^*}_n+
 \sum_{\om^* \in \DDD^n_{\eta,\tau}} \Delta H^{\dr,\om^*}_n. 
\ee
which is the Hamiltonian after the generation $n$ deformation.

\section{Different regimes \& normal forms approaching Diophantine $\om$'s}\label{sec:NormalForms}

In this section we want to analyze the Hamiltonian \eqref{def:HamAfterMolification} when we  approach the Diophantine frequencies through the resonant segments obtained in Theorem \ref{keythm:ResonancesWithNoTriple}. To this end, we fix some $\om^\ast\in\DDD_{\eta,\tau}^n$ for some $n\geq 1$ and we analyze a small neighborhood of it. Recall that the analysis of the generation 0 resonances is done in \cite{BernardKZ11} and \cite{KaloshinZ12}.

As we have explained in Section \ref{sec:Outline} (see also Section \ref{sec:ConstructionNetResonances}), we approach the Diophantine frequencies through  Dirichlet resonant segments, which intersect at double resonances. To analyze them, we cover the resonant segments with balls centered at $\theta$-\emph{strong double resonances} (see Definition \ref{definition:StrongWeakDR}). Then, in each of these balls we will perform certain normal forms. 

First in Section \ref{sec:Covering}, we consider one of the resonant vectors $k_n$ obtained in Theorem \ref{keythm:ResonancesWithNoTriple}. Then, we show that the associated resonant segment $\II_{k_n}^{\om_n}$ in action space can be covered by neighborhoods of double resonances of sufficiently low order. We divide these neighborhoods into two parts where we perform different normal forms.  

In Section \ref{sec:NF:Transition} we analyze the so-called single resonance zones \eqref{def:TransitionZone}. We proceed analogously as the single resonance zones in \cite{BernardKZ11,KaloshinZ12}. Nevertheless, here we need more accurate normal forms. In  Section \ref{sec:NF:DR} we analyze what we call the core of the double resonances \eqref{def:CoreDR-m}, which are small balls around $\theta$-strong double resonances. 

Thanks to these normal form procedure, we obtain good first orders of the Hamiltonian system in the neighborhoods of those double resonances. 
%This is done in Section \ref{sec:DRAnalysis}. We consider two different cases: strong double resonances, those ones 
%which lie 
%in the intersections of two resonant lines of similar order, and the weak double resonances, which are at the intersection of two resonances, one much weaker than the other. 
Later, in Section \ref{sec:NHIC}, we use these normal forms 
and new first orders to prove the existence of NHICs 
along the resonances in the different regimes.

\subsection{Covering the  Dirichlet segments by 
%M, it looks nicer 
%neighborhoods of 
strong double resonances}\label{sec:Covering}
We need to analyze the Hamiltonian \eqref{def:HamAfterMolification} in a neighborhood of the Dirichlet resonant segments $\II_{k_n}^{\om_n}$ obtained 
in Theorem \ref{keythm:ResonancesWithNoTriple} (recall that $\II_{k_n}^{\om_n}$ is a resonant segment in action space and $\SSS_{k_n}^{\om_n}$ is the same resonant segment in frequency space). To analyze this segment 
we follow the approach in \cite{KaloshinZ12}.
% Nevertheless, instead of consider single and double resonance regimes, here we cover $\II_{k_n}^{\om_n}$ by a sequence of balls centered  at $\theta$-strong double resonances. 

In this section, we fix $k=k_n$ and consider the associated resonant segment $\II_k=\II_{k_n}^{\om_n}$. Recall $\II_k$  belongs to the ball $B_{3\rr_{n-1}}(\om_{n_1}^*)$ centered at a frequency of the previous generation $\DDD_{\eta,\tau}^{n-1}$. To simplify notation, we define $\rr=\rr_{n-1}$. 

First we  analyze the density of  $\theta$-strong double resonances in $\II_k$. The measure of how separated they are has been given in Lemma \ref{lemma:NoNeighboringTripleThetaStrongIntersections}.

\begin{lemma}\label{lemma:AbundanceDoubleResonances}
Consider the resonant segment $\II_k$ and consider the double resonant points 
\[
P_{k'}= \II_k\cap \{k'\cdot\omega =0\}
\]
for $|k'|\leq K$. Then, for each such point $P_{k'}$ there exists another point $P_{k''}$ such that
\[
 \left|P_{k'}-P_{k''}\right|\leq C K^{-2}|k|
\]
for some constant $C$ independent of $k$ and $K$.
\end{lemma}
\begin{proof}
To prove this lemma we make first a change of variables in the frequency space 
to put $\Gamma_k$ as a horizontal line. We define the new frequencies 
$v_1=k\cdot \omega$, $v_2 =\om_2$ and $v_3=\om_3=1$.  Now, the  new frequency 
$v=(v_1,v_2,1)\in\Gamma_k$ if $v_1=0$. A double resonance corresponds to 
$v_2\in\Q$. Now we can apply Dirichlet Theorem to ensure that for any 
$v_2\in\Ga_k$ there exist integers $q$ and $p$ with $|q|\leq K$ 
such that
\[
 \left|v_2-\frac{p}{q}\right|\leq \frac{1}{K^2}.
\]
Thus, in the $v$ coordinates, we obtain that the double resonances are at 
a distance at most $2K^{-2}$. Undoing the change of coordinates and measuring 
the distance of the double resonances in the original frequency space we obtain 
the wanted estimates. Recall that the frequency map $\Omega$ has estimates indepedent of $\eps$ and $\rr$ and therefore, the estimates in frequency and action space are equivalent.
\end{proof}

%We use this lemma to cover the resonant segment  $\II_k$  by neighborhoods centered at double resonances. 
%First we rescale the variables so that the Voronoi cell around the Diophantine frequency is of order 1. Assume that the Diophantine frequency corresponds to the action $I=I_0$. Then, we perform the change
%\begin{equation}\label{def:ScalingCloseDiophantine}
%  I=\rr \wt I, 
%\end{equation}
%This change is conformally symplectic and, 
%Take $I_0$ such that $\nabla H_0(I_0)=\om_^*$. We study the Hamiltonian \eqref{def:HamForAveraging} in a ball $B_{C_0}(I_0)=\left\{I\in\RR^n:| I-I_0|\leq C_0\rr\right\}$ in action space. Here $C_0$ is a constant independent of $\rr$ satisfying  $C_0>3$.  
%Note that this constant $C_0$ can be taken independently of 
%$\om^\ast\in D_\eta$, namely only depending on $\eta$ and $\tau$ (I think it can be taken just as a fixed number say $C_0=20$). Now, $\SSS(\Ga_k)$, expressed in the new variables, belongs to 
%the ball $B_{C_0}(0)$ .
%Marcel, does not relevant enough formula for a separate line. 
%Later I replace radius C by radius \rho 
%The new Hamiltonian satisfies 
%\[
% \|H_0\|_{\CCC^s}\leq C, \|H_1\|_{\CCC^s}\leq C.
%\]
%and 
%\[
%\rr D^{-1}\|v\| \le \langle \partial^2 h_0(I) v,v\rangle \le \rr D\|v\|
%\text{ for any }I \in \R^n \text{ and }v\in \R^n. 
%\]
%Then, $k$ satisfies
%\[
% C_1\rr^{-\frac{1}{3}+\tau}\leq |k|\leq C_2\rr^{-\frac{1}{3}-\tau}.
%\]
Denote a ball centered at a $\theta$-strong double resonance point $P_{k'}$ 
of radius $\mu_n$ as 
\[
 B_{\mu_n}({P_{k'}})=\left\{|I-P_{k'}|\leq \mu\right\}.
\]
Recall the Definition \ref{definition:StrongWeakDR} of $\theta$-strong double resonance. Then, taking
$K=R_n^\theta=R_{n-1}^{\theta(1+2\tau)}=\rr^{-\theta\frac{1+2\tau}{3-5\tau}}$,   Lemma \ref{lemma:AbundanceDoubleResonances} 
ensures that, for some constant $C_3>0$, the curve $\II_k$ satisfies
\[
 \II_k\subset\bigcup_{|k'|\leq R_n^\theta}  B_{\mu_n}(P_{k'})
\]
with 
\begin{equation}\label{def:NbdDR}
 \mu_n=C_3\rr^{(2\theta-1)\frac{1+2\tau}{3-5\tau}}.
\end{equation}

%\emph{Vadim: here I'm being sharp with the exponents and therefore are quite terrible. I think at the end I won't need to be that sharp and probably can be changed to more decent expressions.}

%\subsection{Core of double resonances and transition zones}

%We fix a Dirichlet resonant segment $\II_k=\II_{k_n}^{\om_n}$ of the $n$ generation. As explained in Section \ref{sec:OutlineResonances}, we fix $\theta>0$ and we consider all the $\theta$-strong double resonances (see Definition \ref{definition:StrongWeakDR}). Then, following Lemma \ref{lemma:NoNeighboringTripleThetaStrongIntersections}, we consider balls of radius $\rr_n^\theta$ centered at such double resonances. Then, we know that such disk does not contain other $\theta$-strong double resonances.
%The constant $a$ is choosen such that such disks does not contain any other double strong resonance, and therefore depends on $\theta$.
 
As explained in Section \ref{sec:Outline}, we divide these balls into  two different parts: the single resonance zones and the core of the double resonance. The single resonance zone is defined in \eqref{def:TransitionZone} and it is analyzed in Section \ref{sec:NF:Transition}. The core of double resonances is  defined in \eqref{def:CoreDR-m} and is studied in Section \ref{sec:NF:DR}.

\subsection{Normal form along single resonance:
proof of Key Theorem \ref{keythm:Transition:NormalForm}}\label{sec:NF:Transition}
Following \cite{BernardKZ11,KaloshinZ12}, we study the Hamiltonian \eqref{def:HamAfterMolification} in  the single resonance zones $\Tr_n(k_n,k',k'')$ defined in  \eqref{def:TransitionZone}. We perform several steps of normal form instead of just one, since we need the remainder to be smaller than in those papers.

%Recall that we are dealing with a Hamiltonian of the form\begin{equation}\label{def:HamForAveraging}
% H(I,\phi)=H_0(I)+H_1 (I,\phi), 
%\end{equation}
%which satisfies
%\[
% \|H_1\|_{\CCC^s}\leq C\rr^m,\,\,\|H_1\|_{\CCC^{s+\kk}}\leq C\rr^{m-(1+4\tau)\kk}
%\]
%for any $0<\kk\leq r-s$. 

First we add a variable $E$ conjugate to time to Hamiltonian \eqref{def:HamAfterMolification} to have an autonomous Hamiltonian system, as done in \cite{BernardKZ11}. We define the Hamiltonian
\[
 \HH(\varphi,I,t,E)=N' (\varphi,I,t)+E.
\]
We abuse notation and  take $I=(I_1,I_2,E)$ and  $\varphi=(\varphi_1,\varphi_2,t)$. By Theorem \ref{thm:Poschel}, this Hamiltonian can be split as $\HH=\HH_0+\HH_1$ where
\[
%\begin{split}
 \HH_0(I)=H_0'(I_1,I_2)+E\, , \qquad \qquad
\HH_1(\varphi,I)=R(\varphi_1,\varphi_2,I_1,I_2,t).
%\end{split}
\]
 We perform a translation to assume that the double resonance is located at $I=0$ and we perform the following rescaling to the action variables
\begin{equation}\label{def:rescalingSR}
\Phi^{\sc}: J\longrightarrow I=\rr^{m} J
\end{equation}
where $m$ has been introduced in \eqref{def:CoreDR-m}. After a time rescaling, the  system associated to the   Hamiltonian \eqref{def:HamAfterMolification} becomes Hamiltonian with respect to
\begin{equation}\label{def:HamBNFAfterRescaling}
 \HH(J,\varphi)=\HH_0(\rr^{m} J)+\HH_1(\varphi,\rr^{m} J).
\end{equation}
Denote by  $\wt \Tr_n(k_n,k',k'')$ the rescaled single resonance zone. We perform several steps of normal form to this Hamiltonian as in \cite{KaloshinZ12}. %Consider two consecutive annulus $\wt \AAA(k_n,k')$ and $\wt \AAA(k_n,k'')$ which, by construction, overlap. We  do not perform the normal form  in the union of these annulus  but in a strip along the single resonance,
%\[
% \DDD_{C_0}(\II_{k_n})=\{|\pa_J(\HH_0(\rr^{m/2} J))\cdot k_n|\leq C_0\}\cap \left(\wt \AAA(k_n,k')\cup \wt \AAA(k_n,k'')\right).
%\]
%Note that in the original coordinates, this strip would be of width of order ${C_0}\rr^{m/2}$. 
In  $\wt \Tr_n(k_n,k',k'')$, the Hamiltonian \eqref{def:HamBNFAfterRescaling} satisfies
\[
 \|\HH_0\|_{\CCC^{3r-2\tau}}\leq C\qquad ,\qquad \|\HH_1\|_{\CCC^{2r-2\tau-m/2}}\leq C\rr^{r+m/2}.
\]
%and 
%\begin{equation}\label{def:HigherRegularityHamAfterRescaling}
%\begin{split}
% \|\HH_0\|_{\CCC^{r-6-2\tau+\kk}}&\leq C\rr^{-(1+4\tau)\kk}\qquad\,\,0<\kk\leq 6+2\tau\\
%\|\HH_1\|_{\CCC^{r-6-2\tau-m+\kk}}&\leq C\rr^{m-(1+4\tau)\kk}\qquad0<\kk\leq 6+2\tau+m
%\end{split}
%\end{equation}

\begin{theorem}\label{thm:NF:Transition}
Consider the Hamiltonian \eqref{def:HamBNFAfterRescaling}.
% and assume  $m\leq (r-6-2\tau)/4$. Then, t
There exists a  change of coordinates $\Phi: \wt \Tr_n(k_n,k',k'')\longrightarrow \RR^2\times\TT^3$ such that $\HH\circ\Phi$ is  of the form 
\begin{equation}\label{def:hamAfterNF}
 \HH\circ\Phi(\varphi,J)=\wt \HH_0(J)+\ZZZ(k_n\cdot \varphi,J)+\RRR(\varphi,J)
\end{equation}
where $\wt \HH_0(J)=\HH_0(\rr^{m}J)$
and the other terms satisfy
\[
 \|\ZZZ\|_{\CCC^{2}}\leq C\rr^{\frac{4}{3}r}
\qquad ,\qquad\|\RRR\|_{\CCC^{2}}\leq C\rr^{q(r+1)},
\]
for some constant $C$ independent of $\rr$.
% and for any $\kk$ such that $0\leq\kk<r-(r-6-2\tau-3m)=6+2\tau+3m$.

Moreover, 
\[
 \|\Phi-\mathrm{Id}\|_{\CCC^{1}}\leq C\rr^{r-m/2}.
\]
%for any $\kk$ such that $0\leq\kk<r-(r-6-2\tau-3m)=6+2\tau+3m$.
\end{theorem}

As explained in \cite{BernardKZ11},  the normal form procedure done in Theorem \ref{thm:NF:Transition} always keeps the form $H(\varphi,J,t)+E$. Thus, the obtained Hamiltonian $\HH\circ\Phi(\varphi,J)$ can be written as a non-autonomous Hamiltonian.

To obtain the  normal form in Key Theorem \ref{keythm:Transition:NormalForm} it  suffices to undo the rescaling \eqref{def:rescalingSR}. Note that the inverse change has $\CCC^j$ norm of size $\rr^{-jm}$. Then, one recovers the estimates in Key Theorem \ref{keythm:Transition:NormalForm}.

\subsubsection{Proof of Theorem \ref{thm:NF:Transition}: Normal form in the single resonance zones}
The proof of this Theorem follows   \cite{BernardKZ11, KaloshinZ12}. Nevertheless we perform several steps of normal form instead of just one. 
%Here we divide the steps into two parts. In the first parts will reduce the size of the low harmonics (to be defined properly later) and then, in the second part we will reduce the size of the high harmonics. Each part requires several steps of normal form.

%We define as \emph{low harmonics}, those which satisfy $|\wt k|\leq R_n^{4/5}$. First,  we show that, taking $\rr\ll\eta$,  in the original coordinates (before rescaling)  the low harmonic resonances are $100\rr$ away from the ball $B_{3\rr}(\om_n^*)$. Indeed, for such $\wt k$,
%\[
% \dist\left(\Ga_{\wt k},\om_n^*\right)=\frac{\left|\wt k\cdot \om_n^*\right|}{\left|\wt k\right|}\geq \frac{\eta}{\left|\wt k\right|^{3+\tau}}\geq \eta R_n^{-4(3+\tau)/5}.
%\]
%Then, recalling the relation between $\rr=\rr_n$ and $R_n$ given in \eqref{def:rho} and taking $\rr$ small enough with respect to $\eta$, we have that 
%\[
% \dist(\Ga_{\wt k},\om_n^*)\geq 100\rr.
%\]
%Then, $\Ga_{\wt k}$ is  $50\rr$  away from the annulus where we are doing the normal form.

We define the  projections associated to low, high and resonant harmonics harmonics
\[
\begin{split}
 (\pi_< f)(J,\varphi)&=\sum_{k\in\ZZ^3,\, |k|\leq R_n^\theta, \,k\not\in \langle k_n\rangle}f^{[k]}(J)e^{2\pi i k\cdot\varphi}\\
 (\pi_\geq f)(J,\varphi)&=\sum_{k\in\ZZ^3,\, |k|\geq R_n^\theta}f^{[k]}(J)e^{2\pi i k\cdot\varphi}\\
 (\pi_{k_n} f)(J,\varphi)&=\sum_{\ell\in\ZZ,\,|\ell|\leq R_n^\theta/|k_n| }f^{[\ell k_n]}(J)e^{2\pi i \ell k_n\cdot\varphi}
\end{split}
\]
(see the definition of $\theta$ in Appendix \ref{sec:Notations}). These projections satisfy that $f=\pi_<f+\pi_\geq f+\pi_{k_n} f$. To obtain estimates for these projections we use this lemma, proved in \cite{BernardKZ11}. To state it, we define $[k]=\max\{|k_i|\}$ for $k\in\ZZ^m\setminus\{0\}$ and $[(0,\ldots,0)]=1$.
\begin{lemma}\cite{BernardKZ11}\label{lemma:FourierBounds}
For $g(\varphi, I)\in \CCC^r(\TT^m\times B)$, we have
\begin{itemize}
\item If $\ell<r$,  $\|g_k(I)e^{2\pi i k\cdot\varphi}\|_{\CCC^\ell}\leq [k]^{\ell-r}\|g\|_{\CCC^r}$,
\item Let $g_k(I)$ be a series of functions such that the inequality $\|\pa_I^\al g_k\|_{\CCC^0}\leq M[k]^{-|\al|-m-1}$ holds for each multi-index $\al$ with $|\al|\leq \ell$, for some $M$. Then, we have \newline 
$\|\sum_{k\in\ZZ^m}g_k(I)e^{2\pi i k\cdot\varphi}\|_{\CCC^\ell}\leq \kk M$,
\item Let $(\Pi_K^+ g)(\varphi, I)=\sum_{|k|>K}g_k(I)e^{2\pi i k\cdot\varphi}$. Then, for $\ell\leq r-m-1$, $\|\Pi_K^+ g\|_{\CCC^\ell}\leq \kk K^{m-r+\ell+1}\|g\|_{\CCC^r}$,
\end{itemize}
 where $\kk$ is a constant which only depends on $m$.
\end{lemma}

Now we perform the normal form  to reduce the size of the non-resonant harmonics. We see that at step $j$ we have a Hamiltonian $\HH^{j}$ of the form 
\[
 \HH^j(\varphi,J)=\wt\HH_0(J)+ \ZZZ^j( \varphi,J)+\RRR^j( \varphi,J)+\RRR_\geq^j( \varphi,J)
\]
where $\wt\HH_0(J)=\HH_0(\rr^m J)$,  $\ZZZ^j$ only contains resonant  harmonics, $\RRR_\geq^j$ only contains  high harmonics and $\RRR^j$ may contain all harmonics but will be increasingly small. The normal form removes the low harmonics of $\RRR^j$. \emph{A posteriori} we will show that the high harmonics are small due to the high regularity.
% (that is $\pi_\leq \ZZZ^j_\leq=\ZZZ^j_\leq$)  and 
%$\RRR^j_>$ only contains high harmonics (that is $\pi_> \RRR^j_>=\RRR^j_>$). Note that all the $k_n$-resonant  harmonics belong to $\RRR_>^j$.  
In the first step, we take $\ZZZ_0=0$, 
$\RRR^0=\HH_1$, 
$\RRR_\geq^0=0$.

We prove inductively that at the $j$ step, the terms in the Hamiltonian satisfy
\begin{align}
%\|\HH_0^j-\HH_0^0\|_{\CCC^{r+m/2-6j}} &\leq C\rr^{r-m/2}\label{def:NFTransition:H0}\\
%\|\ZZZ^j\|_{\CCC^{s-2j}} &\leq C\rr^m\label{def:NFTransition:Z}\\
\nonumber
\|\ZZZ^j\|_{\CCC^{2r-m/2-2\tau-5j}}&\leq C\rr^{r+m/2}
%\label{def:NFTransition:Z}
\\ \nonumber
\|\RRR^j_< \|_{\CCC^{2r-m/2-2\tau-5j}}\leq C\rr^{r+m/2+j(r-m/2)}
%\label{def:NFTransition:RLow}
\qquad ,& \qquad 
\|\RRR^j_\geq\|_{\CCC^{2r-m/2-2\tau-5j}}\leq C\rr^{r+m/2}.%\label{def:NFTransition:RHigh}.
\end{align}
In the step $0$, these estimates are trivially satisfied.

To reduce  the size of $\RRR^j$,  we perform a change of coordinates defined by the time one map of the flow  associated to the Hamiltonian 
\[
 \Ga^j(J,\varphi)=\sum_{\substack{k\in \ZZ^3,\\0<|k|\leq R_n^\theta, k\notin\langle k\rangle}}\frac{\RRR^{j,k}(J)}{\pa_J\wt\HH_0(J)\cdot k}e^{2\pi ik\cdot \varphi}.
\]
To bound this Hamiltonian we first bound the small divisors $(\pa_J\wt\HH_0(J)\cdot k)\ii$. We only need to bound the small divisors for $k\not\in \langle k_n\rangle$ and $|k|\leq R_n^\theta$. We write $k$ as $k=(\wt k,k')$ where $\wt k\in \RR^2$ and analogusly $k_n=(\wt k_n,k_n')$. Then, we can decompose $\wt k$ as an orthogonal decomposition $\wt k=\al k_n+k_n^\perp$, where $\al\in\ZZ$ and $k_n^\perp$ is orthogonal to $k_n$. We know that 
\[
\al=\frac{|\wt k|}{|k_n|}\leq \frac{2 R_n^\theta}{R_n/4}\leq 8R_n^{\theta-1}\leq 8\rr^{-\frac{1}{3}_\tau(\theta-1)}.
\]
Then, we can write $k$ as 
\[
 k=\al k_n+\wh k,\quad\text{ where }\quad \wh k=(\wt k^\perp, k'-\al k_n').
\]
Thus, recalling that $\wt\HH_0(J)=\HH_0(\rr^m J)$, we have
\[
  \left|k\cdot  \pa_J\wt\HH_0(J)\right|=\left|k\cdot  \pa_J[\HH_0(\rr^m J)]\right|\geq \left|\wh k\cdot  \pa_J[\HH_0(\rr^m J)]\right|-|\al|\left|k_n\cdot  \pa_J[\HH_0(\rr^m J)]\right|.
\]
We bound each term. For the second term, we use the fact that we are looking for estimates in the rescaled single resonance zone and therefore we look at $J\in\RR^2$ such that $\dist(J,\wt\II_{k_n}^{\om_n})\leq c\rr^{\frac{1}{3}\theta+1}$, where $\wt\II_{k_n}^{\om_n}$ is the resonant segment $\II_{k_n}^{\om_n}$ in the resonant variables. Then,
\[
 |k_n\cdot \pa_J[\HH_0(\rr^m J)]|\lesssim \rr^m\dist(J,\wt \II_{k_n}^{\om_n})|k_n|\lesssim \rr^{m+\frac{1}{3}_\tau\theta+1} R_n\leq \rr^{m+\frac{1}{3}_\tau\theta+\frac{2}{3}_\tau}.
\]
The $\rr^m$ comes from the relation between the frequencies and the rescaled actions. 

Now we look for a lower of  bound the first term $\left|\wh  k\cdot \pa_J[\HH_0(\rr^{m}J)]\right|$. In frequency space $\Ga_{\wh k}$ is orthogonal to $\SSS_{k_n}^{\om_n}$. In the actions $I$, the single resonance zone only contains point at a distance bigger than $\rr^m$ from $\theta$-strong resonances and points at a distance of order one in the rescaled actions $J$. Then,
\[
 |\wh k\cdot \pa_J[\HH_0(\rr^m J)]|\gtrsim \rr^m\dist(J,\Ga_{\hat k})|\wh k|\gtrsim \rr^{m}.
\]
where we have used $|\wh k|\geq 1$.

The estimates of $|\al|$, $|k_n\cdot \pa_J[\HH_0(\rr^m J)]|$ and $|\wh k\cdot \pa_J[\HH_0(\rr^m J)]|$ imply
\[
 \left|k\cdot \pa_J[\HH_0(\rr^m J)]\right|\gtrsim \rr^m.
\]
One can easily also see that 
\[
 \left|\pa^\ell_J \left(k\cdot \pa_J[\HH_0(\rr^{m}J)]\right)\ii\right|\lesssim \rr^{-m}.
\]
where $|\ell|=1\ldots 2r-m/2-2\tau$. 
%Moreover, using the induction hypothesis, we have
%\[
% \left|\pa^\ell_J \left(k\cdot \pa_J[\HH^j_0(\rr^{m/2}J)]\right)\ii\right|\lesssim \rr^{-m}.
%\]
%for  $|\ell|=1\ldots r+m/2-1-2j$. 

We use these estimates and Lemma \ref{lemma:FourierBounds}, to bound $\Ga^j$. 
%To this end, we need to estimate the Fourier coefficients. We only need to estimate its $\CCC^{s-m}$ norm. We split $\Ga_0$ in low and high harmonics $\Ga_0=\Ga_0^\leq+\Ga_0^\geq$. The first term containing the harmonics with $|k|\leq R_n^\theta$ and the second with $|k|\geq R_n^\theta$. Then, we have that
\[
 \|\Ga^j\|_{\CCC^{2r-m/2-2\tau-5j-4}}\leq  C\rr^{-m}\|\RRR_<^j\|_{\CCC^{2r-m/2-2\tau-5j}}\leq C\rr^{r+m/2+j(r-m/2)-m}.
\]
%On the other hand, using that $\Ga_0^\geq$ has only high harmonics, we have that
%\[
% \|\Ga_0^\geq\|_{\CCC^{s-m}}\leq C\rr^{\frac{1}{3}_\tau m}\|\Ga_0^\geq\|_{\CCC^{s}}\leq C\rr^{-\frac{2}{3}_\tau m}\|\NNN_0^\geq\|_{\CCC^{s}}\leq C\rr^{-\frac{2}%{3}_\tau m}.
%\]
We define $\Phi^j_t$ the flow associated to the Hamiltonian $\Ga^j$ and the time one map $\Phi^j=\Phi_1^j$. Using the Faa di Bruno formulas, one can see that 
\[
 \|\Phi^j-\mathrm{Id}\|_{\CCC^{2r-m/2-2\tau-5(j+1)}}\leq C \|\Ga^j\|_{\CCC^{2r-m/2-2\tau-5j-4}}\leq C\rr^{r+m/2+j(r-m/2)-m}.
\]
Now, we analyze the new Hamiltonian $\HH^{j+1}=\HH^j\circ\Phi^j$. Define $F_t^j=\ZZZ^j+\RRR_\geq^j+(\RRR^j-\pi_<\RRR^j)+t\pi_<\RRR^j=\ZZZ^j+\RRR_\geq^j+(\pi_{k_n}\RRR^j+\pi_\geq\RRR^j)+t\pi_<\RRR^j$. Then, $\HH^j$ can be written as
\[
 \HH^{j+1}(J,\varphi)=\wt \HH_0+ \ZZZ^j+\RRR_\geq^j+\pi_{k_n}\RRR^j+\pi_\geq\RRR^j+\int_0^1\{F^j_t,\Ga^j\}\circ \Phi_t^j dt.
\]
We define then,
\[
 %\begin{split}
%\HH^{j+1}_0&=\HH^j_0+\left\langle \int_0^1\{F^j_t,\Ga^j\}\circ \Phi_t^j dt\right\rangle\\
\ZZZ^{j+1}=\ZZZ^{j}+\pi_{k_n}\RRR^j, \quad\RRR^{j+1}=\int_0^1\{F^j_t,\Ga^j\}\circ \Phi_t^j dt,\quad \RRR_\geq^{j+1}=\RRR_\geq^j+\pi_\geq \RRR^j.
 %\end{split}
\]
We use the estimates for $\ZZZ^j$, $\RRR^j$, $\RRR_\geq^j$ and $\Ga_j$ to bound the Poisson Bracket appearing in these terms. One can see that $\|F_t\|_{\CCC^{2r-m/2-2\tau-5j-4}}\leq C\rr^{r+m/2}$. Then,
\[
\left\|\{F^j_t,\Ga^j\}\right\|_{\CCC^{2r-m/2-2\tau-5(j+1)}}\leq C\left\|F^j_t\right\|_{\CCC^{2r-m/2-2\tau-5j-4}}\left\|\Ga^j\right\|_{\CCC^{2r-m/2-2\tau-5j-4}}\leq  C\rr^{r+m/2)+(j+1)(r-m/2)}.
\]
Then, applying the Faa di Bruno formula, one obtains
\[
\|R^{j+1}\|_{\CCC^{2r-m/2-2\tau-5(j+1)}}\leq \left\|\{F^j_t,\Ga^j\}\circ\Phi_t^j\right\|_{\CCC^{2r-m/2-2\tau-5(j+1)}}\leq  C\rr^{r+m/2)+(j+1)(r-m/2)}.
\]
To obtain the bounds for $\ZZZ^{j+1}$ and $\RRR_\geq^{j+1}$ it is enough to use  Lemma \ref{lemma:FourierBounds}.

To complete the lemma, we perform $N=\lfloor\frac{1}{5}(r-\frac{m}{2}-2\tau) \rfloor$ steps of normal form. Then,
\[
\begin{split}
\|\ZZZ^N\|_{\CCC^r}&\leq C \rr^{r+\frac{m}{2}}\\
\|\RRR^N\|_{\CCC^r}\leq C \rr^{r+\frac{m}{2}+N(r-\frac{m}{2})}\leq &C\rr^{\frac{1}{5}(r-\frac{m}{2})(r-\frac{m}{2}-2\tau)}\qquad ,
\qquad
\|\RRR_\geq^N\|_{\CCC^r}\leq C \rr^{r+\frac{m}{2}}.
\end{split}
\]
Now, we estimate the $\CCC^2$ norm. Note that the only ones which change the norm are $\ZZZ^N$ and $\RRR_\geq^N$ since they only contain high harmonics. Applying Lemma  \ref{lemma:FourierBounds},
\[
\begin{split}
\|\ZZZ^N\|_{\CCC^2}&\leq C \rr^{r+\frac{m}{2}+\frac{1}{3}_\tau (r-2)}\leq C\rr^{\frac{4}{3}r}\\
\|\RRR^N\|_{\CCC^2}&\leq C\rr^{\frac{1}{5}(r-\frac{m}{2})(r-\frac{m}{2}-2\tau)}\\
\|\RRR_\geq^N\|_{\CCC^2}&\leq C \rr^{(r+\frac{m}{2})+\frac{1}{3}_\tau\theta (r-6)}\leq C \rr^{(1+\frac{1}{3}_\tau\theta)r-2_\tau+\frac{m}{2}}.
\end{split}
\]
Since $\theta=3(q+1)$ (see Appendix \ref{sec:Notations}) and $r\geq m+5q$, we have
\[
 \|\RRR_\geq^N\|_{\CCC^2}\leq C \rr^{(q+2_\tau)r-2_\tau+\frac{m}{2}}\leq C\rr^{q(r+1)}.
\]
and, since $r\geq m+5q$, we have
\[
 \|\RRR^N\|_{\CCC^2}\leq C\rr^{\frac{1}{5}(r-\frac{m}{2})(r-\frac{m}{2}-2\tau)}\leq C\rr^{q(r+1)}.
\]

Taking $\ZZZ=\ZZZ^{N}$ and $\RRR=\RRR^{N}+\RRR^{N}_\geq$,  we obtain the Hamiltonian stated in Theorem \ref{thm:NF:Transition}.
%OLD NORMAL FORM AROUND DOUBLE RESONANCES

%We perform several steps of resonant normal form to the Hamiltonian \eqref{def:HamAfterAveraging} in the balls $B_\mu(P_{k'})$. 
%Note that we still have freedom to choose the parameter $\nu$. This is done in next Theorem, whose prove is deferred to Section \ref{sec:ProofNormalFormDR}.

%\begin{theorem}\label{thm:DR:NormalForm}
%There exists $\rr_0\ll 1$ such that for any $\rr<\rr_0$ and for any $k'$ 
%satisfying $|k'|\leq R_n^\theta$ with 
%\[
% \nu=\frac{m- r_2}{4}+\frac{1}{6}+\frac{\tau}{2}
%\]
%there exists a $\CCC^{2}$ change of coordinates 
%$\Phi_3: B_{\mu}(P_{k'})\times\TT^3\rightarrow B_{2\mu}(P_{k'})\times\TT^3$ 
%which transforms the  Hamiltonian $\wt H$ in \eqref{def:HamAfterAveraging} into 
%\begin{equation}\label{def:HamAfterDRNF}
% \JJ(J,\phi)=\om J+H_0'(J)+Z(J,k\cdot\phi,k'\cdot \phi)+R(J,\phi),
%\end{equation}
%which is $\CCC^{2}$ and satisfies
%\[
%\begin{split}
%\|Z\|_{\CCC^2}&\leq C_4\rr^{m+\frac{1}{3}( r_2-2)}\\ 
%\|R\|_{\CCC^2}&\leq C_4 \rr^{(r_2-2)\left(\frac{m- r_2}{4}-1\right)}
%\end{split}
%\]
%for some constant $C_4>0$ independent of $\rr$.
%\end{theorem}

\subsection{Normal Form in the cores of double resonances: proof of Key Theorem \ref{keythm:CoreDR:NormalForm}}\label{sec:NF:DR}

\subsubsection{Bounemoura normal form}
Consider $\mu$-neighborhoods of a double 
resonance of period $T$ such that $T\mu<1$. 
We derive a normal form describing behavior near 
a double resonance, proposed by Bounemoura \cite{Bounemoura10}.

We state first a result of \cite{Bounemoura10}, which deals with perturbations of linear Hamiltonian systems.
Denote by 
\[
\DDD_R= \TT^2\times B_{R}(I^*) \times  \TT,\qquad B_{R}(I^*)\subset \RR^3.
\]
Call a frequency $\om^\# \in \R^2$ rational if 
\[
T=\inf\{t>0:\ t(\om^\#,1) \in \Z^3\setminus 0\}<\infty
\] 
is well defined. By definition $T$ is the period of 
the corresponding periodic orbit of the unperturbed flow. 
%Below we use the following standard formula: 
%Let $H \in \CCC^z(\DDD_R)$ for some positive integer $z$ and
%$H \circ \Phi^f \in \CCC^z(\DDD_{R-r})$ and the estimate  
%\be \label{faa-di-bruno}
% |H \circ \Phi^f|_{\CCC^z(\DDD_{R-r})} \le  
% |H|_{\CCC^r(\DDD_{R})} |\Phi^f|^z_{\CCC^z(\DDD_{R-r})}. 
%\ee
%The latter follows from the Fa\`a di Bruno formula. 
%Recall that $m=r/10$ has been defined in \eqref{def:CoreDR-m}. Below $p$ stands for the number of steps of averaging.

\begin{theorem}\label{thm:BounemouraNormalForm:linear}
Consider a Hamiltonian $h(\varphi,I)=\ell(I)+f(\varphi,I)$ defined in $\DDD_R$, where $\ell(I)=\om^\#\cdot I$ is linear, $\om^\#$ is a $T$ periodic frequency  and $\| f\|_{\CCC^k}\leq \mu$. Assume $T\mu\leq C$ and $|\om|\leq C$ for some constant $C>0$ independent of $\mu$ and $T$. Then, there exists a constant $C'$ independent of $\mu$ and $T$ and  $\CCC^p$ symplectic change of coordinates $\Phi: \DDD_{R/2}\longrightarrow \DDD_R$ with $\|\Phi-\Id\|_{\CCC^p}\leq C'(T\mu)$ such that
\[
 h\circ \Phi=\ell+g+\wt f
\]
with $\{g,\ell\}=0$ and the estimates
\[
 \|g\|_{\CCC^p}\leq C'\mu\qquad\text{ and }\qquad \|f\|_{\CCC^p}\leq C'(T\mu)^{k-p}\mu.
\]
\end{theorem}
This Theorem is proven in \cite{Bounemoura10} with $p=2$. It is easy to check that it works for any $p\leq k$. 

To apply this Theorem to our setting we rescale the Hamiltonian $N^{\dr}=(H_\eps+\Delta H^\mol+\Delta^\sr)\circ \Phi_\Pos$ and we make it autonomous by adding a variable conjugate to time. We consider the scaling
\begin{equation}\label{def:DR:scaling}
\Phi^\sc: (\varphi,J)\longrightarrow (\varphi,I)=(\varphi,\rr^m J).
\end{equation}
Then, the Hamiltonian $N^{\dr}=(H_\eps+\Delta H^\mol+\Delta^\sr)\circ \Phi_\Pos$ becomes a Hamiltonian of the form 
\[
 \HH(J,\varphi)=\ell(J)+\HH_1(J,\varphi)
\]
where $\ell(J)=\om^\# \cdot J$ is linear and $\HH_1$ satisfies  $\|\HH_1\|_{\CCC^{3r-2\tau-\al}}\lesssim \rr^\al$. We take $\al=2r-m$. Thus,
\[
 \|\HH_1\|_{\CCC^{r+m-2\tau}}\lesssim \rr^{2r-m}.
\]
Note that we are abusing notation and $\varphi$ now includes $t$ and $J$ includes the variable conjugate to time.
Now, we apply Theorem  \ref{thm:BounemouraNormalForm:linear} with $k=r+m-1$ and $p=r$. Since the double resonance is $\theta$-strong, we have that
$\rho^{-(1+2\tau)/3}<|k'|<\rho^{-(1+2\tau)\th/3}$. Thus, the period of the unperturbed flow $\ell(J)$ satisfies
\[
T\le |k|\cdot|k'|<R_n^{1+\th}
\le \rho^{-(1+\th)(1+2\tau)/3}. 
\]
Note that after rescaling, the core of the double resonance corresponds to  $\DDD_C$. We apply Theorem  \ref{thm:BounemouraNormalForm:linear} not to the $\DDD_C$, but to the larger set $\DDD_{2C}$. It gives a change of coordinates $\Phi^\dr:\DDD_{C}\longrightarrow \DDD_{2C}$ such that 
\[
 \wt\HH(J,\varphi,t)=\HH\circ \Phi^\dr(J,\varphi)=\ell(J)+ \ZZZ(J,\varphi)+\RRR(J,\varphi,t)
\]
with 
\[
%\begin{split}
 \| \ZZZ\|_{\CCC^{r}}\leq C'\rr^{2r-m}\qquad, \qquad
 \|\RRR\|_{\CCC^{r}}\leq C'(\rr^{2r-m-(1+\theta)(1+2\tau)/3})^{m-1}\rr^{2r-m}.
%\end{split}
\]
and
\[
 \|\Phi^\dr-\Id\|_{\CCC^{r}}\leq C'\rr^{2r-m-(1+\theta)(1+2\tau)/3}.
\]
for some $C'>0$ independent of $\rr$. 

\begin{comment}

\begin{theorem}\label{thm:BounemouraNormalForm} 
Let $p,m$ be integers with $2\le p<r-6-2\tau$ and $\tau>0$ 
satisfies $4\tau (p-1)<1$. 
For any $0<\rho<1$ and $0< \mu \le C\rho^{m/2}$ 
consider a $\CCC^{\infty}$ 
smooth Hamiltonian of the form 
\[
 \HH(I,\phi,t)=\HH_0(I)+ \HH_1(I,\phi,t)
\]
 satisfying 
\[
\|\HH_0\|_{\CCC^{3r-2\tau-m}(\DDD_\mu(I^\ast))}
\leq C\quad \text{ and }
\quad \|\HH_1\|_{\CCC^{3r-2\tau-m}(\DDD_\mu(I^\ast)}
\leq \rho^m.
\]
Suppose $I^*$ is such that the frequency vector 
$\om^\#=\nabla \HH_0(I^*)$ is rational of some period $T$ 
such that $T\rho^{m/2}<C$ and $|\om^\#|<C$. 
Then there exists a $\CCC^{r-6-2\tau-m-p}$ symplectic transformation
\[
 \Phi^p_B: \DDD_{\mu/2}(I^\ast) 
 \longrightarrow \DDD_{\mu}(I^\ast)
\]
with $\|\Pi_I \Phi^p_B-\mathrm{Id}\|_{\CCC^0(\DDD_{\mu/2}(I^\ast))}
\leq C\ (T\rr^m) $ such that
\[
 \HH\circ\Phi^p_B=\HH_0+\,\ZZZ+\RRR
\]
with $\{\ZZZ,\om^\# I\}=0$, and for any $0\le \kappa \leq 6+2\tau+m+p$ 
we have 
\[
\begin{split}
 \| \ZZZ\|_{\CCC^{3r-2\tau-m-p}(\DDD_{\mu/2}(I^\ast) )}& \leq C \, 
 \rr^{m}\\
 \| \RRR\|_{\CCC^{3r-2\tau-m-p}(\DDD_{\mu/2}(I^\ast) )} &\leq C  
 \, (T\rr^m)^p\rr^m,
\end{split}
\]
and 
\[
\|\Phi^p_B-\Id\|_{\CCC^{3r-2\tau-m-p}(\DDD_{\mu/2}(I^\ast))}
\le  C\,(T\rr^m).
\]
\end{theorem}
This theorem is proved in Section \ref{sec:ProofNFBounemoura}. 
\end{comment}

Notice that $\{\om^\# I,\ZZZ\}=0$ implies that, expanding $\ZZZ$ in 
a Fourier series,  it contains only {\it resonant} 
harmonics, namely, integers $k\in (\Z^2\setminus \{0\})\times \Z$ given 
by $k \cdot \om^\#=0$. By Key Theorem \ref{keythm:ResonancesWithNoTriple}, any vector  $k$  such that  $\om^\#\cdot  k=0$ satisfies $|k|\geq R_n^{1-c\tau} $. Then, since  $\|\ZZZ\|_{\CCC^{r}}\leq C\rr^m$, 
$\rr=\rho_n=R_n^{-(3-5\tau)}$ we have that its Fourier coefficients 
$h_k$ decay as $|h_k|\le c\rho^m |k|^{-r}$. Therefore, 
\[
 \|\ZZZ\|_{\CCC^2}\le c\,\rho^{2r-m}\,|k|^{-(r-2)} \le c\,\rho^{2r-m}\,R_n^{-(r-2)}=
 c\,\rho^{2r-m+\frac{r-2}{3-5\tau}}.
\]
\vskip 0.1in
{\it Warning:} The resonant term $\ZZZ$, in general, is not determined 
by resonant Fourier coefficients of $\HH_1$. Indeed, after one 
step of averaging we have remainders of order $\rho^{2m}$, 
while resonant Fourier coefficients of $\HH_1$ are bounded by 
$  c\,\rho^m\,R_n^{-1_\tau(3r-2\tau-m)}$. For our range of parameters
the former dominates the latter.

Now, we want to remove completely the remainder $\wt \RRR$ from the core of the double resonance. To this end we want to add a $\CCC^r$ small perturbation to the Hamiltonian. Consider a bump function $\chi :\RR^2\longrightarrow \RR$ such that $\chi(J)=1$ if $|J|\leq C$ and $\chi(J)=0$ if $|J|\geq 2C$. Then, we define 
\[
 \Delta H^\dr=\left(\chi \wt\RRR\right)\circ \left(\Phi^\sc\circ \Phi^\dr\right)\ii
\]
where $\Phi^\sc$ is the scaling \ref{def:DR:scaling}. Note that $\|(\Phi^\sc)\ii\|_{\CCC^r}\leq \rr^{-mr}$. Therefore, we can estimate the norms of  $\Delta H^\dr$ as follows. Since we take $m=\theta+1$ and $r\geq m+5q$ (see Appendix \ref{sec:Notations}), for the $\CCC^r$ norm we have
\[
\|\Delta H^\dr\|_{\CCC^r}\leq (\rr^{2r-m-(1+\theta)(1+2\tau)/3})^{m-1}\rr^{2r-m-mr}.\leq C'\rr_n.
\]
For the $\CCC^2$ norm, recalling also that $\theta=3q+1$, we have
\[
\|\Delta H^\dr\|_{\CCC^r}\leq C'(\rr^{m-(1+\theta)(1+2\tau)/3})^{2r-2\tau-m}\rr^{-m}\leq C'\rr^{2q(r+1)}.
\]
Once we add $\Delta H^\dr$, the Hamiltonian $\wt \HH$ becomes 
\begin{equation}\label{def:HAMatDR:NoR}
\wt\HH(J,\varphi,t)=\ell(J)+ \ZZZ(J,\varphi).
\end{equation}
To complete the proof of Key Theorem \ref{keythm:CoreDR:NormalForm} it only remains to perform the rescaling $(\Phi^\sc)\ii$ and to make a change of variables to separate the slow and the fast angles. This is explained in the next section.

\subsubsection{Slow-fast variables}\label{sec:DR:slowfast}
Define
\begin{equation}\label{def:ChangeToFastSlow}
 \left(\begin{matrix}\phi_1^s\\\phi_2^s\\t\end{matrix}\right)=A\left(\begin{matrix}\phi_1\\\phi_2\\t\end{matrix}\right)\,\,\text{ with }A=\left(\begin{matrix}k\\ k'\\ e_3\end{matrix}\right) 
\end{equation}
where $e_3=(0,0,1)$. To have a symplectic change of coordinates, we perform the change of coordinates  
\begin{equation}\label{def:ChangeToFastSlow:2}
\left(\begin{matrix}J_1^s\\ J_2^s\\ E\end{matrix}\right)=A^{-T}\left(\begin{matrix}J_1\\J_2\\E\end{matrix}\right)
\end{equation}
to the conjugate actions, where $E$ is the variable conjugate to time. To simplify notation, we call $J$ to $J=(J_1^s,J_2^s)$. It can be seen that such matrix and its inverse, satisfy
\begin{equation}\label{def:ChangeToFastSlow:Bounds}
\|A\|,\|A^{-1}\|\lesssim \rr^{-\frac{2}{3}_\tau(1+\theta)}
\end{equation}
and the same bounds are satisfied by their transposed matrices.

If we apply the symplectic change of coordinates \eqref{def:ChangeToFastSlow}, \eqref{def:ChangeToFastSlow:2} to the Hamiltonian \eqref{def:HAMatDR:NoR} it can be seen that now $\ell(I)$ only depends on $E$ and therefore it disappears if one considers the non-autonomous setting. Then, we scale back the actions considering the inverse of the change \eqref{def:DR:scaling} and we obtain a Hamiltonian of the form
\begin{equation}\label{def:HAMatDR:2}
\wh\HH(I,\phi)=\wh\HH_0(I)+\wh\ZZZ(I,\phi^s)
\end{equation}
where $\wh\HH_0(I)=\langle \ZZZ\rangle(\rr^{-m}I)$, where the average is taken with respect to $\phi^s$, expressed in slow-fast variables.  The Hamiltonian $\wh\HH_0$ is strictly convex, that is, 
\[
 D\ii \rr^{\frac{2}{3}_\tau(1+\theta)}\,\Id\leq \pa_J^2 \wh\HH_0\leq D \rr^{-\frac{2}{3}_\tau(1+\theta)}\Id
\]
for some $D>1$ independent of $\rr$  since it is close to the original $H_0'\circ\Phi^\sc$.  The appereance of the factors $\rr$ and $\rr\ii$ comes from the change to slow-fast variables. The function $\wh\ZZZ$ is just $\wh\ZZZ=\ZZZ-\wh\HH_0$ expressed in slow-fast variables and rescaled actions. Thus, it satisfies
\[
 \left\|\wh\ZZZ\right\|_{\CCC^2}\leq \rr^{2r+\frac{1}{3}_\tau (r-2)-\frac{4}{3}_\tau(1+\theta)-3m}
\]
Using the choice of parameters in Appendix \ref{sec:Notations} we have that 
\[
 2r+\frac{1}{3}_\tau (r-2)-\frac{4}{3}_\tau(1+\theta)-3m\geq \frac{2}{3}_\tau(1+\theta)
\]
Therefore, the Hamiltonian \eqref{def:HAMatDR:2} is strictly convex with respect to the actions.

This completes the proof of Key Theorem \ref{keythm:CoreDR:NormalForm}.

\section{Deformation of average potentials: proof of Key Theorem \ref{keythm:Transition:Deformation}}
\label{sec:DeformationSinglePotential}
We devote this section to deform the Hamiltonian $N'$,
defined in (\ref{def:HamAfterMolification}), so that the first order 
of the normal form 
along single resonance given by Theorem \ref{thm:NF:Transition} 
is non degenerate. By Key Theorem 
\ref{keythm:Transition:NormalForm}
the leading order is given by 
\[
 \wh\HH (\varphi,J)=\wh\HH_0 (J)+\wh \ZZZ(\varphi^s,J),\qquad
\varphi=(\varphi^s,\varphi^f)\in \T^2,\ J\in \R^2.
\]
We prove thus Key Theorem \ref{keythm:Transition:Deformation}, which ensures the non-degeneracy of this first order. Note that we want to deform $\wh\ZZZ$ so that the deformation, in the original variables, has small $\CCC^r$ norm. We denote this deformation
$\Delta\ZZZ$. It certainly depends on the single resonance $\II_{k_n}^{\om_n}$, where we have this normal form. 
Moreover, we choose $\Delta\ZZZ$ supported in a small 
neighborhood of it. We do not write explicitly this dependence 
to avoid cluttering the notation.

The goal of this section is to prove the following. 
This Theorem is proved through a parameter exclusion procedure. The parameter exclusion procedure proposed 
here is similar to \cite{HK07} and is closely related to notion 
of prevalence \cite{HK10,Kal97}. 

\begin{theorem} Let $\om^*\in \DDD_{\eta,\tau}^n$ be a Diophantine 
number for some $n\in \Z_+,\ \rho=\rho_n,\, 0< \ 50\tau<1$, 
let $\Vor_n(\om_n)$ be its Voronoi 
cell. Let $\{\II_{k_n}^{\om_n}\}_{k_n\in \FF_n(\om^*)}$
be Dirichlet resonant segments inside $\Vor_n(\om^*)$ in action space,
i.e. $\II_{k_n}^{\om_n}\subset \Vor_n(\om^*)$. Let $U_n(\om^*)$
be the $\rho^2$-neighborhood of the union of all 
these segments. Let 
\[
 \wh\HH (\varphi,J)=\wh\HH_0 (J)+\wh \ZZZ(\varphi^s,J)
\]
be the first order given 
%in Lemma  \ref{thm:TransitionZone:NF:slow-fast}, 
%which has been obtained 
by the normal form of Theorem \ref{thm:NF:Transition}.

Then, there exists 
a $\rho^{1+4\tau}$-small in the $\CCC^r$ topology 
pertubation $\Delta Z(J,\phi)$ such that for any 
$k \in \FF_n(\om^*), \ \phi^s=\phi^s_{k}$ the single 
averaged potential 
\[
 \ZZZ^{\sr}_{k}(\phi^s,J^f)= 
 \ZZZ_{k}(\phi^s,J^f)+\Delta \ZZZ_{k}(\phi^s,J^f)
\]
satisfies the following conditions:
\begin{itemize}
\item For each $J^f\in [a^{k}_-,a^{k}_+]$ each global 
maximum of $ \ZZZ^{\sr}_{k}(\phi^s_{k},J^f)$
is $\rho^{d(r+1)}$ nondegenerate, i.e. 
\[
\partial^2_{\phi^s}\ZZZ^{\sr}_{k}(\phi^s,J^f)
\ge \rho^{d(r+1)}.
\]
\item There are finitely many $J^f\in [a^{k}_-,a^{k}_+]$ 
such that $\ZZZ^{\sr}_{k}(\phi^s,J^f)$ has 
exactly two global maxima, denoted $\phi^1_{\min}(J)$ 
and $\phi^2_{\min}(J)$ and we have 
\[
|\partial_{J^f}\ZZZ^{\sr}_{k}(\phi^1_{min}(J),J^f) -
\partial_{J^f}\ZZZ^{\sr}_{k}(\phi^2_{min}(J),J^f)|
\ge \rho^{7d(r+1)/4},
\]
in other words, values at the maxima change in $J^f$ 
with speeds different by $\ge \lb_*^{7/4}$.
\item There are no $J^f\in [a^{k}_-,a^{k}_+]$ such that 
$\ZZZ^{\sr}_{k}(\phi^s_{k},J^f)$ has more than 
two global maxima.
\end{itemize}
%Moreover, in P\"oschel coordinates, $\Delta \ZZZ_{k_n}$ satisfies 
%\[
% \|\Delta \ZZZ_{k_n}\circ \Phi\|_{\CCC^r}\leq C\rr^
%]
\end{theorem} 

\begin{proof} Here we prove this Theorem using 
Theorem \ref{thm:measure-estimate} (proven 
and stated below). 

Denote by $\chi^\rho_j$ a $\CCC^\infty$ bump function given by 
\[
\chi^\rho_j (J^f)=
\begin{cases}
  1\qquad \textup{if \ dist }(J(J^f),\SSS(\Ga_{k_j}))\le \rho^2, \\
  0\qquad \textup{if \ dist }(J(J^f),\SSS(\Ga_{k_j}))\ge 2\rho^2.
  \end{cases}
\]
In the $2\rho^2$ tube neighborhood
we would like to perturb by trigonometric polynomials. 
\[
 \ZZZ^\sg_j(\phi_j,J^f)= \ZZZ_j(\phi_j,J^f)  
 + \chi^\rho_j(J^f)\ (\sg_1 \cos 2\pi \phi_j 
 + \sg_2 \sin 2\pi \phi_j+\sg_3 \cos 4\pi \phi_j+\]
 \[+
 + \sg_4 \cos 4\pi \phi_j + J^f \sg_5 \cos 2\pi \phi_j+
 J^f \sg_6 \sin 2\pi \phi_j ).
\]
Since the perturbation is $\CCC^r$-small, it suffices to have 
\[
|\sg_1|,|\sg_2|, |\sg_3|,|\sg_4|,|\sg_5|,|\sg_6| 
\le \rho^{2(1+4\tau)(r+1)}.
\]
Denote by $D_\nu$ the $6$-dimensional cube of these parameters. 
Recalling  the relation between $R:=R_n$ and $\rho$ and 
$|k|<R$ we have 
\[
\rho^{2(1+4\tau)(r+1)} \le 
R^{-2(3-4\tau)(1+4\tau)(r+1)}\le |k|^{-2(1+4\tau)(3-4\tau)(r+1)}. 
\]
Denote by $\mu_\rho$ the Lebesgue probability measure 
supported on the product of two disks $D_\nu$ of radii 
$\nu=\rho^{2(1+4\tau)(r+1)}$.

Apply Theorem \ref{thm:measure-estimate}. Then,
\[
\begin{split}
\mu_\rho \big\{&(\sg_1,\sg_2,\sg_3,\sg_4,\sg_5,\sg_6) \in D_\nu:
\exists J^f\in [a^{k^*}_-,a^{k^*}_+] \text{ such that either }\  \ZZZ^\sg_j(\phi_j,J^f)\\
& \ \text{has a }
\lb_*\text{--degenerate minimum} 
\text{ or } \ \ZZZ^\sg_j(\phi_j,J^f) 
\text{ has two global minima such }\\
&\text{ that speed of change of values with respect to }
J^f \text{ differ by at least } 
\lb_*^{7/4}\\
&\text{ or in the presence of two global minima of }
\ZZZ^\sg_j(\phi_j,J^f) \text{ values at }\\
&\text{ a (possible) 3rd (local) minima 
and at a global minimum differ by at least } 
\lb_*^{7/4} \big\}\\
&\le 
\dfrac{(3C_3+7) (1+|a_+-a_-|)\,C_3^2\, 
\sqrt \lb_* }{8\,\pi^5 \ \nu^3}.
\end{split}
\]
where $C_3$ is bounded by the $\CCC^4$-norm of $\HH_0$. 

Since $|J_+-J_-|$ is bounded by $\rho$, 
for $\nu=\rho^{2(1+4\tau)(r+1)}$ we have that for 
any perturbation $\ZZZ^\sg_j(\phi_j,J^f)$ with 
$\sg \in D_{\rho^{2(1+4\tau)(r+1)}}$
localized in the $\rho^2$-neighborhood of a Dirichlet 
resonant segment $\II_{k_n}^{\om_n}$ is $\CCC^r$-small. 
In order to assure that measure of the exceptional set 
being small it suffices to pick 
$\lb_*= \rho^{14(r+1)}$. 
Then the upper bound on measure is 
$\frac{(3C_3+7)\,C_3^2\, \rho^{(1-24\tau)(r+1)} }{8\,\pi^5}.$
\end{proof}

\subsection{Perturbation of families of functions on the circle}

Let $f_t:\T \to \T, t\in [a_-,a_+]$ be a $\CCC^r$ smooth 
one-parameter family of periodic functions $\th\in\T=\R /\Z$. Consider the following $6$-parameter extended family:
\[
F_t(\th,\sg) = f_t (\th) + \sg_1 \cos 2\pi\th + \sg_2 \sin 2\pi\th+\]
\[
+\sg_3 \cos 4\pi\th + \sg_4 \sin 4\pi\th + 
t \sg_5 \cos 2\pi\th + t \sg_6 \sin 2\pi\th. 
\]
Fix $\nu>0$. Denote 
$$
D_\nu:=\{(\sg_1,\sg_2,\sg_3,\sg_4):\  \sg_1^2+\sg_2^2\le \nu^2,\ 
 \sg_3^2+\sg_4^2\le \nu^2,\  \sg_5^2+\sg_6^2\le \nu^2\}
$$
the direct products of three disks and by $\mu_\nu$ 
product of the Lebesgue probability measures on each. 

Notice that the set of parameteres $D_\nu$ is invariant with 
respect to rotations. Indeed, using the  
sine and cosine sum formula  for properly chosen $\{\sg_i(\th^*)\}_i$ 
we have that 
\be \label{eqn:shift-fourier}
\beal 
\sg_1 \cos 2\pi(\th+\th^*) + \sg_2 \sin 2\pi(\th+\th^*)+ 
\sg_3 \cos 4\pi(\th+\th^*) + \sg_4 \sin 4\pi(\th+\th^*)+ \\
 t \sg_5 \cos 2\pi(\th+\th^*) + t \sg_6 \sin 2\pi(\th+\th^*) 
\qquad \qquad \qquad 
\\ 
= \sg_1(\th^*) \cos 2 \pi \th + \sg_2(\th^*) \sin 2 \pi \th+ 
\sg_3(\th^*) \cos 4 \pi \th + \sg_4(\th^*) \sin 4 \pi \th+ \\  
t \sg_5(\th^*) \cos 2 \pi \th + t \sg_6(\th^*) \sin 2 \pi \th, 
\qquad \qquad \qquad \qquad 
\enal
\ee
where 
$\sg_1^2+\sg_2^2=\sg_1^2(\th^*)+\sg_2^2(\th^*),\ 
 \sg_3^2+\sg_4^2=\sg_3^2(\th^*)+\sg_4^2(\th^*),\  
\sg_5^2+\sg_6^2=\sg_5^2(\th^*)+\sg_6^2(\th^*).$

Fix $\sg \in D^1_\nu$. For $t\in [a_-,a_+]$ we study 
global minima of $F_t(\cdot,\sg)$. Let 
\[
\th_{\textup{min}}(t,\sg):=\{\th^*:\ F_t(\th^*,\sg)=\min_\th F_t(\th,\sg)\}
\]
be a global minimum. Let 
$ \lb(t,\sg):=\min_{\th_{\textup{min}}(t,\sg)}\  
 \partial^2_\th F_t(\th,\sg)|_{\th=\th_{\textup{min}}(t,\sg)}.$

Call a value $t^*\in [a_-,a_+]$ {\it bifurcation} if there are 
at least two global minima $\th^1_{\textup{min}}(t,\sg)$ and 
$\th^2_{\textup{min}}(t,\sg)$. 
Denote by $\B_\sg \in [a_-,a_+]$ the set of bifurcation points  
of $F_t(\cdot,\sg)$. Introduce the difference between speed of 
change of values at the global minima with respect to $t$: 
\[
 d_1(t^*,\sg)=\min\{
 |\partial_t F_t(\th,\sg)|_{\th=\th^1_{\textup{min}}(t)} -
 \partial_t F_t(\th,\sg)|_{\th=\th^2_{\textup{min}}(t)}|\},
\]
where minimum is taken over all possible pairs of global minima. 
For a bifurcation value introduce we have two global minima  
$\th^1_{\textup{min}}(t)$ and $\th^2_{\textup{min}}(t)$. 
\[
 d_2(t^*,\sg)=\min_{\th_{min}}\{
 F(\th_{min},\sg)-F(\th^1_{\textup{min}}(t),\sg)
 \},
\]
where the minimum is taken over all local minima different 
from $\th^1_{\textup{min}}(t)$ and $\th^2_{\textup{min}}(t)$.

Denote 
$$
C_3 = \max_{t\in [a_-,a_+],\ \th\in \T,\ \sg\in D_\nu}\ 
\{|\partial_* F_t(\th,\sg)|,
|\partial^2_* F_t(\th,\sg)|,|\partial^3_* F_t(\th,\sg)|\},
$$
where partial derivates are taken with respect to 
$\sg$ an $t$ only. 

The goal of this section is to prove the following 
\bthm \label{thm:measure-estimate}
Let $f_t(\th)$ be a $\CCC^4$ smooth one-parameter family 
of periodic functions $t\in [a_-,a_+]$ with 
$\|f_t\|_{\CCC^3}\le C_3$. Then 
\[
 \mu_\nu\{(\sg_1,\sg_2,\sg_3,\sg_4,\sg_5,\sg_6)\in D_\nu:\ \]
 \[
 \min_{t\in [a_-,a_+]} \lb(t,\sg)\le \lb_*,\ \min_{t^*\in \B_\sg} 
 (d_1(t^*,\sg)+d_2(t^*,\sg))\le \lb_*^{7/4}\ 
 \} \le \dfrac{
(3C_3+7)\,C_3^2\, \,\sqrt \lb_* }{8\,\pi^5 \ \nu^3}.
\]
\ethm

\subsection{Condition on local minimum}

In order to determine the set of almost degenerate minima 
define the following set:
\[
\begin{split}
 \Cr(\lb^*)=&\{(\sg_1,\sg_2,\sg_3,\sg_4,\sg_5,\sg_6)\in D_r:\ 
 \exists t\in [a_-,a_+],\\
 &\partial_\th F_t(\th^*,\sg)=0,\ |\partial^2_{\th} F_t(\th,\sg)|\le \lb_*,\ 
 \th^*\text{ is a local minimum}\}.
\end{split}
\]
We embed this set into a bigger set using the following trick
and the discretization method from \cite{HK07}. If $\th^*$ is a local 
minimum of a $\CCC^4$ function $g$ on $\T$ and $g''(\th^*)$ is small, 
then $g^{(3)}(\th^*)$ should also be small. Otherwise, it is not 
a local minimum. Here is the formal claim about relations between 
first, second, and third derivatives.

Let $\th^* \in \T$ be a local minimum of 
$g(\th)$ for some $\th^*\in \T$. This implies that $g'(\th^*) = 0$.
Consider the second and third derivative at the critical 
point $g''(\th^*),\  g^{(3)}(\th^*)$.  
Since $\th^*$ is a local minimum, $g''(\th^*)\ge 0$ and,
in case, $g''(\th^*)=0$ we have $g''(\th^*)(\sg)=0$. 

\blm \label{lm:third-derivative}
Let $g\in \CCC^4,\ \|g\|_{\CCC^4}\le C$ and $\th^*$ is 
a local minimum and $\ 0\le g''(\th^*)\le \lb$, then 
$|g^{(3)}|\le 3(C+2)\sqrt{\lb}$. 
\elm 

\begin{proof} Expand $g$ near its local minimum $\th^*$
using Taylor with the remainder in the Lagrange form.
We have 
\[
 g(\th^*+\dt \th)=g(\th^*)+g'(\th^*) \dt\th
 +\frac 12 g''(\th^*) \dt\th^2
 +\frac 16 g^{(3)}(\th^*+\xi) \dt\th^3,
\]
where $|\xi|<|\dt \th|$. By the Intermediate Value Theorem 
we can rewrite 
\[
g(\th^*+\dt \th)=\frac 12  \dt\th^2 \left(
g''(\th^*)+ \frac 13 g^{(3)}(\th^*+\xi)\right)=
\]
\[
=\frac 12  \dt\th^2 \left(
g''(\th^*)+ \frac 13 g^{(3)}(\th^*)\dt \th +
g^{(4)}(\th^*+\eta)\,\xi\,\dt \th \right) \ge 0. 
\]
Plug in $\dt \th=-\text{sign }(g^{(3)}(\th^*)) \sqrt{\lb}$.
We have 
\[
g''(\th^*)+ \frac 13 g^{(3)}(\th^*)\dt \th 
+ g^{(4)}(\th^*+\eta)\,\xi\,\dt \th\le 
\lb - (C+2)\lb+C\lb\le -\lb<0.
\]
This is a contradiction. \end{proof}

Let $\lb^\#=\lb^*/C_3$ be small positive. Denote by 
${\mathbb Z}^2_{\lb^\#}$ the $\lb^\#$-grid in $[a_-,a_+]\times \T$. 
To estimate measure 
of $\Cr(\lb^*)$  we use the discretization trick (see e.g. \cite{HK07}).
%\vskip 0.1in 
\[
 \Cr^d(\lb^*)=\{(\sg_1,\sg_2,\sg_3,\sg_4,\sg_5,\sg_6)\in D_\nu:\ \ 
 \exists (t^*,\th^*)\in  \mathbb Z^2_{\lb^\#},\] 
 \[ |\partial_{\th} F_t(\th^*,\sg)|\le 2\lb_*,\ 
 |\partial^2_{\th} F_t(\th^*,\sg)|\le 2\lb_*,\ \ 
 |\partial^3_{\th} F_t(\th^*,\sg)|\le (3C_3+7)\lb_*\}. 
\]
One of the key element of the proof is the following estimate 
\blm With the above notations we have 
\[\mu_\nu \{ \Cr^d (\lb^*)\cup \B(\lb_*) \} \le \dfrac{
(3C_3+7)\,C_3^2\,(|a_+-a_-|+1)\, \sqrt \lb_*}{16\,\pi^5 \, \nu^3}.\]
\elm 

This lemma implies Theorem \ref{thm:measure-estimate}
through a simple approximation argument. Suppose 
$\sg \in \Cr(\lb^*)$. Then there exists 
$t'\in [a_-,a_+]$ and $\th'\in \T$ such that 
\[
\partial_{\th} F_{t'}(\th',\sg)=0,\ 
|\partial^2_{\th} F_{t'}(\th',\sg)|\le \lb_*
\]
and $\th'$ is a local minimum. Then by 
Lemma \ref{lm:third-derivative} we have 
\[
|\partial^2_{\th} F_{t'}(\th',\sg)|\le 3(C+2)\sqrt{\lb_*}. 
\]
Now we can approximate $(t',\th',\sg)$
by a $\mathbb \Z^2_{\lb_*/C_3}$-grid point 
with precision $\lb_*/C$. Due to derivatives 
of order up to $4$ being bounded by $C_3$ we can 
transfer the above bounds to a nearby grid point.

\begin{proof}The proof proceeds as follows.
We estimate a probability that given 
$(t^*,\th^*)\in {\mathbb Z}^2_{\lb^\#}$ 
%we have 
\[
 \mu_\nu\{(\sg_1,\sg_2,\sg_3,\sg_4,\sg_5,\sg_6)\in D_\nu:
 |\partial_{\th} F_{t^*}(\th^*,\sg)|\le 2\lb_*,\ 
 \]
 \[
 |\partial^2_{\th} F_{t^*}(\th^*,\sg)|\le 2\lb_*,\ \ 
 |\partial^3_{\th} F_{t^*}(\th^*,\sg)|\le (3C_3+7)\lb_*\}.
\]
Rewrite the family using (\ref{eqn:shift-fourier}) as follows
\[
 F_{t^*}(\th^*+\dt \th,\sg^*+\dt \sg)=
 f_{t^*}(\th^*+\dt \th)+  \sg_1(\th^*) \cos 2\pi\dt\th 
+  \sg_2(\th^*) \sin 2\pi \dt\th
\]\[+ \sg_3(\th^*) \cos 4\pi \dt\th 
 +\sg_4(\th^*) \sin 4\pi \dt\th+t\sg_5(\th^*) \cos 2\pi\dt\th + 
 t\sg_6(\th^*) \sin 2\pi \dt\th. 
\]
Compute the first derivatives 
\[
 \partial_\th F_{t^*}(\th^*+\dt \th,\sg^*+\dt \sg)=
 \partial_\th f_{t^*}(\th^*+\dt \th)+
 \]
 \[
 - 2\pi \sg_1(\th^*) \sin 2\pi \dt\th + 
 2\pi \sg_2(\th^*) \cos 2\pi \dt\th-\]\[
 4\pi \sg_3(\th^*) \sin  4 \dt \th + 
 4\pi \sg_4(\th^*) \cos 4 \dt \th - 
 2\pi t\sg_5(\th^*) \sin 2\pi \dt\th + 
 2\pi t \sg_6(\th^*) \cos 2\pi \dt\th.
\]
For $\dt\th=0$. Fix any values of 
$(\sg_1(\th^*),\sg_3(\th^*), \sg_4(\th^*))$
we see that probability 
\[
\Leb\{\sg_1(\th^*): 
| \partial_\th F_{t^*}(\th^*+\dt \th,\sg^*+\dt \sg)|\le 2\lb_*\}
\le \frac{4\lb^*}{2\pi}.
\]
Compute the second derivatives 
\[
 \partial^2_\th F_{t^*}(\th^*+\dt \th,\sg^*+\dt \sg)=
 \partial^2_\th f_{t^*}(\th^*+\dt \th)
 - 4\pi^2 \sg_1(\th^*) \cos 2\pi \dt\th  
 -4\pi^2 \sg_2(\th^*) \sin 2\pi \dt\th- \]\[
 16\pi^2 \sg_3(\th^*) \cos 4 \dt \th + 
 16\pi^2 \sg_4(\th^*) \sin 4 \dt \th-
 4\pi^2 t\sg_5(\th^*) \cos 2\pi \dt\th  
 -4t\pi^2 \sg_6(\th^*) \sin 2\pi \dt\th.
\]
For $\dt\th=0$. Fix any values of 
$(\sg_2(\th^*),\sg_3(\th^*), \sg_4(\th^*))$
we see that probability 
\[
\Leb\{\sg_2(\th^*): 
| \partial^2_\th F_{t^*}(\th^*+\dt \th,\sg^*+\dt \sg)|\le 2\lb_*\}
\le \frac{4\lb^*}{4\pi^2}.
\]
Compute the third derivatives 
\[
 \partial^3_\th F_{t^*}(\th^*+\dt \th,\sg^*+\dt \sg)=
 \partial^3_\th f_{t^*}(\th^*+\dt \th)
 + 8\pi^3 \sg_1(\th^*) \sin 2\pi \dt\th 
 - 8\pi^3 \sg_2(\th^*) \cos 2\pi \dt\th\]\[
- 64\pi^3 \sg_3(\th^*) \sin 4 \dt \th + 
 64\pi^3 \sg_4(\th^*) \cos 4 \dt \th+ 
+ 8\pi^3 t\sg_5(\th^*) \sin 2\pi \dt\th 
 - 8\pi^3 t\sg_6(\th^*) \cos 2\pi \dt\th.
\]
For $\dt\th=0$. Fix any values of 
$(\sg_1(\th^*),\sg_2(\th^*), \sg_2(\th^*))$
we see that probability 
\[
\Leb\{\sg_4(\th^*): 
| \partial^3_\th F_{t^*}(\th^*+\dt \th,\sg^*+\dt \sg)|
\le (3C_3+7){\sqrt \lb_*}\}
\le \frac{2(3C_3+7)\sqrt {\lb_*}}{64\pi^3}.
\]
The number of grid points is $|a_+-a_-|\lb^{-2}_\#=
C^2_3|a_+-a_-|\lb^{-2}_*$. Using the fact that 
the measure $\mu_\nu$ is the product of two Lebesgue 
probability measures of the disks $\sg_1^2+\sg^2_2\le \nu^2$  
and $\sg_3^2+\sg^4_2\le \nu^2$ we combine the above
estimates and get the required bound. \end{proof}

\subsection{Bifurcation of global minima}

Suppose $t^*\in \B$ be a bifurcation point, i.e. there 
are at least two global minima $\th^1_{\min}(t,\sg)$ and 
$\th^2_{\min}(t,\sg)$. In order to determine the set of 
almost having two global minima having nearly the same 
speed of change of values in $t$ define the following set:
\[
 \B^1(\lb^*)=\{(\sg_1,\sg_2,\sg_3,\sg_4)\in D_\nu:\ 
 \exists t\in [a_-,a_+],\ 
 \] 
\[
|F(\th^1_{\min}(t,\sg),\sg)- F(\th^2_{\min}(t,\sg),\sg)|\le \lb_*^{7/4},
\]\[
|\partial_t F(\th^1_{\min}(t,\sg),\sg)- 
\partial_t F(\th^2_{\min}(t,\sg),\sg)|\le \lb_*^{7/4}\}
\]
As before to estimate measure of this set we discretize 
the set of parameters and of $\th$'s. 
Consider $\lb^\#=\lb_*/C^2_3$ grid.
We can find $\lb^\#$-closest points of the grid to 
each global minima and denote them by 
$\bar \th^1_{\min}(t,\sg)$ and $\bar \th^2_{\min}(t,\sg)$. 
For parameters $\sg \not\in \Cr(\lb_*)$ we have that 
\[
 |\partial_\th F_t(\th^i_{\min}(t,\sg),\sg)|\le \lb_*\qquad 
 \text{ for }i=1,2. 
\]
Now the discretized verions of the set $\B(\lb_*)$ is 
\[
 \B^{1,d}(\lb^*)=\{(\sg_1,\sg_2,\sg_3,\sg_4)\in D_\nu:\ 
 \exists (\th_1,\th_2,t)\in \Z^3_{\lb^\#},
|\partial_\th F(\th_1,\sg)|,|\partial_\th F(\th_2,\sg)|\le 
2\lb_*,\]
\[
|\partial^2_\th F(\th_1,\sg)|,|\partial^2_\th F(\th_2,\sg)|\ge \lb_*,
\ 
|F(\th^1_{\min}(t,\sg),\sg)- F(\th^2_{\min}(t,\sg),\sg)|\le 2\lb_*^{7/4},
\]\[ 
|\partial_t F(\th^1_{\min}(t,\sg),\sg)- 
\partial_t F(\th^2_{\min}(t,\sg),\sg)|\le 2\lb_*^{7/4}\}.
\]
Show that $\B^1(\lb^*)\subset \B^{1,d}(\lb^*)$.
We have 
\[
 |F_t(\bar \th^i_{\min}(t,\sg),\sg)-
 F_t(\th^i_{\min}(t,\sg),\sg)|\le \lb^2_*/C_3.
\]
Therefore, if 
\[
 |F_t(\th^1_{\min}(t,\sg),\sg)-
 F_t(\th^2_{\min}(t,\sg),\sg)|\le \lb^{7/4}_*
\]
implies 
\[
 |F_t(\bar \th^1_{\min}(t,\sg),\sg)-
 F_t(\bar \th^2_{\min}(t,\sg),\sg)|
 \le \lb^{7/4}_*+2\lb_*^2/C_3 \le 2\lb^{7/4}_*
\]
Also if 
\[
 |\partial_t F_t(\th^1_{\min}(t,\sg),\sg)
- \partial_t F_t(\th^2_{\min}(t,\sg),\sg)|\le \lb^{7/4}_*
\]
implies 
\[
 |\partial_t F_t(\bar \th^1_{\min}(t,\sg),\sg)
 -\partial_t F_t(\bar \th^2_{\min}(t,\sg),\sg)|\le 
 \lb^{7/4}_*+2\lb_*^2/C_3\le 2 \lb^{7/4}_*. 
\]
Choose a parameter $\sg$ such that 
\[
|\partial^2_\sg F_t(\bar \th^1_{\min}(t,\sg),\sg)|\ge \lb_*
\quad \text { for all }t\in [a_-,a_+]. 
\]
Then 
\[
 |\th^1_{\min}(t,\sg)-\th^2_{\min}(t,\sg)|\ge \frac{\lb_*}{C_3} 
\ \ \text{ and }\]
\[
 |\bar \th^1_{\min}(t,\sg)-\bar \th^2_{\min}(t,\sg)|\ge 
 \frac{\lb_*}{C_3}-2\lb^\#\le  \frac{\lb_*}{2C_3}. 
\]
Therefore, 
\[
\Leb\{\sg_4(\bar \th^1_{\min}(t,\sg)): 
 |F_t(\bar \th^1_{\min}(t,\sg),\sg) - 
 F_t(\bar \th^2_{\min}(t,\sg),\sg)| \le 2\lb^{7/4}_* \}
\le 8C_3 \lb^{3/4}_* 
\]
and 
\[
\Leb\{\sg_6(\bar \th^1_{\min}(t,\sg)): 
 |\partial_t F_t(\bar \th^1_{\min}(t,\sg),\sg) - 
 \partial_t F_t(\bar \th^2_{\min}(t,\sg),\sg)| \le 2\lb^{7/4}_* \}
\le 
 8C_3\lb^{3/4}_*. 
\]
Combining we get 
\[
 \mu_\nu\{(\sg_1,\sg_2,\sg_3,\sg_4,\sg_5,\sg_6)\in D_\nu:
\ \sg \notin \Cr^d(\lb_*),\ 
 \min_{t^*\in \B_\sg} d(t^*,\sg)\le 2\lb_*^{7/4}\ 
 \}  \le \dfrac{C_3^3\, \,\sqrt \lb_* }{ 2\,\pi^2\ \nu^3}.
\]
We now prove the following 
\blm With the above notations we have \ \ 
$\mu_\nu \{ \B^{1,d} (\lb^*) \} \le \dfrac{
C_3^3\, \sqrt \lb_*}{2\,\pi^2 \, \nu^3}.$
\elm 

\begin{proof}
In order to study quantitative absence of three global 
minima we use a very similar strategy. The condition
of having three global minima can be written as follows.
If a function $G$ has three global minima, then it has 
three critical points with the same value:
\[
 G(\th_0)= G(\th_1)= G(\th_2),\quad 
 G'(\th_0)= G'(\th_1)= G'(\th_2)=0.
\]
We need a quantitative version of when this condition
almost holds. We have 
\[
 \B^2(\lb^*)=\{(\sg_1,\sg_2,\sg_3,\sg_4)\in D_\nu:\ 
 \exists t\in [a_-,a_+],\ 
\max_{i=1,2,3}|\partial_\th F(\th^i_{\min}(t,\sg),\sg)|\le \lb_*
\]
\[
\max_{i\ne j\in \{1,2,3\}} 
|F(\th^i_{\min}(t,\sg),\sg)- F(\th^j_{\min}(t,\sg),\sg)|\le \lb_*^{7/4}\}.
\]
By analogy with we rewrite the family 
\[
F_t(\th,\sg) = f_t (\th) + \sg_1(\th_1) \cos 2\pi\th + \sg_2(\th_1) \sin 2\pi\th+\sg_3(\th_1) \cos 4\pi\th +\]
\[
+ \sg_4(\th_1) \sin 4\pi\th + 
t \sg_5(\th_1) \cos 2\pi\th + t \sg_6(\th_1) \sin 2\pi\th=:
f_t (\th) + P_\sg(\th;\th_1). 
\]
Using independent parameters we evaluate measure of 
\[
 |\partial_{\th}F_t(\th_i,\sg)|\le \lb_*, \quad i=1,2,3 \ 
\text{ and }\]
\[
 |F_t(\th_1,\sg)-F_t(\th_2,\sg)|\le \lb_*, \qquad 
 |F_t(\th_1,\sg)-F_t(\th_3,\sg)|\le \lb_*,
\]

To keep the condition $ \partial_{\th}F_t(\th_1,\sg)=0$
it suffices to have $\partial_{\th} P_\sg(\th_1;\th_1)=0$. 

To keep the condition $ \partial_{\th}F_t(\th_2,\sg)=0$
it suffices to have $\partial_{\th} P_\sg(\th_2;\th_1)=0$. 

To keep the condition $ \partial_{\th}F_t(\th_2,\sg)=0$
it suffices to have $\partial_{\th} P_\sg(\th_3;\th_1)=0$. 

To keep the condition {\small $F_t(\th_1,\sg)-F_t(\th_2,\sg)=0$}
it suffices to have {\small $P_\sg(\th_1;\th_1)-P_\sg(\th_2;\th_1)=0$}. 

To keep the condition {\small $F_t(\th_1,\sg)-F_t(\th_3,\sg)=0$}
it suffices to have {\small $P_\sg(\th_1;\th_1)-P_\sg(\th_3;\th_1)=0$}. 

We obtain five conditions and four variables. Linear algebra 
calculations complete the proof in the same way as before. 
\end{proof}

\subsection{Deformation of single averaged potential along 
Dirichlet  resonances}
\label{sec:single-res-perturbation}
Fix the collection of the  Dirichlet resonances 
$\{k_j\}_{j\in \FF_n}$. Denote this collection 
by $\FF_n$. Each resonance $k\in \Z^3$ has 
the corresponding averaged potential of $\widetilde H_1(J,\phi)$,
given by 
\[
 Z_k(\phi^s,J)= 
 \sum_{m\in \Z} \tilde h_{mk}(J)\, \exp(i m\phi^s), \quad 
\phi^s\in\T, 
\]
where $\tilde h_p, \ p\in \Z^3$ are the Fourier coefficients 
of $\widetilde H_1(J,\phi)$. 
We add a perturbation 
\[
 \widetilde H_1(J,\phi) + \Delta H_1(J,\phi)
\]
so that for each $k\in \{k_j\}_{j\in \FF_n}$
the averaged potential 
\[
 Z'_{k_j}(\phi^s ,J)= Z_k(\phi^s ,J)+ \chi_j^{\rho_n}(J)
 \Delta Z_{k_j}(\phi^s ,J), \quad \phi^s\in \T 
\]
has at most two non-degenerate global minima and 
at each we have 
\[
\min_{J^f} \partial^2_{\phi^s\phi^s} 
Z'_k(\phi^s ,J^s(J^f),J^f)\ge \rho^{5(r+1)}.
\]

The perturbation has the following form: 
\be \nonumber
\beal 
 \Delta Z(\phi,J_f)=\sum_{k_j\in \FF_n}
 \chi_j^{\rho_n}(J)\left(\Delta h^1_{k}\, \cos( k \cdot \phi)
+ \Delta h^2_{k}\, \sin( k \cdot \phi)
+  \Delta h^3_{k}\, \cos(2k \cdot \phi)+\right.\\ \left.
+ \Delta h^4_{k}\, \sin(2k \cdot \phi)
+ J^f\,\Delta h^5_{k}\, \sin(2k \cdot \phi)
+ J^f\,\Delta h^6_{k}\, \sin(2k \cdot \phi)\right).
\enal
\ee

In order to have small $\CCC^r$ norm of 
the perturbation we have 
\[
 |\Delta h^j_k|\le l_n=R_n^{-r-1}\rho^{r+1}\sim 
 \rho^{4(r+1)/3}.
\]
We prove that for the $6$-dimensional Lebesgue 
measure of each resonance

\be 
\beal 
 \text{mes}
 \{\,(\Delta h^1_k,\Delta h^2_k, \Delta h^3_k,\Delta h^4_{k}, 
\Delta h^5_k,\Delta h^6_{k}): 
 \not\exists J_k\in [J_k^-,J_k^+]  
\text{ s.t. } \\ 
 Z_k(\phi^s,J^s(J^f),J^f)+ \Delta Z_k(\phi^s) \qquad \qquad \qquad  \\ 
 \text{ has a }\rho_n^{9(r+1)/2}\text{--degenerate 
global minimum}\}\\
 \le C \ \dfrac{\rho_n^{9(r+1)/2}}{l_n^3} \le C\rho^{(r+1)/2},
\qquad \qquad \qquad \qquad 
\enal 
\ee
where $C$ is a universal constant independent of $Z$. 
The reason we have $l_n^3$ in the demoninator, because 
we find three conditions to be almost fullfilled to 
have a locally degenerate minimum. Each condition, when 
measures with respect to normalized measure get additional 
$l_n$ in the denominator. 
The volume of a $4$-dimensional ball of radius $\dt$ is $\dt^4$, 
so there exists a perturbation of size $\rho^{35(r+1)/24}$.

\section{The NHICs along  single resonances: proof 
of Key Theorem \ref{keythm:Transition:IsolatingBlock}}\label{sec:NHIC}
In this section we prove Key Theorem \ref{keythm:Transition:IsolatingBlock} by proving 
the existence of NHICs in the single resonance zones along 
the Dirichlet resonant segment $\II_{k_n}^{\om_n}$ given by Theorem \ref{KeyThm:SelectionResonances}. We use the normal forms that we have obtained 
in Theorem \ref{thm:NF:Transition}.  We follow the techniques developed in \cite{BernardKZ11}. Nevertheless, the approach 
in that paper needs to be modified since now the different 
parameters involved satisfy different relations.

\subsection{Change to slow-fast variables}\label{sec:NormalForm:DR:StrongWeak}
First we perform a change of variables to Hamiltonian \eqref{def:hamAfterNF} separate the slow and fast angles. This change  is $\rr_n$-dependent and therefore we need  accurate estimates. To simplify notation, in this section we take $\rr=\rr_n$.
%In the neighborhood of weak double resonances,  
%the Hamiltonian system \eqref{def:HamAfterDRNF} can be seen as a perturbation 
%of an integrable one. To show this fact, we do not need to perform any change 
%of variables or rescaling as happened for strong resonances. 

%Recall that since we are dealing with single resonances, we only have one slow angle. 
We define 
\begin{equation}\label{def:ChangeToFastSlow:StrongWeak}
 \left(\begin{matrix}\phi^s\\\phi^f\\t\end{matrix}\right)=
 \wt A\left(\begin{matrix}\phi_1\\\phi_2\\t\end{matrix}\right)\,\,
 \text{ with }\wt A=\left(\begin{matrix}k^*\\ e_2\\ e_3\end{matrix}\right) 
\end{equation}
where $e_i$ are the standard coordinate vectors (if this change is singular, 
replace some of the coordinate vectors). To have a symplectic 
change of coordinates, we perform the change of coordinates  
\begin{equation}\label{def:ChangeToFastSlow:StrongWeak2}
\left(\begin{matrix}J^s\\ J_1^f\\ E\end{matrix}\right)=\wt A^{-T}
\left(\begin{matrix}J_1\\J_2\\E\end{matrix}\right)
\end{equation}
to the conjugate actions. $E$ is the variable conjugate to time, which is not  not modified  when the change \eqref{def:ChangeToFastSlow:StrongWeak2} is performed. 
%As we have done in Section \ref{sec:NormalForm:DR:Strong}. 
We call $J$ to $J=(J^s,J^f)$. The 
matrix $\wt A$ and its inverse, satisfy
\begin{equation}\label{def:ChangeToFastSlow:StrongWeak:Bounds}
\|\wt A\|,\|\wt A^{-1}\|\lesssim \rr^{-\frac{1}{3}-\tau}
\end{equation}
and the same bounds are satisfied by their transposed matrices. The next lemma gives 
estimates for the transformed Hamiltonian.

\begin{lemma}\label{thm:TransitionZone:NF:slow-fast}
If we apply the symplectic change of coordinates \eqref{def:ChangeToFastSlow:StrongWeak}--\eqref{def:ChangeToFastSlow:StrongWeak2} to the Hamiltonian \eqref{def:hamAfterNF}, 
we obtain a Hamiltonian of the form
\begin{equation}\label{def:hamNF:slow-fast}
\wh \HH(J,\phi)=\wh \HH_0(J)+\wh \ZZZ(J,\phi^s)+\wh \RRR(J,\phi^s,\phi^f,t),
\end{equation}
where $\wh \ZZZ$ only depends on the slow angle, $\wh \RRR$ depends on all three angles, 
and they satisfy the following bounds 
\[
 \left\|\wh \ZZZ\right\|_{\CCC^2}\leq C_5\rr^{\frac{4}{3}r-\frac{2}{3}_\tau}\leq C_5\rr^{\frac{4}{3}r-1}
\]
and 
\[
 \left\|\wh \RRR\right\|_{\CCC^2}\leq  
 C_5\rr^{q(r+1)-\frac{2}{3}_\tau}\leq  
 C_5\rr^{q(r+1)-1}. 
\]
Moreover, $\wh\HH_0$ satisfies
\[
\frac{1}{2}D\ii\rr^{2m+1} \mathrm{Id}\leq \frac{1}{2}D\ii\rr^{2m+\frac{2}{3}+\tau} \mathrm{Id}\leq 
\pa_J^2\HH_0(J)\leq 2D\rr^{2m-\frac{2}{3}-\tau} \mathrm{Id}\leq 2D\rr^{2m-1} \mathrm{Id}
\]
where $D$ is the constant introduced in \eqref{def:Convexity:OriginalHam}.
\end{lemma}

In this system of coordinates the resonant vector has become $k_n=(1,0,0)$. We prove the existence of a NHIC along this resonance for the Hamiltonian \eqref{def:hamNF:slow-fast}. The resonance $\II_{k_n}^{\om_n}$  is now defined by 
\begin{equation}\label{def:resonance:slow-fast}
\pa_{J^s} \wh \HH_0(J)=0.
\end{equation}
Since $\pa_{J_f} \wh\HH_0(J)\neq0$ along the resonance. The resonant segment $\II_{k_n}^{\om_n}$ can be parameterized as a graph as $J^s=J_*^s(J^f)$ for $J^f\in [b_{k_n}^-,b_{k_n}^+]$, for some $b_{k_n}^-<b_{k_n}^+$.

\subsection{Existence of NHICs}\label{sec:NHIC:StrongWeakDR}
%In this section we study the Hamiltonian \eqref{def:hamNF:slow-fast}.  
%\subsubsection{Generic properties of a family of averaged potentials}\label{sec:StrongWeakDR:GenericAveragedPotential}
%First, we explain the hyperbolicity hypotheses that we will need. This should come out of the deformation process. So far we assume that bifurcations can be avoided thanks to the smallness of the segment. I'm not sure about that, I'll think about it later.
Key Theorem \ref{keythm:Transition:Deformation} implies that the truncated Hamiltonian $\HH^\trunc$ in \eqref{def:SingleResonance:TruncatedHamiltonian} has a sequence of NHICs (see \eqref{def:Cylinder:truncated}). Now we show persistence of those cylinders for the Hamiltonian \eqref{def:hamNF:slow-fast}.

\begin{comment}
Consider the averaged potential $\wh \ZZZ(J, \varphi^s)$ in  \eqref{def:hamNF:slow-fast} for $J^f\in  [b_{k_n}^-,b_{k_n}^+]$. We can look at it along the resonance, that is, we can take $\wh \ZZZ(J_*^s(J^f),J^f), \varphi^s)$. Call a value $J^f \in  [b_{k_n}^-,b_{k_n}^+]$ \emph{regular} if $\wh \ZZZ(J_*^s(J^f),J^f), \varphi^s)$ has 
a unique global maximum on $\T^s \ni \varphi^s$ at some $\varphi^s*=\varphi^s(J^f)$. 
We say the maximum is non-degenerate if
\[
 \pa_{\varphi^s\varphi^s}\wh \ZZZ(J_*^s(J^f),J^f), \varphi^s)<0.
\]
This maximum depends smoothly on $J^f$ and can be extended to a bigger interval
$\varphi_*^{s}(J^f):[a_--\de, a_+\de]\to \T^2$, in which is a local maximum but might not be the global one.
From the deformation procedure done in \ref{sec:DeformationSinglePotential}, we have proved that $\wh\ZZZ$ has a regular maximum in the needed interval, and that it satisfies
\[
 -\pa_{\varphi^s\varphi^s}\wh \ZZZ(J_*^s(J^f),J^f), \varphi^s)\geq \rr^{\dr}.
\]

\begin{lemma}\label{lemma:cylinder:FirstOrder}
 
\end{lemma}
\end{comment}

The key idea, used in \cite{BernardKZ11}, \cite{KaloshinZ12}, is 
to construct an isolating block around 
\[
\{\varphi_*^{s}(J^f)\}\times \{J^s(J^f)\}\times \T \times [b_--\de,b_++\de]\times \T
\ni (\varphi^s,J^s,\varphi^f,J^f,t).
\]

%Let $\lb\sim l^{m/2}(p-p_0),\quad \dt\sim \sqrt{\eps_0}\, l^m(p-p_0),\quad  
%\eps \sim l^{2m}(p-p_0)$. 

We shall at some occasions lift the map $\varphi^s_*$ to a $\CCC^2$ map
taking values in $\R$ without changing its name.
To simplify notations, we will be using the $\OO(\cdot)$
notation, where $f=\OO(g)$ means $|f|\le Cg$ for
a constant $C$ independent of $\rr$, $\delta$,
and $r$. In particular, we will not be keeping track of
the parameter $D$, which is considered fixed throughout
the paper.

\begin{theorem}\label{thm:ExistenceNHIC} 
There exists  a $\CCC^1$ map
\[
(\Theta^s, P^s)(\varphi^{f},J^{f}):\T^2\times
[a_-,a_+]  \to \T \times \R
\]
such that the cylinder
\[
\mathcal C=\{ (\varphi^s, \wt J^s)=(\Theta^s, P^s)
(\varphi^f, \wt J^f); \quad
(J_*^{s}(p^f),J^f)\in\II_{k_n}^{\om_n},\
\varphi^f\in \T^2\}
\]
is weakly invariant with respect to the vector field associated to the Hamiltonian  \eqref{def:hamNF:slow-fast}, 
in the sense that the  vector field is tangent to $\mathcal C$. 
The cylinder $\mathcal C$ is contained in the set
\beq
\beal
V:=\big\{&(\varphi,\wt J); \wt J^{f}\in[a_-,a_+], \\
&\|\varphi^s-\varphi^s_*(\wt J^{f})\|\leq
M'\rr^{ m-1+2dr},
\quad \|J^s-J^s_*(J^f)\|\le
M'\rr^{ m-1+2dr}
\big\},
\enal
\eeq 
for some constant $M'$ independent of $\rr$
and it contains all the full orbits  of \eqref{def:hamNF:slow-fast}
contained in $V$. We have the estimates
\[
\|\Theta^s(\varphi^f,\wt J^f)-\varphi^s_*(\wt J^f)\|\lesssim
M'\rr^{m-1+5dr}
\]
\[
\| P^s(\varphi^f,\wt J^f,t)-\wt J^s_*(\wt J^f)\|\lesssim
M'\rr^{m-1+5dr},
\]
\[ \left\|\frac{\partial\Theta^s}{\partial \wt J^f}\right\| \lesssim
M'\rr^{\frac{3}{2}dr-\frac{m+1}{2}}
,
\quad \left\|\frac{\partial \Theta^s}{\partial\varphi^f}
\right\| \lesssim
\rr^{\frac{3}{2}dr-\frac{m+1}{2}}.
\]
\end{theorem}

This theorem implies Key Theorem \ref{keythm:Transition:IsolatingBlock}. It only suffices to undo the scaling \ref{def:rescalingSR}.

\subsection{Proof of Theorem \ref{thm:ExistenceNHIC}: 
existence of NHICs}
Theorem \ref{thm:ExistenceNHIC} is a consequence  of 
Theorem 4.1 in \cite{BernardKZ11}. Nevertheless we need 
to slightly modify the system associated to the Hamiltonian  
to be able to apply that Theorem since the different parameters involved in \cite{BernardKZ11} satisfy different relations. Moreover, recall that 
the change \eqref{def:ChangeToFastSlow:StrongWeak2} makes convexity $\rr$-dependent.

%This makes $\wt\HH$ not very well suited to apply Theorem 4.1 in \cite{BernardKZ11} 
%since this paper assumes the convexity to be independent of the small parameter 
%which measures near-integrability. To avoid reproving Theorem 4.1 in \cite{BernardKZ11} 
%including $\rr$-depending convexity we consider a non-Hamiltonian change of coordinates.

We redo the proof in \cite{BernardKZ11}. We start estimating different quantities and computing its $\rr$ dependence.  The Hamiltonian flow admits the following equations of motion
\begin{equation}\label{eq:IsolatingBlock:Original}
 \begin{split}
  \dot\varphi^s&=\pa_{J^s}\wh\HH_0+\pa_{J^s}\wh\ZZZ+\pa_{J^s}\RRR\\
  \dot J^s&=-\pa_{\varphi^s}\wh\ZZZ-\pa_{\varphi^s}\RRR\\
  \dot\varphi^f&=\pa_{J^f}\wh\HH_0+\pa_{J^f}\wh\ZZZ+\pa_{J^f}\RRR\\
  \dot J^f&=-\pa_{\varphi^f}\RRR\\
  \dot t&=1,
 \end{split}
\end{equation}
we denote this vector field by $F$.

First we look for a good first order. Since in the present setting is much smaller than the size of the angle dependent part of the Hamiltonian, the first order considered in \cite{BernardKZ11} cannot be used. Instead, we consider a different one.  This new first order will allow us to look for a system of coordinates suitable for the isolating block procedure. 
\begin{equation}\label{def:ODEForIsolating:FirstOrder}
 \begin{split}
  \dot\varphi^s&=\pa_{J^s}\wh\HH_0+\pa_{J^s}\wh\ZZZ\\
  \dot J^s&=-\pa_{\varphi^s}\wh\ZZZ\\
  \dot\varphi^f&=\pa_{J^f}\wh\HH_0\\
  \dot J^f&=0\\
  \dot t&=1
 \end{split}
\end{equation}

Now, due to Key Theorem \ref{keythm:Transition:Deformation},  for any fixed $J^f\in [J_i^f-\de,J_{i+1}^f+\de]$, the point $(\varphi_{*,i}^s(J^f), J_{*,i}^s(J^f))$ is a hyperbolic critical point for the system 
\[
\dot\varphi^s=\pa_{J^s}\wh\HH_0+\pa_{J^s}\wh\ZZZ,\,\,\, \dot J^s=-\pa_{\varphi^s}\wh\ZZZ.
\]
Moreover, the differential associated to the the hyperbolic point is given by
\[
 M(J^f)=\begin{pmatrix}a & b\\ c&-a\end{pmatrix}
\]
where 
\begin{equation}\label{def:IsolatingBlock:DefABC}
\begin{split}
a&= \pa_{\varphi^sJ^s}\wh\ZZZ(\wh\varphi_*^s(J^f), \wh J_*^s(J^f))\\
b&= \pa_{J^sJ^s}\wh\HH_0( \wh J_*^s(J^f))+\pa_{J^sJ^s}\wh\ZZZ(\wh\varphi_*^s(J^f), \wh J_*^s(J^f))\\
c&= -\pa_{\varphi^s\varphi^s}\wh\ZZZ(\wh\varphi_*^s(J^f), \wh J_*^s(J^f))
\end{split}
\end{equation}
Thanks to Key Theorem \ref{keythm:Transition:Deformation}, these coefficients satisfy 
\begin{equation}\label{def:IsolatingBlock:EstimatesABC}
|a|\lesssim \rr^{\frac{4}{3}r-1}
\ \ \text{ and }\ \ 
\begin{aligned}
 \ \ \ 0<\rr^{2m+1}\lesssim  b\lesssim \rr^{2m-1}\\
\ \ \ 0< \rr^{dr}\lesssim  c\lesssim \rr^{\frac{4}{3}r-1}.%\quad
\end{aligned}
\end{equation}
The eigenvalues of the hyperbolic critical points are given by $\pm \la$ where 
\begin{equation}\label{def:IsolatingBlock:Eigenvalue}
\la=\sqrt{a^2+bc} 
\end{equation}
Therefore, they satisfy $\rr^{\frac{dr+2m+1}{2}}\lesssim\la\lesssim \rr^{\frac{2}{3}r+m-1}$.

We diagonalize the matrix associated to these critical points. To this end we define the matrix 
\[
 S=\begin{pmatrix}a+\sqrt{a^2+bc}&-b\\c&a+\sqrt{a^2+bc} \end{pmatrix}
\]
First, we give some estimates on the critical point $(\varphi_*^s(J^f), J_*^s(J^f))$ and the matrix $S$.
\begin{lemma}\label{lemma:IsolatingBlock:EstimatesDiagonalization}
The hyperbolic critical point $(\varphi_*^s(J^f), J_*^s(J^f))$ and the matrix $S$ satisfy the following estimates 
\[
\left|\pa_{J^f}\varphi_*^s(J^f)\right|,\left| \pa_{J^f}J_*^s(J^f) \right|\lesssim \rr^{\frac{4}{3}r-dr-3}
\]
and 
\[
\begin{aligned}
 |S|\lesssim&\rr^{2m-1}\qquad \qquad \qquad |S^{-1}|\lesssim&\rr^{-dr-2}\\
|\pa_{J^f}S|\lesssim&\rr^{\frac{8}{3}r-\frac{3}{2}dr-7}\qquad 
|\pa_{J^f}S^{-1}|\lesssim&\rr^{\frac{20}{3}r-\frac{7}{2}dr}.
\end{aligned}
\]
\end{lemma}
\begin{proof}
The estimates for the critical point can be obtained applying implicit derivation, since one obtains
\[
M\begin{pmatrix} \pa_{J^f}\varphi_*^s\\\pa_{J^f}J_*^s\end{pmatrix}=\begin{pmatrix}
                                                                     -\pa_{J^fJ^s}\wh \HH_0-\pa_{J^fJ^s}\wh\ZZZ\\\pa_{J^f\varphi^s}\wh\ZZZ
                                                                    \end{pmatrix}
\]
where the right hand side is evaluated at the critcal points. Thus,
\[
 \begin{pmatrix} \pa_{J^f}\varphi_*^s\\\pa_{J^f}J_*^s\end{pmatrix}=\frac{1}{a^2+bc}\begin{pmatrix}a & b\\ c&-a\end{pmatrix}\begin{pmatrix}
                                                                     -\pa_{J^fJ^s}\wh \HH_0-\pa_{J^fJ^s}\wh\ZZZ\\\pa_{J^f\varphi^s}\wh\ZZZ
                                                                    \end{pmatrix}
\]
We have that $\det M\gtrsim \rr^{2m+dr+1}$. Then, using also the estimates for $a$, $b$ and $c$ in \eqref{def:IsolatingBlock:EstimatesABC}, one obtains the desired bounds. 

The estimate for $|S|$ is straightforward. For $S\ii$, one  just needs to take into account that $\det S=2a^2+2bc+2a\sqrt{a^2+bc}$ satisfies
$|\det S|\geq 2bc\gtrsim\rr^{2m+dr+1}$.

Now, we compute the bounds for the derivatives $\pa_{J^f}S$ and $\pa_{J^f}S\ii$. To this end,  we need to compute the derivative for $a$, $b$, $c$ and $\det S$.  By the definiton of $a$ in \eqref{def:IsolatingBlock:DefABC}, we have that 
\[
 \left|\pa_{J^f}a\right|\leq \left\|\wh\ZZZ\right\|_{\CCC^3}\left(1+\left|\pa_{J^f}J_*^s\right|+\left|\pa_{J^f}\varphi_*^s\right|\right)\lesssim \rr^{\frac{8}{3}r-dr-5}
\]
The computations for $\pa_{J^f}c$ is analogous. For $\pa_{J^f}b$ it is enough to recall also that 
\[
\left|\pa_{J^sJ^fJ^f}\wh \HH_0\right|\lesssim \rr^{3m-2}.
\]
Therefore, we obtain also that  $\left|\pa_{J^f}b\right|\lesssim \rr^{\frac{4}{3}r+3m-dr-5}$. 

Using these estimates and also \eqref{def:IsolatingBlock:EstimatesABC}, we obtain the bound for $\pa_{J^f}S$. For $\pa_{J^f}S\ii$, we need  upper bounds for  
\[
\pa_{J^f}\det S=\left(2-\frac{a}{\sqrt{a^2+bc}}\right)\left(a\pa_{J^f}a+c\pa_{J^f}b+b\pa_{J^f}c\right)+2\pa_{J^f}a\sqrt{a^2+bc}
\]
Using the just obtained bounds and  \eqref{def:IsolatingBlock:EstimatesABC}, one can see that $ \left|\pa_{J^f}\det S\right|\lesssim \rr^{\frac{8}{3}r+m-\frac{3}{2}dr-7}\leq \rr^{\frac{8}{3}r-\frac{3}{2}dr}$. Then, using all these estimates, we obtain that 
$ \left|\pa_{J^f}S\ii\right|\lesssim \rr^{\frac{8}{3}r-\frac{7}{2}dr-3m-3}.$
\end{proof}

We define the new coordinates
\begin{equation}\label{def:IsolatingBlock:ChangeDiagonalization}
 \begin{pmatrix} x\\ y\end{pmatrix}=S^{-1}
\begin{pmatrix}\varphi^s-\varphi_*^s(J_f)\\J^s-J_*^s(J_f)\end{pmatrix}
\end{equation}
The inverse change is given by 
\[
\begin{pmatrix}  \varphi^s\\ J^s\end{pmatrix}=\begin{pmatrix}  \varphi_*^s(J_f)\\ J_*^s(J_f)\end{pmatrix}+S \begin{pmatrix} x\\ y\end{pmatrix}
\]
We look for an isolating block in the region $|x|\leq \eta, |y|\leq \eta$ were $\eta$ is a parameter to be determined and depends on $\rr$. We need also to rescale the other variables,
\begin{equation}\label{def:IsolatingBlock:Rescalings}
 \begin{split}
 I&=\nu\ii J^f\\ 
 \Theta&=\ga\varphi^f
 \end{split}
\end{equation}
where the parameters $\nu,\ga\ll1$ will be determined later.

\begin{lemma}\label{lemma:IsolatingBlock:EstimatesDeviationHypPoint}
Assume that $|x|\leq \eta, |y|\leq \eta$. Then, the following estimates are satisfied
\[
\left| \varphi^s-\varphi_*^s(J_f)\right|\lesssim \rr^{2m-2} \eta,\,\,\,\left|J^s-J_*^s(J_f)\right|\lesssim \rr^{2m-2} \eta
\]
\end{lemma}

\begin{lemma}\label{lemma:IsolatingBlock:NewEquation}
In the new variables, the equation \eqref{eq:IsolatingBlock:Original} takes the following form
\begin{equation}
 \begin{split}
\begin{pmatrix}\dot x\\\dot y\end{pmatrix}=\Lambda\begin{pmatrix}x\\y\end{pmatrix}+\OO\left(\rr^{r+(q-2d)r}+\rr^{6r+\left(q-\frac{7}{2}d\right)r}\eta+\rr^{6m-dr}\eta^2\right)
\end{split}
\end{equation}
where $\Lambda=\mathrm{diag}(\la,-\la)$.
\end{lemma}
\begin{proof}
Applying the change \eqref{def:IsolatingBlock:ChangeDiagonalization},\eqref{def:IsolatingBlock:Rescalings}, we have
\[
\begin{split}
 \begin{pmatrix} \dot x\\ \dot y\end{pmatrix}=&\pa_{J^f} S^{-1}\dot J^f \begin{pmatrix}\varphi^s-\varphi_*^s(J_f)\\J^s-J_*^s(J_f)\end{pmatrix}+S^{-1}\begin{pmatrix}\dot \varphi^s\\\dot J^s\end{pmatrix}+S^{-1}\begin{pmatrix}\dot \varphi_*^s\\\dot J_*^s\end{pmatrix}\\
=&\pa_{J^f} S^{-1}\dot J^f S\begin{pmatrix}x\\y\end{pmatrix}+S^{-1}\begin{pmatrix}\dot \varphi^s\\\dot J^s\end{pmatrix}+S^{-1}\begin{pmatrix}\pa_{J^f}\varphi_*^s\\ \pa_{J^f}J_*^s\end{pmatrix}\dot J^f.
\end{split}
\]
We analyze of each three terms. For the first one, we use that $\dot J^f=\OO(\rr^{qr})$ and the estimates in Lemmas \eqref{lemma:IsolatingBlock:EstimatesDiagonalization} and \eqref{lemma:IsolatingBlock:EstimatesDeviationHypPoint} to obtain
\[
\left|\pa_{J^f} S^{-1}\dot J^f S\begin{pmatrix}x\\y\end{pmatrix}\right|\leq \rr^{\frac{20}{3}r+\left(q-\frac{7}{2}d\right)r+2m-1}\eta.
\]
Using the same lemmas, for the third term we obtain 
\[
 \left|S^{-1}\begin{pmatrix}\pa_{J^f}\varphi_*^s\\ \pa_{J^f}J_*^s\end{pmatrix}\dot J^f\right|\lesssim \rr^{\frac{4}{3}r+(q-2d)r-5}
\]
For the second term, we use equation \eqref{eq:IsolatingBlock:Original}. Indeed, we have that 
\[
\begin{split}
 \begin{pmatrix}\dot \varphi^s\\\dot J^s\end{pmatrix}&=M\begin{pmatrix}\varphi^s-\varphi^s_*\\J^s-J^s_*\end{pmatrix}+\rr^{3m-2}\begin{pmatrix}\OO\left(\varphi^s-\varphi^s_*\right)^2\\\OO\left(J^s-J^s_*\right)^2\end{pmatrix}+\OO\left(\rr^{qr}\right)\\
&=M\begin{pmatrix}\varphi^s-\varphi^s_*\\J^s-J^s_*\end{pmatrix}+\OO\left(\rr^{7m-6}\eta^2+\rr^{qr}\right).
\end{split}
\]
Therefore 
\[
 S^{-1}\begin{pmatrix}\dot \varphi^s\\\dot J^s\end{pmatrix}=\Lambda\begin{pmatrix}x\\y\end{pmatrix}+\OO\left(\rr^{7m-dr-8}\eta^2+\rr^{(q-d)r-2}\right)
\]
Using the relation between $d$ and $q$, one obtains the desired estimates.
\end{proof}
To prove the existence of an isolating block, we need to analize also the linearized equation. We first analyze the linearization of the change of coordinates \eqref{def:IsolatingBlock:ChangeDiagonalization}, \eqref{def:IsolatingBlock:Rescalings}. 

\begin{lemma}\label{lemma:IsolatingBlock:DifferentialChanges}
The linearization of the change of coordinates $(x,y,\Theta,I,t)\rightarrow (\varphi^s, J^s,\varphi^f, J^f, t)$
is of the following form 
\[
T= \dps\left(\frac{\pa (\varphi^s, J^s,\varphi^f, J^f, t)}{\pa (x,y,\Theta,I,t)}\right)=
  \begin{pmatrix}
  S &0&\nu A&0\\ 0&0&\ga\ii&0&0\\
0&0&0&\nu&0\\ 0&0&0&0&1 
  \end{pmatrix}
\]
where 
\[
 A=\pa_{J^f}\begin{pmatrix} \varphi_*^s\\J_s^*\end{pmatrix}+\pa_{J^f}S\begin{pmatrix} x\\ y\end{pmatrix}
\]
The linearization of the inverse change of coordinates is of the form 
\[
 T\ii=
  \begin{pmatrix}
  S\ii &0&B&0\\ 0&0&\ga&0&0\\
0&0&0&\nu\ii&0\\ 0&0&0&0&1 
  \end{pmatrix}
\]
where  
\[
 B= \pa_{J^f} S\ii \begin{pmatrix}
                    \varphi^s-\varphi^s_*\\ J^s-J_*^s
                   \end{pmatrix}
+S\ii  \pa_{J^f}\begin{pmatrix}
                    \varphi^s_*\\ J_*^s
                   \end{pmatrix}
\]
Moreover, the norms of $T$ and $T\ii$ have the following bounds
\[
\begin{split}
|T|&\lesssim \ga\ii + \rr^{\frac{4}{3}r-dr-3}\nu+\rr^{\frac{8}{3}r-\frac{3}{2}dr-7}\eta\nu\\
\left|T\ii\right|&\lesssim \nu\ii+\rr^{\frac{4}{3}r-2dr-5}+ \rr^{\frac{20}{3}r-\frac{7}{2}dr+2m-2}\eta.
\end{split}
\]
\end{lemma}

\begin{proof}
 The computation of $T$ and $T\ii$ is straightforward. To compute their norms we first bound $A$ and $B$. By Lemma \ref{lemma:IsolatingBlock:EstimatesDiagonalization}, one can see that $A$ and $B$ satisfy
\[
\begin{split}
|A|&\leq \rr^{\frac{4}{3}r-dr-3}+\rr^{\frac{8}{3}r-\frac{3}{2}dr-7}\eta\\
|B|&\leq \rr^{\frac{4}{3}r-2dr-5}+ \rr^{\frac{20}{3}r-\frac{7}{2}dr+2m-2}\eta.
\end{split}
\]
From these estimates, one can obtain the bounds for $T$ and $T\ii$, recalling that  $m=r/10$ (see \eqref{def:CoreDR-m})  and  that $\ga\ll 1$ and $\nu\ll 1$.
\end{proof}

\begin{lemma}\label{lemma:IsolatingBlock:Differential}
In the coordinate system $(x,y,\Theta,I,t)$, the linearized system is given by the block matrix
\[
\begin{split}
L=&\begin{pmatrix} \Lambda&0\\
   0&0\\
  \end{pmatrix}+\OO\left(\rr^{2m-2}\ga, \rr^{5m-dr-6}\eta,\rr^{\frac{4}{3}r-dr}\nu\ga,\rr^{\frac{8}{3}r-\frac{5}{2}dr}\eta\nu\right)\\&+\OO\left(\rr^{qr}\ga\ii\nu\ii,\rr^{\frac{4}{3}r-(q-2d)r-5}\ga\ii,\rr^{\frac{20}{3}r-(q-\frac{7}{2}d)r+2m-2}\ga\ii\eta, \rr^{\frac{4}{3}r+(q-d)r-3},\rr^{\frac{8}{3}r+(q-3d)r-5}\nu\right).
\end{split}
\]
\end{lemma}

\begin{proof}
Denote by $F$ the vector field \eqref{eq:IsolatingBlock:Original}. Then, 
\[
 DF=\begin{pmatrix}
    M+\OO(\rr^{3m-3}\eta)&0&v_2&0\\v_1^T&0&C&0\\0&0&0&0
    \end{pmatrix}+\OO\left(\rr^{qr} \right)
\]
where
\[
 v_1=\begin{pmatrix}\pa_{J^f\varphi^s}\wh \ZZZ\\ \pa_{J^fJ^s}(\wh \HH_0+\wh \ZZZ)
     \end{pmatrix}
, \,\,\,v_2=\begin{pmatrix}\pa_{J^fJ^s}(\wh \HH_0+\wh \ZZZ)\\ \pa_{J^f\varphi^s}\wh \ZZZ\end{pmatrix}
\]
and $C=\pa_{J^fJ^f}\left(\wh \HH_0+\wh \ZZZ\right)$. Therefore, using the estimates \eqref{thm:TransitionZone:NF:slow-fast}, we have that
\[
 |v_1|\leq \rr^{2m-1},\,\,|v_2|\leq \rr^{m-2},\,\,|C|\leq \rr^{2m-2}.
\]

The matrix $L$ is given by $L=T\ii \cdot DF\cdot T$. Thus, we just need to multiply the matrices.  We obtain 
\[
%\begin{split}
L=
\begin{pmatrix} \Lambda+\OO(|S\ii||S|\rr^{3m-3}\eta) &0&S\ii((M+\OO(\rr^{3m-3}\eta))A+v_2)\nu &0\\
   \ga v_1^TS&0&(v_1^TA+C)\nu\ga&0\\
0&0&0&0\\0&0&0&0
  \end{pmatrix}+\OO\left(|T\ii||T|\rr^{qr}\right).
%\end{split}
\]
To bound each term, we take advantage of a cancellation which arises in $L$. We define
\[
v_2^*=\begin{pmatrix}\pa_{J^fJ^s}(\wh \HH_0( J_*^s, J^f)+\wh \ZZZ(\varphi_*^s, J_*^s, J^f))\\ \pa_{J^f\varphi^s}\wh \ZZZ(\varphi_*^s, J_*^s, J^f)\end{pmatrix}.
\]
Then, we have that 
\[
 M\begin{pmatrix}\pa_{J^f}\varphi_*^s\\ \pa_{J^f} J_*^s\end{pmatrix}+v_2^*=0.
\]
Thus, using the definition of $A$ we have that 
\[
\begin{split}
 \left|S\ii((M+\OO(\rr^{3m-3}\eta))A+v_2)\nu\right|\lesssim & |S\ii||A|\rr^{3m-3}\eta\nu +|S\ii||v_2-v_2^*|\nu\\
&+|S\ii||M||\pa_{J^f}S| \left|\begin{pmatrix}x\\y\end{pmatrix}\right|\nu.
\end{split}
\]
Then, using the estimates of $v_1$, $v_2$ and $C$, the fact that $\ga\ll1$ and $\nu\ll1$, the estimates of $A$ given in Lemma \ref{lemma:IsolatingBlock:DifferentialChanges} and  the estimates given in Lemma \ref{lemma:IsolatingBlock:EstimatesDiagonalization}, we obtain
\[
%\begin{split}
L=\begin{pmatrix} \Lambda&0\\
   0&0\\
  \end{pmatrix}+\OO\left(|T\ii||T|\rr^{qr},\rr^{3m-3}\ga, \rr^{5m-dr-6}\eta,\rr^{\frac{4}{3}r-dr+2m-5}\nu\ga,\rr^{\frac{8}{3}r-\frac{5}{2}dr}\eta\nu\right).
%\end{split}
\]
Finally it is enough to use the estimates for $|T\ii|$ and $|T|$ given in Lemma \ref{lemma:IsolatingBlock:DifferentialChanges}.
\end{proof}

Now, it is enough to choose the parameters $\eta$, $\nu$ and $\ga$ to show that the remainder is smaller than the first order of $\Lambda$. This  allows us to show the existence of the isolating block. We make two different choice of parameters.  The first choice is to construct the smallest possible isolating block. This gives the sharper estimates for the size of the cylinder and also gives good bounds for its derivatives. The second one is to obtain the largest isolating block where we can prove the existence of the cylinder. This gives the maximal set where we can ensure the existence and uniqueness of the cylinder.

First, we take $\eta=\rr^{5dr}$, $\nu=1$ and $\ga=\rr^{2dr}$. Recalling that  $q=18d$ (see Appendix \ref{sec:Notations}), we obtain that the equation obtained in Lemma \ref{lemma:IsolatingBlock:NewEquation} becomes
\begin{equation}\label{eq:IsolatingBlockFinalEq}
 \begin{split}
\begin{pmatrix}\dot x\\\dot y\end{pmatrix}=\Lambda\begin{pmatrix}x\\y\end{pmatrix}+\OO\left(\eta\rr^{\dr}\right)
\end{split}
\end{equation}
and the matrix $L$ obtained in Lemma \ref{lemma:IsolatingBlock:Differential}
\begin{equation}\label{eq:IsolatingBlockFinalDiff}
 \begin{split}
 L=\begin{pmatrix}
    \Lambda & 0\\0&0
   \end{pmatrix}+\OO\left(\rr^{2dr}\right)
  \end{split}
\end{equation}
Now we are ready to apply the isolating block argument as done in \cite{BernardKZ11}. We apply  Proposition A.1
\cite{BernardKZ11} to the system in  coordinates $(x,y,\Theta,I,t)$.
More precisely, with the notations of appendix B \cite{BernardKZ11},
we set
\[
u=x, s=y, c_1=(\Theta,t), c_2=I,
 \Omega=\R^2\times \Omega^{c_2}=\R^2\times 
 \left[\frac{a_- - \lambda^2}{\nu},
 \frac{a_+ + \lambda^2}{\nu}\right].
\]
We fix  $\alpha =\lambda/2$ and we take $B^u=\{u: \|u\|\leq \eta\}$ and $B^s=\{s: \|s\|\leq \eta\}$. Then, by \eqref{eq:IsolatingBlockFinalEq}, one can easily see that 
 \begin{align*}
  \dot x\cdot x&\geq \al x^2\,\,\,\,&\text{ if }|x|=\eta, |y|\leq \eta, (\Theta,I,t)\in\Omega\\
  \dot y\cdot y&\leq -\al y^2\,\,\,\,&\text{ if }|x|\leq\eta, |y|= \eta, (\Theta,I,t)\in\Omega
 \end{align*}
Moreover, using \eqref{eq:IsolatingBlockFinalDiff} and recalling that $\al=\la/2$, we have that 
\[
 \begin{split}
 L_{uu}&=\la+\OO\left(\rr^{2dr}\right)\geq \al\\
 L_{ss}&=\la+\OO\left(\rr^{2dr}\right)\geq \al
 \end{split}
\]
on $B^u\times B^s\times \Omega$. Finally, following the notations of Proposition A.1 of \cite{BernardKZ11}, we have that $m\lesssim \rr^{2dr}$. Therefore, $ K\leq 1/\sqrt{2}$ since
\[
 K=\frac{m}{\al-2m}\lesssim \rr^{\frac{3}{2}dr-\frac{m+1}{2}}\ll 1
\]
Thus, we can apply Proposition A.1 of \cite{BernardKZ11}. This implies that there exists a $\CCC^1$ map 
\[
 w^c=\left(w^c_u, w^c_s\right):\Omega\longrightarrow \RR^2
\]
which satisfies $\|dw^c\|\leq 2K$, is $\ga\ii$-periodic in $\Theta$ and 1-periodic in $t$, and the graph of which is weakly invariant. Now it only remains to go back to the original variables by defining
\[
\begin{pmatrix}
 w_{\varphi^s}^c\\
 w_{J^s}^c\end{pmatrix}=\begin{pmatrix}
 \varphi^s_*(J^f)\\
 J^s_*(J^f)\end{pmatrix}
+S\begin{pmatrix}
w^c_u\left(\ga\varphi^f, \nu\ii J^f,t\right)\\
w^c_s\left(\ga\varphi^f, \nu\ii J^f,t\right)
\end{pmatrix}
\]
Then, we obtain 
\[
%\begin{split}
\left| w_{\varphi^s}^c-\varphi^s_*(J^f)\right|\leq \rr^{m-1}\eta\leq \rr^{m-1+5dr},\qquad 
 \left| w_{J^s}^c- J^s_*(J^f)\right|\leq \rr^{m-1}\eta\leq \rr^{m-1+5dr}
%\end{split}
\]
and for the derivatives 
\[
%\begin{split}
 \left|\pa_{J^f}w_{\varphi^s}^c\right|\lesssim K\nu\ii\lesssim\rr^{\frac{3}{2}dr-\frac{m+1}{2}}, \qquad
 \left|\pa_{(\varphi^f,t)}w_{\varphi^s}^c\right| \lesssim K \lesssim \rr^{\frac{3}{2}dr-\frac{m+1}{2}}
%\end{split}
\]
The second choice of parameters is $\eta=\rr^{2dr}$, $\nu=\rr^{2dr}$, $\ga=\rr^{dr}$. Proceeding analogously one can see that the isolating block argument goes through also with these choice of parameters. This choice of parameters gives the set $V$ given in Theorem \ref{thm:ExistenceNHIC} where there is no other invariant set except the cylinder.

\section{Aubry sets in single resonance zones: proof of Key Theorem \ref{keythm: Transition:AubrySet}}\label{sec:LocAubrySets}
We devote this section to study the properties of the Aubry and 
Ma$\tilde{\mbox n}\acute{\mbox e}$ sets corresponding to 
Dirichlet resonances in the single resonance zones and we prove Key Theorem \ref{keythm: Transition:AubrySet}. The definition of this sets is given in Appendix \ref{app:WeakKAM}. We explain how to modify the approach developed in \cite{BernardKZ11}. We only deal with cohomologies which are not close to the bifurcation values. The other case can be adapted from \cite{BernardKZ11} following the same ideas.

%Theorem \ref{thm:ExistenceNHIC} has given  the existence a normally hyperbolic invariant cylinder along 
%the resonance $\SSS_{k_n}^{\om_n}$. We follow the approach developed 
%in \cite{BernardKZ11} to prove that certain Aubry sets belong to these cylinders.
% Then, in Section \ref{sec:LocAubrySets:StrongDR} we deal with the core of double resonances. This case is more involved since we need to localize the Aubry sets in the different cylinders that we have obtained in Section \ref{sec:NHIC:StrongDR}.
%First in this section we recall the definitions of the Aubry and 
%Ma$\tilde{\mbox n}\acute{\mbox e}$ sets and of weak KAM solutions. 
%Following \cite{BernardKZ11} we also defined the suspended Aubry and Ma$\tilde{\mbox n}\acute{\mbox e}$ sets.
%\subsection{Aubry sets}\label{sec:LocAubrySets:Transition}
%In this section we prove Key Theorem \ref{keythm: Transition:AubrySet}.
We consider the Hamiltonian \eqref{def:hamNF:slow-fast} and show that the Aubry sets related to certain cohomology classes  $c\in H^1(\TT^2)$ belong to the invariant cylinder along the resonance $\II_{k_n}^{\om_n,i}$ obtained in Theorem \ref{thm:ExistenceNHIC}. We also prove the Mather graph property for the Aubry sets.

The proof of Theorem \ref{keythm: Transition:AubrySet} is a consequence of the following two Theorems, which give vertical and horizontal estimates for the Aubry set. 
The first result  replicates Theorem 4.1 in \cite{BernardKZ11} and provides vertical estimates and also a graph property for the Weak KAM solutions analogous to the Mather graph principle. This Theorem is proved in Section \ref{sec:Aubry:vertical}

\begin{theorem}\label{thm:LocAubryVertical:StrongWeak}
Consider the Hamiltonian \eqref{def:hamNF:slow-fast}. 
 Then, for each cohomology class $c\in\RR^2$ and each weak KAM solution $u$ of $H$ at cohomology $c$, the set $\wt\II(u,c)$ is contained in a $C\rr^{2r/3-1}$-Lipschitz graph above $\TT^2$, and in the domain $\|J-c\|\leq C\rr^{2r/3-1}$,  for some constant $C>0$ independent of $\rr$.
\end{theorem}

The next Theorem also gives estimates for the horizontal localization of the Aubry sets. We adapt the statement of \cite{BernardKZ11} to our purposes. This Theorem is proved in Section \ref{sec:Aubry:horizontal}.

\begin{theorem}\label{thm:LocAubryHorizontal:StrongWeak}
Consider the Hamiltonian \eqref{def:hamNF:slow-fast} and a cohomology class $c=(c^s,c^f)=(J_*^s(c^f),c^f)$ with $c^f\in [\wt a_-,\wt a_+]$. Then, for $\rr>0$ small enough, 
 the  Ma$\tilde{\mbox n}\acute{\mbox e}$ set $\wt\NNN(c)$ satisfies
\[
s\wt\NNN(c)\subset B_{\kk\rr^{9dr/4}}(\varphi_*^s)\times\TT^2\times B_{\kk\rr^{9dr/4}}(c^s)\times B_{\kk\rr^{2r/3-1}}(c^f)\subset\TT^2\times\RR^2\times\TT
\]
for some constant $\kk>0$ independent of $\rr$.

Then, it only remains to undo the rescaling \eqref{def:rescalingSR}.
\end{theorem}

With these two Theorems we are ready to prove Key Theorem \ref{keythm: Transition:AubrySet}

\begin{proof}[Proof of Key Theorem \ref{keythm: Transition:AubrySet}]
Theorem  \ref{thm:LocAubryHorizontal:StrongWeak} implies that the suspended   Ma$\tilde{\mbox n}\acute{\mbox e}$ set $s\wt\NNN(c)$ sets belong to the set $V$ introduced in Theorem \ref{thm:ExistenceNHIC}. Since the set   $s\wt\NNN(c)$  is invariant, belonging to $V$ implies that it must belong to the cylinder $\CCC_{k_n}^{\om_n}$. 

Now it only remains to show that the  set $\wt\II(u,c)$ is a Lipschitz graph over the fast angle $\varphi^f$. Let $(\varphi_i, J_i)$, $i=1,2$ be two points in  $\wt\II(u,c)$. By Theorem \ref{thm:LocAubryVertical:StrongWeak} we know that
\[
 \|J_2-J_1\|\leq  C\rr^{2r/3-1}\|\varphi_2-\varphi_1\|\leq  C\rr^{2r/3-1}\left(\|\varphi_2^s-\varphi_1^s\|+\|\varphi_2^f-\varphi_1^f\|\right).
\]
Now, using the fact that $\wt\II(u,c)$ belongs to $\CCC_{k_n}^{\om_n}$, we know that
\[
 \|\varphi_2^s-\varphi_1^s\|\leq C\rr^{3dr/2-(m+1)/2}\left(\|J_2-J_1\|+\|\varphi_2^f-\varphi_1^f\|\right).
\]
Combining these two estimates, we obtain
\[
 \|J_2-J_1\|\leq  C\rr^{2r/3-1}\|\varphi_2^f-\varphi_1^f\|.
\]
\end{proof}

\subsection{Vertical estimates: proof of Theorem \ref{thm:LocAubryVertical:StrongWeak}}\label{sec:Aubry:vertical}
The proof of Theorem \ref{thm:LocAubryVertical:StrongWeak} follows the same lines of the proof of Theorem 4.1 in \cite{BernardKZ11},  taking into account the different choice of parameters. We define $\LL(\varphi, v,t)$ the Lagrangian associated to $\wh \HH$. Then,
\begin{equation}\label{def:RelationHAndL}
 \begin{split}
 \pa_{vv}  \LL(\varphi, v,t)&=\left(\pa_{JJ}\wh\HH(\varphi, \pa_v \LL(\varphi, v,t),t)\right)\ii\\
 \pa_{\varphi v}  \LL(\varphi, v,t)&=-\pa_{\varphi J}\wh\HH(\varphi,  \pa_v \LL(\varphi, v,t),t)\pa_{vv}  \LL(\varphi, v,t)\\
 \pa_{\varphi\varphi}  \LL(\varphi, v,t)&=\pa_{\varphi\varphi}\wh\HH (\varphi,  \pa_v \LL(\varphi, v,t),t)-\pa_{\varphi J}\wh\HH(\varphi,  \pa_v \LL(\varphi, v,t),t) \pa_{\varphi v}  \LL(\varphi, v,t)
\end{split}
\end{equation}
Therefore, we have the estimates
\begin{equation}\label{def:Relation:HamAndLagrangian}
 \|\pa_{vv}\LL\|_{\CCC^0}\lesssim \rr^{m-1},\,\,\|\pa_{\varphi v}\LL\|_{\CCC^0}\lesssim \rr^{\frac{4}{3}r+m-2},\,\,\|\pa_{\varphi \varphi}\LL\|_{\CCC^0}\lesssim \rr^{\frac{4}{3}r-1}.
\end{equation}
Recall that a function  $u:\TT^n\longrightarrow \RR$ is called $K$-semi-concave if the function 
\[
 x\mapsto u(x)-K\frac{\|x\|^2}{2}
\]
is concave on $\RR^n$, where $u$ is seen as a periodic function on $\RR^n$. We state two lemmas. The first is proven in \cite{BernardKZ11} and is a simple case of Lemma A.10 of \cite{Be}. The second one is proven in \cite{FathiBook}.

\begin{lemma}
 If $u:\TT^n\longrightarrow \RR$  is $K$-semi-concave, then it is $(K\sqrt{n})$-Lipschitz.
\end{lemma}
\begin{lemma}
 Let $u$ and $v$ be $K$-semi-concave functions, and let $\II\subset\TT^n$ be the set of points where the sum $u+v$ is minimal. Then, the functions $u$ and $v$ are differentiable at each point of $\II$, and the differential $x\mapsto du(x)$ is $6K$-Lipschitz on $\II$.
\end{lemma}

The Weak KAM solutions of cohomology $c$ are defined as fixed points of the operator $\TTT_c:\CCC(\TT^n)\longrightarrow\CCC(\TT^n)$ defined by
\[
 \TTT_c(u)(\varphi)=\min_\ga u(\ga(0))+\int_0^T\left(\LL(\ga(t),\dot \ga(t),t)+c\cdot \dot \ga(t)\right) dt
\]
where the minimum is taken on the set of $\CCC^1$ curves $\ga:[0,1]\longrightarrow\TT^n$ satisfying the final condition $\ga(T)=\varphi$.

\begin{proposition}
For each $c\in\RR^n$, each Weak KAM solution $u$ of cohomology $c$ is $C\rr^{2r/3-1}$-semi-concave and $C\rr^{2r/3-1}$-Lipschitz for some constant $C>0$ independent of $\rr>0$.
\end{proposition}

\begin{proof}
 Consider $\Phi:[0,T]\longrightarrow \TT^n$ an optimal curve for $\TTT_c$. That is a curve satisfying
\[
 u(\varphi)=u(\Phi(0))+\int_0^T \left(\LL(\Phi(t),\dot \Phi(t),t)+c\dot \Phi (t)\right)dt.
\]
We lift $\Phi$ to a curve in $\RR^n$ and consider $\Phi_x(t)=\Phi(t)+\frac{t}{T}x$. Then, $\Phi_x(T)=\varphi+x$. Then, we have that
\[
 u(\varphi+x)-u(\varphi)\leq \int_0^T  \left(\LL(\Phi_x(t),\dot \Phi_x(t),t)- \LL(\Phi(t),\dot \Phi(t),t)+\frac{c\dot x}{T}\right)dt.
\]
Proceeding as in \cite{BernardKZ11}, one can bound the integral by 
\[
 u(\varphi+x)-u(\varphi)\leq \left(c+\pa_v\LL(\Phi(T), \dot \Phi(T),T)\right)\cdot x+K\left(\rr^{\frac{4}{3}r-1}T+\rr^{\frac{4}{3}r+m-2}+\frac{\rr^{m-1}}{T}\right)|x|^2
\]
Then, taking $T\in [\rr^{-2r/3+m/2}/2, \rr^{-2r/3+m/2}]$ (this interval contains an integer since $\rr\ll1$), we obtain 
\[
\begin{split}
 u(\varphi+x)-u(\varphi)&\leq \left(c+\pa_v\LL(\Phi(T), \dot \Phi(T),T)\right)\cdot x+K\rr^{\frac{2}{3}r+\frac{m}{2}-1}|x|^2\\
&\leq \left(c+\pa_v\LL(\Phi(T), \dot \Phi(T),T)\right)\cdot x+K\rr^{\frac{2}{3}r}|x|^2.
\end{split}
\]
This estimate gives the semi-concavity constant and  therefore, also the Lipschitz constant. 
\end{proof}

\subsection{Horizontal estimates: proof of Theorem \ref{thm:LocAubryHorizontal:StrongWeak}}\label{sec:Aubry:horizontal}
Hamiltonian \eqref{def:hamNF:slow-fast} is of the form $\wh \HH(\varphi, J)=\HH_1(\varphi^s,J)+\wh \RRR(\varphi^s, J)$ where
\begin{equation}\label{def:AubryVertical:HamH_1}
\HH_1(\varphi^s,J)=\wh \HH_0(J)+\wh \ZZZ(\varphi^s,J). 
\end{equation}
First we perform a change of coordinates which puts the $\varphi^s$ coordinate of the  hyperbolic critical point of $\HH_1$ at the origin and removes some terms in its associated linearization. A similar change has  been done in Section \ref{sec:NHIC} but now we need it to be symplectic.

\begin{lemma}\label{lemma:Aubry:Horizontal:ChangeCoordinates}
There exists a exact symplectic change of coordinates 
$(\varphi,J)=\Psi  (\wt \varphi,\wt J)$
such that,  for each fixed  $\wt J^f\in[\wt a_-,\wt a_+]$,  
\begin{equation}\label{def:Ham:AubryStraightened}
\KK_1=\HH_1\circ\Psi\ii,  
\end{equation}
as a function of $(\wt\varphi^s, \wt J^s)$, has a hyperbolic critical point at $(\wt\varphi_*^s,\wt J_*^s)=(0,J_*^s(J^f))$ and moreover it satisfies
\[
\begin{split}
          D\ii\rr^{m+1}\leq&\pa_{J^sJ^s}\KK_1(0,J_*^s(J^f))\leq D\rr^{m-1}\\
          \rr^{dr}\leq&-\pa_{\varphi^s\varphi^s}\KK_1(0,J_*^s(J^f))\lesssim \rr^{\frac{4}{3}r-1}.
\end{split}
    \]
and \[
          \pa_{J^s\varphi^s}\KK_1(0,J_*^s(J^f))=0,\,\,\,\pa_{J^f\varphi^s}\KK_1(0,J_*^s(J^f))=0.
    \]
\end{lemma}

\begin{proof}
 We perform the change of variables in several steps. First we put the curve of critical points at $(0,J_*^s(J^f))$ and then we modify its linearization. Note that even if these changes are only going to matter in a neighborhood of the curve of critical points we need them to be global. The reason is that we want to see that they are exact symplectic.

\begin{comment}
The first change puts the curve of critical points at $(0,0)$. It could be given by the following symplectic change of coordinates
\[
 \begin{split}
  (\wh\varphi^s,\wh J^s) &=(\varphi^s-\varphi^s_*(J^f), J^s-J^s_*(J^f))\\
  (\wh \varphi^f,\wh J^f)&  =(\varphi^f-\pa_{J^f}\varphi^s_*(J^f)J^s+\pa_{J^f}J_*^s(J^f)\varphi^s,J^f).
   \end{split}
\]
Nevertheless, this change is not periodic in $\varphi^s$. We now modify it so that we can extend it periodically in a way that it becomes exact. We construct the modified change as a composition of two changes. The first of them is 
\end{comment}
The first change of coordinates is defined by
\[
 \begin{split}
  (\wh\varphi^s,\wh J^s) &=(\varphi^s-\varphi^s_*(J^f), J^s)\\
  (\wh \varphi^f,\wh J^f)&  =(\varphi^f-\pa_{J^f}\varphi_*^s(J^f)J^s,J^f)
   \end{split}
\]
It is  exact symplectic and it sets the $\varphi^s$ component of the curve of critical points at 0.
\begin{comment}
. To define the second change we consider a function $a:\TT\rightarrow \RR$ such that for $\varphi^s\in (-1/4,1/4)$, $a(\varphi^s)=\varphi^s$ and is extended periodically.
% and it has zero average $\int_0^{2\pi}a(\varphi^s)d\varphi^s0$. 
Then, we define the change of coordinates
\[
 \begin{split}
  (\wh\varphi^s,\wh J^s) &=(\varphi^s, J^s-J^s_*(J^f)a'(\varphi^s))\\
  (\wh \varphi^f,\wh J^f)&  =(\varphi^f+\pa_{J^f}J_*^s(J^f)a(\varphi^s),J^f)
   \end{split}
\]
It can be checked that this change is exact symplectic. Moreover, for $\varphi^s\in (-1/4,1/4)$ it is just 
\[
 \begin{split}
  (\wh\varphi^s,\wh J^s) &=(\varphi^s, J^s-J^s_*(J^f))\\
  (\wh \varphi^f,\wh J^f)&  =(\varphi^f+\pa_{J^f}J_*^s(J^f)\varphi^s,J^f).
   \end{split}
\]
Therefore, it sets the $J^s$ component of curve of critical points at 0.

The composition of these two changes is certainly exact symplectic, since both are, and sets the hyperbolic critical points at $(\wh\varphi^s,\wh J^s)=(0,0)$. 
\end{comment}
Note that $\HH_1$ is independent of $\varphi^f$. Nevertheless, we need to consider it in all these change of coordinates so that we can apply it later on to the full Hamiltonian $\wh \HH$ keeping the symplectic structure. 

Now we eliminate the crossed derivatives. By construction, the Hamiltonian $ \HH_1$ after this change of coordinates is of the form, 
\[
\begin{split}
 \HH_1'(\wh\varphi^s,\wh J)=&a_0(\wh J^f)+a_{20}(\wh J^f)(\wh\varphi^s)^2+a_{11}(\wh J^f)\varphi^s\left(\wh J^s- J_*^s(J^f)\right)+a_{02}(\wh J^f)\left(\wh J^s-J_*^s(J^f)\right)^2\\
&+\text{higher order terms},
\end{split}
\]
where by hypothesis $a_{20}=\pa_{J^sJ^s}\wh \HH_1(0, J_*^s(J^f))>0$ and $D\ii \rr^{m+1}\leq a_{20}\leq D\rr^{m-1}$, $a_{02}=\pa_{\varphi^s\varphi^s}\wh \HH_1(0, J_*^s(J^f))<0$ and $\rr^{dr}\leq- a_{02}\lesssim\rr^{4r/3-1}$, and $a_{11}=\pa_{\varphi^sJ^s}\wh \HH_1(0, J_*^s(J^f))$ and $|a_{11}|\lesssim \rr^{4r/3+m-2}$.

Note that we already have that $\pa_{\varphi^sJ^f}\HH'(0, J_*^s(J^f))=0$. To set the crossed derivative $\pa_{\varphi^sJ^s}$ to zero, it is enough to consider the change of coordinates
\[
 \begin{split}
  (\wt\varphi^s,\wt J^s)&=\left(\wh\varphi^s,\wh J^s+\frac{a_{11}(\wh J^f)}{2a_{20}(\wh J^f)}\wh\varphi^s\right)\\
  (\wt \varphi^f,\wt J^f)&=\left(\wh\varphi^f-\pa_{J^f}\left[\frac{a_{11}(\wh J^f)}{2a_{20}(\wh J^f)}\right]\frac{\left(\wh\varphi^s\right)^2}{2},\wh J^f\right)
   \end{split}
\]
This change is well defined since $a_{20}>0$ but is not periodic in $\varphi^s$. To make it periodic it is enough to consider a function $b:\TT\rightarrow \RR$ such that for $\varphi^s\in (-1/4,1/4)$, $b(\varphi^s)=(\varphi^s)^2/2$. Then, we consider the periodic extension of the change 
\[
 \begin{split}
  (\wt\varphi^s,\wt J^s)&=\left(\wh\varphi^s,\wh J^s+\frac{a_{11}(\wh J^f)}{2a_{20}(\wh J^f)}\pa_{\wh\varphi^s}b(\wh\varphi^s)\right)\\
  (\wt \varphi^f,\wt J^f)&=\left(\wh\varphi^f-\pa_{J^f}\left[\frac{a_{11}(\wh J^f)}{2a_{20}(\wh J^f)}\right]b(\wh\varphi^s),\wh J^f\right).
   \end{split}
\]
This change of coordinates is exact symplectic.

To obtain now the  estimates of the  second derivatives with respect to $\wt J^s$ of Lemma \eqref{lemma:Aubry:Horizontal:ChangeCoordinates}, it is enough to point out that 
$ \pa_{\wt J^s\wt J^s}\KK_1(0, J_*^s(J^f))=a_{20}(J^f)=\pa_{J^sJ^s}\wh \HH_1(0, J_*^s(J^f))$
which satisfies the desired estimates. Moreover,
\[
 \pa_{\wt\varphi^s\wt\varphi^s}\KK_1(0, J_*^s(J^f))=a_{02}(J^f)-\frac{a_{11}^2(J^f)}{2a_{20}(J^f)}\leq a_{02}(J^f)\leq-\rr^{dr}
\]
and 
\[
\left| \pa_{\wt\varphi^s\wt\varphi^s}\KK_1(0, J_*^s(J^f))\right|\leq |a_{02}(J^f)|+\left|\frac{a_{11}^2(J^f)}{2a_{20}(J^f)}\right|\lesssim \rr^{\frac{4}{3}r-1}.
\]
\end{proof}

If we apply the change obtained in Lemma \ref{lemma:Aubry:Horizontal:ChangeCoordinates} to the Hamiltonian $\wh \HH$, we obtain a new Hamiltonian $\KK$ which is of the form 
\[
 \KK=\KK_1+\KK_2
\]
where $\KK_2=\wh\RRR\circ\Psi$ and therefore satisfies $\|\KK_2\|_{\CCC^2}\lesssim \rr^{qr}$. We study the horizontal localization of the Aubry sets along a fixed single resonance for the Hamiltonian $\KK$. Then we deduce them for \eqref{def:hamNF:slow-fast} and we prove Theorem \ref{thm:LocAubryHorizontal:StrongWeak}. The results for $\KK$ are summarized in the next two propositions. The first one gives localization estimates for the Aubry sets with respect to the slow angle.

\begin{proposition}\label{prop:Aubry:Transition:K:slowangle}
Consider a cohomology class $c=(c^s,c^f)=(J^s_*(c^f),c^f)$ with $c^f\in [\wt a_-,\wt a_+]$. Then,  the slow angle coordinate of any point belonging to  the  Ma$\tilde{\mbox n}\acute{\mbox e}$ set $\wt\NNN(c)$ associated to the Hamiltonian $\KK$ satisfies
\[
 |\wt\varphi^s|\leq \kk\rr^{5dr/2}
\]
for some constant $\kk>0$ independent of $\rr$.
\end{proposition}
%Theorem \ref{thm:LocAubryHorizontal:StrongWeak} gives estimates for the slow action. Nevertheless they are not enough accurate to show that the Aubry sets belong to the Aubry set. Next proposition gives more refined estimates.

\begin{proposition}\label{prop:Aubry:Transition:K:slowaction}
Consider a cohomology class $c=(c^s,c^f)=(J^s_*(c^f),c^f)$ with $c^f\in [\wt a_-,\wt a_+]$. Then,  the slow action coordinate of any point belonging to  the  Ma$\tilde{\mbox n}\acute{\mbox e}$ set $\wt\NNN(c)$ associated to the Hamiltonian $\KK$ satisfies
\[
 |\wt J^s|\leq \kk\rr^{9dr/4}
\]
for some constant $\kk>0$ independent of $\rr$.
\end{proposition}
These two propositions are proved respectively in Sections \ref{sec:Aubry:ProfPropSlowAngle} and \ref{sec:Aubry:ProfPropSlowAction}. Now, using the symplectic invariance of the Ma$\tilde{\mbox n}\acute{\mbox e}$ set,  we are ready to prove Theorem \ref{thm:LocAubryHorizontal:StrongWeak}.
Since the change  of coordinates in Lemma \ref{lemma:Aubry:Horizontal:ChangeCoordinates} is exact, it is enough to apply  Theorem \ref{thm:Bernard}, which was proven  in \cite{Be2}.

\begin{proof}[Proof of Theorem \ref{thm:LocAubryHorizontal:StrongWeak}]

Thus, applying the change of coordinates given by \ref{lemma:Aubry:Horizontal:ChangeCoordinates}, we obtain 
\[
 |\varphi^s-\varphi_*^s(J^f)|=|\wt \varphi^s|\leq \kk\rr^{5dr/2}
\]
and
\[
  |J^s-J_*^s(J^f)|=|\wh J^s|\leq |\wt J^s|+\frac{|a_{11}|}{|a_{20}|}|\wt \varphi^s|\leq 2\kk\rr^{9dr/4}
 \]

\end{proof}

\subsubsection{Localization for the slow angle: proof of Proposition \ref{prop:Aubry:Transition:K:slowangle}}\label{sec:Aubry:ProfPropSlowAngle}
We start by studying the Lagrangian $\LL$.
The following lemma is straightforward.

\begin{lemma}\label{lemma:Lagrangian:Remainder}
The Lagrangian $\LL$, the Legendre dual of $\KK$, can be split as  
\[
 \LL=\LL_1+\LL_2
\]
where $\LL_1$ is the Legendre transform of $\KK_1$ and $\LL_2$ satisfies $\|\LL_2\|_{\CCC^2}\lesssim\rr^{qr}$.
\end{lemma}
We study this Lagrangian  and the function $\al$ associated to the Hamiltonian $\KK$  along the cohomology classes belonging to the resonance. Thus, throughout this section we consider cohomology classes $c=(c^s,c^f)=(J^s_*(c^f),c^f)$, $c^f\in [\wt a_-,\wt a_+]$. 
We define $\om=\pa_{J}\KK_1(0,c)$. This implies that $c=\pa_v\LL_1(0,\om)$.

\begin{lemma}\label{lemma:Horizontal:Alpha}
Consider a cohomology class $c=(J^s_*(c^f),c^f)$ with $c^f\in [\wt a_-,\wt a_+]$. The $\al$ function associated to $\KK$ satisfies
\[
 \KK_1(0,c)-C\rr^{qr}\leq \al(c)\leq \KK_1(0,c)+C\rr^{qr}.
\]
for some constant $C>0$ independent of $\rr$.
\end{lemma}
\begin{proof}
It follows the same lines as the proof of Lemma 4.4 in \cite{BernardKZ11}. For the upper bound, it is enough to take into account that 
\[
 \al(c)\leq \max_{(\varphi,t)}\KK(\varphi,c,t)\leq \max_{\varphi^s}K_1(\varphi^s,c)+\max_{(\varphi,t)}\KK_2(\varphi,c,t).
\]
For the lower bound, we consider the Haar measure $\mu$ of the torus $\TT\times\{0\}\times\{\om\}\times\TT$, which is not necessarily invariant but is closed. Then
\[
%\begin{split}
\al(c)\geq c\cdot \om-\int \LL\,d\mu
\geq c\cdot\om -\LL_1(0,\om)-C\rr^{qr}
\geq \KK_1(0,c)-C\rr^{qr}.
%\end{split}
\]
\end{proof}
We use these lemmas to obtain precise upper and lower bounds for the Lagrangian. We use these estimates to obtain good upper and lower bounds for the action along curves belonging to the Aubry sets.

\begin{lemma}\label{lemma:Vertical:BoundsLagrangianAndAlpha}
Consider a cohomology class $c=(J^s_*(c^f),c^f)$ with $c^f\in [\wt a_-,\wt a_+]$. Then,
\[
\begin{split}
\LL(\varphi,v,t)-cv+\al(c)\leq&(\pa_v\LL_1(\varphi^s,\om)-c)(v -\pa_{J}\KK_1(\varphi^s,c))\\
&-\left(\KK_1(\varphi^s,c)-\KK_1(0,c)\right)\\
&+\frac{D\rr^{m-1}}{2}\left\|\pa_v\LL_1(\varphi^s,\om)-\pa_v\LL_1(0,\om)\right\|^2\\
&+\rr^{qr}+\frac{D\rr^{m-1}}{2}\|v-\om\|^2\\
\LL(\varphi,v,t)-cv+\al(c)\geq&(\pa_v\LL_1(\varphi^s,\om)-c)(v -\pa_{J}\KK_1(\varphi^s,c))\\
&-\left(\KK_1(\varphi^s,c)-\KK_1(0,c)\right)\\
&+\frac{\rr^{m+1}}{2D}\left\|\pa_v\LL_1(\varphi^s,\om)-\pa_v\LL_1(0,\om)\right\|^2\\
&-\rr^{qr}+\frac{\rr^{m+1}}{2D}\|v-\om\|^2\\
\end{split}
\] 
\end{lemma}

To prove this lemma, we first prove the following  lemma. 
%We assume that the cylinder is straightened for $\HH_1$. This implies that 
%\[
% \HH_1(\varphi^s,J)=f_1 (J^s)^2+f_2J^s\varphi^s+f_3(\varphi^s)^
%\]

\begin{lemma}\label{lemma:L1:properties}
Consider a cohomology class $c=(J^s_*(c^f),c^f)$ with $c^f\in [\wt a_-,\wt a_+]$. Then, the Lagrangian $\LL_1$ satisfies
\[
\begin{split}
 \LL_1(\varphi^s,\om)-\LL_1(0,\om)\leq&(\pa_v\LL_1(\varphi^s,\om)-c)(\om -\pa_{J}\KK_1(\varphi^s,c))\\
&-\left(\KK_1(\varphi^s,c)-\KK_1(0,c)\right)\\
&+\frac{D\rr^{m-1}}{2}\left\|\pa_v\LL_1(\varphi^s,\om)-\pa_v\LL_1(0,\om)\right\|^2\\
 \LL_1(\varphi^s,\om)-\LL_1(0,\om)\geq&(\pa_v\LL_1(\varphi^s,\om)-c)(\om -\pa_{J}\KK_1(\varphi^s,c))\\
&-\left(\KK_1(\varphi^s,c)-\KK_1(0,c)\right)\\
&+\frac{D\ii\rr^{m+1}}{2}\left\|\pa_v\LL_1(\varphi^s,\om)-\pa_v\LL_1(0,\om)\right\|^2
\end{split}
\]
\end{lemma}
\begin{proof}
 We know that 
\[
 \LL_1(0,\om)=c\om-\KK_1(0,c).
\]
For the other term, we write $ \LL_1(\varphi^s,\om)=\pa_v\LL_1(\varphi^s,\om)\om-\KK_1(\varphi^s,\pa_v\LL_1(\varphi^s,\om))$ and applying Taylor to $\KK_1$ with respect to the actions
\[
\begin{split}
 \LL_1(\varphi^s,\om)\leq &\pa_v\LL_1(\varphi^s,\om)\om-\KK_1(\varphi^s,c)-\pa_J\KK_1(\varphi^s,c)\left(\pa_v\LL_1(\varphi^s,\om)-\pa_v\LL_1(0,\om)\right)\\
&+\frac{D\rr^{m-1}}{2}\left\|\pa_v\LL_1(\varphi^s,\om)-\pa_v\LL_1(0,\om)\right\|^2
\end{split}
\]
Analogously, one can obtain the lower bounds.
\end{proof}

Using Lemma \ref{lemma:L1:properties}, the proof of  Lemma \ref{lemma:Vertical:BoundsLagrangianAndAlpha} follows the same lines of the analogous lemma in \cite{BernardKZ11}
\begin{proof}[Proof of Lemma \ref{lemma:Vertical:BoundsLagrangianAndAlpha}]
To compute  the upper bound, we apply Taylor at $v=\om$ and we use  Lemma \ref{lemma:Lagrangian:Remainder}.
\[
\begin{split}
 \LL(\varphi,v,t)&\leq \LL_1(\varphi^s,v)+\OO\left(\rr^{qr}\right)\\
&\leq\LL_1(\varphi^s,\om)+\pa_v\LL_1(\varphi^s,\om)(v-\om)+\frac{D\rr^{m-1}}{2}\|v-\om\|^2+\OO\left(\rr^{qr}\right)\\
\end{split}
\]
Then, it is enough to use Lemma \ref{lemma:L1:properties}. One can obtain the lower bound analogously.
\end{proof}

%Now, we fix some $r_1\gtrsim \rr^{qr/2-8m/5}$ small and we  obtain estimates for the Lagrangian provided $(\varphi,t)\in \TT\times B(\varphi_*^s,r_1)\times\TT$.
Now, from  Lemma \ref{lemma:Vertical:BoundsLagrangianAndAlpha}, we derive estimates doing Taylor expansion around $\varphi^s=0$.

\begin{lemma}\label{lemma:Lagrangian:LocalBounds}
Consider a cohomology class $c=(J^s_*(c^f),c^f)$ with $c^f\in [\wt a_-,\wt a_+]$. Then, the Lagrangian $\LL$ satisfies
\[
\begin{split}
\LL(\varphi,v,t)-cv+\al(c)\leq& \rr^{\frac{4}{3}r-1}\left|\varphi^s\right|^2+\frac{D\rr^{m-1}}{2}\|v-\om\|^2\\
&+\rr^{\frac{4}{3}r-1}\left|\varphi^s\right|^2\|v-\om\|+\rr^{\frac{8}{3}r-2}\left|\varphi^s\right|^4+\rr^{qr}\\
\LL(\varphi,v,t)-cv+\al(c)\geq& \rr^{dr}\left|\varphi^s\right|^2+\frac{D\rr^{m+1}}{2}\|v-\om\|^2\\
&-\rr^{\frac{4}{3}r-1}\left|\varphi^s\right|^2\|v-\om\|-\rr^{\frac{8}{3}r-2}\left|\varphi^s\right|^4-\rr^{qr}
\end{split}
\]
\end{lemma}
\begin{proof}
We bound each term in Lemma \ref{lemma:Vertical:BoundsLagrangianAndAlpha}. For the first one, we take into account that by \eqref{def:Relation:HamAndLagrangian} and Lemma \ref{lemma:Aubry:Horizontal:ChangeCoordinates}, we have that $ \pa_{\varphi^sv^s}\LL(0,\om)=0$ and $\pa_{\varphi^sv^f}\LL(0,\om)=0$. Now, applying  Taylor at $\varphi^s=0$ we have
\[
\begin{split}
\left|\pa_v\LL_1(\varphi^s,\om)-c\right|&\lesssim \rr^{\frac{4}{3}r-1}|\varphi^s|^2\\
\left|v -\pa_{J}\KK_1(\varphi^s,c)\right|&\lesssim\|v-\om\|+ \rr^{\frac{4}{3}r-1}|\varphi^s|^2
\end{split}
\]
Therefore,we obtain
\[
\left|(\pa_v\LL_1(\varphi^s,\om)-c)(v -\pa_{J}\KK_1(\varphi^s,c))\right|\lesssim \rr^{\frac{4}{3}r-1}|\varphi^s|^2\|v-\om\|+ \rr^{\frac{8}{3}r-2}|\varphi^s|^4\\
\]
For the second term, we use Lemma \ref{lemma:Aubry:Horizontal:ChangeCoordinates} to obtain $\KK_1(\varphi^s,c)-\KK_1(0,c)\lesssim \rr^{4r/3-1}r_1^2$ and $\KK_1(\varphi^s,c)-\KK_1(0,c)\gtrsim \rr^{dr}r_1^2$. The rest of the terms are straightforward.
\end{proof}

This lemma gives very precise positive upper bounds  for the Lagrangian but only  close to $\varphi^s=0$. To obtain the horizontal estimates we also need global positive lower bounds. They are given in the next lemma.

\begin{lemma}\label{lemma:Horizontal:GlobalLowerBounds}
Consider a cohomology class $c=(J^s_*(c^f),c^f)$ with $c^f\in [\wt a_-,\wt a_+]$. Then, the Lagrangian $\LL$ satisfies
 \[
\LL(\varphi,v,t)-cv+\al(c)\geq  \rr^{dr}|\varphi^s|^2-C\rr^{qr}.
\]
\end{lemma}
\begin{proof}
By definition 
\[
%\begin{split}
\LL_1(\varphi^s,v)=\max_J\left\{Jv-\KK_1(\varphi^s,J)\right\}
\geq cv-\KK_1(\varphi^s,c).
%\end{split}
\]
So, we have that $ \LL(\varphi,v,t)-cv\geq -\KK_1(\varphi^s,c)-C\rr^{qr}$. Now it only remains to use the lower bounds of $\al(c)$ obtained in Lemma \ref{lemma:Horizontal:Alpha}, which lead to 
 \[
\LL(\varphi,v,t)-cv+\al(c)\geq -\left(\KK_1(\varphi^s,c)-\KK_1(0,c)\right)-C\rr^{qr}.
\]
To complete the proof it suffices to use the properties of $\KK_1$ obtained in Lemma \ref{lemma:Aubry:Horizontal:ChangeCoordinates}.
\end{proof}

Now, we obtain upper bounds for the action, 
\begin{lemma}\label{lemma:Horitzontal:ActionUpperBounds}
 Let $u(\varphi,t)$ be a weak KAM solution at cohomology $c\in\Ga$. Given $r_1\gtrsim \rr^{qr/2-2r/3+1/2}$ and two points $(\varphi_1,t_1)$ and $(\varphi_2,t_2)\in \TT\times B(0,r_1)\times\TT$, we have
\[
 u(\varphi_2,t_2)-u(\varphi_1,t_1)\lesssim \rr^{\frac{2}{3}r}r_1
\]
\end{lemma}

\begin{proof}
 We proceed as in \cite{BernardKZ11} by taking a curve 
\[
 \varphi(t)=\varphi_1+(t-\wt t_1)\frac{\wt\varphi_2-\wt\varphi_1+(T+\wt t_2-\wt t_1)\om}{T+\wt t_2-\wt t_1}
\]
where $T\in\NN$ is a parameter to be fixed later, and $\wt t_i\in [0,1)$ and $\wt \varphi_i\in [0,1)^2$ are representatives of the angular variables $t_i,\varphi_i$. Then,
\[
 u(\varphi_2,t_2)-u(\varphi_1,t_1)\lesssim\int_{\wt t_1}^{\wt t_2+T}\LL(\varphi(t),\dot \varphi (t),t)-c\cdot\dot\varphi(t)+\al(c)\,dt
\]
By Lemma \ref{lemma:Lagrangian:LocalBounds}, 
\[
u(\varphi_2,t_2)-u(\varphi_1,t_1)\lesssim \int_{\wt t_1}^{\wt t_2+T} \left(D\rr^{m-1}\|\dot\varphi(t)-\om\|^2+\rr^{\frac{4}{3}r-1}r_1^2\|\dot\varphi(t)-\om\|+\rr^{\frac{4}{3}r-1}r_1^2\right)\,dt.
\]
%Note that  now
%\[
% \varphi^s(t)=\varphi_1+(t-\wt t_1)\frac{\wt\varphi_2-\wt\varphi_1}{T+\wt t_2-\wt t_1}
%\]
%since $\om^s=0$. 
Then, 
\[
\begin{split}
 u(\varphi_2,t_2)-u(\varphi_1,t_1)\lesssim &\int_{\wt t_1}^{\wt t_2+T} \frac{D\rr^{m-1}}{(T+\wt t_2-\wt t_1)^2}+\frac{\rr^{\frac{4}{3}r-1}r_1^2}{T+\wt t_2-\wt t_1}+\rr^{\frac{4}{3}r-1}r_1^2\,dt\\
\lesssim & \frac{D\rr^{m-1}}{T+\wt t_2-\wt t_1}+\rr^{\frac{4}{3}r-1}r_1^2+\rr^{\frac{4}{3}r-1}r_1^2(T+\wt t_2-\wt t_1).
\end{split}
\]
Now, taking
\[
\frac{1}{r_1\rr^{\frac{2}{3}r-\frac{m}{2}}}\leq T+\wt t_2-\wt t_1\leq 2\frac{1}{r_1\rr^{\frac{2}{3}r-\frac{m}{2}}}
\]
we obtain
\[
 u(\varphi_2,t_2)-u(\varphi_1,t_1)\lesssim \rr^{\frac{2}{3}r+\frac{m}{2}-1}r_1\leq \rr^{\frac{2}{3}r}r_1
\]

\end{proof}

Now we look for lower bounds of the action so that we can reach a contradiction. We adapt the argument in \cite{BernardKZ11}. We consider the invariant set $\wt\II(v,c)\subset\TT^2\times\RR^2\times\TT$ associated to a suspended weak KAM solution $v$ at cohomology $c$, and its projection  $\II(v,c)\subset\TT^2\times\TT$. The set $\wt\II(v,c)$ is the suspension of the invariant set $\wt\II(u,c)$ associated to $u(\varphi)=v(\varphi,0)$. We consider also a curve $\varphi(t):\RR\longrightarrow \TT^2$ calibrated by $u$. Namely, 
\[
 \int_{t_1}^{t_2}\LL(\varphi(t),\dot \varphi(t),t)-c\cdot\dot\varphi(t)+\al(c)\,dt= u(\varphi(t_2),t_2)-u(\varphi(t_1),t_1)
\]
for each interval $(t_1,t_2)$. Since $u$ is bounded, this integral is bounded independently of the time interval. 
Now, if we take $r_1=\sqrt{2C\rr^{(q-d)r}}$, by Lemma \ref{lemma:Horizontal:Alpha}, we know that 
 \[
\LL(\varphi,v,t)-cv+\al(c)\geq  \frac{1}{2}\rr^{dr}r_1^2
\]
if $\varphi^s$ does not belong to the ball $B(0,r_1)$. Then, the set of times for which $\varphi^s$ does not belong to this ball must have finite measure. Call $[t_1,t_2]$ to one of the connected components of this open set of times. Then, $\varphi^s(t_1)$ ans $\varphi^s(t_2)$ both belong to  $B(0,r_1)$. Therefore, by Lemma \ref{lemma:Horitzontal:ActionUpperBounds}, we have that
\begin{equation}\label{eq:Horitzontal:UpperBoundOutsideBall}
  \int_{t_1}^{t_2}\LL(\varphi(t),\dot \varphi(t),t)-c\dot \varphi(t)+\al(c)\,dt=u(\varphi(t_2),t_2)-u(\varphi(t_1),t_1)\leq \rr^{\frac{2}{3}r}r_1
\end{equation}
Now, we assume that $\varphi^s(t)$ is not contained in the ball $B(0,r_0)$ with $r_0=\rr^{5dr/2}$ and we reach a contradiction. Note that we take $q=18$ and therefore $r_1<r_0$. Assume that the points out of this ball happen in the time interval $[t_1,t_2]$ (if not change the time connected component). Then, call $t_4$ the first time such that $|\varphi^s(t_4)|=r_0$ and $t_3$ the time the largest time in $[t_1,t_4]$ such that $|\varphi^s(t_3)|=r_0/2$. Then, 
\[
  \int_{t_1}^{t_2}\LL(\varphi(t),\dot \varphi(t),t)-c\dot \varphi(t)+\al(c)\,dt\geq \int_{t_3}^{t_4}\LL(\varphi(t),\dot \varphi(t),t)-c\dot \varphi(t)+\al(c)\,dt.
\]
since the function in the integral is always positive.

Now we bound this integral using Lemma \ref{lemma:Horizontal:GlobalLowerBounds} and using that in this time interval $|\varphi^s(t)|\geq r_0/2$.
\[
%\begin{split}
 \int_{t_3}^{t_4}\LL(\varphi(t),\dot \varphi(t),t)-c\dot \varphi(t)+\al(c)\,dt
\geq \int_{t_3}^{t_4}\frac{1}{2}\rr^{dr}(\varphi^s(t))^2\,dt
\geq\frac{1}{8}\rr^{6dr}(t_4-t_3).
%\end{split}
\]
So, joining this estimate with \eqref{eq:Horitzontal:UpperBoundOutsideBall} we have that $\rr^{6dr}(t_4-t_3)/8\leq \rr^{\frac{2}{3}r}r_1$ which implies $ t_4-t_3\leq 8\sqrt{2C}\rr^{\frac{2}{3}r+\frac{q-13d}{2}r}$. Now, we know that 
\[
 \frac{r_0}{2}=\varphi^s(t_4)-\varphi^s(t_3)= \int_{t_3}^{t_4}\dot\varphi^s(t)dt
\]
Using the diferential equation we have that $\dot\varphi^s=\OO(\rr^{m-1} J^s+\rr^{\frac{4}{3}r-1}(\varphi^s)^2)$. By the vertical estimates, for the calibrated curve we know that $|J^s|\leq \rr^{\frac{2}{3}r}$ (since $c^s=0$) and we are assuming that $|\varphi^s|\leq r_1$. So $|\dot\varphi^s(t)|\lesssim  \rr^{m+\frac{2}{3}r-1}\leq \rr^{\frac{2}{3}r}$. This implies that 
\[
 \frac{r_0}{2}=\left|\int_{t_3}^{t_4}\dot\varphi^s(t)dt\right|\leq \rr^{\frac{2}{3}r}(t_4-t_3)
\]
Thus, we obtain that 
\[
 \frac{1}{2}\rr^{5dr/2}\leq  8\sqrt{2C}\rr^{\frac{4}{3}r+\frac{q-13d}{2}r}
\]
Since we take $q=18$ we get a contradiction. This implies that all the curve must be inside the ball of radius $r_0=\rr^{5dr/2}$.
%Consider $r_1\gtrsim \rr^{qr/2-8m/5}$, then by Lemma \ref{lemma:Lagrangian:LocalBounds}, we have that
%\[
%\begin{split}
% \LL(\varphi,v,t)-cv+\al(c)&\geq \frac{D\rr^m}{2}\left(\|v-w\|-\frac{\rr^{3m/5}}{D}r_1^2\right)^2-\frac{\rr^{11m/5}}{D^2}r_1^4+\rr^{\dr}r_1\\
%&\geq \frac{D\rr^m}{2}\left(\|v-w\|-\frac{\rr^{8m/5}}{D}r_1^2\right)^2+\frac{1}{2}\rr^{\dr}r_1\\
%&\geq\frac{1}{2}\rr^{\dr}r_1
%\end{split}
%\]

\subsubsection{Localization for the slow action: proof of Proposition \ref{prop:Aubry:Transition:K:slowaction}}\label{sec:Aubry:ProfPropSlowAction}
To ensure that the Aubry sets belong to the cylinder obtained in Theorem \ref{thm:ExistenceNHIC} we need to improve estimates for the slow action $\wt J^s$ obtained in Proposition \ref{thm:LocAubryVertical:StrongWeak}. As we have done in the previous section we omit the tildes over the variables to simplify notation. We consider the invariant set $\wt\II(v,c)\subset \TT^2\times\RR^2\times\TT$ associated to a (suspended) weak KAM solution $v$ of the Hamiltonian \eqref{def:Ham:AubryStraightened} at a cohomology class $c$, and its projection $\II(v,c)\subset \TT^3$. 
Recall that the set $\wt\II(v,c)$  is the suspension of the invariant set $\wt\II(u,c)$ associated to the restriction $u(\varphi)=v(\varphi,0)$. 
Consider  $\varphi(t):\RR\rightarrow\TT^2$, a calibrated curve by $u$. That is, 
\[
 \int_{t_1}^{t_2}\LL(\varphi(t),\dot \varphi(t),t)-c\cdot\dot\varphi(t)+\al(c)\,dt= u(\varphi(t_2),t_2)-u(\varphi(t_1),t_1)
\]
for each interval $(t_1,t_2)$. In Section \ref{sec:Aubry:ProfPropSlowAngle}, we have seen that these calibrated curves must satisfy  $|\varphi^s|\lesssim\rr^{5dr/2}$.

First, prove that for a calibrated curve $(\varphi(t),J(t))$, the associated action  $J^s(t)$ satisfies $|J^s(t)|\leq \rr^{9dr/4}/2$ for all time except for a bounded set of times. By Lemma \ref{lemma:Lagrangian:LocalBounds},
\[
\begin{split}
\LL(\varphi,v,t)-cv+\al(c)\geq& \rr^{dr}\left|\varphi^s\right|^2+\frac{D\rr^{m+1}}{2}\|v-\om\|^2\\
&-\rr^{\frac{4}{3}r-1}\left|\varphi^s\right|^2\|v-\om\|-\rr^{\frac{8}{3}r-2}\left|\varphi^s\right|^4-\rr^{qr}\\
\geq& \rr^{dr}\left|\varphi^s\right|^2+\frac{D\rr^{m+1}}{2}\left(\|v-\om\|-\frac{\rr^{\frac{4}{3}r-m-2}}{D}|\varphi^s|^2\right)^2\\
&-\frac{\rr^{\frac{8}{3}r-m-3}}{2D}\left|\varphi^s\right|^4-\rr^{\frac{8}{3}r-2}\left|\varphi^s\right|^4-\rr^{qr}
\end{split}
\]
Now,  since $|\varphi^s(t)|\lesssim\rr^{5dr/2}$ for all $t\in\RR$, we know that 
\[
\LL(\varphi(t),\dot \varphi(t),t)-cv+\al(c)\geq \frac{D\rr^{m+1}}{2}\left(\|\dot \varphi(t)-\om\|-\frac{\rr^{\frac{4}{3}r-m-2}}{D}|\varphi^s(t)|^2\right)^2-\rr^{qr}
\]
Consider the corresponding curve in the Hamiltonian setting $(\varphi(t), J(t))$ and assume that $|J^s(t)|\geq \rr^{9dr/4}/2$ for some subset of time $I\subset\RR$. Then, since by the Legendre transform $ |\dot \varphi^s(t)|\gtrsim |\rr^{m+1} J^s(t)|\gtrsim \rr^{9dr/4+m+1}$. Then, 
$\LL(\varphi,v,t)-cv+\al(c)\gtrsim \rr^{9dr/2+3m+3}$. Therefore, the time subset $I\subset\RR$ in which $|J^s(t)|\geq \rr^{9dr/4}/2$ must be bounded. Otherwise, the action along the curve would be infinite.

Assume that there exist times $t_1<t_2$ such that $J^s(t_1)=\rr^{9dr/4}/2$ and $J^s(t_2)=\rr^{9dr/4}$. 
 If they would  not exist, one would have that for all time $|J^s(t)|\leq \rr^{9dr/4}$ and the proof of Proposition \ref{prop:Aubry:Transition:K:slowaction} would be finished. Since the curve is calibrated, we know that, for all $t\in [t_1,t_2]$, $|\varphi^s(t)|\leq \rr^{5dr/2}$. 

We obtain upper bounds for $t_2-t_1$. We have that 
\[
\frac{ \rr^{9dr/4}}{2}=J^s(t_2)-J^s(t_1)=\int_{t_1}^{t_2}\dot J^s(t)dt
\]
Using the equation associated to Hamiltonian \eqref{def:Ham:AubryStraightened}, we have that $| \dot J(t)|\lesssim \rr^{5dr/2+4r/3-1}$. 
Therefore, $ \rr^{9dr/4}/2\lesssim (t_2-t_1) \rr^{5dr/2+4r/3-1}$.
This implies that $t_2-t_1\gtrsim \rr^{-dr/4-4r/3+1}$.
%Therefore we deduce 
%\[
%\begin{split}
%  u(\varphi(t_2),t_2)-u(\varphi(t_1),t_1)=&\int_{t_1}^{t_2}\LL(\varphi(t),\dot \varphi(t),t)-c\cdot\dot\varphi(t)+\al(c)\,dt\\
%\geq &\int_{t_1}^{t_2}\rr^{5dr+3m-2}\,dt\\
%\geq &\rr^{5dr+7m/5-3}.
%\end{split}
%\]
Moreover, we know that $|\varphi^s(t_2)-\varphi^s(t_1)|\leq  2\rr^{5dr/2}$ and 
\[
 \varphi^s(t_2)-\varphi^s(t_1)=\int_{t_1}^{t_2}\dot \varphi^s(t)dt
\]
Using the equation associated to Hamiltonian \eqref{def:Ham:AubryStraightened}, we have that $| \dot \varphi^s|\geq D\rr^{m+1}J^s-C\rr^{4r/3-1}(\varphi^s)^2$
for some constant $C>0$. This implies that $\left| \dot \varphi^s\right|\geq D\rr^{9dr/4+m+1}/4$. Therefore, 
\[
 2\rr^{5dr/2}\geq \int_{t_1}^{t_2}\left|\dot \varphi^s(t)\right|dt\geq \frac{D}{4}\rr^{9dr/4+m+1}(t_2-t_1),\] 
which implies that $
 t_2-t_1\lesssim \rr^{dr/4-m-1}
$.
Thus, we have a contradiction and we can deduce that $|J^s(t)|\leq \rr^{9dr/4}$ for all $t\in\RR$.
%Now we look for lower bounds for $  u(\varphi(t_2),t_2)-u(\varphi(t_1),t_1)$. 
%\subsection{Localization of the Aubry sets in the core of  double resonances}\label{sec:LocAubrySets:StrongDR}

\section{Shadowing: proof of Key Theorem \ref{keythm:equivalence}}
We split the proof of Key Theorem \ref{keythm:equivalence} into 
three steps:
\begin{itemize}
\item single resonance; 
\item double resonance, high energy;
\item double resonance, low energy.
\end{itemize}

{\bf Single resonance:} Consider $n\in\NN$, $k_n\in \FF_n$ and a Dirichlet resonance semgent between 
two consecutive $\theta$-strong resonances, which belongs to  $\Tr_n=\Tr_n (k_n,k',k'')$. Then, by Key Theorem \ref{keythm: Transition:AubrySet}, 
for each $c\in \Tr_n$, the Aubry set $\wt {\AAA}_{N^{\dr}}(c)$ satisfies
\begin{itemize} 
\item {\it vertical localization} or localization in the action component, 
i.e. the projection onto the action space belongs to 
the $C\rho_n^{2r/3+m}+C\rho_n^{2r-2m}$--neighborhood 
of $c$. 

\item {\it horizontal localization} or localized in the slow angle, 
i.e. the projection onto the slow angle $\psi^s$ belongs to 
an interval of length  $2C\rho^{2r-2m}$.

\item {\it cylinder localization}, i.e.  $\wt {\AAA}_{N^{\dr}}(c)$  is 
contained in one (resp. two) of the NHICs $\CCC^{\om_n,i}_{k_n}$
(resp. $\CCC^{\om_n,i}_{k_n}$ and $\CCC^{\om_n,i+1}_{k_n}$), 

\item {\it the $2$-torus graph property} or $\wt {\AAA}_{N^{\dr}}(c)$
is a Lipschitz graph over the $(\psi^f,t)$ component. 
\end{itemize}

It turns out this is enough so that the scheme from \cite{BernardKZ11,KaloshinZ12} of proving $c$-equivalence 
applies to all $c$'s in $\Tr_n$. Note that in the proof
there one studies the discrete Aubry set $\wt\AAA_0(c)=\wt\AAA(c)\cap\{t=0\}$ and the last item
is replaced by the projection onto the fast angle $\psi^f$. 

\vskip 0.1in 

{\bf The core of a $\theta$-strong double resonance:}
Consider the core of the double resonance $\Cor_n(k_n,k')$
of a $\theta$-strong double resonance $(k_n,k')$, 
defined in (\ref{def:CoreDR-m}). Its projection onto the action 
space is given by the $C\rho^m_n$-neighborhood of $I_0$. 
By Key Theorem \ref{keythm:CoreDR:NormalForm} at 
each $\Cor_n(k_n,k')$ the Hamiltonian $N^{\dr}$,
given by (\ref{def:HamAfterSRDR}), has 
a normal form  given by a Hamiltonian 
$$
\HH(\psi,J)=\HH_0(J)+\ZZZ(\psi,J)=:N^{\dr}\circ \Phi_n
$$ 
of two degrees of freedom  \eqref{def:Ham2dofNonMech}
satisfying conditions \eqref{eq:DR-Hamiltonians}. Notice that 
at a $\theta$-strong double resonance we start with 
a linear change of coordinates given by 
(\ref{def:ChangeToFastSlow:StrongWeak:0}--
\ref{def:ChangeToFastSlow:StrongWeak2:0}). 
It can happen that $K=\det \tilde A> 1$. In this case
we consider covering 
\[
\xi:\T^2 \to \T^2 \quad, \quad \varphi \mapsto \xi(\varphi) =K\varphi.
\]
This covering lifts to a symplectic covering
\[
\qquad \Xi:T^*\T^2 \to T^* \T^2\quad , \quad 
(\varphi,J)=(\varphi, I_1/K, I_2/K)
 \to \Xi(\varphi,I)=
(\xi(\varphi),I_1/K, I_2/K). 
\]
Consider the lifted Hamiltonian $\tilde N=N\circ \Xi$. 
It is well known, see \cite{Fathi98,ContrerasP02,Be} that 
\[
\tilde \AAA_{\tilde N}(\xi^* c)=\Xi^{-1}(\tilde \AAA_N(c)).
\]
On the other hand, the inclusion 
\[
\tilde \NNN_{\tilde N}(\xi^* c)\supset \Xi^{-1}(\tilde \NNN_N(c))=
\Xi^{-1}(\tilde \AAA_{\tilde N}(c))
\]
is not an equality. More exactly, the set $\tilde \AAA_{\tilde N}(c)$
is the union of $K^2$ disjoint copies of $\tilde \AAA_N(c)$ 
obtained one from another by a parallel translation, while 
$\tilde \NNN_{\tilde N}(\xi^* c)$ always contains heteroclinic 
orbits connecting these copies.

%%&%

By Theorem \ref{thm:Bernard} for any $c$ 
$$
\AAA_{\HH}(c)=
\Phi_n \,\AAA_{N^{\dr}}(c),
$$
where in notations of Key Theorem \ref{keythm:CoreDR:NormalForm} 
we have $N^{\dr}\circ \Phi_n=\HH$.  Therefore, we can analyze 
the Aubry sets of $\HH$ and relate them to those of $N^{\dr}$. 

It is known that an Aubry set $\wt\AAA_\HH(c)$ of an autonomous Tonelli 
Hamiltonian  $\HH$ belong to the energy surface $\{\HH=\al_\HH(c)\}$, where 
$\al_\HH(c)$ is the alpha function, defined in (\ref{def:alpha_function}).
This function takes its minimum at $0$ and corresponds to $I_0$. 
In Key Theorem \ref{keythm:DR:HighEnergy} depending on properties of 
$\HH$ we determine  a small energy $E_0>0$ and separately study 
\begin{itemize}
\item (high energy) $\al_\HH(c)>\al_\HH(0)+E_0$ 
\item (low energy) $\al_\HH(0)+E_0>\al_\HH(c)>\al_\HH(0)$.
\end{itemize}
By analogy with \cite{KaloshinZ12}, we first, we prove 
$c$-equivalence for high energy and then for low energy.  

It turns out the single resonance and the double resonance 
high energy are similar and we can use \cite{KaloshinZ12}. 
To carry out the proof of $c$-equivalence we need four 
properties of the Aubry sets $\wt\AAA_\HH(c)$,
analogous to the single resonance. Namely, 
\begin{itemize}
\item{\it vertical localization} or the natural projection of 
$\wt\AAA_\HH(c)$ into the action component belongs 
to a small neighborhood of $c$.
\item {\it horizontal localization} the Aubry set $\wt\AAA_\HH(c)$ 
is a periodic orbit or a union of two periodic orbits, therefore,
after a proper smooth change of coordinates it 
belongs to a small neighborhood of $\th^s=\th^s_*$. 

\item {\it cylinder localization} the Aubry set $\wt\AAA_\HH(c)$ is 
contained in one (resp. two) of NHICs  $\MM_h^{E_j,E_{j+1}}$ (resp. 
$\MM_h^{E_j,E_{j+1}}\cup \MM_h^{E_{j+1},E_{j+2}}$) for $h$
given by the $k_n$ and $k_n'$ (see 
(\ref{eq:DR-cylinder}) and (\ref{def:DR:UnionNHICs}) for 
definitions).

\item {\it the $2$-torus graph property} or 
$\wt\AAA_\HH(c)$ is a Lipschitz graph over a certain $2$-torus.
\end{itemize}

\subsection{Equivalent forcing class along single  resonances}
\label{c-equi-SR}

Here we follow the scheme from \cite{BernardKZ11} 
and \cite{KaloshinZ12}, sect 6.1.
In the single resonance $\Tr_n=\Tr_n (k_n,k',k'')$ between
the adjacent cores of double resonances. Let 
$\Gamma^{(n)} \subset \LL\FF_{N^{\dr}}( \Tr_n)$ be 
a monotone connected segment connecting the end 
points $\LL\FF(\partial \Tr_n)$. It is monotone 
in the sense that if crosses $\LL\FF_{N^{\dr}}( \om)$ for each 
$\om \in \Tr_n$ only once. For example, $\{(J^s(J^f),J^f)\}$, 
given by (\ref{eq:resonance-parametrization}). One can show 
that $\LL\FF_{N^{\dr}}(\partial \Tr_n)$ contains an open set. 
  
\begin{itemize}
\item {\it (passage values)}
We say that $c \in \Gamma_1 \subset \Gm^{(n)}$, 
if $\wt {\mathcal N}_{N^{\dr}}(c)$ is contained in only one NHIC, 
has only one static class,  and the projection onto
$\theta^f$ component is not the whole circle. Due to a result 
of Mather \cite{Mather93} (and in the forcing setting, \cite{Be}),
$c$ is in the interior of its forcing class.

\item {\it (bifurcation values)}
We say that $c \in \Gamma_2\subset \Gm^{(n)}$
has exactly two static classes and each one is contained in one NHIC. 
In this case we would like to jump from one cylinder to another.

\item {\it (invariant curve values)}
We say that $c\in \Gamma_3\subset \Gm^{(n)}$, if 
$\wt {\AAA}_{N^{\dr}}(c)$ is contained in a single cylinder, 
but the projection onto $(\theta^f,t)$ component is the whole 
two torus\footnote{in \cite{BernardKZ11} we study the discrete 
Aubry set $\wt\AAA_0(c)=\wt\AAA(c)\cap\{t=0\}$ and obtain 
an invariant circle}. 
In this case it is impossible to move within the cylinder, 
the normal hyperbolicity will be used.
\end{itemize}

\begin{theorem}\label{sr-first-pert} (see  Theorem 12 
\cite{KaloshinZ12}, Proposition 6.9 \cite{BernardKZ11}, 
Proposition 11.4 \cite{KaloshinZ12})
Fix a Dirichlet resonant segment $\SSS^{(n)}$ and the Hamiltonian  
\[
N^{\dr}=(H_\eps+\Delta H^\mol+\Delta H^\sr+\Delta H^\dr)\circ \Phi_\Pos
\]
given by (\ref{def:HamAfterSRDR}). Then, there exists 
a $\CCC^\infty$ small perturbation $\Delta H^\shad_1$ 
supported away from pull backs 
of all NHICs $\wt\CCC_{k_n}^{\om_n,i}$ in single resonance zones 
given by Corollary \ref{coro:cylinderPoschelCoordinates}
and away from pull backs of all manifolds $\MM^{k',k_n}$,
defined in \eqref{def:DR:UnionNHICs}, such that for 
$$
N^{\dr,1}=(H_\eps+\Delta H^\mol+\Delta H^\sr+\Delta H^\dr+\Delta H^\shad_1)\circ \Phi_\Pos,
$$ 
$\Ga_2$ is finite and 
  \begin{equation}\label{cohomology-types-SR}
\Gamma^{(n)} = \Gamma_1 \cup \Gamma_2 \cup \Gamma_3
\subset \LL\FF_{N^{\dr}}(\Tr_n).
  \end{equation}
\end{theorem}

The diffusion for both bifurcation values $\Gamma_2$ and invariant curve
values $\Gamma_3$ requires additional transversalities.  In the case that
$\widetilde{\AAA}_{N^{\dr}}(c)$ is an invariant circle, this transversality condition is 
equivalent to the transversal intersection of the stable and unstable manifolds. 
This condition can be phrased in terms of the barrier functions. 

In the bifurcation set $\Gamma_2$, the Aubry set $\wt{\mathcal A}_{N^{\dr}}(c)$ 
has exactly two static classes. In this case the Ma$\~{n}$e set
$\wt{\mathcal N}_{N^{\dr}}(c)\supsetneq \wt{\mathcal A}_{N^{\dr}}(c)$. 
Let $\psi_0$ and $\psi_1$ be contained in each of the two static 
classes of ${\mathcal A}_{N^{\dr}}(c)$, we define
$$
{b}^+_{N^{\dr},c}(\psi)= h_{N^{\dr},c}( \psi_0, \psi)+ 
h_{N^{\dr},c}(\psi, \psi_1)\ 
\textup{  and }\ 
{b}^-_{N^{\dr},c}(\psi)= h_{N^{\dr},c}( \psi_1, \psi)+
h_{N^{\dr},c}(\psi, \wt \psi_0),
$$
where $h_{N^{\dr},c}$ is the Peierls barrier for cohomology class $c$
associated to the Hamiltonian $N^{\dr}$, defined in 
Appendix~\ref{app:WeakKAM}. Let $\Gamma^*_2\ $ be the set 
of bifurcation $c\in \Gamma_2$ such that the set of minima of each 
$b^+_{N^{\dr},c}$ and $b^-_{N^{\dr},c}$ located outside of ${\mathcal A}_{N^{\dr}}(c)$ 
is totally disconnected. In other words,
these minima correspond to heteroclinic orbits connecting different
components of the Aubry set $\wt {\mathcal A}_{N^{\dr}}(c)$ and heteroclinic 
orbits $\wt {\mathcal N}_{N^{\dr}}(c) \setminus \wt {\mathcal A}_{N^{\dr}}(c)$
form a non-empty and totally disconnected set.
Since the Aubry and the Ma$\~{n}$e sets are {\it symplectic invariants}
\cite{Be2}, it suffices to prove these properties {\it using
a convenient canonical coordinates, e.g. normal forms.}

In the case $c\in \Gamma_3$, we have
$\wt {\mathcal N}(c)=\wt {\mathcal I}(c,u)=\wt {\mathcal A}(c)$
and it is an invariant circle\footnote{see Appendix \ref{app:WeakKAM}
for definition of $\wt {\mathcal I}_{N^{\dr}}(c,u)$,
which implies $\wt {\mathcal A}_{N^{\dr}}(c)\subset \wt {\mathcal I}_{N^{\dr}}(c,u)\subset 
\wt {\mathcal N}_{N^{\dr}}(c)$.}.
We first consider the covering for some positive integer $N$  
\begin{align*}
 \xi:\T^2 \to \T^2\qquad , \qquad 
\psi=(\psi^f, \psi^s) \to \xi(\psi)= (\psi^f, N\psi^s).\qquad 
\end{align*}
This covering lifts to a a symplectic covering
\begin{align*}
\qquad \Xi:T^*\T^2 \to T^* \T^2\qquad , \qquad 
(\psi,J)=(\psi, J^f, J^s)
 \to \Xi(\psi,J)=
(\xi(\psi),J^f, J^s/2),
 \end{align*}
and we define the lifted Hamiltonian $\wt N =N^{\dr} \circ \Xi$.
It is known that
$$
\wt {\mathcal A}_{\wt N}(\wt c)=\Xi^{-1}
\big(\wt {\mathcal I}_{\wt N}(c)\big)
$$
where $\wt c=\xi ^*c =(c^f,c^s/2)$.
On the other hand, the inclusion
$$
\wt {\mathcal N}_{\wt N}(\wt c)\supset \Xi^{-1}
\big(\wt {\mathcal N}_{\wt N}(c)\big)
=\Xi^{-1}\big(\wt {\mathcal A}_{N^{\dr}}(c)\big)
$$
is not an equality for $c\in \Gamma_3$. More precisely, for 
$c\in \Gamma_3(\eps)$, the set $\wt {\mathcal A}_{\wt N}(\wt c)$ 
is the union of two circles, while $\wt {\mathcal I}_{\wt N}(\wt c)$
contains heteroclinic connections between these circles. 
Similarly to the case of $\Gamma_2$, we choose a point $\theta_0$
in the projected Aubry set ${\mathcal A}_{\wt N}(c)$, and consider 
its two preimages $\wt \theta_0$ and $\wt \theta_1$ under $\xi$. 
We define
$$
b^+_{\wt N,c}(\psi)=h_{\wt N,c}(\wt \psi_0, \psi)+
h_{\wt N,c}(\psi, \wt \theta_1)\  \textup{ and }\ 
b^-_{\wt N,c}(\psi)=h_{\wt N,c}(\wt \theta_1, \psi)+
h_{\wt N,c}(\psi, \wt \psi_0)
$$
where $h_{\wt N,c}$ is the Peierl's barrier associated to $\wt N$. 
$\Gamma_3^*\ $ is then the set of cohomologies $c\in \Gamma_3$ 
such that the set of minima of each of the functions 
$b^{\pm}_{\wt N,c}$, located outside of the Aubry set
${\mathcal A}_{\wt N} (\wt c)$, is totally disconnected.
In other words, 
$${\mathcal N}_{\wt N}(\wt c) \setminus
{\mathcal A}_{\wt N}(\wt c) \quad
\text{ is  non-empty and totally disconnected}.
$$ 
Since the Aubry and the Ma$\~{n}$e sets 
are symplectic invariants \cite{Be2}, as before it suffices to prove 
these properties  using a convenient canonical coordinates, 
e.g. normal forms.

We can perform an additional perturbation such that the above
transversality condition is satisfied.

\bthm\label{sr-2nd-pert} 
Let $N^{\dr,1}$ be the Hamiltonian 
from  Theorem~\ref{sr-first-pert}. Then there exists arbitrarily 
$\CCC^r$ small perturbation $\Delta H^{\shad}_2\circ \Phi_{\Pos}$ 
of  $N^{\dr,1}$,
preserving all Aubry sets $\wt {\mathcal A}_{\wt N}(c)$ with
$c\in \Gm^{\sr}_j$, such that for 
$N^{\dr,2}=N^{\dr,1}+\Delta H^{\shad}_2\circ \Phi_{\Pos}$ 
we have $\Gamma_2 = \Gamma_2^*$ and 
$\Gamma_3 = \Gamma_3^*$. Moreover, 
$\Gamma_j^{\sr} = \Gamma_1 \cup \Gamma_2^* 
\cup \Gamma_3^*$ is contained in a single forcing class.
\ethm

The proof of this Theorem is the same as the proof of Thm. 12
\cite{KaloshinZ12} or Theorem 6.4 \cite{BernardKZ11}. The crucial 
element of the proof is a lemma of Cheng-Yan \cite{ChengY09}. 
One can find a small perturbation $\Delta H^{\shad}_2$ localized 
away from NHICs  and, therefore, away from Aubry sets. 
We refer to section 11.6 \cite{KaloshinZ12} for details. 

\subsection{Equivalent forcing class along double resonances}
\label{c-equi-DR}

In the core of a double resonance we diffuse either along 
one resonance $\Gamma$ across this double resonance 
or we diffuse  inside along one $\Gamma$ and diffuse 
outside along another $\Gamma'$. Recall also that at the core of 
each double resonance, modulo rescaling, the dynamics is 
symplectically conjugate to a Tonelli Hamiltonian $\HH=\HH^{k_n,k'}$ 
of two degrees of freedom. By Theorem \ref{thm:Bernard} Aubry 
and Ma$\~n$e sets of conjugate systems are also conjugate. 
Therefore, it suffices to describe these sets for a Hamiltonian 
$\HH$ of the form (\ref{def:Ham2dofNonMech}) satisfying 
(\ref{eq:DR-Hamiltonians}).  

By a result of Diaz Carneiro \cite{DC} each Aubry set $\AAA_{\HH}(c)$ 
and each Ma$\~n$e set $\mathcal N_{\HH}(c)$ is contained in 
a certain energy surface of $\HH$. Thus, it is convenient to 
parametrize these sets by energy. Suppose $\HH$ satisfies 
conditions [H1-H3] of Key Theorem \ref{keythm:DR:HighEnergy}. 
As we discussed in section 
\ref{sec:outline:cylinders}, it is natural to consider tow regimes: 
\begin{itemize}
\item ({\it high energy}) $\al_{\HH}(0)+e\le E\le \al_{\HH}(0)+E^*$ 
\item ({\it low energy}) $\al_{\HH}(0)\le E \le \al_{\HH}(0)+2e,$
\end{itemize} 
where $e$ is small positive and $E^*$ is large positive. 
At the critical energy $S_{\al_{\HH}(0)}$ there is a critical point
denoted $\phi^*$.

In local coordinates of the double resonance each resonance 
$\Gamma$ has a uniquely defined integer homology class 
$h\in H_1(\T^2,\Z)$.  Denote by $h$ (resp. $h'$) cohomology 
class corresponding to $\Gamma$  (resp. to $\Gamma'$).  

In section 4.2.2  \cite{KaloshinZ12} using a Legendre-Fenchel 
transform $\mathcal L\mathcal F$ of the Hamiltonian $\HH$, 
given  homology $h$ we show how to choose cohomology 
$\bar c_h$. In Proposition 4.1 of \cite{KaloshinZ12} we show that 
for Hamiltonians satisfying  [H1]-[H3] and a proper choice of 
cohomology for high energy $E$ each Aubry set 
$\AAA_{\HH}(\bar c_h(E))$ of $\HH$ is 
\begin{itemize}
\item either a periodic orbit, denoted $\gm^E_h$, 
\item or the union of two periodic orbits, denoted $\gm^E_h$ 
and $\bar \gm_h^E$. 
\end{itemize}

In the case of low energy there are three different cases.
In Theorem 7 \cite{KaloshinZ12} we show that for each low energy 
$\al_{\HH}(0)<E\le \al_{\HH}(0)+2e$ the periodic orbit 
$\gm^E_h$ exists, is unique and smoothly depends on $E$. 
Consider the limit $\lim_{E\to \al_{\HH}(0)^+} \gm^E_h=\gm^*_h$. 
Mather \cite{MaSh} showed that generically there are three possibilities: 
\begin{itemize}
\item ({\it simple non-critical}) $\gm^*_h$ does not 
contain the critical point $\varphi^*$.

\item ({\it simple critical}) $\gm^*_h$ contains 
the critical point $\varphi^*$.

\item ({\it non-simple}) $\gm^*_h\ $ is a non-simple curve, 
which contains the critical point $\varphi^*$, and there are 
two simple curves $\gm^*_{h_1}$ and $\gm^*_{h_2}$ and 
two integers $n_1,n_2\in\Z_+$ such that $h=n_1h_1+n_2h_2$.
\end{itemize}

\vskip 0.15in 
\begin{center}
{\bf Equivalent forcing class at double resonance,
high energy }
\end{center}

Naturally, at a double resonance one can associate to 
each reconance $\Gamma$ of the Hamiltonian $N^*$ 
an integer homology class $h\in H_1(\T^2,\Z)$ of the slow 
system $\HH$.  Using the function $\bar c_h(E)$ we can define 
$\Gamma_{e}\subset \Gamma$ such that 
the corresponding  Aubry sets 
$\wt \AAA_{\HH}(\bar c_h(E))=\gm^E_h$ have high energy 
$\al_{\HH}(0)+e\le E\le \al_{\HH}(0)+E^*$. 

By Key Theorem \ref{keythm:DR:AubrySets}, for high energy 
for each  $c_h(E)\in \Gamma_{e}$, the Aubry set 
$\wt {\AAA}_{\HH}(\bar c_h(E))$ is contained in one of the NHICs, 
and has a local graph property (see section 11.1, 
Figure 21 \cite{KaloshinZ12}). Reproducing exactly 
the same double cover construction as in the previous section 
we define the lifted Hamiltonian $\wt N$.  For the resonant 
segment $\Gamma_{e}$ we have an analog of 
Theorem \ref{sr-2nd-pert}.  

\bthm\label{dr-pert} 
Let $N^{\dr,2}$ be the Hamiltonian from  Theorem~\ref{sr-2nd-pert}. 
Then there exists arbitrarily $\CCC^r$ small perturbation 
$\Delta H^{\shad}_3\circ \Phi_{\Pos}$ \ of  $N^{\dr,2}$,
preserving all Aubry sets $\wt {\mathcal A}_{N'}(c)$ 
with $c\in \Gm_{e}$, such that for 
$N^{\dr,3}=N^{\dr,2}+\Delta H^{\shad}_3\circ \Phi_{\Pos}$ 
we have $\Gamma_2 = \Gamma_2^*$ and $\Gamma_3 = \Gamma_3^*$. 
Moreover, $\Gamma_{e} = \Gamma_1 \cup \Gamma_2^* \cup \Gamma_3^*$ is contained in a single forcing class.
\ethm

The proof is quite similar to the proof of Theorem ~\ref{sr-first-pert},
see Key Theorem 9, \cite{KaloshinZ12}.

\vskip 0.15in 
\begin{center}
{\bf Equivalent forcing class at double resonance,
low energy }
\end{center}

Recall that generically at low energy we have three cases:
simple noncritical, simple critical, nonsimple. To achieve 
this condition we perturb $N^{\dr,3}$ to $N^*$. 

The first two cases can be handled in the same way as high energy. 
See Theorem 13 and 14 section 6.2 \cite{KaloshinZ12}.
These two Theorems are analogs of Theorems \ref{sr-first-pert} 
and \ref{sr-2nd-pert}.  

Now we face the problem of passing the double resonance along 
either nonsimple resonance $\Gm$ or entering along one 
resonance $\Gm$ and exiting along another one $\Gm'$. 
See Figure 12 Section 3.5 \cite{KaloshinZ12}. 

Since each simple critical resonance contains a periodic orbit of 
$\HH$ corresponding to the saddle critical point of  $\HH$,
in order to construct diffusion in the above cases it suffices
to construct {\it a jump}. Namely, prove forcing relation 
in the following case.  

Let $\Gm$ define homology $h\in H_1(\T^2,\Z)$ such that 
defined a non-simple $\gm^*_h$. In particular, $\gm^*_h\ $ 
contains the critical point $\varphi^*$ and there are 
two simple curves $\gm^*_{h_1}$ and $\gm^*_{h_2}$ and 
two integers $n_1,n_2\in\Z_+$ such that $h=n_1h_1+n_2h_2$.

We need to prove forcing relation between 
$c_h(E)$ and $c_{h_1}(E)$ for some low energy $E$, i.e. 
$\al_{\HH}(0)\le E\le \al_{\HH}(0)+e$. 

For the slow system $\HH$ this is proven in Section 12.4 
\cite{KaloshinZ12}. Deducing the same result for the original 
system is given by approximation of barrier of $N'$ with 
a barrier of $\HH$ can be done in the same was as in 
section 12.3 \cite{KaloshinZ12}.

\appendix

\section{A time-periodic KAM}
\label{sec:KAM-time-periodic}

Consider a Hamiltonian $H_0(I)$ in action-angle variables, where  
angles $\varphi\in\TT^n$ and $I$ belongs to a bounded open set $U \subset \R^n$. Consider a smooth time-periodic perturbation
\[
H_\eps(\varphi,I,t)=H_0(I)+\eps H_1(\varphi,I,t). 
\]
Using the standard procedure associate to it an autonomous Hamiltonian
\[
H^*_\eps(\varphi,I,t,E)=H_0(I)+E+\eps H_1(\varphi,I,t), 
\]
where $E$ is conjugate to $t$. Therefore,  orbits of $H_\eps$
and projected along $E$ orbits of $H^*_\eps$ are the same as those of $H_\eps$ 
independently of the value of $E$. 

Consider an auxiliary Hamiltonian $\HH_0(I,E)=\exp(H_0+E)$ and 
its frequency map 
$\om(I)=\nabla \HH_0(I,E)=(\exp(H_0+E)\nabla H_0,\exp(H_0+E))$. 
Define 
\[
\HH_\eps(\varphi,I,t,E)=\exp (H_0(I)+E+\eps H_1(\varphi,I,t)). 
\]
Write this Hamiltonian in the form 
\[
\HH_\eps(\varphi,I,t,E)=\HH_0+\eps \HH_1(\varphi,I,t,E)
\]
Notice that orbits of $\HH_\eps$ coincide with orbits of $H_\eps$ 
after constant time change.  

Call $\HH_0$ non-degenerate in $U$ if for each $I\in U$ and each $E$ we have 
\[
\det \left( \dfrac{\partial \om(I,E)}{\partial (I,E)}\right)\neq 0.
\]
Recall that a frequency $\om\in\DDD_{\eta,\tau}$ is called $(\eta,\tau)$-Diophantine and it has been defined in \eqref{def:Diophantine}.
%As before denote by $\DDD_{\eta,\tau}$ this set of frequencies. 
%Denote by $S_\mathcal E=\{\HH_\eps =\mathcal E\}$ energy 
%surface of $\HH_\eps$. 

\begin{theorem}\label{KAM}  \cite{Poschel82} Let $\eta,\tau>0$ and $r \ge 2n+2\tau$.  
Let $\HH_0$ be a real analytic Hamiltonian on $U$ 
and $\HH_1$ be $\CCC^r$ smooth on $U\times \R \times \T^{n+1}$ 
resp. Suppose $\HH_0$ is non-degenerate in $U$.  
Then there exists $\eps_0=\eps_0(\HH_0,\eta,\tau)>0$ and 
$c_0=c_0(\HH_0,r)$ such that for any $\eps$ with $0<\eps<\eps_0$ 
and any  $\om \in \DDD_{\eta,\tau}$
the Hamiltonian $\HH_0+\eps \HH_1$ has a $(n+1)$-dimensional 
(KAM) invariant torus $\TTT_\om$ and dynamics restricted to 
$\TTT_\om$ is smoothly conjugate to the constant flow 
$(\dot \varphi,\dot t)=\om$ on $\TT^{n+1}$. Moreover, 
for the union tori 
\[
Leb(\cup_{\om \in \DDD_{\eta,\tau}}\TTT_\om) > (1 - c_0 \eta)
\ \cdot \ Leb(U\times \T^{n}\times \T),
\]
where $Leb$ is the Lebesgue measure.
\end{theorem}

One can improve $\eta$ to be proportional to $\sqrt \eps$, but we do not 
use this modification.  Notice that dynamics of $\HH_\eps$ is independent of a choice of $E$. 
Namely, for any $E \neq E'$ and any initial condition $(\varphi_0,I_0,t_0)$
orbits of $\HH_\eps$ starting $(\varphi_0,I_0,t_0,E)$ and $(\varphi_0,I_0,t_0,E')$
have the same projection along $E$. In particular, if $(\varphi_0,I_0,t_0,E)$
belongs to a KAM torus $\TTT_\om$, then $(\varphi_0,I_0,t_0,E')$ also belongs 
to a KAM torus $\TTT_{\lb \om}$ for $\lb=\exp(E-E')$. Similarly, 
if $(\varphi_0,I_0,t_0,E)$ does not belong to a KAM torus, then 
$(\varphi_0,I_0,t_0,E')$ also does not belong to any KAM torus. 

Recall that we have assumed that $H_0(I)$ is called strictly convex (see \eqref{def:Convexity:OriginalHam}).

\bcor Let $\eta,\tau>0$ and $r \ge 2n+2\tau$.  Let $H_0$ be a real analytic 
Hamiltonian on $U$ and $H_1$ be $\CCC^r$ smooth on 
$U\times \R \times \T^{n}\times \T$ resp. Suppose $H_0$ is strictly convex on $U$.  
Then there exists $\eps_0=\eps_0(H_0,\eta,\tau)>0$ and $c_0=c_0(H_0,r)$ 
such that for any $\eps$ with $0<\eps<\eps_0$ and any  $\om \in \DDD_{\eta,\tau}$
the Hamiltonian $H_0+\eps H_1$ has a $(n+1)$-dimensional 
(KAM) invariant torus $\TTT_\om$ and dynamics restricted to 
$\TTT_\om$ is smoothly conjugate to the constant flow 
$(\dot \varphi,\dot t)=\lb \om$ on $\TT^{n+1}$ with $\lb=\om^{-1}_{n+1}$. 
Moreover,  for the union tori 
\[
\mu_E(\cup_{\om \in \DDD_{\eta,\tau}}\TTT_\om) > (1 - c_0 \eta)
\ \cdot \ \mu_E(U\times \T^{n}\times \T),
\]
where $\mu_E$ is the Lebesgue measure.
\ecor

To prove this corollary it suffices to show that if $H_0$ is strictly convex in $U$, 
then $\HH_0$ is non-degenerate in $U\times\R$. 

\blm 
If $H_0$ is strictly convex, then the Hessian of $\HH_0$ is nondegenerate.
\elm 

\vskip 0.1in

\begin{proof} Notice that the Hessian of $\HH_0$ has the following form 
\be
\HH_0 \left(\begin{matrix}
\partial^2_{I_1I_1} H_0+\partial_{I_1} H_0 \partial_{I_1} H_0 & \dots & 
\partial^2_{I_1I_n} H_0+\partial_{I_1} H_0 \partial_{I_n} H_0 & \partial_{I_1} H_0 \\
\partial^2_{I_2I_1} H_0+\partial_{I_2} H_0 \partial_{I_1} H_0 & \dots &
\partial^2_{I_2I_n} H_0+\partial_{I_2} H_0 \partial_{I_n} H_0 & \partial_{I_1} H_0 \\
\dots & \dots & \dots \\
\partial^2_{I_nI_1} H_0+\partial_{I_n} H_0 \partial_{I_1} H_0 & \dots &
\partial^2_{I_nI_n} H_0+\partial_{I_n} H_0 \partial_{I_n} H_0 & \partial_{I_n} H_0 \\
\partial_{I_1} H_0 & \dots & \partial_{I_n} H_0 & 1
\end{matrix}
\right).\ee

Denote the last matrix Hess$(I,E)$. The determinant for this matrix is equal to  get 
\[
\HH_0^n \det \mathrm{Hess}(I,E).
\]
In order to compute $\det$ Hess$(I,E)$ for each line $j=1,\dots,n$ we multiply 
the last time by $\partial_{I_j}H_0$ and subtract from $j$-th line. 
\[
\HH_0^n \det \left(\begin{matrix}
\partial^2_{I_1I_1} H_0& \dots & 
\partial^2_{I_1I_n} H_0& 0 \\
\partial^2_{I_2I_1} H_0 & \dots &
\partial^2_{I_2I_n} H_0& 0 \\
\dots & \dots & \dots \\
\partial^2_{I_nI_1} H_0 & \dots &
\partial^2_{I_nI_n} H_0 & 0 \\
\partial_{I_1} H_0 & \dots & \partial_{I_n} H_0 & 1
\end{matrix}
\right).
\]
Notice that $\HH_0$ is always positive and the determinant equals 
$\det$ Hess $H_0$. Since $H_0$ is positive definite, determinant is not zero.  
\end{proof}

\subsection{Application of P\"oschel Theorem to our setting} 
\label{KAM-fin-smooth}

In this section we apply Theorem \ref{KAM} to the situation where $H_0$ is only finitely 
differentiable\footnote{The authors are grateful to Popov for pointing out this approach}.

We use a more precise version of Theorem \ref{KAM}. Fix $\al>1$ and $\lb>\tau+n>n$ and let $H$ be of class $\CCC^{\al,\lb \al}$ if it is of class
$\CCC^\al$ in the $I$-variables, but of class $\CCC^{\lb \al}$ in 
the $\phi$-variables. 
%The formal definition is given below. 

\bthm \label{thm:poschel} (P\"oschel \cite{Poschel82}) 
Let $H_0$ be real analytic and nondegenerate, 
such that the frequency map $\nabla \bar H_0$ is a diffeomorphism 
$I \to \R^n$, and consider a differentiable perturbation 
$\bar H = \bar H_0 + \eps \bar H_1$ of class $\CCC^{a \lb +\lb + n+\tau}$ 
with $\lb > \tau + n$ and $\al > 1$ not in the discrete set 
$\Lb = \{ i / \lb + j,\ i, j \ge 0$ integers$\}$. Then, 
for small $\eta> 0$, one can choose $|\eps|$ proportional 
to $\eta^2$ so that there exists a diffeomorphism
\[
 T = \Psi_0 \circ \Phi_\eps  : \T^n \times \D_{\eta,\tau} \to 
\T^n \times \R^n
 \]
which on $\T^n \times \D_{\eta,\tau}$, transforms
the Hamiltonian equations of motion into
\[
 \dot \th = \om, \quad \dot \om=0.
\]
More precisely, $\Psi_0$ is the real analytic inverse of
the frequency map, and $\Phi_\eps$ is a diffeomorphism of class
$\CCC^{\al\lb ,\al}$ close to the identity. Its Jacobian determinant
uniformly bounded from above and below. In addition, if $H$ is
of class $\CCC^{\bt\lb+\lb+n-1+\tau}$  with $\al \le \bt \le \infty$,
one can modify $\Phi_\eps$ outside $\T^n \times \D_{\eta,\tau}$
so that $\Phi_\eps$ is of class $\CCC^{\bt \lb,\bt}$ for $\bt \not\in \Lb$.
\ethm

Now we apply this result to  Hamiltonians with finite regularity. Expand the Hamiltonian $H_0$ near $\xi \in D$ 
using Taylor formula and write $H_\eps$ as follows: 
\[
 H_\eps(\th,\xi+\Delta I) =  H_0(\xi)
 + \langle \nabla H_0(\xi),\Delta I \rangle 
 + \dfrac 12 \langle \partial^2 H_0(\xi) \Delta I,\Delta I \rangle 
 + Q_3(\Delta I,\xi)+ \eps H_1(\th,\xi + \Delta I),
\]
where $Q_3(\Delta I,\xi)$ is the cubic remainder.
Consider the ball of radius $d=\eps^{1/3}$ around $\xi$. 
Denote 
\[
 \bar H_0(I)=H_0(\xi)
 + \langle \nabla H_0(\xi),\Delta I \rangle 
 + \dfrac 12 \langle \partial^2 H_0(\xi) \Delta I,\Delta I \rangle. 
\]
and
\[
 \eps \bar H_1=\bar H-\bar H_0=Q_3(\Delta I,\xi)+ 
 \eps H_1(\th,\xi + \Delta I).
\]
Let $r=\al \lb+\lb+n-1+\tau$. Assume that $H_0\in \CCC^{r+3}$ and 
$H_1\in \CCC^{r}$. Then the remainder $\bar H_1  \in \CCC^{r+3}$. 
The condition 
\[
|\partial^2_{II}\bar H_0|_{\Om+\rho},
\qquad 
|(\partial^2_{II}\bar H_0)^{-1}|_{\Om+\rho} \le R
\]
are clearly satisfies, because 
$\partial^2_{II}\bar H_0 \equiv \partial^2 H_0(\xi)$. 

Application of Theorem \ref{thm:poschel} gives 
existence of a diffeomorphism $\Psi_0\circ \Phi_\eps$
defined in $\T^n \times \DDD_{\eps^{1/3}}(\xi)$. Covering 
$\T^n \times \DDD$ by balls of radius $\xi$ and applying 
Whitney Extension Theorem (see e.g. \cite{Poschel82}) 
we obtain a global change of coordinates. 
Moreover, we have $\HH_0 \in \CCC^{\al+1}$ and $\HH_1 \in \CCC^{\al}$ where $\al=(r-n+1-\tau)/\lb$. 

\section{Smoothness of compositions and the inverse}
In this appendix we state several results dealing with the  smoothness of the composition of $\CCC^r$ functions. Let $x=(x_1,\dots,x_n)\in \R^n$ be coordinate
system in $\R^n$.

\blm\label{lm:norm-composition-bound} 
(\cite{DelshamsH09} Appendix C) Let $F\in \CCC^k(\R^n)$ and 
$G\in \CCC^k(\R^n,\R^n)$, $k\ge 1$. Then for some $c=c(n,k)$ we have 
\[
 \| F \circ G\|_{\CCC^k}\le 
 c (\|F\|_{\CCC^1}\ \|G\|_{\CCC^k} + \sum_{p=2}^k \|F\|_{\CCC^p}\
 \sum_{j_1+\dots+ j_p=k} \|G\|_{\CCC^{j_1}} \dots \|G\|_{\CCC^{j_p}}).
\]
In particular, 
\[
 \| F \circ G\|_{\CCC^k}\le 
 c (\|F\|_{\CCC^1}\ \|G\|_{\CCC^k} + \|F\|_{\CCC^k} \|G\|_{\CCC^{k-1}}^k).
\]
\elm

Consider $\D_R=D_R \times \T^{n+1} \subset
\R^n \times \T^{n+1} \ni (I,\phi,t)$.  Let 
$H \in \CCC^k(\D_R)$ be a $\CCC^k$ smooth function. Denote 
by $X_H$ the Hamiltonian vector field associated to 
$H$ give by 
\[
 X_H = (\partial_\phi H, - \partial_\phi H),
\]
where $\partial_*$ are partial derivatives
with respect to the coordinate variables. 
Clearly we have $X_H \in \CCC^{k-1}(\D_R,\R^{2n})$ and 
\[
 \|X_H\|_{\CCC^{k-1}(\D_R)} \le 
 \|H\|_{\CCC^{k}(\D_R)}. 
\]
Assume that $\|\partial_\phi H\|_{\CCC^0(\D_R)}<\dt$ for 
some $\dt<R$. Then, by the Mean Value Theorem the time 
$t$ map of the vector field $\Phi^H_t$ is a well-defined
$\CCC^{k-1}$ map and satisfies 
\[
 \Phi^H=\Phi^H_1:\D_{R-\dt}\to \D_R.
\]
If $H$ is integrable, then $\dt$ can be chosen equal to $0$. 

In what follows we need to estimate the $\CCC^k$ norm of 
$\Phi^H$ in terms of the $\CCC^k$ norm of the vector field 
$X_H$. One can expect to have $\Phi^H$ being $\CCC^k$ 
close to the identity when $X_H$ is $\CCC^k$ close to zero.
In our case, for smaller $k$ is will indeed be true,
but for higher $k$ estimates will deteriorate. 
Recall the standard relation 
\[
 \Phi^H_t= \Id + \int_0^t \ X_H\circ \Phi^H _s \ ds.
\]
It follows from the classical formula due to Fa\`a 
di Bruno (see e.g. \cite{AbrahamR67}) that 
\[
 \|F\circ G\|_{\CCC^k} \le c \|F\|_{\CCC^k} \|G\|^k_{\CCC^k}
\]
for some $c=c(n,k)$ and for $\CCC^k$ vector-valued functions
defined on appropriate domains. Recall that 
at some point of the proof we consider 
a molification of a $\CCC^k$ smooth vector field with
a molifing parameter $\sigma$. In this setting we have 
the following 

\blm \label{lm:Hamiltonian-flow-bound}
Let $X_{H_\sigma} \in \CCC^k(\D_R,\R^{2n})$ be a one-parameter family 
of $\CCC^k$ vector field, $\dt,\sigma_0$ be small positive, $0<m<k$.  
Assume that 
\[
\max_{0\le \sigma \le \sigma_0} \|X_{H_\sigma}\|_{\CCC^0(\D_R)}<\dt,
\]
\[
\max_{0\le \sigma \le \sigma_0} \|X_{H_\sigma}\|_{\CCC^k(\D_R)}< 1, \quad 
\textup{ and }\quad 
\|X_{H_\sigma}\|_{\CCC^{k+m}(\D_R)}< \sigma^{-m},
\]
then 
\[
 \|\Phi^{H_\sigma}-\textup{Id}\|_{\CCC^k(\D_{R-\dt})}\le c
 \|X_{H_\sigma}\|_{\CCC^k(\D_R)}
\qquad 
 \|\Phi^{H_\sigma}-\textup{Id}\|_{\CCC^{k+m}(\D_{R-\dt})}\le c
 \|X_{H_\sigma}\|_{\CCC^{k+m}(\D_R)}
\]
for some $c=c(n,k,R)$. In addition, 
\[
 \|F \circ \Phi^{H_\sigma}\|_{\CCC^k(\D_{R-\dt})}<
 c \|F\|_{\CCC^k(\D_R)}\,\|\Phi^{H_\sigma}\|^k_{\CCC^k(\D_R)}.
\]
\elm

This lemma for $m=0$ is lemma 3.1 \cite{Bounemoura10}.
(see also similar to lemma 3.15 \cite{DelshamsH09}).
For $m>0$ we use estimates from the proof of lemma 3.15 
\cite{DelshamsH09}.

{\it Proof:} By the Fundamental Theorem of calculus 
we have 
\[
 \Phi^{H_\sigma}_t(x)= x+ \int_0^t 
 \dfrac{\partial \Phi^{H_\sigma}_\tau}{\partial \tau}(x)
 d\tau=x+ \int_0^t X_{H_\sigma} \circ \Phi^{H_\sigma}_\tau(x) \ d\tau,
 d\tau,
\]
where $x=(I,\phi,t) \in \D_R$ and for the canonical matrix $J$ of 
the symplectic form $\om = dI \wedge d\phi + dE \wedge dt$ we have 
$X_H=J \nabla H$. The extra variable $E$, conjugated to the angle $s$, 
was introduced to make apparent the symplectic character of the change 
of variables. Using the last formula of lemma \ref{lm:norm-composition-bound}
we obtain
\[
 \|\Phi^{H_\sigma}_1\|_{\CCC^l}\le \|\textup{Id}\|_{\CCC^l} +
 \int_0^1  \|X_{H_\sigma} \circ \Phi^{H_\sigma}_\tau\|_{\CCC^l} \ d\tau 
 \le  \|\textup{Id}\|_{\CCC^l} +
\]
\[
 \int_0^1 \left( \|X_{H_\sigma}\|_{\CCC^1} \| \Phi^{H_\sigma}_\tau\|_{\CCC^l}+
  \sum_{p=2}^l \| X_{H_\sigma} \|_{\CCC^p} \sum_{j_1+\dots+j_p=l}
 \|\Phi^{H_\sigma}_\tau\|_{\CCC^{j_1}}\dots\| \Phi^{H_\sigma}_\tau\|_{\CCC^{j_p}}
\right) d\tau,
\]
where $l=2,\dots,k+m$ and $c$ is a constant depending on $l$ and $n$. 
Notice also that in this and the similar sums given below all $\{j_s\}_s$
are strictly positive.

For $l=1$ we have 
\[
 \|\Phi^{H_\sigma}_1\|_{\CCC^1}\le  \|\textup{Id}\|_{\CCC^1} + \int_0^1  
\|X_{H_\sigma}\|_{\CCC^1} \| \Phi^{H_\sigma}_\tau\|_{\CCC^1}\,d\tau.
\]
Now by induction define 
$$
a_l=\max_{0\le t\le 1,\ 0\le \sigma \le \sigma_0} \|\Phi^{H_\sigma}_t\|_{\CCC^l}.
$$
Denote $d_l= \max_{0\le \sigma \le \sigma_0}\|X_{H_\sigma}\|_{\CCC^l}$. 
We know that $d_l\le  1$ for $1\le l\le k$ and $d_{k+l} \le \sigma^{-l}$ 
for $1\le l \le m$. Then, 
\[
 a_1 \le 1+d_1 a_1,
\]
and 
\[
 a_l \le 1 + d_1 a_l  +d_l a_1^l+c  \sum_{p=2}^{l-1} d_p 
 \sum _{j_1+\dots+j_p=l} a_{j_1} \dots a_{j_p}.
\]

As in the proof of Lemma  3.15 \cite{DelshamsH09}
we see that for $1\le l\le k$ we have $a_l\le D_l$ with $D_l$ 
being independent of $\sigma$. For $a_{k+l},\ 1 \le l \le m$
we proceed by induction. Let $l=1$. 
\[
a_{k+1}\le 1 + d_1 a_{k+1}  +c d_{k+1} a_1^{k+1}
+ c\sum_{p=2}^k d_p 
 \sum _{j_1+\dots+j_p=k} a_{j_1} \dots a_{j_p}\le c + c\sigma^{-1}.  
\]
Let $1<l\le m<k$. We have 
\[
a_{k+l}\le 1 + d_1 a_{k+l}  +c d_{k+l} a_1^{k+l}
+ c\sum_{p=2}^{k+l-1} d_p 
 \sum _{j_1+\dots+j_p=k+l-1} a_{j_1} \dots a_{j_p} \le c + c\sigma^{-l}.  
\]
Due to the fact that $m<k$ and all $\{j_s\}_s$ are strictly positive
in the second line sum in each of its elements there is at most one 
$j_{s^*}>k$. In this case $p<l-1\le m-1<k-1$ and all other 
$j_s<l-1\le m-1<k-1$. Thus, $j_{s^*}\le k+l-p\le k+l-2$ and such 
an element in the sum is bounded by $c\sigma^{-l+2}$. All others are 
uniformly bounded. However, $d_{k+l}\le \sigma^{-l}$ and gives 
the leading term. \qed

\blm \label{lm:norms-of-inverse} Let $0<s<r$. 
Let $\Phi_\sigma:\Omega \to \R^d$ be a one-parameter
family of $\CCC^r$ smooth maps of an open set 
$\Om \subset \R^d$  which is $\CCC^2$-close to the identity 
on the closure and such that for some $C>0$ we have 
\[
\|\Phi_\sigma\|_{\CCC^s}\le C \qquad 
\|\Phi_\sigma\|_{\CCC^{s+j}}\le C \sigma^{-j} \text{ for }j\le r-s. 
\]
Then for the inverse $\Phi^{-1}_\sigma$ is $\CCC^r$-smooth on 
$\Om'_\sigma= \Phi_\sigma(\Om)$, namely, 
$\|\Phi^{-1}_\sigma(y)\|_{\CCC^r(\Om')}$ is bounded. Moreover, 
\[
\|\Phi^{-1}_\sigma\|_{\CCC^s}\le C \qquad 
\|\Phi^{-1}_\sigma\|_{\CCC^{s+j}}\le C \sigma^{-j} 
\text{ for }j\le r-s. 
\]
\elm 

{\it Proof:}\ Denote for $\bt, \al \in \Z^n_+$ 
$\bt\prec \al$ if $\bt_i\le \al_i$ for $i=1,\dots,n$. 

Differentiate $\Phi_\sigma^{-1}\circ \Phi_\sigma(x)$ wrt $x$. 
In the domain of definition $x\in \Om$ we have 
\[
 D \Phi_\sigma^{-1}(\Phi_\sigma(x)) D \Phi_\sigma(x)\equiv Id,
\]
where $D \Phi_\sigma^{-1}(y)$ and $D\Phi_\sigma(x)$ 
are $n \times n$ linearization matrices. 

Let $\al \in \Z^n_+$ with $|\al|<r$. 
Consider 
\[
0\equiv \partial^\al 
(D \Phi_\sigma^{-1}(\Phi_\sigma(x)) D \Phi_\sigma(x)).   
\]
For any integer $k$ with $0\le k\le r$ denote by 
$D^k \Phi_\sigma(x)$ the collection of all partial derivatives 
of order up to $k$. Let $e_j=(0,\dots,1_j,\dots,0)$.
By induction we can prove the following statement:
\[
 0\equiv \partial^\al (D \Phi_\sigma^{-1}(y)|_{y=\Phi(x)} 
 D \Phi_\sigma(x))=
 \]
 \[
= \sum_{\bt\prec \al} 
\partial^\bt D \Phi_\sigma^{-1}(y)|_{y=\Phi_\sigma(x)}
 P_{\bt,\al} (D\Phi_\sigma(x),\dots,D^{|\al-\bt|}\Phi_\sigma(x)),
\]
where $P_{\bt,\al}$ are $n\times n$ matrix polynomials
with $\bt\prec \al$ satisfying the following inductive 
formula: for any $j=1,\dots,n$  
\[
 P_{\bt,\al+e_j} (D\Phi_\sigma(x),\dots, 
 D^{|\al-\bt|+1}\Phi_\sigma(x))
=\]
\[
= \partial_{x_j} 
P_{\bt,\al} (D\Phi_\sigma(x),\dots, D^{|\al-\bt|}\Phi_\sigma(x))
+ \partial_{x_j}\Phi_\sigma(x)\,
P_{\bt-e_j,\al} (D\Phi_\sigma(x),\dots, D^{|\al-\bt|}\Phi_\sigma(x)),
\]
and $P_{0,0}(D\Phi_\sigma(x))=D\Phi_\sigma(x)$.
%and $P_{\al-e_j,\al}=0$ provided $\al_j\ge 1$. 
Notice that all partial derivatives of $\Phi^{-1}_\sigma$ of 
order $\le r$ can be explicitly computed using partial 
derivatives of $\Phi_\sigma$ of order $\le r$ and of $\Phi^{-1}$ 
of order $\le r-1$. This implies the claim.

\section{Generic Tonelli Hamiltonians \& a Generalized Ma\-pertuis Principle}\label{app:NonMechanicalAtDR}

Consider a $\CCC^r$ smooth function $\HH:T^*\T^n \to \R$.  
It is called {\it a Tonelli Hamiltonian } if 
\begin{itemize}
\item $\HH(\vf,I)$ is {\it positive definite}, 
i.e., the Hessian $\partial^2_{I_iI_j}\HH$ is positive definite 
$\forall (\phi,I)\in \T^n$.

\item $\HH$ has {\it superlinear growth}, i.e. 
$\HH (\vf, I)/\|I\| \to \infty$  as $\|I\| \to \infty$.
\end{itemize}

Recall that $x=(\vf,I)$ is a critical point if the gradient of $\HH$ vanishes. 
A function $\HH$ is called {\it a Morse function} if  at any critical point, 
the Hessian $\partial^2_{xx}\HH$ is non-singular and critical points 
have pairwise distinct values.  Morse functions form an open dense set 
in a properly chosen topology (see e.g. \cite{Milnor63}). 

We say that a property is {\it Ma$\tilde{\mbox n}\acute{\mbox e}$'s  $\CCC^r$ generic }
if there is a $\CCC^r$ generic set of potentials $U$ on $\T^2$
such that $\HH_U = \HH + U$ satisfies this property.
We prove that properties in Lemma  \ref{lemma:SelfIntersectingGeodesics} and
Key Theorems \ref{keythm:DR:HighEnergy} and  \ref{keythm:DR:LowEnergy:MapComposition} are 
Ma$\tilde{\mbox n}\acute{\mbox e}$'s  $\CCC^r$ generic. 

We introduce some notations: $\nabla \HH(\vf,I)$ is 
the gradient, i.e. the set of all partial derivatives, $\pi:T^*\T^n\to \T^n$
is the natural projection, $\nabla_I \HH(\vf,I)$ and $\nabla_\vf \HH(\vf,I)$
are the $I$ and the $\vf$-gradients, given by the set of partial 
derivatives with respect to $I$ and  $\vf$ respectively.

\blm The property that $\HH$ is a Morse function is 
Ma$\tilde{\mbox n}\acute{\mbox e}$'s  $\CCC^r$ generic.
\elm

\begin{proof}
Due to convexity for each $\vf^*\in \T^n$ the map 
$\nabla_I\HH(\vf^*,\cdot):\R^n\to \R^n$ is a diffeomorphism. 
Therefore,  for each $\vf^*$ there is at most one point $I^*(\vf^*)$ 
such that $\nabla_I\HH(\vf^*,I^*(\vf^*))=0$. By the Implicit Function Theorem $I^*(\vf^*)$ is a smooth function. 
In particular, outside of some compact set in $T^*\T^n$ 
there are no critical points of $\nabla \HH(\vf,I)=0$. 
Due to superlinearity this compact set is 
contained in $\{\HH\le E\}$ for some large $E$. 

Let $U:\T^n\to \R$ be a potential. Consider a family $\HH_U=\HH+U$ 
of potential perturbations.  Notice that adding a potential
does not change the function $I^*:\T^n \to T^*\T^n$ and 
equations for a critical point $\nabla \HH_U(\vf,I)=0$ can be 
rewritten in the form 
\[\nabla_\vf \HH(\vf,I)|_{I=I^*(\vf)}+\nabla U(\vf)=0.
\] 
Generically in $U$ this function as a function of $\vf$ has 
isolated solutions to this equation. These solutions lift 
to critical points of $\HH_U$ by means of the function $I^*$. 
Once critical points of $\HH_U$ ae isolated an arbitrary small 
localized potential perturbation $\HH'_U=\HH_U+U'$ of $\HH_U$ 
we can assure that $\HH'_U$ is Morse or, equivalently, 
all critical points are in $\{\HH'_U\le E\}$ and nondegenerate. 
In addition, critical points have pairwise distinct values. 
\end{proof}

In \cite{ContrerasIPP98} the following definition of $\al_\HH(0)$ 
is given. Let $\om$ be a smooth one form on $\TT^n$. It is 
a section of the bundle $\TT^n\to T^*\TT^n$. Let $G_f \subset  T^*\TT^n$ 
be the graph of the differential $df$ of the smooth function $f$. It is well 
known that $G_f$ is a Lagrangian submanifold of $T^*\TT^n$.  Define 
\[
\al_\HH(0)=\inf_{E\in \R} \HH^{-1}(-\infty,E) \text{ contains } 
G_f \text{ for all }\CCC^\infty \text{ function }f.
\] 
This definition is equivalent to the one of Mather as proven in \cite{ContrerasIPP98}. 

In \cite{ContrerasIPP98} it is also shown that if $E > \al_\HH(0)$,  
the dynamics of the Hamiltonian flow of $\HH$ restricted to 
the energy surface $\{\HH=E\}$ is a reparametrization of the geodesic 
flow on the unit tangent bundle of an appropriately chosen Finsler 
metric on $\TT^n$. Below we define the corresponding Finsler metric 
for $E> \al_\HH(0)$ and describe this relation. 

\blm \label{lm:mane-critical-value-critical-point}
 Let $\HH$ be a Tonelli Hamiltonian and a Morse function. 
Then $\al_\HH(0)$ is a critical value of $\HH$. Moreover, for 
$n=2$ after a possible potential perturbation
%, keeping $\HH$ being a Morse function, we have that 
the corresponding unique critical point $(\vf^*,I^*)$ is 
a saddle of the Hamiltonian flow and has real eigenvalues. 
\elm 

Notice that a critical point of a Hamiltonian $\HH$ 
corresponds to a fixed point of the Hamiltonian flow of $\HH$. 
Thus, this Lemma motivates the following definition: 
Let $\HH:T^*\T^n\to \R$ be a Tonelli Hamiltonian,
then a fixed point $(\vf^*,I^*)$ is called minimal 
if 
\be \label{eq:min-crit-point}
\HH(\vf^*,I^*)=\al_\HH(0). 
\ee

Since a Ma$\tilde{\mbox n}\acute{\mbox e}$  $\CCC^r$ 
generic Hamiltonian $\HH$ is a Morse function, each 
critical value corresponds to a unique critical point. 
For $n=2$ we also have that this critical point 
is a saddle fixed point.

\begin{proof}
Notice that $E^*=\al(0)$ cannot be a regular value of $\HH$ as the fact that 
$\HH^{-1}(-\infty,E^*)$  contains the graph $G_f$, implies that 
$\HH^{-1}(-\infty,E^*-\dt)$ also contains $G_{f'}$ for small enough $\dt>0$ 
and $f'$ close to $f$.  

Since critical points have pairwise distinct values, on 
the critical energy surface $\{\HH=E^*\}$ we have exactly one non-degenerate 
critical point. Show that $\{\HH=E^*\}$ has a saddle critical point. 

Let $G_f \subset \HH^{-1}(-\infty,E^*)$ be a graph of the differential $df$ 
for some smooth function $f$. The projection $\pi \HH^{-1}(-\infty,E^*-\dt)$ 
cannot contain $\T^n$, otherwise, we can deform $f$ to $f'$ and have 
$G_{f'}\subset \HH^{-1}(-\infty,E^*-\dt)$. 

Consider a neighborhood $V$ of the only critical point $(\vf^*,I^*)\in \T^*\T^n$ with $\HH(\vf^*,I^*)=E$. Denote by $\pi_n$ the natural projection 
of $\T^*\T^n \to \T^n$. By Morse lemma (see e.g. \cite{Milnor63}) in local coordinates $(x_1,\dots,x_{2n})=\Phi(\vf,I)$ for some $a$ we have 
%\small 
\[
\HH\circ \Phi^{-1} (x)=a+\sum_{i=1}^k x_i^2 - \sum_{i={k+1}}^{2n} x_i^2.
\]
Thus, $\pi_n$ is a smooth projection onto 
a $n$-dimensional surface $\Phi\left((\T^n\times 0) \cap V\right)$ 
given by $\Phi \pi_n \circ \Phi^{-1}$. Denote by $K_L$ the kernel 
of $\pi_m(0)$. Perturbing $\HH$, if necessary, make 
the subspaces $x_1\R,\dots,x_{2n}\R$ transverse to $K_L$ and, 
therefore, transverse to the kerner of $\pi_n(x)$ for all $x\in U=\Phi(V)$, 
i.e. near $0$. Due to fiber convexity of $\HH$ we have that for all $x\in U$ 
the set $\{\HH^{-1}(a')\cap \T^*_x\T^n\}$ is convex (if nonempty). 
We know that for small $\dt>0$ 
\begin{itemize}
\item the projection of $\pi_n(\HH^{-1}(a+\dt))$ contains $U$, 
\item the projection of $\pi_n(\HH^{-1}(a-\dt))$ does not contain $U$. 
\end{itemize}
%For $k=2n$ the first item fails, for $k=0$ the second. 
If $k<n$ convexity in the fiber $K_L$ fails. 
If $k>n$, $\pi_n \HH^{-1}(a+\dt)$ does not contain $U$. 

In the cases $k=n$ the energy surface $\{\HH=a\}$ locally consists of 
two transverse $n$-dimensional surfaces, that are invariant. If $n=2$, 
then linearizing the Hamiltonian vector field at $0$ we have two invariant 
subspaces of dimension two. Consider the eigenvalues of the linearization. Since the vector field is Hamiltonian, each egenvalue $\lb$ has 
a pair $\lb^{-1}$. The linearization is real, thus, all
complex eigenvalues come in pairs: $\lb\notin \R$ and its conjugate 
$\bar \lb$. Therefore, if eigenvalues are away from the unit cirle, 
they form quadruples. 

By a perturbation move eigenvalues away from $1$. 
If a pair of complex conjugate eigenvalues are on the unit circle, 
then one pair of conjugate variables is the sum of squares and enters 
into the Hamiltonian with one sign, while the other pair has 
the opposite sign to keep $k=2$. This contradicts fiber convexity. 
Therefore, eigenvalues should be away from the unit circle. 
If some nonreal eigenvalues are away from the unit circle, they 
come in quadruples. In this case there is no pair of invariant 
$2$-dimensional subspaces. 
Thus, $(\vf^*,I^*)$ is a saddle with real eigenvalues. 
\end{proof}
 
Denote by $E_{min}=\al(0)$. Consider 
a mechanical system $\HH(\vf,I)=K(I)-U(\vf)$, where $K(I)$ is a positive 
definite quadratic form and $U(\vf)$ is a smooth potential on $\T^2$. 
One can associate $\HH$ the Jacobi metric $g_E(\vf)=2\sqrt{E+U(\vf)}K$.
Let $E>E_{min}$. According to the Mapertuis principle the orbits of $\HH$ 
on the energy surface $\{\HH=E\}$ are time reparametrization of geodesics
of $g_E$. In \cite{KaloshinZ12} as a slow system in  a double resonance we study mechanical 
systems of the above form $\HH=K-U$ and state all non-degeneracy 
conditions on the Hamiltonian $\HH$ using Mapertuis principle.

In our case we cannot assure that the slow system $\HH$ is a mechanical system. 
However, for a Tonelli Hamiltonian $\HH$ there is a generalized Mapertuis
principle, described below.

\subsection{A generalized Mapertuis principle}\label{sec:gen-mapertuis}

Let $L$ be a Lagrangian given as a Legendre transform of $\HH$. Given 
any $c \in H^1(\T^n, \R) \cong \R^n$, define
\[
\al(c)=−\inf\int(L-c\cdot I)\, d\mu(\vf,I), 
\]
where the infimum is taken over all probability measures $\mu$ in $T\T^n$.
It is called Mather's $\al$-function. $\al(c)$ is the average 
of the Hamiltonian on the support of the $c$-minimal measures.

Denote by $L_c(\vf,v)=L(\vf,v)-c\cdot v$ for any $c\in \R^n$. 
Let $\HH_c$ be the Legendre-Fenchel dual of $L_c(\vf,v)$, then 
it has the form $\HH(\vf,I+c)$. For any $t>0$, $x,y\in \T^n$ 
and $c\in \R^n$ we introduce the following quantity: 
\[
h^t_c(x,y)=\int_0^t L_c(\xi(s),\dot \xi(s))\,ds,
\] 
where the infimum is taken over of the piecewise $\CCC^1$ curve 
$\xi : [0,t] \to  \T^n$ such that $\xi(0) = x$ and $\xi(t) = y$.
It is well know that we can define the Ma$\tilde{\mbox n}\acute{\mbox e}$e's critical potential 
and Peierls’ barrier respectively as follows 
 \[
\begin{split}
\phi_c(x, y) &= \inf_{t>0} h_c^t (x, y) + \al(c)\,t\\ 
h_c(x, y) &= \liminf_{t\to \infty} h_c^t (x, y) + \al(c)\,t.
\end{split}
\]
For given $c \in \R^n$, define the projected Aubry set
\[
\mathcal A_c=\{x\in \T^n:\ h_c(x,x)=0\}. 
\]
In order to define a Finsler metric associate to the flow of $\HH$ 
on the energy surface $\{\HH=E\}$ we need the following 
definition. We use the variant approach in \cite{FathiS04}, where only the case 
$c = 0$ was considered. For any fixed $x \in \T^n$ , define the sublevel set
\[
\bar Z _c(x) = \{I \in \R^n : \HH _c(x,I) \le \al(c)\},
\ \textup{ for }\ c \in \R^n, \ \ \textup{ and }
\]
\[
\dt_c(x,v) = \sigma_{\bar Z_c(x)}(v),\ x \in \T^n,\ v \in \R^n,\qquad 
\]
where $\sigma_C(v)$ is the support function with respect to the compact 
convex set $C$, i.e.,
\[
\sigma_C(v)=\max\{\langle x,v \rangle :x\in C\}, v\in \R^n.
\]
Denote by $\sigma_{\al(c)}(x,v)=\sigma_{\bar Z_c(x)}(v)$, 
where $Z_{\al(c)}(x) = \{ I \in \R^n : \HH(x,I) \le \al(c)\}.$

For any $x,y \in \T^n$, define
\[
S_c(x, y) = \inf \int_0^1 \dt_c (\xi(t), \dot \xi(t))dt = 
\inf \int_0^1 (\sigma_c (\xi(t), \dot \xi(t)) - c \dot \xi(t)) dt, 
\]
where the infimum is taken over of the piecewise $\CCC^1$ curve 
$\xi : [0,t] \to  \T^n$ such that $\xi(0) = x$ and $\xi(t) = y$. 

It is natural to call the quantity $\dt_c(x,v)$ --- 
the Jacobi--Finsler metric for the generalized Maupertuis’ principle 
with the restriction of the energy $\al(c)$. 
If the kinetic energy function is of the form of Riemannian metric 
$g_x(v,v) = \langle v,v \rangle_x,\  \dt_0(x,v)$ is the usual Jacobi metric
$\sqrt{ E-U(x)}$, where $E=\al(c)>\min_c \al(c)$. In \cite{Cheng13} the following useful statement if proven.  

\bthm The Finsler length of $\dt_c(x,v)$ and action $L_c+\al(c)$
coincide, i.e.  
\[\phi_c(x,y)=S_c(x,y).
\] 
\ethm 
It follows from Lemma \ref{lm:mane-critical-value-critical-point} that 
\bcor \label{generic-critical-point-critical-energy}
Let $n=2$ and $\HH(\vf,I)$ be a Tonelli Hamiltonian, then the property 
that $\HH$ is a Morse function %, $\al(0)$ is a critical value, 
and the only critical point   $(\vf^*,I^*)\in \HH^{-1}(\al(0))$ 
of the Hamiltonian flow associated to $\HH$ is a saddle with real eigenvalues
is Ma$\tilde{\mbox n}\acute{\mbox e}$ generic. 

For Tonelli Hamiltonians satisfying this property the associated 
Jacobi-Finsler metric $\dt_0$ has exactly one critical point at $\vf^*$.  
\ecor

\subsection{Proof of Lemma \ref{lemma:SelfIntersectingGeodesics}}

\begin{proof} We use a proof from unpublished notes by Kaloshin and Zhang.  
It is easy to see that all the properties are open. To prove density, we show that by an arbitrarily small perturbation by a potential, the system satisfies 
condition of this lemma.  

Let $S\subset \T^3$ be a smooth global section not containing $m_0$, and intersecting every curve in homology $\bar{h}$. Lift $S$ and the metric $\dt_0$ to the universal cover. Denote by $l_\dt(\gamma)$
$\dt_0$-length of $\gamma$. Consider the variational problem 
	$$ \min_{\vf\in S}l_\dt(\gamma), \quad \gamma(0) = \vf, \gamma(T) = x + \bar{h}. $$
Any minimizer $\gamma_0$ projects to a shortest geodesic of homology $\bar{h}$. Perturbing the section if necessary, we may assume at least one of the minimizer intersect $S$ transversally, i.e. $\gamma_0 \cap S$ at isolated points. Let $\vf_0$ be one such point, for $\delta>0$, let $U_0(\vf)$ be a function satisfying: 
\begin{itemize}
	\item $U(\vf_0) = 0$ and $U(\vf) \ge 0$ for all $\vf \ne \vf_0$.
	\item $U(\vf) = \mathrm{const}$ for $\vf\notin B_\delta(\vf_0)$ and 
$\|U\|_{\CCC^r} = o(1)$ as $\delta\to 0$. 
\end{itemize}
Consider the new potential function $U + U_0$, for the new variational problem there exists a unique minimizer intersecting $S$ at $\vf_0$. 

Since $\gamma_0$ is a geodesic outside of the critical point $m_0$, it has 
no self-intersection outside of $m_0$. 

If $m_0 \notin \gamma_0$, then $\gamma_0$ must be a simple closed curve of homology $\bar{h}$. 

If $m_0 \in \gamma_0$ and $\gamma_0$ is not a simple closed curve, 
then it must be concatenation of simple closed curves, intersecting at 
$m_0$. We assume that for each integer homology class $g$ with
$l_\dt(g) \le l_{\dt}(\bar{h})$ the minimizer is unique. This is open and 
dense from the earlier argument. Then the components of $\gamma_0$ 
must have different homology classes. 

We claim: for a collection of simple curves with different homologies in 
$\T^2$, to avoid intersection at more than one point, the only possibility 
is the following (This is similar to the discussion of Mather in \cite{Mather08}): 
\begin{itemize}
	\item There are  two curves $\gamma_1$ and $\gamma_2$ with homologies $h_1$ and $h_2$ generating $\Z^2$, and if there is a third 
simple curve, it must have homology $\pm (h_1 + h_2)$. 
\end{itemize}
 To see this, make a (integer linear) change of basis such that $h_1 = (0,1)$. It's easy to see that unless $h_2 = \pm (1,n)$, the curves $\gamma_1, \gamma_2$ intersect at two points. $h_1, h_2$ clearly form a basis. Now 
$\T^2 \setminus \gamma_1\cup \gamma_2$ is an open rectangle with $m_0$ at four vertices, and the only room left is the diagonal, i.e. a simple curve of homology $\pm (h_1 + h_2)$. 

 We now make an additional perturbation, such that 
\[
l_{\dt}(h_1) + l_{\dt}(h_2) < l_{\dt}(h_1+h_2),
\] 
the system now fall into the non-simple case as given in the condition 
of the Lemma. 
\end{proof}

\section{Weak KAM theory, the Aubry and Ma$\tilde{\mbox n}\acute{\mbox e}$ sets}\label{app:WeakKAM}
In this appendix we review the concepts we need from Mather Theory \cite{Mather91, Mather91a, Mather93}. We use the approach developed by Fathi \cite{FathiBook} and Bernard \cite{Bernard10}. 
%This short review can be found in \cite{KaloshinZ12}. 

Let $(M,g)$ be a smooth compact Riemannian manifold and consider a $\CCC^2$ Hamiltonian $H:T^*M\times\RR\longrightarrow \RR$. We denote $(\varphi,I)$ the points in $T^*M$,  $\Psi_s^t(\varphi,I)$ the time $(t-s)$-map of the Hamiltonian vector field of $H$ with initial time $s$ and $\Psi^t=\Psi_0^t$. Assume the following hypotheses:
\begin{itemize}
\item {\it Periodicity:} $H(\varphi,I,t+1)=H(\varphi,I,t)$ for each $(\varphi,I)\in T^*M$ and $t\in\RR$.
\item {\it Completeness:} the Hamiltonian vector field of $H$ generates a complete flow.
\item {\it Convexity:} for each $(\varphi,t)\in M\times\TT$,  $H(\varphi,I,t)$ is strictly convex on $T_\varphi^*M$.
\item {\it Super-linearity:} for each $(\varphi,t)\in M\times\TT$, the function  $H(\varphi,I,t)$ is super-linear in $I$, i. e. $\lim_{|I|\rightarrow+\infty}H(\varphi,I,t)/|I|\rightarrow\pm\infty$.
\end{itemize}
We call the the Hamiltonians satisfying these hypotheses \emph{Tonelli Hamiltonians}. To each Tonelli Hamiltonian $H$ we associate the Lagrangian given by the Legendre transform
\[
 L(\varphi,v,t)=\sup_{I\in T^*xM}\left\{v\cdot I-H(\varphi,I,t)\right\},
\]
The Legendre transform is an involution, that is, the Legendre transform of $L$ is $H$. The Legendre transform defines a diffeomorphism
\[
\LL:(\varphi,I,t)\longrightarrow (\varphi,v,t)=(\varphi,\pa_IH(\varphi,I,t),t),
\]
whose inverse is given by
\[
\LL\ii:(\varphi,v,t)\longrightarrow (\varphi,I,t)=(\varphi,\pa_vL(\varphi,v,t),t).
\]
The Hamiltonian flow is mapped to the Euler-Lagrange flow given by $L$. The Lagrangian $L$ satisfies properties analogous to the ones stated for $H$: periodicity, fiber convexity on $T_\varphi M$ and superlinear growth as $|v|\rightarrow+\infty$.

\subsection{Overlapping pseudographs}
Given $K>0$, a function $f:M\longrightarrow\RR$ is called a $K$-semi-concave function if for any $\varphi\in M$, there exists $I_\varphi\in T_*M$, such that for each chart $\psi$ at $x$, we have
\[
 f\circ\psi(y)-f\circ\psi(x)\leq I_\varphi (d\psi_x(y-x))+K\|y-x\|^2
\]
for all $y$. The linear form $I_\varphi$ is called a super-differential at $\varphi$. A function $f$ is called semi-concave if it is $K$-semi-concave for some $K>0$. A semi-concave function is Lipschitz. 

Given a Lipschitz function $u:M\longrightarrow \RR$ and a closed smooth form $\eta$ on $M$, one can consider the subset $\cG_{\eta,u}\subset T^*M$ defined by
\[
 \cG_{\eta, u}=\left\{(x,\eta_x+du_x):x\in M \text{ such that }du_x \text{ exists}\right\}.
\]
We call a set $\cG\subset T^*M$ an \emph{overlapping pseudograph}  if there exist a closed one form $\eta$ and a semi-concave function $u$ such that $\cG=\cG_{\eta,u}$. The overlapping pseudographs are well suited to study unstable invariant manifolds. 
%To describe stable manifolds, one can define the anti-overlapping pseudographs
%\[
% \breve\cG_{\eta, u}=\left\{(x,\eta_x-du_x):x\in M \text{ such that }du_x \text{ exists}\right\}.
%\]
If $\cG$ can be represented in two different ways $\cG_{\eta,u}$ and $\cG_{\eta',u'}$, $\eta$ and $\eta'$ must belong to the same cohomology class. Therefore, each pseudograph $\cG$ has a well defined cohomology class $c=c(\cG)=c(\cG_{\eta,u})=[\eta]$, where $[\eta]$ is the De Rahm cohomology of $\eta$. The function $u$ is then uniquely defined up to an additive constant. Thus, if we denote by $\SS$ the set of semi-concave functions on $M$, and by $\PP$ the set of overlapping pseudographs we have the identification
\[
 \PP=H^1(M,\RR)\times \SS/\RR.
\]
We denote by $\PP_c$, the set of overlapping pseudographs with cohomology $c$. 
%We define in the same way the set $\breve \PP$ for anti-overlapping pseudographs $\breve \cG$.

\subsection{The Lax-Oleinik semi-group}
Denote by $\Sigma(t,\varphi_1;s,\varphi_2)$ the set of absolutely continuous curves $\ga:[t,s]\longrightarrow M$ such that $\ga(t)=\varphi_1$ and $\ga(s)=\varphi_2$. Denote by $d\ga(\tau)=(\ga(\tau),\dot\ga(\tau),\tau)$ the one-jet of $\ga$. It is well defined for almost every $\tau$. Given a a closed form $\eta$ and a cohomology class $c\in H^1(M,\RR)$ we define 
\begin{equation}\label{def:LagrangianMinusForm}
 L_c(d\ga(\tau))=L(d\ga(\tau))-c(\dot\ga(\tau)),\,\,\,L_\eta(d\ga(\tau))=L(d\ga(\tau))-\eta(\dot\ga(\tau))
\end{equation}
The Lagrangians $L_c$ and $L_\eta$ have the same Euler-Lagrange flow as $L$. 

We define the Lax-Oleinik operator on $\CCC^0(M,\RR)$,
\[
 T_\eta u(\varphi)=\min_{\varphi\in M, \ga\in\Sigma(0,\theta;1,\varphi)}\left(u(\theta)+\int_0^1L_\eta(d\ga(\tau))\,d\tau\right).
\]
The operator $T_\eta$ is contracting, that is $\|T_\eta u-T_\eta v\|_\infty\leq \|u-v\|_\infty$.

It turns out that there exists a unique map $\Phi:\PP\longrightarrow\PP$ such that 
\[
 \Phi(\cG_{\eta,u})=\cG_{\eta,T_\eta u}
\]
for all smooth forms $\eta$ and all semi-concave functions $u$. Moreover, $c(\Phi(\cG))=c(\cG)$. The map $\Phi$ is continuous and preserves $\PP_c$ for each $c\in H^1(M,\RR)$. Fathi proved that $\Phi$ has 
a fixed point in each $\PP_c$ (see e.g. \cite{FathiBook}). Denote 
by $\V_c$ the set of these fixed points and by $\V=\cup_c \V_c$ 
their union. 

\subsection{Mather, Aubry and Ma$\tilde{\mbox n}\acute{\mbox e}$ sets}
The set of fixed points $\V_c$ satisfies 
\[
 \ol\cG\subset\Psi(\cG)
\]
where $\Psi=\Psi^1$ is the time one map of the Hamiltonian flow. Then, we can define the invariant compact sets
\[
 \wt\II(\cG)=\bigcap_{n\geq 0}\Psi^{-n}(\cG).
\]
We use the notation $\wt I(u,c)$, where $\cG=\cG_{\eta,u}$ and $c=[\eta]$. Using these invariant sets, we define the Mather, Aubry and Ma$\tilde{\mbox n}\acute{\mbox e}$ sets. To each cohomology class $c\in H^1(M,\RR)$, we associate the non-empty invariant sets
\[
 \wt\MM(c)\subset\wt\AAA(c)\subset\wt\NNN(c),
\]
where $\wt\AAA(c)=\cap_{\cG\in\V_c}\wt\II (\cG)$ and  $\wt\NNN(c)=\cup_{\cG\in\V_c}\wt\II (\cG)$ are the Aubry and Ma$\tilde{\mbox n}\acute{\mbox e}$ sets. The Mather sets $\wt\MM(c)$ is the union of the supports of the invariant 
measures of the Hamiltonian flow $\Psi^t$ on 
$\wt\AAA(c)$. In the case dependence of the Hamiltonian 
$H$ is essential use subindex, i.e. $\wt\MM_H(c),
\wt\AAA_H(c),\wt\NNN_H(c).$

We define also the Mather alpha function. 
%We define it through 
%the following proposition in \cite{FathiBook}.
Among several equivalent ways, we use the definition 
in \cite{Sorrentino11}. Let $\mathfrak{M}(L)$ be the space of 
probability measures on $TM$ that are invariant under
the Euler-Lagrange flow of $L$ and such that $\int_{TM}Ld\mu<\infty$. It can be seen that this set is non-empty. Recall that the Hamiltonian $L_c$ where $c$ is a cohomology class has been defined in \eqref{def:LagrangianMinusForm} and it has the same Euler-Lagrange flow as $L$. This implies $\mathfrak{M}(L_c)=\mathfrak{M}(L)$. Then, we define the $\al$ function as
%\begin{equation}
 \begin{eqnarray}\label{def:alpha_function}
 \al:\quad  H^1(M,\RR)  \longrightarrow \RR,\qquad 
                c  \mapsto \min_{\mu\in \mathfrak{M}(L)}\int_{TM}L_c \,d\mu.
 \end{eqnarray}
Note that the action of $L_c$ does not depend on the chosen representative in the cohomology class and therefore, the $\al$ function is well defined. It can be seen that $\al(c)$ is convex and super-linear. 
%\end{equation}

%\begin{proposition}
%There exists a function $\al:H^1(M,\RR)\longrightarrow\RR$  such that for each continuous function $u$ and each closed one form $\eta$ of cohomology $c$, the sequence $T_\eta^nu(x)+n\al(c)$, $n\geq 1$ of continuous functions is equi-bounded and equi-Lipschitz. The function $c\mapsto\al(c)$ is convex and super-linear. More precisely, there exists a constant $K(c)$, which does not depend on the continuous function $u$, such that
%\[
% \min u -K(c)\leq T_\eta^n u(x)+n\al(c)\leq \max u+K(c)
%\]
%for each positive integer $n$ and each point $x\in M$.
%\end{proposition}

We also define the Peierls' barrier. For $n\in \NN$, consider the function $h^n_L: M 
\times M \to \R$ by 
\[
	h^n_c(\varphi, \psi) = \min_{\gamma(0)= \varphi, \gamma(n) = \psi}\int_0^n (L(\gamma , \dot{\gamma},t)-c\cdot \dot\gamma + 
\alpha(c)) dt. 
\]
The Peierls' barrier is 
\be \label{def:barrier}
h_c(\varphi, \psi) = 
\liminf_{n \to \infty} h_c^n(\varphi, \psi).
\ee
The limit exists, and the function $h_c$ is 
Lipschitz in both variables.

\subsection{Forcing relation}
Following \cite{Bernard10} we define the forcing relation $\dashv\vdash$. Given two pseudographs $\cG$ and $\cG'$ we define the relation $\cG\vdash_N\cG'$ as
\[
 \cG\vdash_N\cG' \quad \Leftrightarrow \quad\ol{\cG'}\subset\cup_{n=1}^N \Psi^n(\cG).
\]
We say that $\cG$ forces $\cG'$, and we write $\cG\vdash\cG'$, if there exists $N\in\NN$ such that $\cG\vdash_N\cG'$. If $\cG$ is a subset of $T^*M$ and $c\in H^1(M,\RR)$, the relations $\cG\vdash c$ and $\cG\vdash_N c$ mean that there exists an overlapping pseudograph $\cG'$ of cohomology $c$ such that $\cG\vdash\cG'$ (respectively $\cG\vdash_N\cG'$). Finally, for two cohomology classes $c,c'\in H^1(M,\RR)$, the relation $c\vdash_Nc'$ means that for each pseudograph $\cG$ with cohomology $c$, we have $\cG\vdash_N c'$. The relation $\vdash$  is transitive either between overlapping pseudographs or cohomology classes. We introduce the symmetric relation 
\[
c\dashv\vdash c' \quad\Leftrightarrow\quad c\vdash c'\,\,\text{ and }\,\,c'\vdash c.
\]
We say that $c$ and $c'$ force each other. In \cite{Bernard10} it is shown that the forcing 
relation $\dashv\vdash$ is an equivalence relation (see Proposition 5.1 there).

We use the forcing relation to drift along the cylinders. 
We use the following result from \cite{Bernard10}
\begin{proposition}[Proposition 0.10 in \cite{Bernard10}] 
\label{prop:c-equiv}
(i) Let $\cG$ an $\cG'$ be two Lagrangian graphs of 
cohomologies $c,c'\in H^1(M,\RR)$. If $c\dashv\vdash c'$, 
then there exists an integer time $n$ such that $\Psi^n(\cG)$ intersects $\cG'$.

(ii) If $c\dashv\vdash c'$, there exist two heteroclinic trajectories of the Hamiltonian flow from $\wt\AAA(c)$ to $\wt\AAA(c')$ and from $\wt\AAA(c')$ to $\wt\AAA(c)$.

(iii) Let $\{c_i\}_{i\in\ZZ}$ be a sequence of cohomology classes all of which force the others. Fix for each $i$ a neighborhood $U_i$  of the corresponding Mather set $\wt\MM(c_i)$ in $T^*M$. There exists a trajectory of the Hamiltonian flow $\Psi^t$ which visits in turn all the sets $U_i$. In addition, if the sequence stabilizes to $c_-$ on the left, or (and) to $c_+$ on the right, the trajectory can be selected to be negatively asymptotic to $\wt \AAA(c_-)$ or (and) asymptotic to $\wt\AAA(c_+)$.
 
\end{proposition}

\section{Notations}\label{sec:Notations}
\subsection{Fixed constants}
\begin{small}
\begin{itemize}
\item $d$ -- measures hyperbolicity of NHICs $\CCC_{k_n}^{\om_n,i}$ along single resonances. We take $d=14$.
\item $q$ -- measures the smallness of the remainder of the single resonance normal 
form $\RRR^{k_n}$. We take $q=18d=252$. 
\item $\theta$ -- measures the amount of double resonances considered. We choose it $\theta=3q+1=757$.
\item $m$ -- measures the radius of the core of double resonances $\Cor_n(k_n,k')$. We choose $m=\theta+1=758$.
\item $r$ -- regularity of the original Hamiltonian $H_0+\eps H_1$. It satisfies $r\geq m+5q=2018$.
\end{itemize}
\end{small}
\subsection{Other notations}
\begin{small}
\begin{itemize}
\item $H_0+\eps H_1$ -- the original nearly integrable Hamiltonian.
\item $D$ -- constant measuring the convexity of $H_0$: $D^{-1}\|v\| \le \langle \partial^2 h_0(I) v,v\rangle \le D\|v\|$.
\item $\DDD_{\eta,\tau}$ -- the set of diophantine frequencies
satisfying $|\om\cdot k|\geq \eta |k|^{-(2+\tau)}$, 
with diophantine exponent $\tau$ and diophantine constant $\eta$.
\item $\DDD_{\eta,\tau}^U=\DDD_{\eta,\tau}\cap U$
\item $\KAM_{\eta,\tau}^U$ -- KAM tori of $H_0+\eps H_1$ with frequency in $\DDD_{\eta,\tau}^U$.
\item $\{\rr_n\}_{n\geq 1}$, $\rr_{n+1}=\rr_n^{1+2\tau}$ -- sequence measuring the density of resonances at generation $n$.
\item $\DDD_{\eta,\tau}^n\subset \DDD_{\eta,\tau}^U$ -- grid of Diophantine frequencies which are $3\rr_n$ dense in $ \DDD_{\eta,\tau}^U$ such that no two points in $\DDD_{\eta,\tau}^n$ are closer than $\rr_n$. %They are defined in Section \ref{sec:Outline}.
\item $\Ga_k$ -- resonance line in frequency space defined by $\om\cdot k=0$.
\item $\{k_n\}$ -- integer vectors whose associated resonance $\Ga_{k_n}$ intersects $B_{3\rr_n}(\om_n)$ of the diophantine frequency $\om_n\in \DDD_{\eta,\tau}^n$. They depend on $\om_n$ even if not explicitly stated.
\item $\{R_n\}_{n\geq 1}$, $\rr_n=R_n^{-3+5\tau}$, $R_{n+1}=R_n^{1+2\tau}$ --  sequence dual to $\{\rr_n\}_{n\geq 1}$. It will measure size of $\Ga_{k_n}$ which are $\rr_n$ dense.
\item $\SSS_k\subset\Ga_k$ -- resonant segments of the zero generation.
\item $\SSS_{k_n}^{\om_n}\subset\Ga_k$ --  resonant segment along which we obtain drifting orbits.
\item $\bS_n$ -- union of resonant segments $\SSS_{k_n}^{\om_n}\subset\Ga_k$ of generation $n$.
\item $\bS=\{\bS_n\}_{n\geq 0}$ -- the union of all resonant segments.
\item $H_\eps'=H_0+\eps H_1+\eps\Delta H^\mol$ -- mollification of the Hamiltonian $H_0+\eps H_1$.
\item $\Phi_\Pos$ -- canonical transformation which gives the P\"oschel normal form, see Theorem \ref{thm:Poschel}.
\item $N'=H_\eps\circ \Phi_\Pos=H_0'+\eps R$ -- P\"oschel normal form.
\item $\Omega=\pa_I H_0'$ -- the frequency map
\item $I_n=\Om\ii (\om_n)$ -- action associated by 
the  map $\Om$ to a Diophantine frequency in  $\DDD_{\eta,\tau}^n$.
\item $\II_{k}=\Om\ii(\SSS_{k})$ -- resonant segment of the first generation in action space.
\item $\II_{k_n}^{\om_n}=\Om\ii(\SSS_{k_n}^{\om_n})$ -- dirichlet resonant segment in action space.
\item $\Delta H^\sr$ -- deformation of $N'$ to attain non-degeneracy along single resonances.
\item $\Delta H^\dr$ -- deformation of $N'$ to attain non-degeneracy in double resonances.
\item $\Delta H^\shad$ -- deformation of $N'$ to attain non-degeneracy which allows shadowing.
\item $\Tr_n (k_n, k',k'')$ -- single resonance zone along $\SSS_{k_n}^{\om_n}$ between the punctures given by $\Ga_{k'} \,\& \,\Ga_{k''}$.
\item $\Cor_n (k_n, k')$ -- core of the double resonance around $\Ga_{k_n}\cap \Ga_{k'}$.
\item $\HH^{k_n}=\HH_0^{k_n}+\ZZZ^{k_n}+\RRR^{k_n}$ -- normal form along single resonances, see Key Theorem \ref{keythm:Transition:NormalForm}.
\item $\Phi_n=\Phi^{k_n}$ -- canonical transformation which gives the single resonance normal form.
\item $\CCC_{k_n}^{\om_n,i}$, $i=1,\ldots, N$ -- normally hyperbolic invariant cylinders in the transition zone $\Tr_n (k_n, k',k'')$. It is given in Key Theorem \ref{keythm:Transition:IsolatingBlock}.
\item $\HH^{k_n,k'}=\HH_0^{k_n,k'}+\ZZZ^{k_n,k'}$ -- normal form in the core of double resonance $\Cor_n (k_n, k')$, see Key Theorem \ref{keythm:CoreDR:NormalForm}.
\item $\de_E$ -- Finsler metric associated to a Hamiltonian $\HH=\HH_0+\ZZZ$ at the energy $\{\HH=E\}$.
\item $\ell_E$ -- length associated to the Finsler metric $\de_E$.
\item $\ga_h^E$ -- closed geodesic of homology $h\in H^1(\TT^2,\ZZ)$ given by the length $\ell_E$.
\item $\MM_h^{E_j,E_{j+1}}$ -- NHICs made by the union of geodesics $\ga_h^E$ with $E\in [E_j-\de,E_{j+1}+\de]$.
\end{itemize}
\end{small}

\bibliography{references}
\bibliographystyle{alpha}
\end{document}